%% file: universal-markov-partitions-v5.tex
\documentclass[11pt,reqno]{amsart}
\pdfoutput=1

\usepackage{amsmath}
\usepackage{amsfonts}
\usepackage{amssymb}
\usepackage{enumerate}
\usepackage{amstext}
\usepackage{amsbsy}
\usepackage{amsopn}
\usepackage{bbm,amsthm}
\usepackage{amscd}
\usepackage[pdftex]{color}
\usepackage{amsxtra}
\usepackage{upref}
\usepackage{epstopdf}
\usepackage{graphicx,color}
\usepackage{hyperref}
\usepackage{longtable}
\usepackage{graphicx}
\usepackage[francais, english]{babel}
\usepackage{soul}
\usepackage{array}
\usepackage{todonotes}

\DeclareFontFamily{OML}{rsfs}{\skewchar\font'177}
\DeclareFontShape{OML}{rsfs}{m}{n}{ <5> <6> rsfs5 <7> <8> <9> rsfs7
  <10> <10.95> <12> <14.4> <17.28> <20.74> <24.88> rsfs10 }{}
\DeclareMathAlphabet{\mathfs}{OML}{rsfs}{m}{n}

\newtheorem{theorem}{Theorem}
\newtheorem*{maintheorem}{Main Theorem}

\newtheorem{lemma}[theorem]{Lemma}
\newtheorem{proposition}[theorem]{Proposition}
\newtheorem{corollary}[theorem]{Corollary}

\theoremstyle{definition}

\theoremstyle{remark}
\newtheorem{remark}[theorem]{\bf Remark}

\numberwithin{equation}{section}
\numberwithin{theorem}{section}

\newcommand{\intav}[1]{\mathchoice {\mathop{\vrule width 6pt height 3 pt depth  -2.5pt
\kern -8pt \intop}\nolimits_{\kern -6pt#1}} {\mathop{\vrule width
5pt height 3  pt depth -2.6pt \kern -6pt \intop}\nolimits_{#1}}
{\mathop{\vrule width 5pt height 3 pt depth -2.6pt \kern -6pt
\intop}\nolimits_{#1}} {\mathop{\vrule width 5pt height 3 pt depth
-2.6pt \kern -6pt \intop}\nolimits_{#1}}}

\newcommand{\intavl}[1]{\mathchoice {\mathop{\vrule width 6pt height 3 pt depth  -2.5pt
\kern -8pt \intop}\limits_{\kern -6pt#1}} {\mathop{\vrule width 5pt
height 3  pt depth -2.6pt \kern -6pt \intop}\nolimits_{#1}}
{\mathop{\vrule width 5pt height 3 pt depth -2.6pt \kern -6pt
\intop}\nolimits_{#1}} {\mathop{\vrule width 5pt height 3 pt depth
-2.6pt \kern -6pt \intop}\nolimits_{#1}}}

\newcommand{\un}{\underline}

\newcommand{\ve}{\varepsilon}
\newcommand{\wt}{\widetilde}
\newcommand{\vf}{\varphi}

\newcommand{\R}{\mathbb{R}}
\newcommand{\N}{\mathbb{N}}

\newcommand{\Z}{\mathbb{Z}}

\newcommand{\Hol}[1]{{\textrm{H\"ol}}_{#1}}

\renewcommand{\exp}[1]{{\rm exp}_{#1}}

\newcommand{\Lip}{{\rm Lip}}
\newcommand{\Sas}{d_{\rm Sas}}
\newcommand{\inj}{{\rm inj}}
\newcommand{\nuh}{{\rm NUH}}
\newcommand{\simN}{\stackrel{N}\sim}

\newcommand{\colvec}[2][.8]{%
  \scalebox{#1}{%
    \renewcommand{\arraystretch}{.8}%
    $\begin{bmatrix}#2\end{bmatrix}$%
  }
}

\title[Symbolic dynamics of 3--dimensional flows]{Symbolic dynamics for large non-uniformly hyperbolic sets of
three dimensional flows}
\author{J\'er\^ome Buzzi, Sylvain Crovisier, and Yuri Lima}
\date{\today}
\keywords{}
\thanks{SC was partially supported by the ERC project 692925 \emph{NUHGD}. YL was also supported by CNPq 
and Instituto Serrapilheira, grant ``Jangada Din\^{a}mica: Impulsionando Sistemas Din\^{a}micos na 
Regi\~{a}o Nordeste''.}

\begin{document}
\maketitle

\begin{abstract}
We construct symbolic dynamics for three dimensional flows with positive speed. 
More precisely, for each $\chi>0$, we code a set of full measure for every invariant probability measure which is $\chi$--hyperbolic. These include all ergodic measures with entropy bigger than $\chi$ as well as all hyperbolic periodic orbits of saddle-type with Lyapunov exponent outside of $[-\chi,\chi]$.
This contrasts with a previous work of Lima \& Sarig which built a coding associated
to a given invariant probability measure~\cite{Lima-Sarig}. As an application, we code homoclinic classes of measures
by suspensions of irreducible countable Markov shifts.
\end{abstract}

\tableofcontents

\section{Introduction}

Let $M$ be a smooth closed  three dimensional manifold
and $X$ be a $C^{1+\beta}$ vector field on $M$ with $\beta>0$
which is non-singular, i.e. $X_p\neq 0$ for all $p\in M$.
We want to code a ``large" subset of $M$ with some non-uniform hyperbolicity
for the flow $\vf=\{\vf^t\}_{t\in\R}$ generated by $X$. This subset carries
all $\vf$--invariant hyperbolic ergodic probability measures with the following nonuniform hyperbolic property.
Let $\chi>0$.

\medskip
\noindent
{\sc $\chi$--hyperbolic measure:} A $\vf$--invariant probability measure $\mu$ on $M$
is {\em $\chi$--hyperbolic} if $\mu$--a.e. point has one Lyapunov exponent $>\chi$
and one Lyapunov exponent $< -\chi$.
The Lyapunov exponent along the flow, which always vanishes, is called trivial.

\medskip
This defines a rather natural, large, and uncountable class of measures. For instance, by the Ruelle inequality,
every $\vf$--invariant ergodic probability measure with metric entropy
larger than $\chi$ is $\chi$--hyperbolic. Also, every $\vf$--invariant probability measure defined by
a closed orbit with nontrivial Lyapunov exponents larger than $\chi$ in absolute value is
$\chi$--hyperbolic. In this paper, for each $\chi>0$ we construct a symbolic
system which lifts all $\chi$--hyperbolic measures.

\begin{maintheorem}\label{maintheorem}
Let $X$ be a non-singular $C^{1+\beta}$ vector field ($\beta>0$) on a closed $3$--manifold $M$.
Given $\chi>0$, there exist a locally compact topological Markov flow
$(\Sigma_r,\sigma_r)$ and a map $\pi_r:\Sigma_r\to M$ such that
$\pi_r\circ \sigma_r^t=\vf^t\circ\pi_r$, for all $t\in\R$, and satisfying:
\begin{enumerate}[{\rm (1)}]
\item The roof function $r$ and the projection $\pi_r$ are H\"older continuous.
\item $\pi_r[\Sigma_r^\#]$ has full measure for every $\chi$--hyperbolic measure on $M$.
\item $\pi_r$ is finite-to-one on $\Sigma_r^\#$, i.e. $\operatorname{Card}(\{z\in \Sigma_r^\#:\pi_r(z)=x\})<\infty$, for all $x\in \pi_r[\Sigma_r^\#]$.
\end{enumerate}
\end{maintheorem}

\noindent
A more precise version of the Main Theorem is stated in Section~\ref{s.detailed}, see
Theorem~\ref{t.main}.
\smallskip

A topological Markov flow is the unit speed vertical flow on a suspension space whose
basis is a topological Markov shift and whose roof function is continuous, everywhere positive and uniformly bounded.
We can endow $(\Sigma_r,\sigma_r)$
with a natural metric, called the {\em Bowen-Walters metric}, that makes $\sigma_r$ a continuous flow.
It is with respect to this metric that $\pi_r$ is H\"older continuous.
The set $\Sigma_r^\#$ is the {\em regular}  set of $(\Sigma_r,\sigma_r)$, consisting of all elements of
$\Sigma_r$ for which the symbolic coordinate has a symbol repeating infinitely often in the future
and a symbol repeating infinitely often in the past. See Section \ref{Section-Preliminaries} for the definitions.

The Main Theorem provides a {\em single} symbolic extension that codes all
$\chi$--hyperbolic measures at the same time, and that is finite-to-one almost everywhere.
This improves on the seminal result by
Lima \& Sarig \cite{Lima-Sarig},
whose codings depend on the choice of a measure (or a countable class of measures).
We will mention later the importance of this novelty.

In applications, it is useful to work with {\em irreducible} Markov shifts
since, among other properties, they are topologically transitive and they 
carry at most one measure of maximal entropy
(see Section~\ref{subsection-symbolic}).
This is related to the notion of homoclinically related measures and of
\emph{homoclinic classes of measures}, defined in Section~\ref{sec.homoclinic}.
In this context, we prove the following theorem.

\begin{theorem}\label{thm.homoclinic}
In the setting of the Main Theorem,
let $\mu$ be a hyperbolic ergodic measure.
Then $\Sigma_{r}$ contains an irreducible component
${\Sigma'_{r}}$ which lifts any $\chi$--hyperbolic ergodic measure $\nu$
homoclinically related to $\mu$.
\end{theorem}

This implies the following local uniqueness result for measures of maximal entropy.

\begin{corollary}\label{cor.local-uniq}
In the setting of the Main Theorem,
let $\mu$ be a hyperbolic ergodic measure.
Then there is at most one measure $\nu$ which is homoclinically related to $\mu$ and 
maximizes the entropy,  i.e. satisfies
$h(\vf,\nu)=\sup\{h(\vf,\mu):\mu \text{ is homoclinically related to $\mu$}\}$.
\end{corollary}

Results about uniqueness of the measure of maximal entropy for flows have been obtained previously under various settings, see for instance~\cite{Bowen-mme-flow,Knieper-Rank-One-Entropy,climenhaga-thompson,BCFT,Gelfert-Ruggiero,CKP-20,CKW, pacifico-yang-yang}.

The field of symbolic dynamics has been extremely successful in analyzing systems displaying hyperbolic
behavior. Its modern history includes (but is not restricted to) the construction of Markov partitions in various 
uniformly and non-uniformly hyperbolic settings:
\begin{enumerate}[$\circ$]
\item Adler \& Weiss for two dimensional hyperbolic toral automorphisms \cite{Adler-Weiss-PNAS}.
\item Sina{\u\i} for Anosov diffeomorphisms \cite{Sinai-Construction-of-MP}.
\item Ratner for Anosov flows \cite{Ratner-MP-three-dimensions,Ratner-MP-n-dimensions}.
\item Bowen for Axiom A diffeomorphisms \cite{Bowen-MP-Axiom-A,Bowen-LNM}
and Axiom A flows without fixed points \cite{Bowen-Symbolic-Flows}.
\item Katok for sets approximating hyperbolic measures of diffeomorphisms \cite{KatokIHES}.
\item Hofbauer \cite{Hofbauer-PMM} and Buzzi \cite{Buzzi-IJM,Buzzi-Invent} for piecewise maps on the interval and beyond.
\item Sarig for surface diffeomorphisms \cite{Sarig-JAMS}.
\item Lima \& Matheus for two dimensional non-uniformly hyperbolic billiards \cite{Lima-Matheus}.
\item Ben Ovadia for diffeomorphisms in any dimension \cite{Ben-Ovadia-high-dimension}.
\item Lima \& Sarig for three dimensional flows without fixed points \cite{Lima-Sarig}.
\item Lima for one-dimensional maps \cite{Lima-AIHP}.
\item Araujo, Lima, Poletti for non-invertible maps with singularities in any dimension \cite{ALP}.
\end{enumerate}
In the first four settings above, that dealt with uniformly hyperbolic systems, 
the coding is surjective and one-to-one in a large (Baire generic) set. 
Katok was the first to treat non-uniformly hyperbolic systems \cite{KatokIHES}.
When applied to surface diffeomorphisms, it implies the existence of horseshoes of large
(but not necessarily full) topological entropy.
Sarig was the first to construct non-uniformly hyperbolic horseshoes of full topological entropy \cite{Sarig-JAMS}.
His work improved Katok's to a great extent, proving that for each $\chi>0$ there is
a symbolic coding with good properties, among them the finiteness-to-one property
in the regular set $\Sigma^\#$ (see Section ~\ref{subsection-symbolic}
for the definition of $\Sigma^\#$).
It codes all $\chi$--hyperbolic measures {\em simultaneously},
and it implies many dynamical consequences such as estimates on the number of closed orbits \cite{Sarig-JAMS},
an at most countable set of ergodic measures of maximal entropy \cite{Sarig-JAMS},
ergodic properties of equilibrium measures \cite{Sarig-JMD},
and the almost Borel structure of surface diffeomorphisms \cite{Boyle-Buzzi}.
In recent years, more advances are being obtained, such as the
coding of homoclinic classes of measures by irreducible Markov shifts and
finiteness/uniqueness of measures of maximal entropy \cite{BCS-MME}, and 
continuity properties of Lyapunov exponents \cite{BCS-continuity}.

The work of Lima \& Sarig was the first to construct, 
for three dimensional flows, horseshoes of full topological entropy \cite{Lima-Sarig}.
It is not as strong as Sarig's, since it only codes one $\chi$--hyperbolic measure at a time
(actually, by an easy adaptation in the proof, it codes countably many such measures).
It implies some dynamical consequences, such as 
estimates on the number of closed orbits \cite{Lima-Sarig}, the countability on the number of measures of
maximal entropy \cite{Lima-Sarig}, and ergodic properties of equilibrium measures \cite{Ledrappier-Lima-Sarig}.
Unfortunately, their techniques do not seem to extend to, say, the coding of all $\chi$-hyperbolic measures
as in the case of diffeomorphisms.

Our Main Theorem identifies a subset
of points in $M$ with non-uniform hyperbolicity at least $\chi$ possessing local product
structure, and constructs a finite-to-one extension of this set by a locally compact topological Markov flow.
This set carries all $\chi$--hyperbolic measures. As an application, we code
homoclinic classes of measures by irreducible Markov flows.

\subsection{Method of proof}\label{ss-method-proof}

We build on the groundbreaking works of Sarig \cite{Sarig-JAMS} and Lima \& Sarig \cite{Lima-Sarig}. 
Lima \& Sarig study a flow by considering the Poincar\'e return map to a section. 
This yields a surface map to which they apply a version of Sarig's result.
Of course, the Poincar\'e map has singularities which they control at the price of choosing
the section in a way that 
almost all orbits slowly approach the boundary of the section. Here is where their construction 
becomes specific to a single measure.
It is still unknown whether there is a global Poincar\'e section such that this latter property holds
for every $\chi$--hyperbolic measure. We call the presence of boundary
the {\em boundary effect}. 

Additionally to the work of Sarig \cite{Sarig-JAMS} and Lima \& Sarig \cite{Lima-Sarig},
we are also inspired by the work of Bowen \cite{Bowen-Symbolic-Flows}. Bowen's idea
to construct Markov partitions for flows is to replace the Poincar\'e map by good returns
(suitable holonomy maps), which are smooth by construction: the artificial singularities of 
\cite{Lima-Sarig} have disappeared. In this way, we proceed as follows:
\begin{enumerate}[(1)]
\item Construct two global Poincar\'e sections $\Lambda,\widehat{\Lambda}$ such that
$\Lambda\subset\widehat{\Lambda}$. We use $\Lambda$ as the reference section
for our construction, and $\widehat{\Lambda}$ as a security section.
\item Let $f:\Lambda\to\Lambda$ be the Poincar\'e return map of $\Lambda$
(note: $f$ is not the Poincar\'e return map of $\widehat{\Lambda}$).
If $\mu$ is $\chi$--hyperbolic and $\nu$ is the measure induced on $\Lambda$,
then $\nu$--almost every $x\in\Lambda$ has a Pesin chart
$\Psi_x:[-Q(x),Q(x)]^2\to \widehat\Lambda$ whose size satisfies $\lim\tfrac{1}{n}\log Q(f^n(x))=0$.
Note that the center of the chart is in $\Lambda$, while the image is on the security section
$\widehat\Lambda$. Local changes of coordinates by linear maps of norm $Q^{-1}$
allow to conjugate $f$ to a uniformly hyperbolic map.
\item Introduce $\ve$--double charts $\Psi_x^{p^s,p^u}$, which are some version of Pesin charts
that controls separately the local stable and local unstable hyperbolicity at $x$ (the parameters
$p^s/p^u$ can be seen as choices of sizes of the stable/unstable manifolds). Define the
transition between $\ve$--double charts so that the parameters $p^s,p^u$ are {\em almost} maximal.
\item Construct a countable collection $\mathfs A$ of $\ve$--double charts that are dense
in the space of all $\ve$--double charts. The notion of denseness is defined in terms of
finitely many parameters of $x$. Using pseudo-orbits, shadowing
and the graph transform method, the collection $\mathfs A$ defines
a Markov cover $\mathfs Z$. Unfortunately, $\mathfs Z$ defines a symbolic coding that
is usually infinite-to-one. Fortunately, $\mathfs Z$ is locally finite.
\item $\mathfs Z$ satisfies a Markov property: for every $x\in\bigcup_{Z\in \mathfs Z}Z$
there is $k>0$ such that $f^k(x)$ satisfies a Markov property in the stable direction and
$\ell>0$ such that $f^{-\ell}(x)$ satisfies a Markov property in the unstable direction. The values
of $k,\ell$ are uniformly bounded.
\item The local finiteness of $\mathfs Z$ and the uniform bounds on $k,\ell$ allow 
to apply a refinement method to obtain a countable Markov partition, which 
defines a topological Markov flow $(\Sigma_r,\sigma_r)$ and a map
$\pi_r:\Sigma_r\to M$ satisfying the Main Theorem.
\end{enumerate}
In analogy with Bowen \cite{Bowen-Symbolic-Flows}, in our case a good return of the center of
a chart is a return to $\Lambda$. The ideas of \cite{Bowen-Symbolic-Flows} are also used in steps
(5) and (6).

We use the same method of \cite{Sarig-JAMS} to obtain step (2). Steps (3) and (4) use ideas
of \cite{Sarig-JAMS,Lima-Sarig}, but they require {\em novel}
ideas. The main difficulty is the following: there is no canonical way to parse a flow orbit into good returns,
hence a single orbit might be cut into different ways. We call this the {\em parsing problem}.
It relates to the {\em inverse problem}, whose goal is to prove that the parameters of the $\ve$--double
charts coding an orbit are defined ``up to bounded error''.
Firstly, since the flow transition times of good returns might belong to a continuum (hence uncountable),
our definition of a transition between $\ve$--double charts requires inequalities between the parameters $p^s,p^u$,
see relations (\ref{gpo2-a}) and (\ref{gpo2-b}).
This contrasts with all previous recent literature, whose definitions of transition require equalities. 
Secondly, we compare the parameters of an orbit directly with the parameters
of the $\ve$--double charts coding it. 
To do that, we introduce analogues of the parameters $p^s,p^u$ for points of $M$. Indeed, we introduce
continuous and discrete versions of such parameters, see Sections \ref{section-parameters-qs-qu}
and \ref{section-Z-indexed}. The continuous version is intrinsic and only depends on the flow,
while the discrete depends on the parsing. The 
discrete one can more easily be compared with the parameters of the $\ve$--double charts.
These new parameters, already used in \cite{Lima-Matheus} in a non-essential way, 
are essential to our us.

The definition of transition between $\ve$--double charts introduces new difficulties. Since
equalities between the parameters no longer hold, a single orbit can be shadowed by two different sequences of 
$\ve$--double charts whose transition times accumulate. To investigate this shear,
we first show that parameters at {\em hyperbolic times} are defined ``up to bounded error'',
and then prove that between two hyperbolic times the shear is uniformly bounded, regardless
the number of hits to $\widehat\Lambda$ and $\Lambda$, see Section \ref{section-control-ps-pu}.

Another difficulty we encounter related to Step (4) above is the {\em coarse graining}, which
consists on selecting a countable collection $\mathfs A$ of $\ve$--double charts that are dense
in the space of all $\ve$--double charts and such that the pseudo-orbits they generate shadow all
relevant orbits. This also relates to the definition of transition between $\ve$--double charts,
that has to be loose enough to code all relevant orbits and tight enough to impose that
charts parameters are defined ``up to bounded error''. To guarantee that the definition is loose
enough, the countable collection we consider is much larger than those constructed in the 
recent literature. Yet, proving that this family is sufficient also requires an analysis at hyperbolic times,
where parameters are essentially uniquely defined.
We can then define
parameters between successive hyperbolic times. 
See Section \ref{section-coarse-graining}.
 
\medbreak
{\bf Applications.} In recent years, the ergodic theory of non-uniformly hyperbolic 
diffeomorphisms has seen great advances, notably with the coding of homoclinic classes of measures
by irreducible Markov shifts and finiteness/uniqueness of measures of maximal entropy \cite{BCS-MME},
and continuity properties of Lyapunov exponent \cite{BCS-continuity}.
Theorem \ref{thm.homoclinic} obtains the corresponding result for flows of the 
coding of homoclinic classes. The remaining questions for flows are not known, and 
are expected to be much harder to obtain than those for diffeomorphisms.

Symbolic codings have been used to obtain exponential decay of correlations
for measures of maximal entropy of smooth surface diffeomorphisms. 
The groundbreaking work of Dolgopyat \cite{Dolgopyat-Annals} shows that, for uniformly hyperbolic
flows, the study of decay of correlations is much more difficult than for diffeomorphisms. 
We believe the corresponding question for non-uniformly hyperbolic flows
is also much harder than for diffeomorphisms.  

In contrast to Lima \& Sarig \cite{Lima-Sarig}, our coding considers the flow displacement between 
returns to the section. We expect to use this more intrinsic construction to code flows {\em with} fixed points. 

We also expect to extend the techniques of this paper to flows in higher dimensions. 
As part of this extension, the tools obtained by Ben Ovadia \cite{Ben-Ovadia-high-dimension} 
when dealing with diffemorphisms in any dimension
will play a key role.

\subsection{Preliminaries}\label{Section-Preliminaries}

\subsubsection{Symbolic dynamics}\label{subsection-symbolic}
Let $\mathfs G=(V,E)$ be an oriented graph, where $V,E$ are the vertex and edge sets.
We denote edges by $v\to w$, and assume that $V$ is countable.

\medskip
\noindent
{\sc Topological Markov shift (TMS):} It is a pair $(\Sigma,\sigma)$
where
$$
\Sigma:=\{\text{$\Z$--indexed paths on $\mathfs G$}\}=
\left\{\un{v}=\{v_n\}_{n\in\Z}\in V^{\Z}:v_n\to v_{n+1}, \forall n\in\Z\right\}
$$
is the symbolic space and $\sigma:\Sigma\to\Sigma$, $[\sigma(\un v)]_n=v_{n+1}$, is the {\em left shift}. 
We endow $\Sigma$ with the distance $d(\un v,\un w):={\rm exp}[-\inf\{|n|\in\Z:v_n\neq w_n\}]$.
The {\em regular set} of $\Sigma$ is
$$
\Sigma^\#:=\left\{\un v\in\Sigma:\exists v,w\in V\text{ s.t. }\begin{array}{l}v_n=v\text{ for infinitely many }n>0\\
v_n=w\text{ for infinitely many }n<0
\end{array}\right\}.
$$

\medskip
We only consider TMS that are \emph{locally compact}, i.e.
for all $v\in V$ the number of ingoing edges $u\to v$ and outgoing edges $v\to w$ is finite.

\medskip
Given $(\Sigma,\sigma)$ a TMS, let $r:\Sigma\to(0,+\infty)$ be a continuous function.
For $n\geq 0$, let
$r_n=r+r\circ\sigma+\cdots+r\circ \sigma^{n-1}$ be $n$--th {\em Birkhoff sum} of $r$,
and extend this definition for $n<0$
in the unique way such that the {\em cocycle identity} holds: $r_{m+n}=r_m+r_n\circ\sigma^m$, $\forall m,n\in\Z$.

\medskip
\noindent
{\sc Topological Markov flow (TMF):} The TMF defined
by $(\Sigma,\sigma)$ and the \emph{roof function} $r$ is the pair $(\Sigma_r,\sigma_r)$ where
$\Sigma_r:=\{(\un v,t):\un v\in\Sigma, 0\leq t<r(\un v)\}$
and $\sigma_r:\Sigma_r\to\Sigma_r$ is the flow on $\Sigma_r$ given by
$\sigma_r^t(\un v,t')=(\sigma^n(\un v),t'+t-r_n(\un v))$, where
$n$ is the unique integer such that $r_n(\un v)\leq t'+t<r_{n+1}(\un v)$.
We endow $\Sigma_r$ with a natural metric $d_r(\cdot,\cdot)$,
called the {\em Bowen-Walters metric}, such that $\sigma_r$ is a continuous flow \cite{Bowen-Walters-Metric}.
The {\em regular set} of $(\Sigma_r,\sigma_r)$ is $\Sigma_r^\#=\{(\un v,t)\in\Sigma_r:\un v\in \Sigma^\#\}$.

\medskip
In other words, $\sigma_r$ is the unit speed vertical flow on $\Sigma_r$ with the identification
$(\un v,r(\un v))\sim (\sigma(\un v),0)$. 
The roof functions we will consider will be H\"older continuous.
In this case, there exist $\kappa,C>0$ such that $d_r(\sigma_r^t(z),\sigma_r^{t}(z'))\leq C d_r(z,z')^\kappa$
for all $|t|\leq 1$ and $z,z'\in\Sigma_r$, see \cite[Lemma 5.8]{Lima-Sarig}.

\medskip
\noindent
{\sc Irreducible component:}
If $\Sigma$ is a countable Markov shift defined by an oriented graph
$\mathfs{G}=(V,E)$, its \emph{irreducible components} are the subshifts $\Sigma'\subset \Sigma$ over
maximal subsets $V'\subset V$ satisfying the following condition:
$$\forall v,w\in V',\;\exists \un v\in \Sigma \text{ and } n\geq 1\text{ such that } v_0=v \text{ and } v_n=w.$$
An irreducible component $\Sigma'_r$ of a suspended shift $\Sigma_r$ is a set of
elements $(\un v,t)\in \Sigma_r$ with $\un v$ in an irreducible component $\Sigma'$ of $\Sigma$.

\subsubsection{Notations.}
For $a,b,\varepsilon>0$, we write $a=e^{\pm\ve}b$ when $e^{-\ve}\leq \frac{a}{b}\leq e^\ve$.
We also write $a\wedge b:=\min (a,b)$.
We write $\bigsqcup A_n$ to represent the {\em disjoint union} of sets $A_n$.

The \emph{Frobenius norm}
of a $2\times 2$ matrix is
$\left\|{   \left[\begin{smallmatrix}
a & b\\ c & d\end{smallmatrix}\right]}
\right\|_{\rm Frob}:=\sqrt{a^2+b^2+c^2+d^2}$. It is equivalent to the sup norm $\|\cdot\|$,
since $\|\cdot\|\leq \|\cdot\|_{\rm Frob}\leq \sqrt{2}\|\cdot\|$. The {\em co-norm}
of an invertible matrix A is denoted by $m(A)=\|A^{-1}\|^{-1}$.

\subsubsection{Metrics}\label{s.metric}

If $M$ is a smooth Riemannian manifold, we denote by
$d_M$ the distance induced by the Riemannian metric.
The Riemannian metric induces a Riemannian metric $\Sas(\cdot,\cdot)$ on $TM$,
called the {\em Sasaki metric}, see e.g. \cite[\S2]{Burns-Masur-Wilkinson}.
For nearby small vectors, the Sasaki metric is
almost a product metric in the following sense. Given a geodesic $\gamma$ in $M$
joining $y$ to $x$, let $P_\gamma:T_yM\to T_xM$
be the parallel transport along $\gamma$. If $v\in T_xM$, $w\in T_yM$ then
$\Sas(v,w)\asymp d(x,y)+\|v-P_\gamma w\|$ as $\Sas(v,w)\to 0$, see e.g.
\cite[Appendix A]{Burns-Masur-Wilkinson}. The rate of convergence depends on the
curvature tensor of the metric on $M$.

Given an open set $U\subset \R^n$ and $h:U\to \R^m$,
let $\|h\|_{C^0}:=\sup_{x\in U}\|h(x)\|$ denote the $C^0$ norm of $h$. For $0<\beta\leq 1$,
let $\Hol{\beta}(h):=\sup\frac{\|h(x)-h(y)\|}{\|x-y\|^\beta}$ 
where the supremum ranges over distinct elements $x,y\in U$. Note that $\Hol{1}(h)$ is a
Lipschitz constant of $h$, that we will also denote by $\Lip(h)$.
If $h$ is differentiable, let
$\|h\|_{C^1}:=\|h\|_{C^0}+\|dh\|_{C^0}$ denote its $C^1$ norm, and
$\|h\|_{C^{1+\beta}}:=\|h\|_{C^1}+\Hol{\beta}(dh)$ its $C^{1+\beta}$ norm.

For any $x,y$ close to some point $z$ in a Riemanniann manifold $M$,
the parallel transport along the shortest geodesic between $x$ and $y$
induces a linear map $P_{x,y}:T_xM\to T_yM$.
To  any linear map $A\colon T_xM\to T_yM$, one associates a map
$\widetilde A :=P_{y,z} \circ A\circ P_{z,x}$.
By definition, $\widetilde{A}$ depends on $z$ but different basepoints $z$ define
a map that differs from $\widetilde{A}$ by pre and post composition with isometries.
In particular, $\|\widetilde{A}\|$ does not depend on the choice of $z$.
With this notation, a map
$f:U\subset M\to M$ is $C^{1+\beta}$ if it is $C^1$ and $\exists C>0$ such that
$\|\widetilde{df_x}-\widetilde{df_y}\|\leq C d(x,y)^\beta$ for all nearby $x,y\in U$.
In this case, define $\Hol{\beta}(df):=\sup\frac{\|\widetilde{df_x}-\widetilde{df_y}\|}{d(x,y)^\beta}$ 
where the supremum ranges over distinct nearby elements $x,y\in U$.

\subsection{Standing assumptions}\label{ss.standing}
Let $M$ be a three dimensional closed smooth Riemannian manifold, and let
$X:M\to TM$ be 
a $C^{1+\beta}$ vector field such that $X(x)\neq 0$, $\forall x\in M$, and let $\varphi=(\varphi^t)_{t\in \R}$ 
be the flow generated by $X$.
We will denote the value of the vector field $X$ at $x$ by either $X_x$ or $X(x)$.
Given a set $Y\subset M$ and an interval $I\subset\R$, write $\vf^I(Y):=\bigcup_{t\in I}\vf^t(Y)$.

We assume from now on
that $\|\nabla X\|_0\leq1$ (up to multiplying the Riemannian metric by a small constant). This assumption
avoids the introduction of some multiplicative constants.
For instance, since an application of the Gr\"onwall inequality implies that $\|d\vf^t\|\leq e^{\|\nabla X\|_0 |t|}$ for all $t\in\R$
(see e.g. \cite{Kunzinger-flow}), we will simply write that $\|d\vf^t\|\leq e^{|t|}$, $\forall t\in\R$.
Another consequence is that every Lyapunov exponent of $\vf$ has absolute value
at most $1$, hence we can take $\chi\in (0,1)$ in the definition of $\chi$-hyperbolicity.

\section{Poincar\'e sections}\label{Section-sections}

In this section, we:
\begin{enumerate}[(1)]
\item Construct two sections $\Lambda,\widehat{\Lambda}$ with good geometrical properties
such that $\Lambda\subset\widehat{\Lambda}$, $d(\Lambda,\partial\widehat{\Lambda})>0$,
and the orbit under $\vf$ of every point of $M$ hits $\Lambda$ after some time $\rho\ll 1$.
The section $\Lambda$ induces a Poincar\'e return map $f$ and a return time $r$.
We call $\Lambda$ the {\em reference section} and $\widehat{\Lambda}$ the {\em security section}.
\item Introduce the {\em induced linear Poincar\'e flow} $\Phi$, which is a flow that describes the local behavior
of $\vf$ in the complimentary direction to $X$.
\item Introduce the {\em holonomy} maps $g^+_x$, $g^-_x$ for each $x\in \Lambda$, which are local and 
continuous versions of Poincar\'e return
maps. It is for these maps that we will construct suitable systems of coordinates in Section \ref{s.nhu}.
\end{enumerate}

\subsection{Transverse discs and flow boxes}\label{section-discs}

Let $\rho>0$.

\medskip
\noindent
{\sc $\rho$--transverse disc:} An open disc $D\subset M$ is {\em $\rho$--transverse} if:
\begin{enumerate}[$\circ$]
\item $D$ is compactly contained in a $C^\infty$ disc of $M$.
\item ${\rm diam}(D)<4\rho$.
\item For every $x\in D$, $\angle(X(x),T_xD^\perp)<\rho$.
\end{enumerate}

\medskip
In other words, a $\rho$--transverse disc is a small disc that is almost orthogonal to $X$.
It is easy to build $\rho$--transverse discs. For instance, we know by the tubular neighborhood theorem
that $\vf$ can be conjugated in local charts to the flow $(x,t_0)\in\R^2\times\R\mapsto(x,t_0+t)$.
If $\rho'$ is small enough, then the image of $B(0,\rho')\times\{t_0\}$ under the local chart
is a $\rho$--transverse disc. 

\medskip
\noindent
{\sc Flow box:} Every $\rho$--transverse disc $D$ defines a {\em flow box} $\vf^{[-4\rho,4\rho]}D$.

\medskip
The assumption that $X$ does not vanish implies
that for all $\rho>0$ small enough,
the map $\Gamma_D\colon (y,t)\in D\times [-4\rho,4\rho]\mapsto \vf^t(y)$
is a diffeomorphism onto the flow box $\vf^{[-4\rho,4\rho]}D$. We denote its inverse by
$x\in \vf^{[-4\rho,4\rho]}D\mapsto (\mathfrak q_D(x),\mathfrak t_D(x))$,
where $\mathfrak q_D:\vf^{[-4\rho,4\rho]}D\to D$ and
$\mathfrak t_D:\vf^{[-4\rho,4\rho]}D\to [-4\rho,4\rho]$. 

\begin{lemma}\label{Lemma-regularity-q-t} 
There is a $\rho_0=\rho_0(M,X)>0$
such that for every $\rho_0$--transverse discs $D,D'$:
\begin{enumerate}[{\rm (1)}]
\item The maps $\mathfrak q_D,\mathfrak t_D$ are $C^{1+\beta}$.
\item The map $\mathfrak q_D$ has a Lipchitz constant smaller than $2$.
\item If $D'$ intersects the flow box $\vf^{[-4\rho_0,4\rho_0]}D$,
then the restriction to $D'$ of the map $\mathfrak t_D$ has a Lipschitz constant smaller than $1$.
\end{enumerate}
\end{lemma}

\begin{proof}
By the implicit function theorem, for any $\rho$--transverse disc $D$
the chart $\Gamma_D\colon (x,t)\in D\times [-4\rho,4\rho]\mapsto \vf^t(x)$
is $C^{1+\beta}$, hence the inverse maps $\mathfrak q_D,\mathfrak t_D$ are also $C^{1+\beta}$.
This proves part (1).

Let us consider the foliation of $\R^3$ whose leaves are the verticals $\{(x,y)\}\times \R$,
and let $\Delta$ be a section whose tangent spaces $T_x\Delta$ define angles with the horizontal planes
smaller than $\gamma>0$.
Then the holonomy along the vertical lines define a projection to $\Delta$ which is $(1/\cos\gamma)$--Lipschitz.

Given $e>0$ arbitrarily small, there exists a covering of
$M$ by finitely many charts $\Theta \colon (-a,a)^3\to M$ which are $(1+e)$-biLipschitz
and such that the lifted vector field $\widehat X:=\Theta^*X$ is tangent to the vertical lines.
Choosing $e$ and then $\rho_0>0$ small enough,
for any $\rho_0$--transverse disc $D$, the set $\varphi_{[-4\rho,4\rho]}(D)$ is contained
in the image of a chart $\Theta$; moreover
$\Theta^{-1}(\Delta)$ is a disc whose tangent spaces define angles with the horizontal planes
smaller than $\gamma>0$.
The projection $\mathfrak q_D$ to $D$ is conjugated by $\Theta$ to the projection by holonomy along
vertical lines to the set $\Theta^{-1}(\Delta)$. Consequently, the map $\mathfrak q_D$ is 
$(1+e)^2/\cos\gamma$--Lipschitz,
which can be chosen arbitrarily close to $1$ if $e,\gamma$, and hence $\rho_0$, are small enough.
This proves part (2).

Now consider two $\rho_0$--transverse discs $D,D'$ such that
$D'$ intersects $\vf^{[-4\rho_0,4\rho_0]}D$.
Since $\rho_0$ is chosen small, both $D,D'$
are contained in the image of a same chart $\Theta$.
As before, let $\widehat X:=\Theta^*X$ be the vector field $X$ lifted in the chart.
The discs $\Theta^{-1}(D'),\Theta^{-1}(D)$ are graphs $\{(x,y,\varphi_i(x,y))\}$
over the horizontal hyperplane of $C^{1+\beta}$ maps $\varphi_1$ and $\varphi_2$ respectively.
For $z=\Theta(x,y,\varphi_1(x,y))\in D'$ which intersects the flow box $\vf^{[-4\rho_0,4\rho_0]}D$,
the projection time to $D$ can be computed in the chart as:
$$\mathfrak t_D(z)=\int_{\varphi_1(x,y)}^{\varphi_2(x,y)} \tfrac 1 {\|\widehat X(x,y,t)\|}dt .$$
The $C^1$--norm of  $\tfrac 1 {\|\widehat X\|}$ is bounded, independently from the charts $\Theta$.
Taking $\rho_0$ small, the derivatives $D\varphi_i$ and the differences
$\varphi_2(x,y)-\varphi_1(x,y)$ are close to $0$.
Hence the derivative of the map $z\mapsto \mathfrak t_D(z)$ for $z\in D'$ is close to $0$,
proving part (3).
\end{proof}

\subsection{Proper sections and Poincar\'e return maps}\label{ss.proper-section}
We begin with some definitions.

\medskip
\noindent
{\sc Proper section:} A {\em proper section of size} $\rho$
is a finite union $\Lambda=\bigcup_{i=1}^{\mathfrak{n}} D_i$ of $\rho$--transverse
discs $D_1,\ldots,D_\mathfrak{n}$ such that:
\begin{enumerate}[(1)]
\item {\sc Cover:} $M=\bigcup_{i=1}^{\mathfrak{n}} \vf^{[0,\rho)}D_i$.
\item {\sc Partial order:} For all $i\neq j$, at least one of the sets
$\overline{D_i}\cap \vf^{[0,4\rho]}\overline{D_j}$
or $\overline{D_j}\cap \vf^{[0,4\rho]}\overline{D_i}$ is empty;
in particular $\overline{D_i}\cap\overline{D_j}=\emptyset$.
\end{enumerate}
Define the {\em return time function} $r_\Lambda:\Lambda\to (0,\rho)$ by $r_\Lambda(x):=\inf\{t>0:\vf^t(x)\in\Lambda\}$.

\medskip
\noindent
{\sc Poincar\'e return map:} The {\em Poincar\'e return map} of a proper section $\Lambda$
is the map $f_\Lambda:\Lambda\to\Lambda$ defined by $f_\Lambda(x):=\vf^{r_\Lambda(x)}(x)$.

\medskip
In the following, we fix $\rho<\min\{0.25,\rho_0\}$ small and
consider two proper sections
$\Lambda,\widehat{\Lambda}$ of size $\rho/2$ such that
$\Lambda\subset\widehat{\Lambda}$ and $d_M(\Lambda,\partial\widehat{\Lambda})>0$.
We let $d=d_{\widehat{\Lambda}}$ be the metric on $\widehat{\Lambda}$ defined
by the induced Riemannian metric on $\widehat{\Lambda}$.
For $x\in\widehat{\Lambda}$ and $r>0$, we write:
\begin{enumerate}[$\circ$]
\item $B(x,r)\subset\widehat{\Lambda}$ for the ball in the distance $d$ with center $x$ and radius $r$;
\item $B_x[r]\subset T_x\widehat{\Lambda}$ for the ball with center $0$ and radius $r$;
\item $R[r]:=[-r,r]^2\subset\R^2$. 
\end{enumerate}
Since the associated flow boxes are $C^{1+\beta}$,
there exists $L>0$ such that for any transverse disc $D_i$ defining the section $\Lambda$,
the maps $\mathfrak q_{D_i},\mathfrak t_{D_i}$ satisfy:
$$\Hol{\beta}(d \mathfrak q_{D_i}) <L \, \text{ and }\, \Hol{\beta}(d \mathfrak t_{D_i}) <L.$$

\subsection{Exponential maps}
Given $x\in\widehat{\Lambda}$, let $\inj(x)$ denote the {injectivity radius} of $\widehat{\Lambda}$ at $x$,
and let $\exp{x}$ be the {\em exponential map} of $\widehat{\Lambda}$ at $x$, wherever it can be defined.
Below we list the properties of $\exp{x}$ that we will use.

\medskip
\noindent
{\sc Regularity of $\exp{x}$:} There is $\mathfrak r\in (0,\rho)$ such that for every $x\in\Lambda$
the following properties hold on the ball $B_x:=B(x,2\mathfrak r)\subset\widehat{\Lambda}$:
\begin{enumerate}[ii\, )]
\item[(Exp1)] If $y\in B_x$ then $\inj(y)\geq 2\mathfrak r$, the map $\exp{y}^{-1}:B_x\to T_y\widehat{\Lambda}$
is a diffeomorphism onto its image, and for all
$v\in T_x\widehat{\Lambda},w\in T_y\widehat{\Lambda}$ with $\|v\|,\|w\|\leq 2\mathfrak r$ it holds
$$\tfrac{1}{2}(d(x,y)+\|v-P_{y,x}w\|)\leq \Sas(v,w)\leq 2(d(x,y)+\|v-P_{y,x} w\|),$$ 
where $P_{y,x}$
is the parallel transport along the geodesic joining $y$ to $x$.
\item[(Exp2)] If $y_1,y_2\in B_x$ then $d(\exp{y_1}v_1,\exp{y_2}v_2)\leq 2\Sas(v_1,v_2)$ for 
$\|v_1\|$, $\|v_2\|\leq 2\mathfrak r$, and
$\Sas(\exp{y_1}^{-1}z_1,\exp{y_2}^{-1}z_2)\leq 2[d(y_1,y_2)+d(z_1,z_2)]$ for $z_1,z_2\in B_x$
whenever the expression makes sense.
In particular, $\|d(\exp{x})_v\|\leq 2$ for $\|v\|\leq 2\mathfrak r$ and 
$\|d(\exp{x}^{-1})_y\| \leq 2$ for $y\in B_x$.
\end{enumerate}

\medskip
Conditions (Exp1)--(Exp2) say that the exponential maps and their inverses
are well-defined and have Lipschitz constants bounded by 2 in balls
of radius $2\mathfrak r$. The existence of $\mathfrak r$ follows from compactness, since 
$d_M(\Lambda,\partial\widehat{\Lambda})>0$ and $d(\exp{x})_0$ is the identity map.

The next two assumptions describe the regularity of $d\exp{x}$. For $x,x'\in\widehat{\Lambda}$, let
$\mathfs L _{x,x'}:=\{A:T_x\widehat{\Lambda}\to T_{x'}\widehat{\Lambda}:A\text{ is linear}\}$
and $\mathfs L _x:=\mathfs L_{x,x}$. 
In particular, $P_{y,x}$ considered in (Exp1) is in $\mathfs L_{y,x}$.
Given $y\in B_x,z\in B_{x'}$ and $A\in \mathfs L_{y,z}$,
let $\widetilde{A}\in\mathfs L_{x,x'}$, $\widetilde{A}:=P_{z,x'} \circ A\circ P_{x,y}$.
The norm $\|\widetilde{A}\|$ does not depend on the choice of $x,x'$.
If $A_i\in\mathfs L_{y_i,z_i}$ then $\|\widetilde{A_1}-\widetilde{A_2}\|$ does
depend on the choice of $x,x'$, but if we change the basepoints $x,x'$ to $w,w'$ then
the respective differences differ by precompositions and postocompositions
with norm of the order of the areas of the geodesic triangles formed by $x,w,y_i$
and by $x',w',z_i$, which will be negligible to our estimates.
For $x\in\Lambda$,
define the map $\tau=\tau_x:B_x\times B_x\to \mathfs L_x$
by $\tau(y,z)=\widetilde{d(\exp{y}^{-1})_z}$, where we use the identification
$T_v(T_{y}\widehat{\Lambda})\cong T_{y}\widehat{\Lambda}$ for all $v\in T_y\widehat{\Lambda}$.

\medskip
\noindent
{\sc Regularity of $d\exp{x}$:} There is $\mathfrak K>1$ such that for all $x\in\Lambda$ the following holds:
\begin{enumerate}[ii\, )]
\item[(Exp3)] If $y_1,y_2\in B_x$ then
$\|\widetilde{d(\exp{y_1})_{v_1}}-\widetilde{d(\exp{y_2})_{v_2}}\|\leq \mathfrak K\Sas(v_1,v_2),$
for all $\|v_1\|,\|v_2\|\leq 2\mathfrak r$, and
$\|\tau(y_1,z_1)-\tau(y_2,z_2)\|\leq \mathfrak K[d(y_1,y_2)+d(z_1,z_2)]$
for all $z_1,z_2\in B_x$.
\item[(Exp4)] If $y_1,y_2\in B_x$ then the map $\tau(y_1,\cdot)-\tau(y_2,\cdot):B_x\to \mathfs L_x$
has Lipschitz constant $\leq \mathfrak Kd(y_1,y_2)$.
\end{enumerate}

\medskip
Condition (Exp3) controls the Lipschitz constants of the derivatives of $\exp{x}$,
and (Exp4) controls the Lipschitz constants of their second derivatives.
The existence of $\mathfrak K$ is guaranteed whenever the curvature tensor of
$\widehat{\Lambda}$ is uniformly bounded, and this happens because $\widehat{\Lambda}$
is the restriction to a compact subset of a finite union of $\rho$--transverse (open) discs.

\subsection{Induced linear Poincar\'e flows}\label{section-induced}

Classically, the linear Poincar\'e flow is the $\R$--cocycle induced by $d\vf$ in the bundle orthogonal to $X$.
In this paper we employ a different definition: we fix a 1--form $\theta$ and consider parallel projections
to $X$ onto the bundle ${\rm Ker}(\theta)$. We begin choosing a suitable 1--form.

\begin{lemma}\label{Lemma-1-form}
If $\widehat \Lambda$ is a proper section of size $\rho/2$, there exists a $1$--form $\theta$ on $M$ such that:
\begin{enumerate}[{\rm (1)}]
\item $\theta(X(x))=1$ and $\angle(X(x),{\rm Ker}(\theta_x)^\perp)<\rho$, $\forall x\in M$.
\item ${\rm Ker}(\theta_x)=T_x\widehat \Lambda$, $\forall x\in\widehat \Lambda$.
\end{enumerate}
\end{lemma}

\begin{proof}
Take $\eta(v)=\tfrac{\langle v,X(x)\rangle}{\|X(x)\|^2}$ for $v\in T_xM$.
Clearly $\eta$ is a 1--form on $M$ satisfying (1) above. Let $U$
be a small neighborhood of $\widehat \Lambda$. By the tubular neighborhood theorem,
there exists a 1--form $\zeta$ on $U$ such that $\zeta(X(x))=1$ and
$\angle(X(x),{\rm Ker}(\zeta_x)^\perp)<\rho$ for all $x\in U$,
and ${\rm Ker}(\zeta_x)=T_x\widehat \Lambda$ for all $x\in\widehat \Lambda$. Let
$V$ be a neighborhood of $\widehat \Lambda$ with $\widehat \Lambda\subset V\subset U$, and take
a bump function $h:M\to[0,1]$ such that $h\restriction_{V}\equiv 0$ and $h\restriction_{ M\backslash U}\equiv 1$.
The 1--form $\theta:=h\eta+(1-h)\zeta$ satisfies the following:
\begin{enumerate}[$\circ$]
\item $\theta(X(x))=1$, $\forall x\in M$: clear, since $\eta(X(x))=\zeta(X(x))=1$.
\item $\angle(X(x),{\rm Ker}(\theta_x)^\perp)<\rho$, $\forall x\in M$:
to see this, write $\eta_x(\,\cdot\,)=\langle \,\cdot\,,v_x\rangle$ and $\zeta_x(\,\cdot\,)=\langle \,\cdot\,,w_x\rangle$,
where $v_x=\tfrac{X(x)}{\|X(x)\|}$ and $\angle(X(x),w_x)<\rho$. Since
${\rm Ker}(\theta_x)^\perp$ is generated by the linear combination $h(x)v_x+(1-h(x))w_x$, we have
$\angle(X(x),{\rm Ker}(\theta_x)^\perp)\leq \angle(X(x),w_x)<\rho$.
\item ${\rm Ker}(\theta_x)=T_x\widehat \Lambda$, $\forall x\in\widehat \Lambda$: since $h(x)=0$,
we have ${\rm Ker}(\theta_x)={\rm Ker}(\zeta_x)=T_x\widehat \Lambda$.
\end{enumerate}
The proof is complete.
\end{proof}

From now on, we fix a 1--form $\theta$ satisfying Lemma \ref{Lemma-1-form}.
Introduce the two dimensional bundle
$$N:=\displaystyle\bigsqcup_{x\in M}{\rm Ker}(\theta_x).$$
For each $x\in M$, let $\mathfrak p_x:T_xM\to N_x$ be the projection to $N_x$ parallel to $X(x)$.
By Lemma \ref{Lemma-1-form}(1), for all $x\in M$ we have:
$$
\|\mathfrak p_x\|=\tfrac{1}{\cos\angle(X(x),{\rm Ker}(\theta_x)^\perp)}<\tfrac{1}{\cos\rho}<1+\rho.
$$

\medskip
\noindent
{\sc Induced linear Poincar\'e flow:} The {\em linear Poincar\'e flow of $\vf$ induced by $\theta$}
is the flow $\Phi=\{\Phi^t\}_{t\in\R}:N\to N$ defined by
$\Phi^t(v)={\mathfrak p}_{\vf^t(x)}[d\vf^t_x(v)]$ for $v\in N_x$.

\medskip
When the context is clear, we will omit the subscripts $x$ and $\vf^t(x)$.
Clearly $\Phi$ is H\"older continuous, and
$\|\Phi^t_x\|\leq \|\mathfrak p_{\vf^t(x)}\|\|d\vf^t_x\|\leq (1+\rho)e^{|t|}<e^{\rho+|t|}$, $\forall t\in\R$.
In particular:
\begin{equation}\label{definition-d}
\|\Phi^t\|=e^{\pm 4\rho},\,\forall |t|\leq2\rho.
\end{equation}

\begin{lemma}\label{Lemma-flow-Phi}
The following hold.
\begin{enumerate}[{\rm (1)}]
\item $\Phi$ is a flow: $\Phi^{t+t'}=\Phi^t\circ\Phi^{t'}$, $\forall t,t'\in\R$.
\item If $D\subset{\widehat \Lambda}$ is a transverse disc, then for all $x\in D$ it holds
$d(\mathfrak q_D)_x=\mathfrak p_x$.
\end{enumerate}
\end{lemma}
\begin{proof}
(1) If $v\in N_x$ and $t,t'\in\R$, then there is $\gamma\in\R$ such that
\begin{align*}
&\ \Phi^{t'}(\Phi^t(v))=\Phi^{t'}({\mathfrak p}_{\vf^t(x)}[d\vf^t_x(v)])=\Phi^{t'}(d\vf^t_x(v)+\gamma X({\vf^t(x)}))\\
&={\mathfrak p}_{\vf^{t'+t}(x)}[d\vf^{t'}_{\vf^t(x)}(d\vf^t_x(v)+\gamma X({\vf^t(x)}))]
={\mathfrak p}_{\vf^{t'+t}(x)}[d\vf^{t'+t}_{x}(v)+\gamma X({\vf^{t'+t}(x)})]\\
&={\mathfrak p}_{\vf^{t'+t}(x)}[d\vf^{t'+t}_{x}(v)]=\Phi^{t'+t}(v).
\end{align*}

\noindent
(2) Fix $x\in D\subset {\widehat \Lambda}$. It is enough to show that $d(\mathfrak q_D)_x[X(x)]=0$
and $d(\mathfrak q_D)_x[v]=v$ for all $v\in N_x$.
We have $d(\mathfrak q_D)_x[X(x)]=\tfrac{d}{dt}\big|_{t=0}[\mathfrak q_D(\vf^t(x))]=0$
because $\mathfrak q_D(\vf^t(x))=x$ for small $t$. Now, since $N_x=T_x\widehat \Lambda$ and 
$\mathfrak q_D\restriction_{{\widehat \Lambda}}$ is the identity,
$d(\mathfrak q_D)_x[v]=v$ for all $v\in N_x$.
\end{proof}

\subsection{Holonomy maps}
We have fixed $\Lambda,\widehat{\Lambda}$,
two proper sections of size $\rho/2$.
From now on, write $f:=f_\Lambda$.  
The maps $f,r_\Lambda$ admit discontinuities, hence we introduce a related family of local diffeomorphisms.
Recall that $\mathfrak r>0$ is a fixed small parameter, and that $B_x:=B(x,2\mathfrak r)$. 
Write $\widehat{\Lambda}=\bigcup_{i=1}^{\mathfrak{n}} D_i$ as the disjoint union of $\rho$--transverse
discs $D_i$, and let $\mathfrak q_{D_i}$ as before. By Lemma \ref{Lemma-regularity-q-t},
$\Lip(\mathfrak q_{D_i})<2$.

Assume that $x,\vf^t(x)\in\Lambda$ for some $|t|\leq\rho$, with $x\in D_i$ and $\vf^t(x)\in D_j$.
In this case, the restrictions
$\mathfrak q_{D_j}\restriction_{B_x}$ and $\mathfrak q_{D_i}\restriction_{B_{\vf^t(x)}}$ are diffeomorphisms
onto their images, and one is the inverse of the other when the compositions makes sense.
When this happens, we call these restrictions {\em holonomy maps}.

\begin{lemma}\label{Lemma-map-g}
Under the above conditions, the holonomy map $\mathfrak q_{D_j}\restriction_{B_x}$ is a $2$--bi-Lipschitz
$C^{1+\beta}$ diffeomorphism onto its image, and its derivative at $x$ equals $\Phi^t\restriction_{N_x}$.
\end{lemma}

\begin{proof}
Write $g=\mathfrak q_{D_j}\restriction_{B_x}$. The first statement follows from Lemma \ref{Lemma-regularity-q-t}.
Now, since $g=\mathfrak q_{D_j}\circ \vf^t$,
Lemma \ref{Lemma-flow-Phi}(2) implies
$dg_x=d(\mathfrak q_{D_j})_{\vf^t(x)}\circ d\vf^{t}_x\restriction_{T_x\widehat{\Lambda}}=
\mathfrak p_{\vf^t(x)}\circ d\vf^t_x\restriction_{N_x}=\Phi^t\restriction_{N_x}$.
\end{proof}

In the sequel we will investigate some particular holonomy maps, defined as follows. 
Let $0<t,t'<\rho$ such that $f(x)=\vf^t(x)\in D_j$ and $f^{-1}(x)=\vf^{-t'}(x)\in D_k$. 

\medskip
\noindent
{\sc Holonomy maps:} The {\em forward holonomy map} at $x$ is 
$g_x^+:=\mathfrak q_{D_j}\restriction_{B_x}$. Similarly, the {\em backward holonomy map} at $x$ is 
$g_x^-:=\mathfrak q_{D_k}\restriction_{B_x}$.

\medskip
Note that $g_x^+$ differs from $f$ and from the Poincar\'e return to $\widehat\Lambda$,
see Figure \ref{figure-holonomy}.
Also, $(g_x^+)^{-1}=g_{f(x)}^-$.

\begin{figure}[hbt!]
\centering
\def\svgwidth{9cm}
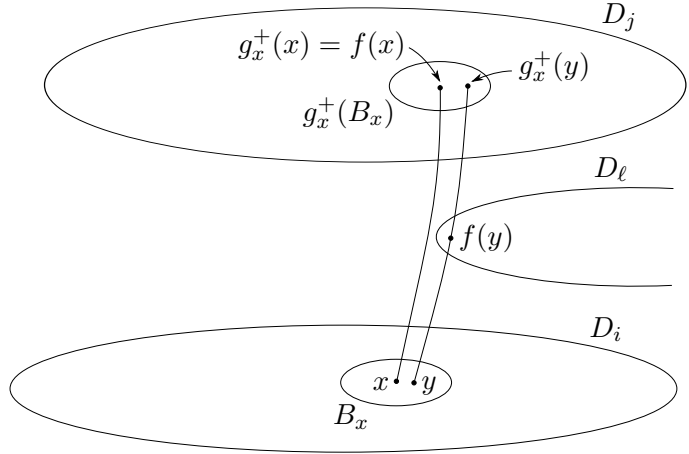\caption{The holonomy map $g_x^+$: it may differ
from $f$ and $f_{\widehat{\Lambda}}$.
}\label{figure-holonomy}
\end{figure}

\section{The non-uniformly hyperbolic locus}\label{s.nhu}
Up to now, we have fixed $\vf,\chi,\rho,\Lambda,\widehat\Lambda$ and $\theta$. In this section, we:
\begin{enumerate}[(1)]
\item Define the set $\nuh$ of points that exhibit a hyperbolicity of scale at least $\chi$. We fix $\varepsilon>0$ small
enough, and associate to each $x\in \nuh$ a number $Q(x)\in (0,1)$ that approaches zero 
as the quality of the hyperbolicity at $x$ deteriorates.
\item Introduce numbers $q(x)\in [0,1)$, that measure how fast $Q(\varphi^t(x))$
decreases to zero as $|t|\to\infty$.
We also associates analogous number $q^s(x)$ and $q^u(x)$
for future and past orbits.
\item Define the set $\nuh^\#$ of points $x\in\nuh$ whose hyperbolicity satisfies a recurrence property: there is $c(x)>0$ such that
$q(\varphi^t(x))>c(x)$ for some values of $t$ arbitrarily close to $\pm\infty$.
This set carries all $\chi$--hyperbolic measures.
\item Define Pesin charts $\Psi_x$ for each $x\in \Lambda\cap\nuh$.
We then prove that, in Pesin charts, the holonomy maps $g_x^{\pm}$ are close to hyperbolic linear maps.
\end{enumerate}

\subsection{The non-uniformly hyperbolic locus $\nuh$}\mbox{}

\medskip
\noindent
{\sc Non-uniformly hyperbolic locus $\nuh=\nuh(\vf,\chi,\rho,\theta)$:} It is the invariant set of points
$x\in M$ for which there are unitary vectors $n^s_x,n^u_x$ with the following properties:
\medskip

\begin{enumerate}[(NUH1)]
\item {\sc $s$--direction:} $\liminf\limits_{t\to+\infty}\tfrac{1}{t}\log\|\Phi^{-t} n^s_x\|>0$,
$\limsup\limits_{t\to+\infty}\tfrac{1}{t}\log\|\Phi^t n^s_x\|\leq -\chi$, and
$$ \displaystyle
s(x):=2e^{2\rho}\left(\int_0^{+\infty} e^{2\chi t}\|\Phi^t n^s_x\|^2dt\right)^{1/2}<+\infty.
$$
\medskip

\item  {\sc $u$--direction:} $\liminf\limits_{t\to+\infty}\tfrac{1}{t}\log\|\Phi^t n^u_x\|>0$,
$\limsup\limits_{t\to+\infty}\tfrac{1}{t}\log\|\Phi^{-t} n^u_x\|\leq -\chi$, and
$$\displaystyle
u(x):=2e^{2\rho}\left(\int_0^{+\infty} e^{2\chi t}\|\Phi^{-t}n^u_x\|^2dt\right)^{1/2}<+\infty.
$$
\end{enumerate}

\medskip
It is clear that $n^s_x,n^u_x$ are unique up to a choice of signs. We choose the sign
so that their angle is less than or equal to $\pi/2$, and make the following definition.

\medskip
\noindent
{\sc angle: $\alpha(x):=\angle(n^s_x,n^u_x)$.}

\medskip
Let us remind  that $\chi\in (0,1)$, see Section~\ref{ss.standing}.
From the estimate before~\eqref{definition-d}, we have
$$\int_0^{+\infty} e^{2\chi t}\|\Phi^t n^s_x\|^2dt\geq \int_0^{+\infty} e^{2\chi t}e^{-4\rho-2t}dt=\tfrac{e^{-4\rho}}{2(1-\chi)}>
\tfrac{e^{-4\rho}}{2},$$
therefore for each $x\in\nuh$ we have $s(x),u(x)\in[\sqrt{2},+\infty)$ and $\alpha(x)\neq 0$.

Conditions (NUH1)--(NUH2) are weaker than Lyapunov regularity, hence
${\rm NUH}$ contains all Lyapunov regular points with exponents greater than $\chi$
in absolute value. Moreover, a periodic point $x$ is in ${\rm NUH}$ iff all of its
exponents are greater than $\chi$ in absolute value. 
But ${\rm NUH}$ might contain points with Lyapunov exponents equal to $\pm\chi$,
and even points which are not Lyapunov regular, where the contraction rates oscillate infinitely often.
These conditions are different from \cite{Sarig-JAMS,Lima-Sarig,Lima-Matheus}, and
similar to \cite{ALP,Ben-Ovadia-ETDS}. This allows to identify the coded set, as
done in these latter works.
 
\begin{proposition}\label{Prop-NUH}
If $\mu$ is a $\chi$--hyperbolic measure, then $\mu[\nuh]=1$.
\end{proposition}

\begin{proof}
Fix a $\chi$--hyperbolic measure $\mu$. By the Oseledets theorem, there is a set $X\subset M$
with $\mu[X]=1$ such that for all $x\in X$ there are unitary vectors $e^s_x,e^u_x\in T_xM$ satisfying:
\begin{enumerate}[(1)]
\item $\lim\limits_{t\to\pm\infty}\tfrac{1}{t}\log\|d\vf^te^s_x\|<-\chi$ and
$\lim\limits_{t\to\pm\infty}\tfrac{1}{t}\log\|d\vf^te^u_x\|>\chi$.
\item $\lim\limits_{t\to\pm\infty}\tfrac{1}{t}\log|\sin\angle(e^s_{\vf^t(x)},e^u_{\vf^t(x)})|=
\lim\limits_{t\to\pm\infty}\tfrac{1}{t}\log|\sin\angle(X_{\vf^t(x)},e^{s/u}_{\vf^t(x)})|=0$.
\end{enumerate}
If $x\in X$ then $e^{s/u}_x\not\in{\rm span}(X_x)$, hence there are scalars $\gamma^{s/u}(x)$,
$\delta^{s/u}(x)$ such that
\begin{equation}\label{definition-normal-vectors}
n^s_x=\gamma^s(x)e^s_x+\delta^s(x)X_x\ \text{ and }\ n^u_x=\gamma^u(x)e^u_x+\delta^u(x)X_x
\end{equation}
are unitary vectors in $N_x$. If $X(x)\perp N_x$ then 
$\gamma^{s/u}(x)=\pm\tfrac{1}{\sin\angle (X_x,e^{s/u}_x)}$. Since by construction we have
$\angle(N_x,X(x)^\perp)<\rho$ (see Lemma \ref{Lemma-1-form}),
we conclude that $\gamma^{s/u}(x)=\pm\tfrac{e^{\pm 4\rho}}{\sin\angle (X_x,e^{s/u}_x)}$
and so condition (2) translates to
$$
\lim\limits_{t\to\pm\infty}\tfrac{1}{t}\log|\gamma^s(\vf^t(x))|=\lim\limits_{t\to\pm\infty}\tfrac{1}{t}\log|\gamma^u(\vf^t(x))|=0.
$$

We claim that $X\subset\nuh$, and we prove this showing that
$$\lim_{t\to\pm\infty}\tfrac{1}{t}\log\|\Phi^tn^s_x\|<-\chi\text{ and }\lim_{t\to\pm\infty}\tfrac{1}{t}\log\|\Phi^{-t}n^u_x\|<-\chi.
$$
We show the first estimate (the second is analogous). For that, we claim that
$\|\Phi^tn^s_x\|=\tfrac{|\gamma^s(x)|}{|\gamma^s(\vf^t(x))|}\|d\vf^t e^s_x\|$ for all $x\in X$.
By (\ref{definition-normal-vectors}),
\begin{align*}
d\vf^t n^s_x=d\vf^t[\gamma^s(x)e^s_x+\delta^s(x)X_x]=
\gamma^s(x)\|d\vf^t e^s_x\|e^s_{\vf^t(x)}+\delta^s(x)X_{\vf^t(x)},
\end{align*}
hence
\begin{align*}
&\ \Phi^tn^s_x={\mathfrak p}_{\vf^t(x)}[\gamma^s(x)\|d\vf^t e^s_x\|e^s_{\vf^t(x)}+\delta^s(x)X_{\vf^t(x)}]=
\gamma^s(x)\|d\vf^t e^s_x\|{\mathfrak p}_{\vf^t(x)}[e^s_{\vf^t(x)}]\\
&=\gamma^s(x)\|d\vf^t e^s_x\|{\mathfrak p}_{\vf^t(x)}\left[\tfrac{1}{\gamma^s(\vf^t(x))}n^s_{\vf^t(x)}-
\tfrac{\delta^s(\vf^t(x))}{\gamma^s(\vf^t(x))}X_{\vf^t(x)}\right]=
\tfrac{\gamma^s(x)}{\gamma^s(\vf^t(x))}\|d\vf^t e^s_x\|n^s_{\vf^t(x)}.
\end{align*}
Taking norms, we get that $\|\Phi^tn^s_x\|=\tfrac{|\gamma^s(x)|}{|\gamma^s(\vf^t(x))|}\|d\vf^t e^s_x\|$.
Hence
\begin{align*}
&\ \lim_{t\to\pm\infty}\tfrac{1}{t}\log\|\Phi^tn^s_x\|=-\lim_{t\to\pm\infty}\tfrac{1}{t}\log|\gamma^s(\vf^t(x))|+
\lim_{t\to\pm\infty}\tfrac{1}{t}\log\|d\vf^t e^s_x\|\\
&=\lim_{t\to\pm\infty}\tfrac{1}{t}\log\|d\vf^t e^s_x\|<-\chi.
\end{align*}
\end{proof}

\subsection{Oseledets-Pesin reduction}
Let $e_1=(1,0),e_2=(0,1)$ be the canonical basis of $\R^2$.
We define a change of coordinates that diagonalizes the induced
linear Poincar\'e flow.

\medskip
\noindent
{\sc Linear map $C(x)$:} For $x\in{\rm NUH}$, let $C(x):\R^2\to N_x$ be the linear map defined by
$$
C(x):e_1\mapsto \frac{n^s_x}{s(x)}\ ,\ C(x):e_2\mapsto\frac{n^u_x}{u(x)}\cdot
$$

\begin{lemma}\label{Lemma-linear-reduction}
The following holds for all $x\in\nuh$.
\begin{enumerate}[{\rm (1)}]
\item $\|C(x)\|\leq \|C(x)\|_{\rm Frob}\leq 1$
and $\|C(x)^{-1}\|_{\rm Frob}=\tfrac{\sqrt{s(x)^2+u(x)^2}}{|\sin\alpha(x)|}$.
\item $C(\vf^t(x))^{-1}\circ \Phi^t\circ C(x)$ is a diagonal matrix with diagonal entries
$A_t(x),B_t(x)$ satisfying:
$$e^{-4\rho}<|A_t(x)|<e^{-\chi t} \text{ and } e^{\chi t}<|B_t(x)|<e^{4\rho},
\;  \forall\, 0<t\leq 2\rho.$$
\item For all $|t|\leq 2\rho$:
$$
\frac{s(\vf^t(x))}{s(x)}=e^{\pm 10\rho},\ \frac{u(\vf^t(x))}{u(x)}=e^{\pm 10\rho},
\ \frac{|\sin\alpha(\vf^t(x))|}{|\sin\alpha(x)|}=e^{\pm 8\rho}.
$$
In particular, $$\frac{\|C(\vf^t(x))^{-1}\|_{\rm Frob}}{\|C(x)^{-1}\|_{\rm Frob}}=e^{\pm 18\rho}.$$
\end{enumerate}
\end{lemma}

The proof is in Appendix \ref{Appendix-proofs}.
Part (2) is known as Oseledets-Pesin reduction, and represents the diagonalization of $\Phi$.

\subsection{Quantification of hyperbolicity: the parameters $\mathbf{Q(x),q(x), q^{s/u}(x)}$}
\label{section-parameters-qs-qu}

We now introduce another small parameter $\ve\in (0,\mathfrak r)$ such that
$\ve\ll \rho\ll 1$ (each symbol $\ll$ is defined by means of a finite number of inequalities
that need to be satisfied throughout the paper).
Instead of working with $\|C(x)^{-1}\|_{\rm Frob}$, it is more convenient to introduce:

\medskip
\noindent
{\sc The parameter $Q(x)$:} For $x\in\nuh$, let
$Q(x):=\ve^{3/\beta}\|C(x)^{-1}\|_{\rm Frob}^{-12/\beta}$.

\medskip
The hyperbolicity degenerates as $Q$ goes to $0$. Lemma~\ref{Lemma-linear-reduction}
immediately implies that
\begin{equation}\label{ratio-Q}
\tfrac{Q(\vf^t(x))}{Q(x)}=e^{\pm\frac{250\rho}{\beta}}, \ \ \forall x\in \nuh,\forall\, 0<t\leq 2\rho,
\end{equation}
and the following result.

\begin{proposition}
An invariant set $K\subset \nuh$ is uniformly hyperbolic if and only if $\inf\limits_{x\in K}Q(x)>0$.
\end{proposition}

\medskip
It will be important to identify the orbits in $\nuh$ whose hyperbolicity satisfies some recurrence
(to ensure e.g. the existence of stable and unstable manifolds) and ask how fast
$Q(\vf^t(x))$ can go to zero when $k\to \pm\infty$.
For that reason, we introduce:

\medskip
\noindent
{\sc The parameters $q(x),q^s(x),q^u(x)$:}  For $x\in\nuh$, define:
\begin{align*}
q(x)&:=\ve\inf\{e^{\ve|t|}Q(\vf^t(x)):t\in\R\}\\
q^s(x)&:=\ve\inf\{e^{\ve|t|}Q(\vf^t(x)):t\geq 0\}\\
q^u(x)&:=\ve\inf\{e^{\ve|t|}Q(\vf^t(x)):t\leq 0\}.
\end{align*}

\medskip
Clearly $0\leq q(x),q^s(x),q^u(x)\leq\ve Q(x)$, hence
these parameters are much smaller than $Q(x)$. Also,
$q^s(x)\wedge q^u(x)=q(x)$.
The families $\{q^s(\vf^t(x))\}_{t\in\R}$ and $\{q^u(\vf^t(x))\}_{t\in\R}$
represent the \emph{local quantifications of hyperbolicity} along the orbit $\{\vf^t(x)\}_{t\in\R}$.
We collect the following simple lemma, for later use.

\begin{lemma}\label{Lemma-q}
For all $x\in\nuh$ and $t\in\R$, it holds ${q(\vf^t(x))}=e^{\pm\ve |t|}q(x)$.
\end{lemma}

\begin{proof}
Using that $|t'|=|t'+t|\pm|t|$, we have
$$
q(\vf^t(x))=e^{\pm \ve|t|}\ve \inf\{e^{\ve|t'+t|}Q(\vf^{t'+t}(x)):t'\in\R\}=e^{\pm\ve|t|}q(x).$$
The proof is complete.
\end{proof}

\subsection{The recurrently non-uniformly hyperbolic locus $\nuh^\#$}\mbox{}

\medskip
\noindent
{\sc Recurrently non-uniformly hyperbolic locus  $\nuh^\#=\nuh^\#(\vf,\chi,\rho,\theta,\ve)$:} 
It is the invariant set of points $x\in\nuh$ such that:
\begin{enumerate}[(NUH3)]
\item[(NUH3)] $q(x)>0$.
\item[(NUH4)] $\limsup\limits_{t\to+\infty}q(\vf^t(x))>0$ and $\limsup\limits_{t\to-\infty}q(\vf^t(x))>0$.
\end{enumerate}

\medskip
Note that if (NUH3) holds then $q(\vf^t(x)),q^s(\vf^t(x)),q^u(\vf^t(x))$ are positive for all $t\in\R$.
Condition (NUH4) requires that these values do not degenerate to zero in the limit. The set $\nuh^\#$
carries all $\chi$--hyperbolic measures, as we now prove.

\begin{proposition}\label{Prop-adaptedness}
If $\mu$ is a $\vf$--invariant probability measure with $\mu[\nuh]=1$, then $\mu[\nuh^\#]=1$. In particular,
if $\mu$ is $\chi$--hyperbolic then $\mu[\nuh^\#]=1$.
\end{proposition} 

\begin{proof}
Note that
$$
\lim\limits_{n\to\pm\infty}\tfrac{1}{n}\log Q(\vf^{n\rho}(x))=0 \ \Longrightarrow \ 
\lim\limits_{t\to\pm\infty}\tfrac{1}{t}\log Q(\vf^t(x))=0 \ \Longrightarrow \ q(x)>0.
$$
To establish the first limit, we use the following basic fact of ergodic theory.

\medskip
\noindent
{\sc Fact:} Let $(X,\mu,T)$ be an invertible probability-preserving system, and $u:X\to(0,+\infty)$ measurable.
If there is $C>0$ such that $C^{-1}\leq \tfrac{u(Tx)}{u(x)}\leq C$ for $\mu$--a.e. $x\in X$, then
$\lim\limits_{n\to\pm\infty}\tfrac{1}{n}\log u(T^nx)=0$ for $\mu$--a.e. $x\in X$.

\begin{proof}[Proof of the fact.]
By the Poincar\'e recurrence theorem, $\liminf\limits_{n\to\pm\infty}u(T^nx)<+\infty$ a.e., 
hence $\liminf\limits_{n\to\pm\infty}\tfrac{1}{n}\log u(T^nx)=0$ a.e. Now, 
applying the Birkhoff ergodic theorem to the bounded function $U:=\log u\circ T-\log u$,
$\lim\limits_{n\to\pm\infty}\tfrac{1}{n}\log u(T^nx)$ exists a.e. Therefore 
$$
\lim\limits_{n\to\pm\infty}\tfrac{1}{n}\log u(T^nx)=\liminf\limits_{n\to\pm\infty}\tfrac{1}{n}\log u(T^nx)=0
$$
for $\mu$--a.e. $x\in X$.
\end{proof}

Now we prove the proposition. Assume that $\mu[\nuh]=1$. By (\ref{ratio-Q}),
$\tfrac{Q(\vf^\rho(x))}{Q(x)}=e^{\pm\frac{250\rho}{\beta}}$
for all $x\in\nuh$. Applying the Fact to the transformation $T=\vf^\rho$ and the function $Q$,
we conclude that (NUH3) holds $\mu$--a.e. Finally, by the Poincar\'e recurrence theorem,
(NUH4) also holds $\mu$--a.e.
\end{proof}

\subsection{The $\Z$--indexed versions of $\mathbf{ q^{s/u}(x)}$: the parameters $\mathbf{p^{s/u}(x)}$}
\label{section-Z-indexed}
We now define discrete time approximate versions of $q^s(x),q^u(x)$
that satisfy recursive explicit formulas, which we call {\em $\Z$--indexed versions} of $q^s(x),q^u(x)$.
Recall that $r_\Lambda$ is the Poincar\'e return time of the proper section $\Lambda$ of size $\rho/2$.
In particular, $0<\inf(r_\Lambda)\leq \sup(r_\Lambda)\leq\rho/2$.

\medskip
\noindent
{\sc $\Z$--indexed versions of $q^s,q^u$:} Let $x\in\nuh$. For each sequence
$\mathcal T=\{t_n\}_{n\in\Z}$ of real numbers with $\tfrac{1}{2}\inf(r_\Lambda)\leq t_{n+1}-t_n\leq 2\sup(r_\Lambda)$,
define:
\begin{align*}
p^s(x,\mathcal T,n)&:=\ve\inf\{e^{\ve(t_m-t_n)}Q(\vf^{t_m}(x)):m\geq n\} \\
p^u(x,\mathcal T,n)&:=\ve\inf\{e^{\ve(t_n-t_m)}Q(\vf^{t_m}(x)):m\leq n\}.
\end{align*}

\medskip
Clearly, $p^{s/u}(x,\mathcal T,n)\geq q^{s/u}(\vf^{t_n}(x))$.
As the choice of $\mathcal T$ will be always clear in the context, 
we will simply write
$p^{s/u}(\vf^{t_n}(x))$ for $p^{s/u}(x,\mathcal T,n)$. As a matter of fact, although the values
$p^{s/u}(\vf^{t_n}(x))$ do depend on the choice of $\mathcal T$,
they are not very sensitive to this choice.

\begin{proposition}\label{Prop-Z-par}
The following holds for all $x\in\nuh^\#$ and $\mathcal T= \{t_n\}_{n\in\Z}$ with 
$\tfrac{1}{2}\inf(r_\Lambda)\leq t_{n+1}-t_n\leq 2\sup(r_\Lambda)$.

\begin{enumerate}[{\rm (1)}]
\item {\sc Robustness:} Let $\mathfrak H:=\ve\rho+\tfrac{250\rho}{\beta}$. For all $n\in\Z$ and $t\in[t_n,t_{n+1}]$,
it holds:
$$
\frac{p^{s/u}(\vf^{t_n}(x))}{q^{s/u}(\vf^t(x))}=e^{\pm\mathfrak H}.
$$
\item {\sc Greedy algorithm:} For all $n\in\Z$ it holds:
\begin{align*}
p^s(\vf^{t_n}(x))&=\min\left\{e^{\ve(t_{n+1}-t_n)}p^s(\vf^{t_{n+1}}(x)),\ve Q(\vf^{t_n}(x))\right\}\\
p^u(\vf^{t_n}(x))&=\min\left\{e^{\ve(t_n-t_{n-1})}p^u(\vf^{t_{n-1}}(x)),\ve Q(\vf^{t_n}(x))\right\}.
\end{align*}
In particular:
\begin{align*}
\ve Q(\vf^{t_n}(x)) \geq\;\; & p^s(\vf^{t_n}(x))\geq e^{-\ve(t_n-t_m)}p^s(\vf^{t_m}(x)),\ \forall n\geq m, \\
\ve Q(\vf^{t_n}(x)) \geq\;\; & p^u(\vf^{t_n}(x))\geq e^{-\ve(t_m-t_n)}p^s(\vf^{t_m}(x)),\ \forall m\geq n.
\end{align*}
\item {\sc Maximality:} $p^s(\vf^{t_n}(x))=\ve Q(\vf^{t_n}(x))$
for infinitely many $n>0$, and $p^u(\vf^{t_n}(x))=\ve Q(\vf^{t_n}(x))$
for infinitely many $n<0$. 
\end{enumerate}
\end{proposition}

\begin{proof}
We prove the statements for $p^s$ (the proofs for $p^u$ are analogous).

\medskip
\noindent

\noindent
(1) Fix $x\in\nuh^\#$, $n\in\Z$, $t\in [t_n,t_{n+1}]$.
By Lemma \ref{Lemma-q}, we have $\tfrac{p^s(\vf^{t_n}(x))}{q^s(\vf^t(x))}=\tfrac{p^s(\vf^{t_n}(x))}{q^s(\vf^{t_n}(x))}
\cdot\tfrac{q^s(\vf^{t_n}(x))}{q^s(\vf^t(x))}=e^{\pm\ve\rho}\tfrac{p^s(\vf^{t_n}(x))}{q^s(\vf^{t_n}(x))}$,
hence we need to estimate $\tfrac{p^s(\vf^{t_n}(x))}{q^s(\vf^{t_n}(x))}$.
For $m\geq n$, let $\gamma_m:=e^{\ve(t_m-t_n)}Q(\vf^{t_m}(x))$ and
$\delta_m:=\inf\{e^{\ve(t-t_n)}Q(\vf^t(x)):t_m\leq t\leq t_{m+1}\}$. By definition,
we have $p^s(\vf^{t_n}(x))=\ve\inf\{\gamma_m:m\geq n\}$ and $q^s(\vf^{t_n}(x))=\ve\inf\{\delta_m:m\geq n\}$.
Since $\tfrac{Q(\vf^t(x))}{Q(\vf^{t_m}(x))}=e^{\pm\frac{250\rho}{\beta}}$ for
$t_m\leq t\leq t_{m+1}$ (see (\ref{ratio-Q})), we get:
\begin{align*}
&\ \gamma_m\geq \delta_m=e^{\ve(t_m-t_n)}\inf\{e^{\ve(t-t_m)}Q(\vf^t(x)):t_m\leq t\leq t_{m+1}\}\\
&\geq e^{\ve(t_m-t_n)}e^{-\frac{250\rho}{\beta}}Q(\vf^{t_m}(x))=e^{-\frac{250\rho}{\beta}}\gamma_m.
\end{align*}
Hence $1\leq \tfrac{p^s(\vf^{t_n}(x))}{q^s(\vf^{t_n}(x))}\leq e^{\frac{250\rho}{\beta}}$
and so $\tfrac{p^s(\vf^{t_n}(x))}{q^s(\vf^t(x))}=e^{\pm\mathfrak H}$.

\medskip
\noindent
(2) We have
\begin{align*}
&\ p^s(\vf^{t_n}(x))=\ve\inf\left\{e^{\ve(t_m-t_n)}Q(\vf^{t_m}(x)):m\geq n\right\}\\
&=\min\left\{\ve\inf\left\{e^{\ve(t_m-t_n)}Q(\vf^{t_m}(x)):m\geq n+1\right\},\ve Q(\vf^{t_n}(x))\right\}\\
&=\min\left\{e^{\ve(t_{n+1}-t_n)}p^s(\vf^{t_{n+1}}(x)),\ve Q(\vf^{t_n}(x))\right\},
\end{align*}
which proves the recursive relation. Clearly $p^s\leq \ve Q$. For, the other side of the inequality,
note that if $n\geq m$ then:
\begin{align*}
&\ p^s(\vf^{t_n}(x))=\ve\inf\{e^{\ve(t_\ell-t_n)}Q(\vf^{t_\ell}(x)):\ell\geq n\}\\
&=e^{-\ve(t_n-t_m)}\ve\inf\{e^{\ve(t_\ell-t_m)}Q(\vf^{t_\ell}(x)):\ell\geq n\}\\
&\geq e^{-\ve(t_n-t_m)}\ve\inf\{e^{\ve(t_\ell-t_m)}Q(\vf^{t_\ell}(x)):\ell\geq m\}=e^{-\ve(t_n-t_m)}p^s(\vf^{t_m}(x)).
\end{align*}

\medskip
\noindent
(3) The proof is based on \cite[Prop. 8.3]{Sarig-JAMS}.
Since $x\in\nuh^\#$, $\limsup_{t\to\infty}q^s(\vf^t(x))>0$.
By part (1), $\limsup_{n\to\infty}p^s(\vf^{t_n}(x))>0$ hence 
$\exists\delta_0>0$ such that $p^s(\vf^{t_n}(x))>\delta_0$ for infinitely many $n>0$.
By contradiction, assume $\exists n_0>0$ such that
$p^s(\vf^{t_n}(x))<\ve Q(\vf^{t_n}(x))$ for all $n\geq n_0$. By the greedy algorithm on part (2),
$p^s(\vf^{t_n}(x))=e^{\ve(t_{n+1}-t_n)}p^s(\vf^{t_{n+1}}(x))$ for all $n\geq n_0$.
This implies that $p^s(\vf^{t_{n_0}}(x))=e^{\ve(t_{n_0+\ell}-t_{n_0})}p^s(\vf^{t_{n_0+\ell}}(x))$
for all $\ell\geq 0$, hence $p^s(\vf^{t_{n_0}}(x))>e^{\ve(t_{n_0+\ell}-t_{n_0})}\delta_0$
for infinitely many $\ell\geq 0$, which is a contradiction since $e^{\ve(t_{n_0+\ell}-t_{n_0})}\to\infty$ as $\ell\to\infty$.
\end{proof}

\subsection{Pesin charts $\Psi_x$} Recall that $R[\mathfrak r]:=[-\mathfrak r,\mathfrak r]^2\subset\R^2$. 
We define Pesin charts for $x\in\Lambda\cap\nuh$.

\medskip
\noindent
{\sc Pesin chart at $x$:} It is the map $\Psi_x:R[\mathfrak r]\to \widehat{\Lambda}$ defined by
$\Psi_x:=\exp{x}\circ C(x)$.

\medskip
The center $x$ of the Pesin chart $\Psi_x$ always belongs to the reference section $\Lambda$,
while its image is contained in the security section $\widehat{\Lambda}$. 
In particular, when $x$ is close to the boundary of $\Lambda$, the image of $\Psi_x$ is {\em not}
contained in $\Lambda$. This definition is different from \cite{Lima-Sarig}, and it is the first
step to bypass the boundary effect mention in Section \ref{ss-method-proof}.

For $x\in \widehat{\Lambda}$, let $\iota_x:T_x\widehat{\Lambda}\to\R^2$ be an isometry.
If $x\in\Lambda,y\in \widehat{\Lambda}$ with $d(x,y)\leq 2\mathfrak r$,
we consider as in section~\ref{s.metric} an isometry $P_{y,x}:T_yM\to T_xM$.
If $A:\R^2\to T_y\widehat{\Lambda}$
is a linear map, we define $\widetilde{A}:\R^2\to \R^2$ by $\widetilde{A}:=\iota_x\circ P_{y,x}\circ A$.
The map $\widetilde A$ depends on $x$ but $\|\widetilde{A}\|$ does not.

\begin{lemma}\label{Lemma-Pesin-chart}
For all $x\in\Lambda\cap\nuh$, the Pesin chart $\Psi_x$ is a diffeomorphism onto its image and:
\begin{enumerate}[{\rm (1)}]
\item $\Psi_x$ is $2$--Lipschitz and $\Psi_x^{-1}$ is $2\|C(x)^{-1}\|$--Lipschitz.
\item $\|\widetilde{d(\Psi_x)_{v_1}}-\widetilde{d(\Psi_x)_{v_2}}\|\leq \mathfrak K\|v_1-v_2\|$
for all $v_1,v_2\in R[\mathfrak r]$.
\end{enumerate}
\end{lemma}

\begin{proof}
Since $C(x)$ is a contraction,
$C(x)R[\mathfrak r]\subset B_x[2\mathfrak r]$
and so $\Psi_x$ is well-defined with inverse $C(x)^{-1}\circ \exp{x}^{-1}$.
It is a diffeomorphism because $C(x)$ and $\exp{x}$ are.

\medskip
\noindent
(1) $C(x)$ is a contraction and $\exp{x}$ is 2--bi-Lipschitz in $B_x[2\mathfrak r]$.
Therefore $\Psi_x$ is $2$--Lipschitz and $\Psi_x^{-1}$ is $2\|C(x)^{-1}\|$--Lipschitz.

\medskip
\noindent
(2) Since $C(x)v_i\in B_x[2\mathfrak r]$, condition (Exp3) gives that
\begin{align*}
&\ \|\widetilde{d(\Psi_x)_{v_1}}-\widetilde{d(\Psi_x)_{v_2}}\|=
\|\widetilde{d(\exp{x})_{C(x)v_1}}\circ C(x)-
\widetilde{d(\exp{x})_{C(x)v_2}}\circ C(x)\|\\
&\leq \mathfrak K\|C(x)v_1-C(x)v_2\|\leq \mathfrak K\|v_1-v_2\|.
\end{align*}
The proof is complete.
\end{proof}

\subsection{Holonomy maps $g_x^{\pm}$ in Pesin charts}
The parameter $Q(x)$ defines the size of the domain where we can control $g_x^{\pm}$ in Pesin charts:
in these coordinates, $g_x^{\pm}$ are small perturbations of hyperbolic linear maps.

\newcounter{thm-nlp}
\setcounter{thm-nlp}{\value{theorem}}

\begin{theorem}\label{Thm-non-linear-Pesin}
The following holds for all $\ve>0$ small enough. For all $x\in\Lambda\cap\nuh$
the map $f_x^+:=\Psi_{f(x)}^{-1}\circ g_x^+\circ\Psi_x$ is well-defined on
$R[10Q(x)]$ and satisfies:
\begin{enumerate}[{\rm (1)}]
\item $d(f_x^+)_0=C(f(x))^{-1}\circ \Phi^{r_\Lambda(x)}\circ C(x)$ and
$e^{-4\rho}<m(d(f_x^+)_0)\leq \|d(f_x^+)_0\|< e^{4\rho}$.
\item $f_x^+=\begin{bmatrix}A & 0 \\ 0 & B\end{bmatrix}+H$ where:
\begin{enumerate}[{\rm (a)}]
\item $e^{-4\rho}<|A|<e^{-\chi r_\Lambda(x)}$ and $e^{\chi r_\Lambda(x)}<|B|<e^{4\rho}$,
cf. Lemma \ref{Lemma-linear-reduction}{\rm (2)}.
\item $H(0)=0$ and $dH_0=0$.
\item $\|H\|_{C^{1+\frac{\beta}{2}}}<\ve$.
\end{enumerate}
\end{enumerate}
A similar statement holds for $f_x^-:=\Psi_x^{-1}\circ g_{f(x)}^-\circ \Psi_{f(x)}$.
\end{theorem}

The proof is in Appendix \ref{Appendix-proofs}.

\subsection{The overlap condition}\label{section-overlap}

We now control the change coordinates from $\Psi_x$ to $\Psi_y$ when $x,y$
are ``sufficiently close''. This can only be made when both $x,y$ and
$C(x),C(y)$ are very close. In the sequel we will make extensive use of Pesin
charts with different domains.

\medskip
\noindent
{\sc Pesin chart $\Psi_x^\eta$:} It is restriction of $\Psi_x$ to $R[\eta]$, where $0<\eta\leq Q(x)$.
\medskip

Recall that $d$ is the distance on $\widehat \Lambda$
associated to the induced Riemannian metric.

\medskip

\noindent
{\sc $\ve$--overlap:} We say that two Pesin charts $\Psi_{x_1}^{\eta_1},\Psi_{x_2}^{\eta_2}$
{\em $\ve$--overlap} if $\tfrac{\eta_1}{\eta_2}=e^{\pm\ve}$ and
$d(x_1,x_2)+\|\widetilde{C(x_1)}-\widetilde{C(x_2)}\|<(\eta_1\eta_2)^4$.
In particular,  $x_1,x_2$ belong to the same local connected component of $\Lambda$.
We write $\Psi_{x_1}^{\eta_1}\overset{\ve}{\approx}\Psi_{x_2}^{\eta_2}$.

\begin{lemma}
The following holds for $\ve>0$ small. If
$\Psi_{x_1}^{\eta_1}\overset{\ve}{\approx}\Psi_{x_2}^{\eta_2}$, then
$$\Psi_{x_i}(R[10Q(x_i)])\subset B_{x_1}\cap B_{x_2}.$$
\end{lemma}

In particular, it makes sense to consider $\|\widetilde{C(x_1)}-\widetilde{C(x_2)}\|$.

\begin{proof}
Let $i=1$.
By Lemma \ref{Lemma-Pesin-chart}(1), $\Psi_{x_1}(R[10Q(x_1)])\subset B(x_1,40Q(x_1))$. This latter ball
is contained in $B_{x_1}$ since $40 Q(x_1)<40\ve^{3/\beta}<2\mathfrak r$ when $\ve>0$ is small.
Also:
$$
\Psi_{x_1}(R[10Q(x_1)])\subset B(x_1,40 Q(x_1))\subset
B(x_2,40 Q(x_1)+d(x_1,x_2)).
$$
Since $40Q(x_1)+d(x_1,x_2)< 40\ve^{3/\beta}+\ve^{24/\beta}<2\mathfrak r$
for small $\ve>0$, $\Psi_{x_1}(R[10Q(x_1)])\subset B_{x_2}$.
\end{proof}
The next result guarantees that $\ve$--overlap is strong enough to change coordinates.

\begin{proposition}\label{Lemma-overlap}
The following holds for $\ve>0$ small.
If $\Psi_{x_1}^{\eta_1}\overset{\ve}{\approx}\Psi_{x_2}^{\eta_2}$ then:
\begin{enumerate}[{\rm (1)}]
\item {\sc Control of $s,u$:}
$\frac{s(x_1)}{s(x_2)}=e^{\pm(\eta_1\eta_2)^3}$ and $\frac{u(x_1)}{u(x_2)}=e^{\pm(\eta_1\eta_2)^3}$.
\item {\sc Control of $\alpha$:} $\frac{|\sin\alpha(x_1)|}{|\sin\alpha(x_2)|}=e^{\pm(\eta_1\eta_2)^3}$.
\item {\sc Overlap:} $\Psi_{x_i}(R[e^{-2\ve}\eta_i])\subset \Psi_{x_j}(R[\eta_j])$ for $i,j=1,2$.
\item {\sc Change of coordinates:} For $i,j=1,2$, the map $\Psi_{x_i}^{-1}\circ\Psi_{x_j}$
is well-defined in $R[\mathfrak r]$,
and $\|\Psi_{x_i}^{-1}\circ\Psi_{x_j}-{\rm Id}\|_{C^2}<\ve(\eta_1\eta_2)^2$
where the norm is taken in $R[\mathfrak r]$.
\end{enumerate}
\end{proposition}

The proof is in Appendix \ref{Appendix-proofs}.

\subsection{The maps $f_{x,y}^+,f_{x,y}^-$}

Let $x,y\in\Lambda\cap\nuh$ such that $\Psi_{f(x)}^{\eta}\overset{\ve}{\approx}\Psi_y^{\eta'}$.
In this section, we change $\Psi_{f(x)}$ by $\Psi_y$ in the definition of $f_x^+$ and obtain a result
similar to Theorem \ref{Thm-non-linear-Pesin}.

\medskip
\noindent
{\sc The maps $f_{x,y}^+$ and $f_{x,y}^-$:} If $\Psi_{f(x)}^{\eta}\overset{\ve}{\approx}\Psi_y^{\eta'}$,
we define the map $f_{x,y}^+:=\Psi_y^{-1}\circ g_x^+\circ \Psi_x$.
If $\Psi_{x}^{\eta}\overset{\ve}{\approx}\Psi_{f^{-1}(y)}^{\eta'}$, we define
$f_{x,y}^-:=\Psi_x^{-1}\circ g_y^-\circ \Psi_y$. 

\medskip
Since any meaningful estimate of $f_{x,y}^{\pm}$ in the $C^{1+\beta/2}$ norm cannot be better than
that of Theorem \ref{Thm-non-linear-Pesin}, and to keep estimates of size $\ve$, we
consider the $C^{1+\beta/3}$ norm of $f_{x,y}^\pm$.

\newcounter{keep-thm}
\setcounter{keep-thm}{\value{theorem}}
\setcounter{theorem}{\value{thm-nlp}}
\renewcommand{\thetheorem}{\arabic{section}.\arabic{theorem}'}

\begin{theorem}\label{Thm-non-linear-Pesin-2}
The following holds for all $\ve>0$ small enough.
If $x,y\in\Lambda\cap\nuh$ and $\Psi_{f(x)}^{\eta}\overset{\ve}{\approx}\Psi_{y}^{\eta'}$, then
$f_{x,y}^+$ is well-defined on $R[10Q(x)]$ and can be written as
$f_{x,y}^+=\begin{bmatrix}A & 0 \\ 0 & B\end{bmatrix}+H$ where:
\begin{enumerate}[{\rm (1)}]
\item $e^{-4\rho}<|A|<e^{-\chi r_\Lambda(x)}$, $e^{\chi r_\Lambda(x)}<|B|<e^{4\rho}$, cf. Lemma \ref{Lemma-linear-reduction}{\rm (2)}.
\item For $i=1,2$, it holds $\|H(0)\|<\ve\eta$, $\|dH_0\|<\ve\eta^{\beta/3}$,
$\Hol{\beta/3}(dH)<\ve$.
\end{enumerate}
If $\Psi_{x}^{\eta}\overset{\ve}{\approx}\Psi_{f^{-1}(y)}^{\eta'}$
then a similar statement holds for $f_{x,y}^-$.
\end{theorem}
\renewcommand{\thetheorem}{\arabic{section}.\arabic{theorem}}
\begin{proof}
We write $f_{x,y}^+=(\Psi_y^{-1}\circ\Psi_{f(x)})\circ f_x^+=:g\circ f_x^+$ and see it as a
small perturbation of $f_x^+$. By Theorem \ref{Thm-non-linear-Pesin},
$$
f_x^+(0)=0,\ \|d(f_x^+)\|_{C^0}<2e^{4\rho},\ \|d(f_x^+)_v-d(f_x^+)_w\|\leq \ve\|v-w\|^{\beta/2},\forall v,w\in R[10Q(x)],
$$
where the $C^0$ norm is taken in $R[10Q(x)]$,
and by Proposition \ref{Lemma-overlap}(4) we have
$$
\|g-{\rm Id}\|<\ve(\eta\eta')^2,\ \|d(g-{\rm Id})\|_{C^0}<\ve(\eta\eta')^2,\ \|dg_v-dg_w\|\leq\ve(\eta\eta')^2\|v-w\|^{\beta/2}
$$
for $v,w\in R[\mathfrak r]$, where the $C^0$ norm is taken in $R[\mathfrak r]$.

We first prove that $f_{x,y}^+$ is well-defined on $R[10Q(x)]$. For $\ve>0$ small enough we have
$f_x^+(R[10Q(x)])\subset B(0,40e^{4\rho}Q(x))\subset R[\mathfrak r]$
since $40e^{4\rho}Q(x)<40e^{4\rho}\ve^{3/\beta}<\mathfrak r$.
By Proposition \ref{Lemma-overlap}(4), $f_{x,y}^+$ is well-defined.
 
Letting $A,B$ as in Lemma  \ref{Lemma-linear-reduction}, part (1) is clear, so we focus on part (2).
We have $\|H(0)\|=\|g(0)\|<\ve(\eta\eta')^2<\ve\eta$
and for $\ve>0$ small enough:
$$
\|dH_0\|\leq \|dg_0\circ d(f_x^+)_0-d(f_x^+)_0\|\leq \|d(g-{\rm Id})_0\|\|d(f_x^+)_0\|
<\ve(\eta\eta')^2 e^{4\rho}<\ve\eta^{\beta/3}.
$$
Finally, since $f_x^+(R[10Q(x)])\subset R[\mathfrak r]$, if $\ve>0$ is small then
for $v,w\in R[10Q(x)]$ it holds:
\begin{align*}
&\ \|dH_v-dH_w\|=\|dg_{f_x^+(v)}\circ d(f_x^+)_v-dg_{f_x^+(w)}\circ d(f_x^+)_w\|\\
&\leq \|dg_{f_x^+(v)}-dg_{f_x^+(w)}\|\|d(f_x^+)_v\|+\|dg_{f_x^+(w)}\|\|d(f_x^+)_v-d(f_x^+)_w\|\\
&\leq \ve(\eta\eta')^2\|f_x^+(v)-f_x^+(w)\|^{\beta/2}\|d(f_x^+)\|_{C^0}+\ve\|dg\|_{C^0}\|v-w\|^{\beta/2}\\
&\leq \left[\ve(\eta\eta')^2\|d(f_x^+)\|_{C^0}^{1+\beta/2}+40\ve\|dg\|_{C^0}Q(x)^{\beta/6}\right]\|v-w\|^{\beta/3}\\
&\leq \left[\eta^2\eta'^2(2e^{4\rho})^{1+\beta/2}+ 80Q(x)^{\beta/6}\right]\ve\|v-w\|^{\beta/3}\\
&\leq \left[\ve^{12/\beta}(2e^{4\rho})^{1+\beta/2}+80\ve^{1/2}\right]\ve\|v-w\|^{\beta/3}<\ve\|v-w\|^{\beta/3}.
\end{align*}
The proof is now complete.
\end{proof}

\section{Invariant manifolds and shadowing}

Up to now, we have fixed $\vf,\chi,\rho,\Lambda,\widehat\Lambda,\theta$ and $\ve$, where
$\rho,\ve$ are small parameters. In this section, we:
\begin{enumerate}[(1)]
\item Define \emph{$\ve$--double charts} $\Psi_x^{p^s,p^u}$, which are double versions of Pesin charts
whose stable and unstable sizes $p^s,p^u$ may differ. The parameters $p^s/p^u$ control separately the 
local stable/unstable hyperbolicity at $x$.
\item Define \emph{generalized pseudo-orbit}, which is a sequence $\underline v$ of $\varepsilon$--double charts
satisfying {\em edge conditions}, which are nearest neighbor conditions relating the parameters of consecutive
$\ve$--double charts.
\item Associate to each generalized pseudo-orbit its \emph{local stable and unstable manifolds}
$V^s[\underline v]$ and $V^u[\underline v]$. As a consequence, we obtain a \emph{shadowing lemma}.
\end{enumerate}

\subsection{Pseudo-orbits}\label{ss.pseudo.orbits}\mbox{}

\medskip
\noindent
{\sc $\ve$--double chart:} An {\em $\ve$--double chart} is a pair of Pesin charts
$\Psi_x^{p^s,p^u}=(\Psi_x^{p^s},\Psi_x^{p^u})$ where
$0<p^s,p^u\leq \ve Q(x)$.

\medskip
The parameters $p^s/p^u$ are local quantifications of the hyperbolicity at $x$.
One can think of them as a definite size for the stable and unstable manifolds at $x$.
\medskip

\noindent
{\sc Transition time:}
For two $\ve$--double charts $v=\Psi_x^{p^s,p^u}$, $w=\Psi_y^{q^s,q^u}$
we define $T(v,w)$ by
$$
\min\left\{\min\{T^+(z):z\in \Psi_x(R[\tfrac{1}{20}(p^s\wedge p^u)])\},
\\ \min\{-T^-(z):z\in \Psi_y(R[\tfrac{1}{20}(q^s\wedge q^u)])\}\right\},
$$
where $T^+:B_x\to\R$ and $T^-:B_y\to\R$ are the $C^{1+\beta}$ functions satisfying
$g_x^+=\vf^{T^+}$, $g_{f^{-1}(y)}^{-}=\vf^{T^-}$ with $T^+(x)=r_\Lambda(x)$ and $T^-(y)=-r_\Lambda(f^{-1}(y))$. 

\medskip
\noindent
{\sc Edge $v\overset{\ve}{\rightarrow}w$:} Given two $\ve$--double charts $v=\Psi_x^{p^s,p^u}$,
$w=\Psi_y^{q^s,q^u}$, we draw an {\em edge} from $v$ to $w$ if the two following conditions are
satisfied:
\medskip

\begin{enumerate}[iii\,]
\item[(GPO1)] $\Psi_{f(x)}^{q^s\wedge q^u}\overset{\ve}{\approx}\Psi_y^{q^s\wedge q^u}$ and
$\Psi_{f^{-1}(y)}^{p^s\wedge p^u}\overset{\ve}{\approx}\Psi_x^{p^s\wedge p^u}$.
\smallskip

\item[(GPO2)] The following estimates hold:
\begin{align}
\label{gpo2-a}&
e^{-\ve p^s}\min\{e^{\ve T(v,w)}q^s,e^{-\ve}\ve Q(x)\}\leq p^s\leq \min\{e^{\ve T(v,w)}q^s,\ve Q(x)\}\\
\label{gpo2-b}&
e^{-\ve q^u}\min\{e^{\ve T(v,w)}p^u,e^{-\ve}\ve Q(y)\}\leq q^u\leq \min\{e^{\ve T(v,w)}p^u,\ve Q(y)\}.
\end{align}
\end{enumerate}

\begin{remark}\label{rmk-time}
In the above notation, if $v\overset{\ve}{\rightarrow} w$ then by Theorem \ref{Thm-non-linear-Pesin-2} 
we have
$$
g_y^{-}(\Psi_y(R[\tfrac{1}{20}(q^s\wedge q^u)]))\subset \Psi_x(R[\tfrac{1}{15}(p^s\wedge p^u)])
$$
and so $T(v,w)=T^+(z)$ for some $z\in \Psi_x(R[\tfrac{1}{15}(p^s\wedge p^u)])$.
In particular, $T(v,w)\leq \rho$.
\end{remark}

\medskip
\noindent
{\sc $\ve$--generalized pseudo-orbit ($\ve$--gpo):} An {\em $\ve$--generalized pseudo-orbit ($\ve$--gpo)}
is a sequence $\un{v}=\{v_n\}_{n\in\Z}$ of $\ve$--double charts
such that $v_n\overset{\ve}{\rightarrow}v_{n+1}$ for all $n\in\Z$. We say that
$\un v$ is {\em regular} if there are $v,w$ such that $v_n=v$ for infinitely many $n>0$ and
$v_n=w$ for infinitely many $n<0$.

\medskip
\noindent
{\sc Positive and negative $\ve$--gpo:}
A {\em positive $\ve$--gpo} is a sequence $\un{v}^+=\{v_n\}_{n\geq 0}$ of $\ve$--double charts
such that $v_n\overset{\ve}{\rightarrow}v_{n+1}$ for all $n\geq 0$. A {\em negative $\ve$--gpo}
is a sequence $\un{v}^-=\{v_n\}_{n\leq 0}$ of $\ve$--double charts
such that $v_n\overset{\ve}{\rightarrow}v_{n+1}$ for all $n\leq -1$. 

\medskip
Condition (GPO1) allows to pass from an $\ve$--double chart at $x$
to an $\ve$--double chart at $y$ and vice-versa. Condition (GPO2) is a greedy recursion
that implies that the local quantifications of hyperbolicity are ``as large as possible''.
The need of (GPO2) will be clear in the proof of Theorem \ref{Thm-coarse-graining} (coarse graining)
and Theorem \ref{Thm-inverse} (inverse theorem).

\begin{lemma}\label{Lemma-minimum}
If $v=\Psi_x^{p^s,p^u},w=\Psi_y^{q^s,q^u}$ are $\ve$--double charts
satisfying {\rm (GPO2)} then $\tfrac{p^s\wedge p^u}{q^s\wedge q^u}=e^{\pm 2\ve}$.
\end{lemma}

\begin{proof}
We have $e^{-\ve p^s}\min\{e^{\ve T(v,w)}q^s,e^{-\ve}\ve Q(x)\}\leq p^s\leq
\min\{e^{\ve T(v,w)}q^s,\ve Q(x)\}$, therefore
$e^{-\ve p^s}\min\{e^{\ve T(v,w)}q^s,e^{-\ve} p^u\}\leq p^s\wedge p^u\leq \min\{e^{\ve T(v,w)}q^s, p^u\}$
and so
$$
e^{-\ve-\ve p^s}\min\{e^{\ve T(v,w)}q^s,p^u\}\leq p^s\wedge p^u\leq \min\{e^{\ve T(v,w)}q^s, p^u\}.
$$
By the same reason,
$e^{-\ve-\ve q^u}\min\{e^{\ve T(v,w)}p^u,q^s\}\leq q^s\wedge q^u\leq \min\{e^{\ve T(v,w)}p^u,q^s\}$
hence
$$
e^{-\ve-\ve q^u-\ve T(v,w)}\min\{e^{\ve T(v,w)}q^s,p^u\}\leq q^s\wedge q^u\leq e^{\ve T(v,w)}\min\{e^{\ve T(v,w)}q^s,p^u\}.
$$
Together, these inequalities imply that
$$
e^{-\ve[1+p^s+T(v,w)]}\leq\tfrac{p^s\wedge p^u}{q^s\wedge q^u}\leq e^{\ve[1+q^u+T(v,w)]}.
$$
Since $p^s,q^u<\ve<0.25$ and $T(v,w)\leq\rho<0.25$,
it follows that $\tfrac{p^s\wedge p^u}{q^s\wedge q^u}=e^{\pm 2\ve}$.
\end{proof}

\begin{remark}
There is a big difference between (GPO2) above and all previous definitions
used in \cite{Sarig-JAMS,Lima-Matheus,Ben-Ovadia-high-dimension,Lima-Sarig,Lima-AIHP,ALP}.
The first is that we we only require inequalities, while previous work required equalities. One reason
is the following: while for diffeomorphisms the hyperbolicity acquired in an edge $v\overset{\ve}{\rightarrow} w$
is at least $e^\ve$, for flows
it is at least $e^{\ve T(v,w)}$. Since $T(v,w)$ usually does not belong to 
a countable set, neither does $\min\{e^{\ve T(v,w)}q^s,\ve Q(x)\}$. Therefore,
instead of requiring $p^s$ to be equal to this minimum we relax the assumption
to an ``approximate equality''. This approximate equality implies that either $p^s$ is of the
order of $e^{\ve T(v,w)}q^s$ and/or it is essentially maximal (of the order of $\ve Q(x)$).
The conditions we consider are weak enough to code all relevant orbits 
(Theorem \ref{Thm-coarse-graining}(2)) but still strong enough for the coding to be
``unique up to bounded error" (Theorem \ref{Thm-inverse}).
\end{remark}

\subsection{Graph transforms and invariant manifolds}\label{ss.graph.transform}

Let $v=\Psi_x^{p^s,p^u}$ be an $\ve$--double chart.

\medskip
\noindent
{\sc Admissible manifolds:} An {\em $s$--admissible manifold at $v$} is a set
of the form $$V=\Psi_x\{(t,F(t)):|t|\leq p^s\}$$ where $F:[-p^s,p^s]\to\R$ is a $C^{1+\beta/3}$ function
such that:
\begin{enumerate}
\item[(AM1)] $|F(0)|\leq 10^{-3}(p^s\wedge p^u)$.
\item[(AM2)] $|F'(0)|\leq \tfrac{1}{2}(p^s\wedge p^u)^{\beta/3}$.
\item[(AM3)] $\|F'\|_{C^0}+\Hol{\beta/3}(F')\leq\tfrac{1}{2}$ where the norms are taken in $[-p^s,p^s]$.
\end{enumerate}
The function $F$ is called the \emph{representing function} of $V$.
Similarly, a {\em $u$--admissible manifold at $v$} is a set
of the form $\Psi_x\{(G(t),t):|t|\leq p^u\}$ where $G:[-p^u,p^u]\to\R$ is a $C^{1+\beta/3}$ function
satisfying (AM1)--(AM3), with norms taken in $[-p^u,p^u]$.

\medskip
If $V_1,V_2$ are two $s$--admissible manifold at $v$, with representing functions
$F_1,F_2$, for $i\geq 0$ define $ d_{C^i}(V_1,V_2):=\|F_1-F_2\|_{C^i}$ where the
norm is taken in $[-p^s,p^s]$. The same applies to $u$--admissible manifolds.

In the sequel, we introduce {\em graph transforms}, which is the tool used to construct invariant 
manifolds. Since the proofs are adaptations of \cite{Sarig-JAMS}, we restrict the discussion to
stable manifolds. The main result of this section, Theorem \ref{Thm-stable-manifolds},
collects the basic properties of invariant manifolds. 
Given a $\ve$--double chart $v=\Psi_x^{p^s,p^u}$, we denote
by $\mathfs M^s(v)$ the set of its $s$--admissible manifolds.

\medskip
\noindent
{\sc The graph transform $\mathfs F_{v,w}^s$:} To any edge
$v\overset{\ve}{\rightarrow}w$ between $\ve$-double charts $v=\Psi_x^{p^s,p^u}$ and $w=\Psi_y^{q^s,q^u}$,
we associate the {\em graph transform} $\mathfs F_{v,w}^s:\mathfs M^s(w)\to\mathfs M^s(v)$
as being the map that sends an $s$--admissible manifold at $w$ with representing function $F:[-q^s,q^s]\to\R$ to the unique
$s$--admissible  manifold at $v$ with representing function $G:[-p^s,p^s]\to\R$ such that
$\{(t,G(t)):|t|\leq p^s\}\subset f_{x,y}^-\{(t,F(t)):|t|\leq q^s\}$.

\begin{lemma}\label{Prop-graph-transform}
If $\ve>0$ is small enough, then $\mathfs F_{v,w}^s$ is well-defined for any edge $v\overset{\ve}{\rightarrow}w$.
Furthermore, if
$V_1,V_2\in \mathfs M^s(w)$ then:
\begin{enumerate}[{\rm (1)}]
\item $ d_{C^0}(\mathfs F_{v,w}^s(V_1),\mathfs F_{v,w}^s(V_2))\leq e^{-\chi\inf(r_\Lambda)/2} d_{C^0}(V_1,V_2)$.
\item $ d_{C^1}(\mathfs F_{v,w}^s(V_1),\mathfs F_{v,w}^s(V_2))\leq
e^{-\chi\inf(r_\Lambda)/2}( d_{C^1}(V_1,V_2)+ d_{C^0}(V_1,V_2)^{\beta/3})$.
\end{enumerate}
\end{lemma}

When $M$ is compact and $f$ is a $C^{1+\beta}$ diffeomorphism,
this is \cite[Prop. 4.12 and 4.14]{Sarig-JAMS}.
The same proofs work by changing $C_f$ and $\chi$ in \cite{Sarig-JAMS} to
$e^{4\rho}$ and $\chi\inf(r_\Lambda)$ in our case, and observing that by Lemma \ref{Lemma-linear-reduction}(2)
and Theorem \ref{Thm-non-linear-Pesin-2}(1) we have
$e^{-4\rho}<|A|<e^{-\chi\inf(r_\Lambda)}$ and $e^{\chi\inf(r_\Lambda)}<|B|<e^{4\rho}$.

\medskip
\noindent
{\sc The stable manifold of positive $\ve$--gpo:} The {\em stable manifold} of a positive 
$\ve$--gpo $\un v^+=\{v_n\}_{n\geq 0}$ is 
$$
V^s[\un v^+]:=(\mathfs F_{v_0,v_1}^s\circ\cdots\circ\mathfs F_{v_{n-2},v_{n-1}}^s\circ\mathfs F_{v_{n-1},v_n}^s)(V_n)
$$
for some (any) choice  $(V_n)_{n\geq 0}$ with $V_n\in \mathfs M^s(v_n)$. The convergence occurs
in the $C^1$ topology.

\medskip
The proof of the good definition and $C^1$ convergence is done as in \cite[Prop. 4.15, part (1)]{Sarig-JAMS}.
Similarly, we introduce the {\em unstable manifold} $V^u[\un v^-]$ of a negative $\ve$--gpo.
We then arrive at the basic properties of $V^s[\un v^+]$ and $V^u[\un v^-]$.

\begin{theorem}[Stable manifold theorem]\label{Thm-stable-manifolds}
The following holds for all $\ve>0$ small enough.
Let ${\un v}^+=\{v_n\}_{n\geq 0}=\{\Psi_{x_n}^{p^s_n,p^u_n}\}_{n\geq 0}$
be a positive $\ve$--gpo.

\begin{enumerate}[{\rm (1)}]
\item {\sc Admissibility.} The set $V^s[{\un v}^+]$ is an $s$--admissible manifold at $v_0$, equal to
$$
V^s[{\un v}^+]=\{x\in \Psi_{x_0}(R[p^s_0]):(g^+_{x_{n-1}}\circ\dots\circ g^+_{x_0})(x)\in \Psi_{x_n}(R[10Q(x_n)]),\,\forall n\geq 0\}.
$$
\item {\sc Invariance.} $g_{x_0}^+(V^s[\{v_n\}_{n\geq 0}])\subset V^s[\{v_n\}_{n\geq 1}]$.
\smallskip
\item {\sc Hyperbolicity.} For all $y,y'$ in $V^s[\un v^+]$ and all $n\geq 0$: 
$$d(g^+_{x_{n-1}}\circ\dots\circ g^+_{x_0}(y),g^+_{x_{n-1}}\circ\dots\circ g^+_{x_0}(y'))
\leq d(\Psi^{-1}_{x_0}(y),\Psi^{-1}_{x_0}(y'))\;e^{-\frac{\chi \; \inf(r_\Lambda)}{2} n}.$$
For any unit vector $w$ tangent to $V^s[{\un v}^+]$ at a point $y$ and all $n\geq 0$:
\begin{align*}
\|d(g^+_{x_{n-1}}\circ\dots\circ g^+_{x_0})_yw\|&\leq 8p^s_0\; e^{-\frac{\chi \; \inf(r_\Lambda)}{2} n}\ \text{ and}\\
\|d(g^-_{x_{-n+1}}\circ\dots\circ g^-_{x_0})_yw\|&\geq 
\tfrac{1}{8}(p^s_0\wedge p^u_0)^{\frac{\beta}{12}}\; e^{\left(\frac{\chi \; \inf(r_\Lambda)}{2}-\frac{\beta\ve}{6}\right) n}.
\end{align*}
\item {\sc Bounded distortion.} For all $y,y'$ in $V^s[\un v^+]$, unit vectors $w,w'$ tangent
to $V^s[{\un v}^+]$ at $y,y'$ respectively and all $n\geq 0$,
$$
\left| \log\|d(g^+_{x_{n-1}}\circ\dots\circ g^+_{x_0})_yw\|-\log\|d(g^+_{x_{n-1}}\circ\dots\circ g^+_{x_0})_{y'}w'\| \right|
\leq Q(x_0)^{\beta/4}.
$$
\item {\sc H\"older property.} The map $\un v^+\mapsto V^s[\un v^+]$ is H\"older continuous:\\
There are $K>0$ and $\theta\in(0,1)$ such that for all $N\geq 0$, if $\un v^+,\un w^+$ are positive $\ve$--gpo's
with $v_n=w_n$ for $n=0,\ldots,N$
then $ d_{C^1}(V^s[\un v^+],V^s[\un w^+])\leq K\theta^N$.
\end{enumerate}
The curve $V^s[\un v^+]$ is called \emph{local stable manifold} of $\un v^+$.
A similar statement holds for unstable manifold $V^u[\un v^-]$ of
a negative $\varepsilon$-gpo $\un v^-$.
\end{theorem}

\medskip
The above theorem is a strenghtening of the Pesin stable manifold theorem \cite{Pesin-Izvestia-1976}.
Its statement is similar to \cite{Sarig-JAMS}, and its proof is perfomed exactly as in
\cite[Prop. 4.15 and 6.3]{Sarig-JAMS}, noting that in Pesin charts the composition
$g^+_{x_{n-1}}\circ\dots\circ g^+_{x_0}$ is represented by $f_{x_{n-1},x_n}^+\circ\cdots\circ f_{x_0,x_1}^+$.
Since each $f_{x_i,x_{i+1}}^+$ is hyperbolic (Theorem \ref{Thm-non-linear-Pesin-2})
and each $\mathfs F_{v_i,v_{i+1}}^s$ is contracting (Lemma \ref{Prop-graph-transform}), the proof follows.
We note that the second estimate of part (3) is proved as in \cite[Prop. 6.5]{Sarig-JAMS},
see also the proof of \cite[Prop. 4.11]{ALP}.

\subsection{Shadowing}
We say that an $\ve$--gpo $\{\Psi_{x_n}^{p^s_n,p^u_n}\}_{n\in\Z}$ {\em shadows}
a point $x\in \widehat \Lambda$ if:
$$(g^+_{x_{n-1}}\circ\dots\circ g^+_{x_0})(x)\in \Psi_{x_n}(R[p^s_n\wedge p^u_n])\text{ for all $n\geq 0$},$$
$$(g^-_{x_{n+1}}\circ\dots\circ g^-_{x_0})(x)\in \Psi_{x_n}(R[p^s_n\wedge p^u_n])\text{ for all $n\leq 0$}.$$
An important property is the following.

\begin{proposition}\label{Prop-shadowing}
If $\ve$ is small enough, then every $\ve$--gpo $\un v$ shadows a unique point
$\{x\}=V^s[\un v]\cap V^u[\un v]$.
\end{proposition}

\medskip
The proof uses the following property of admissible manifolds.

\begin{lemma}\label{Lemma-admissible-manifolds}
The following holds for all $\ve>0$ small enough. If $v=\Psi_x^{p^s,p^u}$ is an $\ve$--double chart,
then for every $V^{s/u}\in\mathfs M^{s/u}(v)$ it holds:
\begin{enumerate}[{\rm (1)}]
\item $V^s$ and $V^u$ intersect at a single point
$P\in \Psi_x(R[10^{-2}(p^s\wedge p^u)])$.
\item $\tfrac{\sin\angle(V^s,V^u)}{\sin\alpha(x)}=e^{\pm(p^s\wedge p^u)^{\beta/4}}$
and $|\cos\angle(V^s,V^u)-\cos\alpha(x)|<2(p^s\wedge p^u)^{\beta/4}$, where
$\angle(V^s,V^u)$ is the angle of intersection of $V^s$ and $V^u$ at $P$.
\end{enumerate}
\end{lemma}

When $M$ is compact and $f$ is a $C^{1+\beta}$ diffeomorphism,
the above lemma is \cite[Prop. 4.11]{Sarig-JAMS}.
The same proof works in our case, since inside $\Psi_x(R[10Q(x)])$ the estimates
(Exp1)--(Exp4) hold.

\begin{proof}[Proof of Proposition~\ref{Prop-shadowing}]
Let $\un v=\{v_n\}_{n\in\Z}=\{\Psi_{x_n}^{p^s_n,p^u_n}\}_{n\in\Z}$ be an $\ve$--gpo.
The proof is the same of \cite[Theorem 4.2]{Sarig-JAMS}, and follows the steps below: 
\begin{enumerate}[$\circ$]
\item By Theorem \ref{Thm-stable-manifolds}(1),
any point shadowed by $\un v$ must lie in $V^s[\{v_n\}_{n\geq 0}]\cap V^u[\{v_n\}_{n\leq 0}]$.
By Lemma \ref{Lemma-admissible-manifolds}(1), this intersection is a single point $\{x\}$.
We claim that $\un v$ shadows $x$. 
\item The definition of shadowing is equivalent to the following weaker definition: 
$\un v$ shadows $x$ if and only if
$$(g^+_{x_{n-1}}\circ\dots\circ g^+_{x_0})(x)\in \Psi_{x_n}(R[10Q(x_n)])\text{ for all $n\geq 0$},$$
$$(g^-_{x_{n+1}}\circ\dots\circ g^-_{x_0})(x)\in \Psi_{x_n}(R[10Q(x_n)])\text{ for all $n\leq 0$}.$$
\item By Theorem \ref{Thm-stable-manifolds}(2), if $n\geq 0$ then 
$g^+_{x_{n-1}}\circ\dots\circ g^+_{x_0}(P)\in V^s[\{v_{n+k}\}_{k\geq 0}]\subset \Psi_{x_n}(R[10Q(x_n)])$, 
and if $n\geq 0$ then $(g^-_{x_{n+1}}\circ\dots\circ g^-_{x_0})(x)\in V^u[\{v_{n+k}\}_{k\leq 0}]\subset \Psi_{x_n}(R[10Q(x_n)])$, and so the weaker definition of shadowing holds.
\end{enumerate}
This concludes the proof.
\end{proof}

\subsection{Additional properties}

Now, we relate stable/unstable manifolds of $\ve$--gpo's with stable/unstable manifolds of the flow
$\vf$. 

\begin{proposition}\label{Prop-center-stable}
The following holds for all $\ve>0$ small enough. Let $\un v=\{v_n\}_{n\geq 0}$ be a positive $\ve$--gpo
with $v_0=\Psi_x^{p^s,p^u}$, and let $F:[-p^s,p^s]\to\R$ be the representing function of $V^s=V^s[\un v^+]$.
Then there exists a function $\Delta:[-p^s,p^s]\to\R$ with $\Delta(0)=0$ such that
the curve $\wt V^s:=\{\vf^{\Delta(t)}[\Psi_x(t,F(t))]:|t|\leq p^s\}$ satisfies 
$d(\vf^t(\wt y),\vf^t(\wt z))\leq e^{-\frac{\chi\inf(r_\Lambda)}{2\sup(r_\Lambda)}t}$
for all $\wt y,\wt z\in \wt V^s$ and $t\geq 0$.
An analogous statement holds for negative $\ve$--gpo's. 
\end{proposition}

In other words, $\wt V^s$ is a lift of $V^s$ to a curve that contracts in the future under the flow (we are not claiming
$\wt V^s$ is the local stable manifold of $\vf$ at $x$).

\begin{proof}
Write $v_n=\Psi_{x_n}^{p^s_n,p^u_n}$ with $\Psi_{x_0}^{p^s_0,p^u_0}=\Psi_x^{p^s,p^u}$. 
The idea is simple: $\Delta$ is the cumulative shear of a point of $V^s$ under iterations
of the maps $g_{x_n}^+$. Write $g_{x_n}^+=\vf^{T_n}$ where $T_n:B_{x_n}\to\R$
is a $C^{1+\beta}$ function with $T_n(x_n)=r_\Lambda(x_n)$. Let $G_0={\rm Id}$ and
$G_n:=g_{x_{n-1}}^+\circ\cdots\circ g_{x_0}^+$, $n\geq 1$. For $n\geq 0$, define $\tau_n:[-p^s,p^s]\to\R$ by
$$
\tau_n(t):=\sum_{k=0}^{n-1}T_k(G_k[\Psi_x(t,F(t))]),
$$
equal to the flow displacement of the point $\Psi_x(t,F(t))$
under the maps $g_{x_0}^+,g_{x_1}^+$,\ldots, $g_{x_{n-1}}^+$. Define $\Delta_n:[-p^s,p^s]\to\R$
by $\Delta_n(t):=\tau_n(t)-\tau_n(0)$ for $n\geq 0$,
and $\Delta:[-p^s,p^s]\to\R$ by $\Delta(t):=\lim\limits_{n\to+\infty}\Delta_n(t)$. We have:
\begin{enumerate}[$\circ$]
\item $\Lip{}(T_n)<1$, by Lemma \ref{Lemma-regularity-q-t}(3).
\item $\|\Delta-\Delta_n\|_{C^0}<\ve e^{-\frac{\chi}{2}n}$ for all $n\geq 0$, since
\begin{align*}
&\, \|\Delta-\Delta_n\|_{C^0}\leq \sum_{k=n}^\infty \|T_k(G_k[\Psi_x(\cdot,F(\cdot))])-T_k(G_k[\Psi_x(0,F(0))])\|_{C^0}\\
&\overset{!}{\leq}\sum_{k=n}^\infty \Lip{}(T_k)6p^s e^{-\frac{\chi\inf(r_\Lambda)}{2}k}
\leq \frac{6p^s}{1-e^{-\frac{\chi\inf(r_\Lambda)}{2}}}e^{-\frac{\chi\inf(r_\Lambda)}{2}n}\overset{!!}{<}
\ve e^{-\frac{\chi\inf(r_\Lambda)}{2}n},
\end{align*}
where in $\overset{!}{\leq}$ we used Theorem~\ref{Thm-stable-manifolds}(3) and
in $\overset{!!}{<}$ we used that
$\frac{6p^s}{1-e^{-\frac{\chi\inf(r_\Lambda)}{2}}}<\frac{6\ve^{3/\beta}}{1-e^{-\frac{\chi\inf(r_\Lambda)}{2}}}<\ve$
when $\ve>0$ is small enough.
\end{enumerate}

\medskip
Let $\wt V^s:=\{\vf^{\Delta(t)}[\Psi_x(t,F(t))]:|t|\leq p^s\}$. Fix
$\wt y,\wt z\in\wt V^s$, say $\wt y=\vf^{\Delta(t_0)}[\Psi_x(t_0,F(t_0))]=\vf^{\Delta(t_0)}(y)$
and $\wt z=\vf^{\Delta(t_1)}[\Psi_x(t_1,F(t_1))]=\vf^{\Delta(t_1)}(z)$ with $t_0,t_1\in[-p^s,p^s]$.
By definition, $y,z\in V^s$.
Fix $t\geq 0$, and take the unique $n\geq 0$ such that $\tau_{n-1}(0)<t\leq\tau_n(0)$. For such $n$,
write $\Delta=\Delta_n+E$, with $\|E\|_{C^0}<\ve e^{-\frac{\chi\inf(r_\Lambda)}{2}n}$.
Therefore
$$
\vf^t(\wt y)=\vf^{t+\Delta(t_0)}(y)=\vf^{t+\Delta_n(t_0)+E(t_0)}(y)=
\vf^{t-\tau_n(0)+E(t_0)}[G_n(y)],
$$
and similarly $\vf^t(\wt z)=\vf^{t-\tau_n(0)+E(t_1)}[G_n(z)]$, hence
\begin{align*}
d(\vf^t(\wt y),\vf^t(\wt z))& \leq d(\vf^{t-\tau_n(0)+E(t_0)}[G_n(y)],\vf^{t-\tau_n(0)+E(t_0)}[G_n(z)])+\\
&\ \ \ \ d(\vf^{t-\tau_n(0)+E(t_0)}[G_n(z)],\vf^{t-\tau_n(0)+E(t_1)}[G_n(z)])\\
&\leq \sup_{|\zeta|\leq 1}\Lip{}(\vf^\zeta) d(G_n(y),G_n(z))+\|X\|_{C^0}|E(t_0)-E(t_1)|\\
&\leq \left[6p^s\sup_{|\zeta|\leq 1}\Lip{}(\vf^\zeta)+2\ve \|X\|_{C^0}\right]e^{-\frac{\chi\inf(r_\Lambda)}{2}n}
\leq e^{-\frac{\chi\inf(r_\Lambda)}{2}n}
\end{align*}
for $\ve>0$ small. Since $t\leq \tau_n(0)\leq n\sup(r_\Lambda)$,
we get that $d(\vf^t(\wt y),\vf^t(\wt z))\leq e^{-\frac{\chi\inf(r_\Lambda)}{2\sup(r_\Lambda)}t}$.
\end{proof}

We note two important facts. Firstly, the choice of $\Delta(0)=0$ is arbitrary: given $y=\Psi_x(t,F(t))\in V^s$,
we can choose $\Delta$ so that $\Delta(t)=0$. The resulting smooth curve
$\widetilde V^s\ni y$ also satisfying Proposition \ref{Prop-center-stable}. 

The second is more relevant. Given $y\in V^s=V^s[\un v]$, lift $V^s$ to $\wt V^s\ni y$, and let
$\widetilde e^s_y$ be a unitary vector tangent to $\widetilde V^s$ at $y$ (it is defined up to a sign). By construction,
the projection of $\widetilde e^s_y$ in the flow direction is a multiple of $\widetilde n^s_y$.
Taking the angle $\angle(N_y,X(y))$ into account, we can prove that
$d\vf^t \widetilde e^s_y$ contracts exponentially fast as $t\to+\infty$, i.e.
$\limsup\limits_{t\to+\infty}\tfrac{1}{t}\log\|d\vf^t \widetilde e^s_y\|<0$ (just a $\limsup$, not necessarily a $\lim$).
The same holds for $u$--admissible manifolds $V^u[\un v]$.
Therefore, given an $\ve$--gpo $\un v=\{\Psi_{x_n}^{p^s_n,p^u_n}\}_{n\in\Z}$,
if $V^s[\un v]\cap V^u[\un v]=\{x\}$,
there are two smooth curves $\widetilde V^s,\widetilde V^u$ passing through $x$ satisfying the following:
\begin{enumerate}[$\circ$]
\item If $\widetilde e^s_x$ is a unitary vector tangent to $\widetilde V^s$ at $x$, then
$\limsup\limits_{t\to+\infty}\tfrac{1}{t}\log\|d\vf^t\widetilde e^s_x\|<0$.
\item If $\widetilde e^u_x$ is a unitary vector tangent to $\widetilde V^u$ at $x$, then 
proceeding as in \cite[Prop. 6.5]{Sarig-JAMS} we show that
$\limsup\limits_{t\to+\infty}\tfrac{1}{t}\log\|d\vf^t\widetilde e^u_x\|>0$.\footnote{One important ingredient in the proof
of \cite[Prop. 6.5]{Sarig-JAMS} is the estimate $p^s_{n+1}\wedge p^u_{n+1}\leq e^{\ve}(p^s_n\wedge p^u_n)$.
We have a similar result, by Lemma \ref{Lemma-minimum}.}
\end{enumerate}
These two properties above uniquely define the directions $\widetilde e^s_x,\widetilde e^u_x$ (up to a sign).
Therefore we can consider $\alpha(x),s(x),u(x)$, although we do not know that $s(x),u(x)$ are finite. 
Remember that $\widetilde n^s_x,\widetilde n^u_x\in N_x$, the tangent vectors to $V^s,V^u$ at $x$,
are the projections of $\widetilde e^s_x,\widetilde e^u_x$ in the flow direction. 

We finish this section proving another property about invariant manifolds.
\begin{proposition}\label{Prop-disjointness}
The following holds for $\ve>0$ small enough.
Let $\un v^+=\{v_n\}_{n\geq 0}$ and $\un w=\{w_n\}_{n\geq 0}$ be positive $\ve$--gpo's,
with $v_0=\Psi_x^{p^s,p^u}$ and $w_0=\Psi_x^{q^s,q^u}$. Then 
either $V^s[\un v^+],V^s[\un w^+]$ are disjoint or one contains the other.
\end{proposition}

\begin{proof}
For $C^{1+\beta}$ surface diffeomorphism, this is \cite[Prop. 6.4]{Sarig-JAMS}.
We apply a similar idea, using Proposition \ref{Prop-center-stable}. 
Write $V^s=V^s[\un v^+]$ and $U^s=V^s[\un w^+]$.
If $V^s\cap U^s=\emptyset$, we are done, 
so assume there is $z\in V^s\cap U^s$. Assuming without loss of generality
that $q^s\leq p^s$, we will prove that $U^s\subset V^s$. The proof will follow from three claims as
in \cite[Prop. 6.4]{Sarig-JAMS}. Write $\un v^+=\{\Psi_{x_n}^{p^s_n,p^u_n}\}_{n\geq 0}$.
We continue using the same terminology
of the previous proposition, with $g_{x_ n}^+=\vf^{T_n}$ for $n\geq 0$, $G_0={\rm Id}$,
and $G_n=g_{x_{n-1}}^+\circ\cdots\circ g_{x_0}^+$ for $n\geq 1$.

\medskip
\noindent
{\sc Claim 1:} If $n$ is large enough then $G_ n(V^s)\subset \Psi_{x_n}(R[\tfrac{1}{2}Q(x_n)])$.

\begin{proof}[Proof of Claim $1$.] Same as \cite[Prop. 6.4]{Sarig-JAMS}, using that the representation 
of $g_{x_n}^+$ in Pesin charts satisfies Theorem \ref{Thm-non-linear-Pesin-2}.
\end{proof}

\medskip
\noindent
{\sc Claim 2:} If $n$ is large enough then $G_ n(U^s)\subset \Psi_{x_n}(R[Q(x_n)])$.

\begin{proof}[Proof of Claim $2$.] Lift $U^s$ to a curve $\wt U^s$ passing through $z$
and satisfying Proposition \ref{Prop-center-stable}. Fix $n\geq 0$, and let
$t_n=\sum_{k=0}^{n-1}T_k(G_k(z))$ be the total flow time of $z$ under $G_n$.
Let $z_n=G_n(z)=\vf^{t_n}(z)$. If $D\subset\widehat\Lambda$ is the disc containing $x_n$ then 
$$
G_n(U^s)=\mathfrak q_D[\vf^{t_n}(\wt U^s)].
$$
Let $c:=\inf(r_\Lambda)^2/2\sup(r_\Lambda)$. Since $\mathfrak q_D$ is $2$--Lipschitz (Lemma \ref{Lemma-regularity-q-t}(2)),
Lemma \ref{Lemma-map-g} and Proposition \ref{Prop-center-stable}
imply that
$$
{\rm diam}(G_n(U^s))={\rm diam}(\mathfrak q_D[\vf^{t_n}(\wt U^s)])\leq 2{\rm diam}(\vf^{t_n}(\wt U^s))
\leq 2e^{-\frac{\chi\inf(r_\Lambda)}{2\sup(r_\Lambda)}t_n}\leq 2e^{-\chi c n},
$$
since $t_n\geq \inf(r_\Lambda)n$. Hence
$\Psi_{x_n}^{-1}[G_n(U^s)]$ is contained in the ball with center $\Psi_{x_n}^{-1}(z_n)$
and radius $4\|C(x_n)^{-1}\|e^{-\chi c n}$. Since by Claim 1 we have
$\Psi_{x_n}^{-1}(z_n)\in R[\tfrac{1}{2}Q(x_n)]$, it is enough to prove that 
$4\|C(x_n)^{-1}\|e^{-\chi c n}<\tfrac{1}{2}Q(x_n)$. Using that
$Q(x_n)<\|C(x_n)^{-1}\|^{-1}$, we just need to prove
that $8Q(x_n)^{-2}e^{-\chi cn}<1$. We claim that $Q(x_n)^{-2}e^{-\chi cn}$
converges to zero exponentially fast as $n$ increases. Indeed, by Lemma \ref{Lemma-minimum}
we have
$Q(x_n)\geq p^s_n\wedge p^u_n\geq e^{-2\ve n}(p^s_0\wedge p^u_0)$
and so 
$$
Q(x_n)^{-2}e^{-\chi cn}\leq e^{4\ve n}(p^s_0\wedge p^u_0)^{-2}e^{-\chi cn}=
(p^s_0\wedge p^u_0)^{-2}e^{-(\chi c-4\ve)n}
$$
which converges to zero if $\ve>0$ is small enough.
\end{proof}

By Theorem \ref{Thm-stable-manifolds}(1), we conclude that
$G_n(U^s)\subset V^s[\{\Psi_{x_k}^{p^s_k,p^u_k}\}_{k\geq n}]$ for every
$n$ large enough. 

\medskip
\noindent
{\sc Claim 3:} $U^s\subset V^s$.

\begin{proof}[Proof of Claim $3$.] Fix $n$ large enough so that
$G_n(U^s)\subset V^s[\{\Psi_{x_k}^{p^s_k,p^u_k}\}_{k\geq n}]$, 
and proceed as in Claim 3 of \cite[Prop. 6.4]{Sarig-JAMS}.
\end{proof}

The proof of the proposition is complete.
\end{proof}

\section{First coding}\label{section-coarse-graining}

Up to now, we have fixed $\vf,\chi,\rho,\Lambda,\widehat\Lambda,\theta,\ve$ such that
$\ve\ll \rho\ll 1$, and we have constructed invariant manifolds for $\ve$--gpo's. 
We also defined shadowing. In this section, we:
\begin{enumerate}[$\circ$]
\item Construct a countable family of $\ve$--double 
charts whose $\ve$--gpo's they define shadow the whole set $\Lambda\cap\nuh^\#$.
\item Define a first coding, that is usually infinite-to-one.
\end{enumerate}

\subsection{Coarse graining}

This self-contained section comprises an important part of this work that cannot be obtained
using the methods of \cite{Sarig-JAMS,Lima-Sarig,Lima-Matheus}. Indeed,
condition (GPO2) in our definition of edge between $\ve$--double charts is a set of inequalities, 
so we need to show it is loose enough to shadow all points of $\Lambda\cap\nuh^\#$.
The proof of this fact requires an analysis of orbits at hyperbolic times,
where parameters are essentially uniquely defined.

\begin{theorem}[Coarse graining]\label{Thm-coarse-graining}
For all $0<\ve\ll \rho\ll 1$, there exists a countable family $\mathfs A$ of $\ve$--double charts
with the following properties:
\begin{enumerate}[{\rm (1)}]
\item {\sc Discreteness:} For all $t>0$, the set $\{\Psi_x^{p^s,p^u}\in\mathfs A:p^s,p^u>t\}$ is finite.
\item {\sc Sufficiency:} If $x\in\Lambda\cap\nuh^\#$ then there is a regular $\ve$--gpo
$\un v\in{\mathfs A}^{\Z}$ that shadows $x$.
\item {\sc Relevance:} For each $v\in \mathfs A$, $\exists\un{v}\in\mathfs A^\Z$ an $\ve$--gpo
with $v_0=v$ that shadows a point in $\Lambda\cap\nuh^\#$.
\end{enumerate}
\end{theorem}

Recall that $\un v=\{v_n\}_{n\in\Z}$ is regular if there are $v,w$ such that
$v_n=v$ for infinitely many $n>0$ and $v_n=w$ for infinitely many $n<0$.
According to Proposition \ref{Prop-adaptedness} and part (2) above,
the $\ve$--gpo's in $\mathfs A$ shadow almost every point with respect to every
$\chi$--hyperbolic measure.

\begin{proof}
When $M$ is a closed surface and $f$ is a diffeomorphism, the above statement is consequence
of Propositions 3.5, 4.5 and Lemmas 4.6, 4.7 of \cite{Sarig-JAMS}. When $M$ is a compact surface with boundary
and $f$ is a local diffeomorphism with bounded derivatives, this is Proposition 4.3 of \cite{Lima-Sarig}.
When $M$ is a surface and $f$ is a local diffeomorphism with unbounded derivatives, this is
Theorem 5.1 of \cite{Lima-Matheus}. Our proof follows a similar strategy of \cite{Sarig-JAMS,Lima-Sarig}
but the implementation is significantly harder, since the definition of edge is more complicated.
In particular, we need to control the cumulative shear between an orbit and an $\ve$--gpo.

\medskip
Let $\N_0=\N\cup\{0\}$, and let $X:=\Lambda^3\times {\rm GL}(2,\R)^3\times (0,1]$.
For $x\in\Lambda\cap\nuh^\#$, let
$\Gamma(x)=(\un x,\un C,\un Q)\in X$ with
\begin{align*}
\un x=(f^{-1}(x),x,f(x)),\ \un C=(C(f^{-1}(x)),C(x),C(f(x))),\ \un Q=(Q(x),q(x)).
\end{align*}
Let $Y=\{\Gamma(x):x\in\Lambda\cap\nuh^\#\}$. We want to construct a countable dense subset
of $Y$. Since the maps $x\mapsto C(x),Q(x),q(x)$ are usually just measurable,
we apply a precompactness argument.
For each $\un{\ell}=(\ell_{-1},\ell_0,\ell_1)\in\N_0^3$ and $m,j\in\N_0$, define
$$
Y_{\un \ell,m,j}:=\left\{\Gamma(x)\in Y:
\begin{array}{cl}
e^{\ell_i}\leq\|C(f^i(x))^{-1}\|<e^{\ell_i+1},&-1\leq i\leq 1\\
e^{-m-1}\leq Q(x)< e^{-m}&\\
e^{-j-1}\leq q(x)< e^{-j}&\\

\end{array}
\right\}.
$$

\medskip
\noindent
{\sc Claim 1:} $Y=\bigcup\limits_{\un\ell\in\N_0^3\atop{m,j\in\N_0}}Y_{\un\ell,m,j}$, and each
$Y_{\un\ell,m,j}$ is precompact in $X$.

\begin{proof}[Proof of Claim $1$.] 
The first statement is clear. We focus on precompactness.
Fix $\un\ell\in \N_0^3$, $m,j\in\N_0$, and take $\Gamma(x)\in Y_{\un\ell,m,j}$. Then
$\un x\in \Lambda^3$, a precompact subset of $M^3$.
For $|i|\leq 1$, $C(f^i(x))$ is an element of ${\rm GL}(2,\R)$ with norm $\leq 1$ and
inverse norm $\leq e^{\ell_i+1}$, hence it belongs to a compact subset of ${\rm GL}(2,\R)$.
This guarantees that $\un C$ belongs to a compact subset of ${\rm GL}(2,\R)^3$. Also,
$\un Q\in [e^{-m-1},1]\times [e^{-j-1},1]$, a compact subinterval of $(0,1]$. Since the product of precompact sets
is precompact, the claim is proved.
\end{proof}

By Claim 1, there exists a finite set
$Z_{\un\ell,m,j}\subset Y_{\un\ell,m,j}$ such that for every $\Gamma(x)\in Y_{\un\ell,m,j}$
there exists $\Gamma(y)\in Z_{\un\ell,m,j}$ with:
\begin{enumerate}[{\rm (a)}]
\item $ d(f^i(x),f^i(y))+\|\widetilde{C(f^i(x))}-\widetilde{C(f^i(y))}\|<\tfrac{1}{2}q(x)^8$,
$|i|\leq 1$.
\item $\tfrac{Q(x)}{Q(y)}=e^{\pm \ve/3}$ and $\tfrac{q(x)}{q(y)}=e^{\pm \ve/3}$.
\end{enumerate}
A fortiori, (a) implies that $f^i(x),f^i(y)$ belong to the same disc of $\Lambda$, for $|i|\leq 1$.
For $\eta>0$, let $I_{\ve,\eta}:=\{e^{-\ve^2\eta k}:k\geq 0\}$, a countable discrete set whose ``thickness''
depends on $\eta$.

\medskip
\noindent
{\sc The alphabet $\mathfs A$:} Let $\mathfs A$ be the countable family of $\Psi_x^{p^s,p^u}$ such that:
\begin{enumerate}[i i)]
\item[(CG1)] $\Gamma(x)\in Z_{\un\ell,m,j}$ for some
$(\un\ell,m,j)\in\N_0^3\times \N_0\times \N_0$.
\item[(CG2)] $0<p^s,p^u\leq \ve Q(x)$ and $p^s,p^u\in I_{\ve,q(x)}$.
\item[(CG3)] $e^{-\mathfrak H-1}\leq \tfrac{p^s\wedge p^u}{q(x)}\leq e^{\mathfrak H+1}$, where $\mathfrak H$
is given by Proposition \ref{Prop-Z-par}(1).
\end{enumerate}

\medskip
\noindent
{\em Proof of discreteness.}
Fix $t>0$, and let $\Psi_x^{p^s,p^u}\in\mathfs A$ with $p^s,p^u>t$.
If $\Gamma(x)\in Z_{\un\ell,m,j}$ then:
\begin{enumerate}[$\circ$]
\item Finiteness of $\un\ell$: we have $e^{\ell_0}\leq \|C(x)^{-1}\|<Q(x)^{-1}<t^{-1}$,
hence $\ell_0<|\log t|$. By Lemma \ref{Lemma-linear-reduction}(3), for $i=\pm 1$ we have
$$
e^{\ell_i}\leq \|C(f^i(x))^{-1}\|\leq \|C(f^i(x))^{-1}\|_{\rm Frob}\leq e^{18\rho}\|C(x)^{-1}\|_{\rm Frob}
<e^{18\rho}t^{-1},
$$
hence $\ell_{-1},\ell_1<18\rho+|\log t|=:T_t$, which is bigger than $|\log t|$.
\item Finiteness of $m$: $e^{-m}>Q(x)>t$, hence $m<|\log t|$.
\item Finiteness of $j$: $e^{-j}>q(x)\geq e^{-\mathfrak H-1}(p^s\wedge p^u) >e^{-\mathfrak H-1} t$,
hence $j\leq |\log t|+\mathfrak H+1$.
\end{enumerate}
Therefore
\begin{align*}
\#\left\{\Gamma(x):
\Psi_x^{p^s,p^u}\in\mathfs A\text{ s.t. }p^s,p^u>t\right\}\leq
\sum_{j=0}^{\lceil |\log t|+\mathfrak H\rceil+1}\sum_{m=0}^{\lceil |\log t|\rceil}
\sum_{-1\leq i\leq 1\atop{\ell_i=0}}^{T_t} \# Z_{\un\ell,m,j}
\end{align*}
is the finite sum of finite terms, hence finite. For each such $\Gamma(x)$,
$$
\#\{(p^s,p^u):\Psi_x^{p^s,p^u}\in\mathfs A\text{ s.t. }p^s,p^u>t\}\leq (\# I_{\ve,q(x)}\cap (t,1))^2$$
is finite, hence 
\begin{align*}
\#\left\{\Psi_x^{p^s,p^u}\in\mathfs A:p^s,p^u>t\right\}\leq 
\sum_{j=0}^{\lceil |\log t|+\mathfrak H\rceil+1}\sum_{m=0}^{\lceil |\log t|\rceil}\sum_{-1\leq i\leq 1\atop{\ell_i=0}}^{T_t}
\sum_{\Gamma(x)\in Z_{\un\ell,m,j}}(\# I_{\ve,q(x)}\cap (t,1))^2
\end{align*}
is the finite sum of finite terms, hence finite. This proves the discreteness of $\mathfs A$.

\medskip
\noindent
{\em Proof of sufficiency.}
Let $x\in\Lambda\cap\nuh^\#$. Take $(\ell_i)_{i\in\Z},(m_i)_{i\in\Z},(j_i)_{i\in\Z}$  such that:
\begin{align*}
& \|C(f^i(x))^{-1}\|\in [e^{\ell_i},e^{\ell_i+1}),Q(f^i(x))\in [e^{-m_i-1},e^{-m_i}),\\
&q(f^i(x))\in[e^{-j_i-1},e^{-j_i}).
\end{align*}
For $n\in\Z$, let $\un\ell^{(n)}=(\ell_{n-1},\ell_n,\ell_{n+1})$. Then $\Gamma(f^n(x))\in Y_{\un\ell^{(n)},m_n,j_n}$.
Take $\Gamma(x_n)\in Z_{\un\ell^{(n)},m_n,j_n}$ such that:
\begin{enumerate}[aaa)]
\item[(${\rm a}_n$)] $ d(f^i(f^n(x)),f^i(x_n))+
\|\widetilde{C(f^i(f^n(x)))}-\widetilde{C(f^i(x_n))}\|<\tfrac{1}{2}q(f^n(x))^8$, $|i|\leq 1$.
\item[(${\rm b}_n$)] $\tfrac{Q(f^n(x))}{Q(x_n)}=e^{\pm\ve/3}$ and $\tfrac{q(f^n(x))}{q(x_n)}=e^{\pm\ve/3}$.
\end{enumerate}
From now on the proof differs from \cite{Sarig-JAMS,Lima-Sarig,Lima-Matheus}.
Take $\{t_n\}_{n\in\Z}$ such that $f^n(x)=\vf^{t_n}(x)$, with $t_0=0$ and $g_{x_n}^+[f^n(x)]=\vf^{t_{n+1}-t_n}[f^n(x)]$.
Define
\begin{align*}
P_n^s&:=\ve\inf\{e^{\ve|t_{n+k}-t_n|}Q(x_{n+k}):k\geq 0\},\\
P_n^u&:=\ve\inf\{e^{\ve|t_{n+k}-t_{n}|}Q(x_{n+k}):k\leq 0\}.
\end{align*}
There is no reason for $\Psi_{x_n}^{P^s_n,P^u_n}$ belonging to $\mathfs A$ nor
for $\{\Psi_{x_n}^{P^s_n,P^u_n}\}_{n\in\Z}$ being an $\ve$--gpo. Indeed,
with the above definitions 
one of the inequalities in (GPO2) holds in the reverse direction. 
To satisfy (GPO2), we will slightly decrease each $P^s_n,P^u_n$.
Below we show how to make this ``surgery'' for $P^s_n$ (the method for
$P^u_n$ is symmetric).

Start noting the greedy recursion $P^s_n=\min\{e^{\ve(t_{n+1}-t_n)}P^s_{n+1},\ve Q(x_n)\}$
and that
\begin{align*}
&P^s_n=e^{\pm \ve}\ve\inf\{e^{\ve|t_{n+k}-t_n|}Q(f^{n+k}(x)):k\geq 0\}
=e^{\pm \ve}p^s(x,\mathcal T,n)=e^{\pm\left(\mathfrak H+\frac{\ve}{3}\right)}q^s(f^n(x)),
\end{align*}
by (${\rm b}_n$) above and Proposition \ref{Prop-Z-par}(1), where $\mathcal T=\{t_n\}_{n\in\Z}$.
We fix $\lambda:=\exp{}[\ve^{1.5}]$
and divide the indices $n\in\Z$ into two groups:
\begin{center}
$n$ is {\em growing} if $P^s_n\geq \lambda P^s_{n+1}$
and it is {\em maximal} otherwise.
\end{center}
Note that $\lambda$ has an exponent 
with order smaller than $\ve$. The definition of growing/maximal indices is motivated by the following: the parameter
$P^s_n$ gives a choice on the size of the stable manifold at $x_n$,
therefore we expect $P^s_n$ to be larger than $P^s_{n+1}$ at least by a multiplicative
factor bigger than $\lambda$, unless it reaches the maximal size $\ve Q(x_n)$.
In the first case the index is growing, and in the second it is maximal.
Assuming that $\ve>0$ is sufficiently small, we note two properties of this notion:
\begin{enumerate}[$\circ$]
\item If $n$ is maximal then $P^s_n=\ve Q(x_n)$: otherwise
$P^s_n=e^{\ve(t_{n+1}-t_n)}P^s_{n+1}\geq e^{\ve\inf(r_\Lambda)}P^s_{n+1}>\lambda P^s_{n+1}$,
which contradicts the assumption that $n$ is maximal.
\item There are infinitely many maximal indices $n>0$,
and infinitely many maximal indices $n<0$: the first claim follows exactly as in the proof of
Proposition \ref{Prop-Z-par}(3)  (remember we are assuming that $x\in\nuh^\#$
and so $\limsup\limits_{n\to+\infty} P^s_n>0$).
The second claim follows from direct computation: if there is $n_0$ such that every $n<n_0$ is growing then
$P^s_n\geq \lambda^{n_0-n}P^s_{n_0}$ for all $n<n_0$, which cannot hold since $\lambda^{n_0-n}\to\infty$
as $n\to-\infty$.
\end{enumerate}

We define $p^s_n=a_n P^s_n$ where $e^{-\ve}<a_n\leq 1$ are appropriately chosen.
We first define $a_n$ for the maximal indices $n\in\Z$ as the largest value in $(0,1]$ with
$a_nP^s_n\in I_{\ve,q(x_n)}$.
In particular, $e^{-\ve^2 q(x_n)}\leq a_n\leq 1$. Then we define $a_n$ for the growing indices.
Fix two consecutive maximal indices $n<m$ and define $a_{n+1},\ldots,a_{m-1}$ with a backwards 
induction as follows.
If $n<k<m$ and $a_{k+1}$ is well-defined then we choose $a_k$ largest as possible satisfying:
\begin{enumerate}[aa)]
\item[(i)] $e^{-\frac{\ve}{4}P^s_k}a_{k+1}\leq e^{\frac{\ve}{4}P^s_k}a_k\leq a_{k+1}$;
\item[(ii)] $a_k P^s_k\in I_{\ve,q(x_k)}$.
\end{enumerate}
This choice is possible because the interval $(e^{-\frac{\ve}{4}P^s_k}a_{k+1},a_{k+1}]$
intersects $I_{\ve,q(x_k)}$,
since $\tfrac{\ve}{4}P^s_k\geq \tfrac{\ve}{4}e^{-\left(\mathfrak H+\frac{\ve}{3}\right)}q^s(f^k(x))\geq
\tfrac{\ve}{4}e^{-\left(\mathfrak H+\frac{\ve}{3}\right)}q(f^k(x))
\geq \tfrac{\ve}{4}e^{-\left(\mathfrak H+\frac{2\ve}{3}\right)}q(x_k)>\ve^2 q(x_k)$.
The first condition implies that $0<a_{n+1}\leq\cdots\leq a_{m-1}\leq a_m\leq 1$.
The maximality on the choice of $a_k$ indeed implies the inequality
$e^{-\ve^2 q(x_k)}a_{k+1}\leq e^{\frac{\ve}{4}P^s_k}a_k\leq a_{k+1}$
for every growing $k$ (this is stronger than (i)).

Before continuing, we collect some estimates relating $q(x_k),P^s_k,p^s_k$.
Fix two consecutive maximal indices $n<m$. Then then following holds for all $\ve>0$ small enough:
\begin{enumerate}[$\circ$]
\item $\displaystyle\sum_{k=n+1}^m P^s_k<\ve^{\frac{3}{\beta}-1}$:
every $k=n+1,\ldots,m-1$ is growing,
thus $P^s_k\leq \lambda^{n+1-k}P^s_{n+1}$ for $k=n+1,\ldots,m$. This implies that
$$
\sum_{k=n+1}^m P^s_k\leq P^s_{n+1}\sum_{i=0}^{m-n-1}\lambda^{-i}<\ve^{\frac{3}{\beta}+1}\frac{1}{1-\lambda^{-1}}<
2\ve^{\frac{3}{\beta}-0.5}<\ve^{\frac{3}{\beta}-1},
$$
since $\lim\limits_{\ve\to 0}\tfrac{\ve^{1.5}}{1-\lambda^{-1}}=1$.
\item $\displaystyle\sum_{k=n+1}^m q(x_k)<\ve^{\frac{3}{\beta}-1}$: by the previous item,
$$
\sum_{k=n+1}^m q(x_k)\leq e^{\mathfrak H+\frac{2\ve}{3}}\sum_{k=n+1}^m P^s_k
<2e^{\mathfrak H+\frac{2\ve}{3}}\ve^{\frac{3}{\beta}-0.5}<\ve^{\frac{3}{\beta}-1}.
$$
\item $a_{n+1}>\lambda^{-1}$: using that $a_m\geq e^{-\ve^2 q(x_m)}>e^{-\ve P^s_m}$ and that
$e^{-\ve P^s_k}a_{k+1}\leq a_k$ for every growing $k$, we have
$$
a_{n+1}\geq \exp{}\left[-\ve\sum_{k=n+1}^{m-1}P^s_k\right]a_m\geq \exp{}\left[-\ve\sum_{k=n+1}^{m}P^s_k\right]
> \exp{}\left[-\ve^{\frac{3}{\beta}}\right]>\lambda^{-1},
$$
since $\ve^{\frac{3}{\beta}}<\ve^{1.5}$. 
\end{enumerate}
In particular, $a_k>\lambda^{-1}>e^{-\ve}$ for all $k\in\Z$.

\medskip
\noindent
{\sc Claim 2:} $\Psi_{x_n}^{p^s_n,p^u_n}\in\mathfs A$ for all $n\in\Z$.

\begin{proof}[Proof of Claim $2$.] We have to check (CG1)--(CG3).

\medskip
\noindent
(CG1) By definition, $\Gamma(x_n)\in Z_{\un\ell^{(n)},m_n,j_n}$.

\medskip
\noindent
(CG2) We have $p^s_n\leq P^s_n\leq \ve Q(x_n)$, and the same holds for $p^u_n$.
By definition, $p^s_n,p^u_n\in I_{\ve,q(x_n)}$.

\medskip
\noindent
(CG3) Since $e^{-\ve}< a_n\leq 1$ and $P^s_n=e^{\pm\left(\mathfrak H+\frac{2\ve}{3}\right)}q^s(x_n)$,
we have $e^{-\mathfrak H-2\ve}\leq
\tfrac{p^s_n}{q^s(x_n)}\leq e^{\mathfrak H+\ve}$. By the same reason,
$^{-\mathfrak H-2\ve}\leq\tfrac{p^u_n}{q^u(x_n)}\leq e^{\mathfrak H+\ve}$.
These inequalities imply that
$e^{-\mathfrak H-2\ve}\leq
\tfrac{p^s_n\wedge p^u_n}{q(x_n)}\leq e^{\mathfrak H+\ve}$
and so $e^{-\mathfrak H-1}\leq \tfrac{p^s_n\wedge p^u_n}{q(x_n)}\leq e^{\mathfrak H+1}$.
\end{proof}

\medskip
\noindent
{\sc Claim 3:} $\Psi_{x_n}^{p^s_n,p^u_n}\overset{\ve}{\rightarrow}\Psi_{x_{n+1}}^{p^s_{n+1},p^u_{n+1}}$
for all $n\in\Z$.

\begin{proof}[Proof of Claim $3$.] We have to check (GPO1)--(GPO2).

\medskip
\noindent
(GPO1) By (${\rm a}_n$) with $i=1$ and (${\rm a}_{n+1}$) with $i=0$, we have
\begin{align*}
&\ d(f(x_n),x_{n+1})+\|\widetilde{C(f(x_n))}-\widetilde{C(x_{n+1})}\|\\
&\leq  d(f^{n+1}(x),f(x_n))+
\|\widetilde{C(f^{n+1}(x))}-\widetilde{C(f(x_n))}\|\\
&\ \ \ \,+ d(f^{n+1}(x),x_{n+1})+
\|\widetilde{C(f^{n+1}(x))}-\widetilde{C(x_{n+1})}\|\\
&<\tfrac{1}{2}q(f^n(x))^8+\tfrac{1}{2}q(f^{n+1}(x))^8
\overset{!}{\leq} \tfrac{1}{2}(1+e^{8\ve})q(f^{n+1}(x))^8\\
&\overset{!!}{\leq} \tfrac{1}{2}e^{8\mathfrak H+\frac{56\ve}{3}}(1+e^{8\ve})(p^s_{n+1}\wedge p^u_{n+1})^8
\overset{!!!}{<}(p^s_{n+1}\wedge p^u_{n+1})^8,
\end{align*}
where in $\overset{!}{\leq}$ we used Lemma \ref{Lemma-q}, in $\overset{!!}{\leq}$ we used
(${\rm b}_n$)and the estimate used to prove (CG3) in the previous paragraph, and in $\overset{!!!}{<}$
we used that $\tfrac{1}{2}e^{8\mathfrak H+\frac{56\ve}{3}}(1+e^{8\ve})<1$ when $\ve,\rho>0$ are sufficiently small.
This proves that
$\Psi_{f(x_n)}^{p^s_{n+1}\wedge p^u_{n+1}}\overset{\ve}{\approx}\Psi_{x_{n+1}}^{p^s_{n+1}\wedge p^u_{n+1}}$.
Similarly, we prove that
$\Psi_{f^{-1}(x_{n+1})}^{p^s_n\wedge p^u_n}\overset{\ve}{\approx}\Psi_{x_n}^{p^s_n\wedge p^u_n}$.

\medskip
\noindent
(GPO2) We show that relation (\ref{gpo2-a}) holds for all $k\in\Z$:
$$
e^{-\ve p^s_k}\min\{e^{\ve T(v_k,v_{k+1})}p^s_{k+1},e^{-\ve}\ve Q(x_k)\}\leq p^s_k\leq
\min\{e^{\ve T(v_k,v_{k+1})}p^s_{k+1},\ve Q(x_k)\}.
$$
Relation (\ref{gpo2-b}) is proved similarly. For ease of notation, write $T_k=T(v_k,v_{k+1})$
and $\Delta_k=(t_{k+1}-t_k)-T_k$. Since $T_k$ is the minimal time, we have $\Delta_k\geq 0$.
Using Lemma \ref{Lemma-regularity-q-t}(3), condition (${\rm a}_n$) and Remark
\ref{rmk-time}, we also have the following upper bound for $\Delta_k$:
$$
\Delta_k\leq {\rm diam}(R[\tfrac{1}{15}(p^s_k\wedge p^u_k)])=\tfrac{\sqrt{2}}{15}(p^s_k\wedge p^u_k)
< \tfrac{p^s_k}{4}\cdot 
$$
We fix two consecutive maximal indices $n<m$ and establish the above inequality for $k=n,\ldots,m-1$.
We divide the proof into two cases: $k=n$ and $k\neq n$.
Assume first that $k=n$. For $\ve>0$ small enough (remember $a_{n+1}>\lambda^{-1}$),
$$
e^{\ve T_n}p^s_{n+1}=e^{\ve T_n} a_{n+1}P^s_{n+1}>\exp{}\left[\inf(r_\Lambda)\ve-\ve^{1.5}\right]P^s_{n+1}
>\lambda P^s_{n+1}>P^s_n=\ve Q(x_n).
$$
Therefore
$$
e^{-\ve p^s_n}\min\{e^{\ve T_n}p^s_{n+1},e^{-\ve}\ve Q(x_n)\}=e^{-\ve p^s_n}e^{-\ve}\ve Q(x_n)
<e^{-\ve}\ve Q(x_n)<a_nP^s_n=p^s_n
$$
and
$$
\min\{e^{\ve T_n}p^s_{n+1},\ve Q(x_n)\}=\ve Q(x_n)=P^s_n\geq p^s_n.
$$
This proves (\ref{gpo2-a}) for $k=n$.

Now let $k\neq n$, and call
${\rm I}=\min\{e^{\ve T_k}p^s_{k+1},e^{-\ve}\ve Q(x_k)\}$,
${\rm II}=\min\{e^{\ve T_k}p^s_{k+1},\ve Q(x_k)\}$.
We wish to show that $e^{-\ve p^s_k}{\rm I}\leq p^s_k\leq {\rm II}$.
Since $a_{k+1}\geq e^{-\ve\Delta_k}a_{k+1}>\exp{}\left[-\ve \tfrac{p^s_k}{4}-\ve^{1.5}\right]>\exp{}[-\ve]$,
we have
\begin{align*}
&\ {\rm I}=\min\{e^{-\ve\Delta_k}a_{k+1}e^{\ve (t_{k+1}-t_k)}P^s_{k+1},e^{-\ve}\ve Q(x_k)\}\\
&\leq a_{k+1}\min\{e^{\ve (t_{k+1}-t_k)}P^s_{k+1},\ve Q(x_k)\}=a_{k+1}P^s_k.
\end{align*}
Therefore $e^{-\ve p^s_k}{\rm I}\leq e^{-\frac{\ve}{2}P^s_k}a_{k+1}P^s_k\leq a_kP^s_k=p^s_k$,
where in the second inequality we used property (i) in the definition of $a_k$.

For the other inequality, start observing that
$$
p^s_k=a_kP^s_k = a_k\min\{e^{\ve(t_{k+1}-t_k)}P^s_{k+1},\ve Q(x_k)\}=
\min\{e^{\ve(t_{k+1}-t_k)}a_kP^s_{k+1},a_k \ve Q(x_k)\}.
$$
Clearly $a_k \ve Q(x_k)\leq \ve Q(x_k)$. Using that $\Delta_k\leq \tfrac{P^s_k}{4}$,
we have $e^{\ve\Delta_k}a_k\leq e^{\frac{\ve}{4}P^s_k}a_k\leq a_{k+1}$, where in the last
passage we used property (i) in the definition of $a_k$. Hence
$$
e^{\ve(t_{k+1}-t_k)}a_kP^s_{k+1}=e^{\ve T_k}e^{\ve \Delta_k}a_k P^s_{k+1}\leq e^{\ve T_k}a_{k+1}P^s_{k+1}
=e^{\ve T_k}p^s_{k+1}.
$$
The conclusion is that
$p^s_k\leq {\rm II}$.
The proof of Claim 3 is now complete.
\end{proof}

\medskip
\noindent
{\sc Claim 4:} $\{\Psi_{x_n}^{p^s_n,p^u_n}\}_{n\in\Z}$ is regular.

\begin{proof}[Proof of Claim $4$.]
Since $x\in\nuh^\#$ and $\tfrac{p^s_n\wedge p^u_n}{q(f^n(x))}=e^{\pm(\mathfrak H+1)}$,
we have
$\limsup\limits_{n\to+\infty}p^s_n\wedge p^u_n>0$ and $\limsup\limits_{n\to-\infty}p^s_n\wedge p^u_n>0$.
By the discreteness of $\mathfs A$, it follows that $\Psi_{x_n}^{p^s_n,p^u_n}$ repeats infinitely
often in the future and infinitely often in the past.
\end{proof}

\medskip
\noindent
{\sc Claim 5:} $\{\Psi_{x_n}^{p^s_n,p^u_n}\}_{n\in\Z}$ shadows $x$.

\begin{proof}[Proof of Claim $5$.]
By (${\rm a}_n$) with $i=0$, we have
$\Psi_{f^n(x)}^{p^s_n\wedge p^u_n}\overset{\ve}{\approx}\Psi_{x_n}^{p^s_n\wedge p^u_n}$, hence 
by Proposition \ref{Lemma-overlap}(3) we have $f^n(x)=\Psi_{f^n(x)}(0)\in \Psi_{x_n}(R[p^s_n\wedge p^u_n])$,
thus $\{\Psi_{x_n}^{p^s_n,p^u_n}\}_{n\in\Z}$ shadows $x$.
This concludes the proof of sufficiency.
\end{proof}

\medskip
\noindent
{\em Proof of relevance.} The alphabet $\mathfs A$ might not a priori satisfy
the relevance condition, but we can easily reduce it to a sub-alphabet $\mathfs A'$ satisfying (1)--(3).
Call $v\in\mathfs A$ relevant if there is $\un v\in\mathfs A^\Z$ with $v_0=v$ such that $\un{v}$ shadows
a point in $\Lambda\cap\nuh^\#$. Since $\nuh^\#$ is $\vf$--invariant, every $v_i$ is relevant.
Hence $\mathfs A'=\{v\in\mathfs A:v\text{ is relevant}\}$ is discrete
because $\mathfs A'\subset\mathfs A$, it is sufficient and relevant by definition.
\end{proof}

\subsection{First coding}\label{ss.first.coding}

Let $\Sigma$ be the TMS associated to the graph with vertex set $\mathfs A$ given by
Theorem \ref{Thm-coarse-graining} and
edges $v\overset{\ve}{\to}w$. An element $\un v\in\Sigma$ is an $\ve$--gpo,
so let $\pi:\Sigma\to \widehat \Lambda$ by
$$
\{\pi(\un v)\}:=V^s[\un v]\cap V^u[\un v].
$$
Here are the main properties of the triple $(\Sigma,\sigma,\pi)$.

\begin{proposition}\label{Prop-pi}
The following holds for all $0<\ve\ll \rho\ll 1$.
\begin{enumerate}[{\rm (1)}]
\item Each $v\in\mathfs A$ has finite ingoing and outgoing degree, hence $\Sigma$ is locally compact.
\item $\pi:\Sigma\to \widehat \Lambda$ is H\"older continuous.
\item $\pi[\Sigma^\#]\supset\Lambda\cap\nuh^\#$.
\end{enumerate} 
\end{proposition}

Part (1) follows from Lemma \ref{Lemma-minimum} and Theorem \ref{Thm-coarse-graining}(1),
part (2) follows from Theorem \ref{Thm-stable-manifolds}(5),
and part (3) follows from Theorem \ref{Thm-coarse-graining}(2).
It is important noting that $(\Sigma,\sigma,\pi)$ is {\em not} the TMS that satisfies the Main Theorem,
since $\pi$ might be (and usually is) infinite-to-one.
We use $\pi$ to induce a locally finite cover of $\Lambda\cap\nuh^\#$, which will then be
refined to generate a new TMS whose TMF is the one satisfying the Main Theorem.

We finish this section introducing the TMF generated by $(\Sigma,\sigma,\pi)$.
Remember that $r_\Lambda:\Lambda\to(0,\rho/2]$ is the first return time to $\Lambda$.

\medskip
\noindent
{\sc The roof function $r:\Sigma\to (0,\rho)$:} Given $\un v=\{\Psi_{x_n}^{p^s_n,p^u_n}\}_{n\in\Z}\in\Sigma$,
let $x=\pi(\un v)$ and assume that $x_1$ belongs to the disc $D\subset{\widehat \Lambda}$.
Define $r(\un v):=r_\Lambda(x_0)-\mathfrak t_D[\vf^{r(x_0)}(x)]$.

\medskip
Since $g_{x_0}^+=\mathfrak q_D\circ \vf^{r(x_0)}$,
$r(\un v)$ is the time increment for $\vf$ between the points $\pi(\un v)$ and $g_{x_0}^+[\pi(\un v)]$.
In particular, $\vf^{r(\un v)}[\pi(\un v)]=\pi[\sigma(\un v)]$ belongs to $\widehat\Lambda$ but not necessarily to
$\Lambda$. (Note: even if $\pi(\un v),\vf^{r(\un v)}[\pi(\un v)]\in\Lambda$, the values of 
$r(\un v)$ and $r_\Lambda[\pi(\un v)]$ may be different.)

\medskip
\noindent
{\sc The triple $(\Sigma_r,\sigma_r,\pi_r)$:} We take $(\Sigma_r,\sigma_r)$
to be the TMF associated to the TMS $(\Sigma,\sigma)$ and roof function $r$,
and $\pi_r:\Sigma_r\to M$ to be the map defined by $\pi_r[(\un v,t)]=\vf^t[\pi(\un v)]$.

\medskip
The next proposition collects the main properties of $(\Sigma_r,\sigma_r,\pi_r)$.

\begin{proposition}\label{Prop-pi_R}
The following holds for all $0<\ve\ll \rho\ll 1$.
\begin{enumerate}[{\rm (1)}]
\item $\pi_r\circ\sigma_r^t=\vf^t\circ\pi_r$, for all $t\in\R$.
\item $\pi_r$ is H\"older continuous with respect to the Bowen-Walters distance.
\item $\pi_r[\Sigma_r^\#]\supset\nuh^\#$.
\end{enumerate}
\end{proposition}

\begin{proof}
Part (1) is direct from the definition of $\pi_r$. The proof of Part (2) uses Proposition \ref{Prop-pi}(2),
and follows by the same methods used in the proof of \cite[Lemma 5.9]{Lima-Sarig}.
To prove part (3), let $S:=\Sigma^\#\times\{0\}\subset\Sigma_r^\#$.
By Proposition \ref{Prop-pi}(3), $\pi_r[S]\supset\Lambda\cap\nuh^\#$.
Since $\pi_r[\Sigma_r^\#]=\bigcup\limits_{t\in\R}\vf^t[\pi_r(S)]$ and
$\nuh^\#=\bigcup\limits_{t\in\R}\vf^t[\Lambda\cap\nuh^\#]$,
we get that $\pi_r[\Sigma_r^\#]\supset\nuh^\#$. 
\end{proof}

\section{Inverse theorem}\label{Section-inverse}

In the previous section, we have constructed a first coding $\pi:\Sigma\to \widehat\Lambda$.
As mentioned, it is usually infinite-to-one. In this section, we investigate how $\pi$ loses injectivity:
if $\un v\in\Sigma$ and $x=\pi(\un v)$, what is the relation between the parameters defining $\un v$ and
those associated to the orbit of $x$? Our goal is to analyze this as an {\em inverse problem}: 
fixed $x\in \widehat\Lambda$,
the parameters of $\un v$ are defined ``up to bounded error''. The answer to this inverse problem is what
we call an {\em inverse theorem}. From now on, we require that
$\un v\in\Sigma^\#$, where $\Sigma^\#$ is the {\em regular set} of $\Sigma$:
$$
\Sigma^\#:=\left\{\un v\in\Sigma:\exists v,w\in V\text{ s.t. }\begin{array}{l}v_n=v\text{ for infinitely many }n>0\\
v_n=w\text{ for infinitely many }n<0
\end{array}\right\}.
$$

\medskip
Recall $r:\Sigma\to (0,\rho)$, the roof function defined before
Proposition \ref{Prop-pi_R}. Let $r_n$ denote its $n$--th Birkhoff sum with respect to the shift map
$\sigma:\Sigma\to\Sigma$. Let $\un v=\{\Psi_{x_n}^{p^s_n,p^u_n}\}_{n\in\Z}\in\Sigma$,
and let $x=\pi(\un v)$. Then:
\begin{enumerate}[$\circ$]
\item $\vf^{r_n(\un v)}(x)=\pi[\sigma^n(\un v)]$, a point in $\widehat\Lambda$ that is close to $x_n$.
\item $g_{x_n}^+[\vf^{r_n(\un v)}(x)]=\vf^{r_{n+1}(\un v)}(x)$.
\end{enumerate}
Let $p^{s/u}(\vf^{r_n(\un v)}(x))$ be the $\Z$--indexed version of the parameter $q^{s/u}$
with respect to the sequence of times $\{r_n(\un v)\}_{n\in\Z}$ (see Section \ref{section-Z-indexed} for the definition).

\begin{theorem}[Inverse theorem]\label{Thm-inverse}
The following holds for all $0<\ve\ll \rho\ll 1$.
If $\un v=\{\Psi_{x_n}^{p^s_n,p^u_n}\}_{n\in\Z}\in\Sigma^\#$
and $x=\pi(\un v)$, then $x\in\nuh^\#$ and the following are true.
\begin{enumerate}[{\rm (1)}]
\item $d(\vf^{r_n(\un v)}(x),x_n)<50^{-1}(p^s_n\wedge p^u_n)$.
\item $\tfrac{\sin\alpha(x_n)}{\sin\alpha(\vf^{r_n(\un v)}(x))}=e^{\pm (p^s_n\wedge p^u_n)^{\beta/4}}$,
$|\cos\alpha(x_n)-\cos\alpha(\vf^{r_n(\un v)}(x))|<2(p^s_n\wedge p^u_n)^{\beta/4}$.
\item $\tfrac{s(x_n)}{s(\vf^{r_n(\un v)}(x))}=e^{\pm \sqrt{\ve}}$ and
$\tfrac{u(x_n)}{u(\vf^{r_n(\un v)}(x))}=e^{\pm \sqrt{\ve}}$.
\item $\tfrac{Q(x_n)}{Q(\vf^{r_n(\un v)}(x))}=e^{\pm \sqrt[3]{\ve}}$.
\item[{\rm (5)}] $\tfrac{p^s_n}{p^s(\vf^{r_n(\un v)}(x))}=e^{\pm\sqrt[3]{\ve}}$ and
$\tfrac{p^u_n}{p^u(\vf^{r_n(\un v)}(x))}=e^{\pm\sqrt[3]{\ve}}$.
\item[{\rm (6)}] $\Psi_{x_n}^{-1}\circ\Psi_{\vf^{r_n(\un v)}(x)}$ and $\Psi_{\vf^{r_n(\un v)}(x)}^{-1}\circ\Psi_{x_n}$
can be written in the form $(-1)^\sigma v+\delta+\Delta(v)$ for $v\in R[10Q(x)]$,
where $\sigma\in\{0,1\}$, $\delta$ is a vector with $\|\delta\|<50^{-1}(p^s_0\wedge p^u_0)$ and
$\Delta$ is a vector field such that $\Delta(0)=0$ and $\|d\Delta\|_{C^0}<\sqrt[3]{\ve}$ on $R[10Q(x)]$.
\end{enumerate}
\end{theorem}

Part (1) is a direct consequence of Lemma \ref{Lemma-admissible-manifolds}.
Indeed, since $\vf^{r_n(\un v)}(x)=\pi[\sigma^n(\un v)]$, this point is the intersection
of a $s$--admissible and a $u$--admissible manifold at $\Psi_{x_n}^{p^s_n,p^u_n}$.
By Lemma \ref{Lemma-admissible-manifolds}(1) and since Pesin charts are $2$--Lipschitz,
we get that $d(\vf^{r_n(\un v)}(x),x_n)<50^{-1}(p^s_n\wedge p^u_n)$.

\subsection{An improvement lemma}

This section comprises the core of the proof that $x\in\nuh$ and of part (3) above.
It states that the graph transforms $\mathfs F^s/\mathfs F^u$ improve the ratios of $s/u$--parameters,
therefore we call it an {\em improvement lemma}.

\begin{lemma}[Improvement lemma]\label{Lemma-improvement}
The following holds for all $0<\ve\ll \rho\ll 1$. Let $v\overset{\ve}{\to}w$ with
$v=\Psi_x^{p^s,p^u},w=\Psi_y^{q^s,q^u}$, let $W^s\in\mathfs M^s[w]$ be the stable manifold
of a positive $\ve$--gpo, and let $V^s=\mathfs F^s_{v,w}(W^s)$, then:
\begin{enumerate}[{\rm (1)}]
\item  If $s(z)<\infty$ for some (every) $z\in W^s$, then $s(z')<\infty$ for every $z'\in V^s$.
\item Let $z\in W^s$ with $g_{y}^-(z)\in V^s$. For $\xi\geq {\sqrt{\ve}}$, if
$\tfrac{s(z)}{s(y)}=e^{\pm\xi}$ then $\tfrac{s(g_{y}^-(z))}{s(x)}=e^{\pm(\xi-Q(y)^{\beta/4})}$.
\end{enumerate} 
\end{lemma}

We note that the ratio improves.

\begin{proof}
When $M$ is a closed surface and $f$ is a $C^{1+\beta}$ diffeomorphism,
this is \cite[Lemma 7.2]{Sarig-JAMS}. When $M$ is a surface (possibly with boundary) and
$f$ is a local diffeomorphism with unbounded derivatives, this is \cite[Lemma 6.3]{Lima-Matheus}.
The main difference from these results to what we will do below is that our parameters $s,u$ 
envolve integrals instead of sums. So we need to be careful on how to split the integrals in a way
that we can control each part reasonably. In the sequel, we will use the parallel transports $P_{z,y}$
and the maps $\widetilde A$ defined in the beginning of Section \ref{s.metric}.
We will also use estimate (\ref{definition-d}), which states that 
$\|\Phi^t\|=e^{\pm 4\rho}$ for $|t|\leq 2\rho$.

\medskip
\noindent
{\sc Claim 1:} $\exists\mathfrak C=\mathfrak C(M,\vf,\theta)>0$ such that
if $z\in B_y$ and $v\in T_y\Lambda,w\in T_z\Lambda$ with $\|v\|=\|w\|=1$ then for all $|t|\leq 2\rho$:
\begin{align*}
&\, |\|\Phi^t(v)\|-\|\Phi^t(w)\||\leq\mathfrak C[d(y,z)^\beta+\|v-P_{z,y}w\|]\hspace{.2cm}\text{and}\\
&\left|\frac{\|\Phi^t(v)\|}{\|\Phi^t(w)\|}-1\right|\leq \mathfrak C[d(y,z)^\beta+\|v-P_{z,y}w\|].
\end{align*}
In particular
$\left|\log\|\Phi^t(v)\|-\log\|\Phi^t(w)\|\right|\leq \mathfrak C[d(y,z)^\beta+\|v-P_{z,y}w\|]$.

\begin{proof}[Proof of Claim $1$] The inequalities are direct consequences of the H\"older continuity of $\Phi$,
as follows: if $\mathfrak C_0=\mathfrak C_0(M,\vf,\theta)>0$ is a constant such that 
$$
|\|\Phi^t(v)\|-\|\Phi^t(w)\||\leq\mathfrak C_0[d(y,z)^\beta+\|v-P_{z,y}w\|]\\
$$
for all $y,z,v,w$ as above, then the claim holds with
$\mathfrak C:=e^{4\rho}\mathfrak C_0$.
\end{proof}

Now we start the proof of the lemma. We have $g_y^-(y)=f^{-1}(y)$, therefore
$\tfrac{s(g_y^-(z))}{s(x)}=\tfrac{s(g_y^-(z))}{s(g_y^-(y))}\cdot \tfrac{s(f^{-1}(y))}{s(x)}$.
Since $(p^s\wedge p^u)^3(q^s\wedge q^u)^3\ll Q(y)^{\beta/4}$,
Proposition \ref{Lemma-overlap}(1) implies $\tfrac{s(f^{-1}(y))}{s(x)}=e^{\pm Q(y)^{\beta/4}}$.
Thus it is enough to show that $\tfrac{s(g_y^-(z))}{s(g_y^-(y))}=e^{\pm(\xi-2Q(y)^{\beta/4})}$.
We show one side of the inequality (the other is similar). Note that this is the term that gives the improvement.

Write $g_y^-=\vf^{T^-}$ where $T^-$ is a $C^{1+\beta}$ function with $T^-(y)=-r_\Lambda(f^{-1}(y))$. Then
$g_y^-(y)=\vf^{T^-(y)}(y)$ and $g_y^-(z)=\vf^{T^-(z)}(z)$. For simplicity of notation,
let $t_0=-T^-(y)$ and $t_1=-T^-(z)$, then $g_y^-(y)=\vf^{-t_0}(y)$ and $g_y^-(z)=\vf^{-t_1}(z)$.
In the proof of Lemma \ref{Lemma-linear-reduction}
(see Appendix \ref{Appendix-proofs}), we saw that
$$
s(x)^2=4e^{4\rho}\int_0^t e^{2\chi t'}\|\Phi^{t'}n^s_x\|^2dt'+e^{2\chi t}\|\Phi^t n^s_x\|^2 s(\vf^t(x))^2
$$
for $x\in\nuh$ and $t\in\R$.
Therefore we can decompose $s(g_y^-(y))^2$ and $s(g_y^-(z))^2$ as follows:
\begin{eqnarray*}
&s(g_y^-(y))^2=\underbrace{4e^{4\rho}\int_0^{t_0} e^{2\chi t}\|\Phi^{t}n^s_{g_y^-(y)}\|^2dt}_{=:I_1}
+\underbrace{e^{2\chi t_0}\|\Phi^{t_0} n^s_{g_y^-(y)}\|^2}_{=:I_2} s(y)^2=:I_1+I_2 s(y)^2\\
&\ s(g_y^-(z))^2=\underbrace{4e^{4\rho}\int_0^{t_1} e^{2\chi t}\|\Phi^{t}n^s_{g_y^-(z)}\|^2dt}_{=:I_3}
+\underbrace{e^{2\chi t_1}\|\Phi^{t_1} n^s_{g_y^-(z)}\|^2}_{=:I_4} s(z)^2=:I_3+I_4 s(z)^2.
\end{eqnarray*}
Using that $\|\Phi^{t_0} n^s_{g_y^-(y)}\|=\|\Phi^{-t_0}n^s_y\|^{-1}$ and an analogous equation for $z$, we have
\begin{align*}
&I_1=4e^{4\rho}\int_0^{t_0} e^{2\chi t}\|\Phi^{-t}n^s_{y}\|^{-2}dt\,,\ \  I_2=e^{2\chi t_0}\|\Phi^{-t_0} n^s_{y}\|^{-2},\\
&I_3=4e^{4\rho}\int_0^{t_1} e^{2\chi t}\|\Phi^{-t}n^s_{z}\|^{-2}dt\,,\ \ I_4=e^{2\chi t_1}\|\Phi^{-t_1} n^s_z\|^{-2}.
\end{align*}

Before continuing, we need to make some estimates.

\medskip
\noindent
{\sc Claim 2:} $d(y,z)<Q(y)$ and $\|n^s_y-P_{z,y}n^s_z\|<4\ve^{1/4}Q(y)^{\beta/4}$.

\begin{proof}[Proof of Claim $2$]
We proceed as in \cite[Lemma 6.3]{Lima-Matheus}.
Let $F,G$ be the representing function of $V^s,W^s$ respectively, and let $z=\Psi_y(t,G(t))$.
Since $\Lip(G)<\ve$, we have $\|\colvec[.6]{t\\ G(t)}\|\leq |t|+|G(t)|\leq |t|(1+\Lip(G))+|G(0)|\leq (1+\ve)q^s+10^{-3}(q^s\wedge q^u)<2q^s$, therefore $d(y,z)<4q^s\leq 4\ve Q(y)<Q(y)$ for small $\ve>0$.

To bound the second term, we first estimate $\sin\angle(n^s_y,P_{z,y}n^s_z)$.
Since $n^s_y$ is the unitary vector in the direction of
$d(\Psi_y)_0\colvec{1\\0}=d(\exp{y})_0\circ C(y)\colvec{1\\0}$
and $n^s_z$ is the unitary vector in the direction of
$d(\Psi_y)_{(t,G(t))}\colvec{1\\ G'(t)}=d(\exp{y})_{C(y)\colvec[.6]{t\\G(t)}}\circ C(y)\colvec{1\\ G'(t)}$,
the angles they define are the same. In other words, if
$$
A=\widetilde{d(\exp{y})_0\circ C(y)},B=\widetilde{d(\exp{y})_{C(y)\colvec[.6]{t\\G(t)}}\circ C(y)},
v_1=\colvec{1\\0},v_2=\colvec{1\\ G'(t)}
$$
then $\sin\angle(n^s_y,P_{z,y}n^s_z)=\sin\angle(Av_1,Bv_2)$. Using (\ref{gen-ineq-angles}) 
with $L=A$, $v=v_1$, $w=A^{-1}Bv_2$, we get
\begin{align*}
&\, |\sin\angle(Av_1,Bv_2)|\leq \|A\|\|A^{-1}\||\sin\angle(v_1,A^{-1}Bv_2)|\\
&\leq \|C(y)^{-1}\|\left[|\sin\angle(v_1,v_2)|+|\sin\angle(v_2,A^{-1}Bv_2)|\right].
\end{align*}
We have $|\sin\angle(v_1,v_2)|\leq |G'(t)|\leq (q^s\wedge q^u)^{\beta/3}\leq Q(y)^{\beta/3}$.
Also, by (Exp3):
\begin{align*}
&\,\|A^{-1}B-{\rm Id}\|\leq \|A^{-1}\|\|A-B\|\leq
\|C(y)^{-1}\| \left\|\widetilde{d(\exp{y})_0}-\widetilde{d(\exp{y})_{C(y)\colvec[.6]{t\\ G(t)}}}\right\|\\
&\leq 2\mathfrak K q^s\|C(y)^{-1}\|
\leq 2\mathfrak K \ve^{1/4}Q(y)^{1-\beta/12}<\tfrac{1}{4}Q(y)^{\beta/3}\ll 1.
\end{align*}
This implies that $v_2,A^{-1}Bv_2$ are almost unitary vectors, therefore
$$
|\sin\angle(v_2,A^{-1}Bv_2)|\leq 2\|v_2-A^{-1}Bv_2\|\leq 4\|A^{-1}B-{\rm Id}\|<Q(y)^{\beta/3},$$
and so $|\sin\angle(n^s_y,P_{z,y}n^s_z)|<2\|C(y)^{-1}\|Q(y)^{\beta/3}$.
Since $\|n^s_y\|=\|P_{z,y}n^s_z\|=1$ and the angle between them is small, we conclude that
for small $\ve>0$:
$$
\|n^s_y-P_{z,y}n^s_z\|\leq 2|\sin\angle(n^s_y,P_{z,y}n^s_z)|<4\|C(y)^{-1}\|Q(y)^{\beta/3}\leq
4\ve^{1/4}Q(y)^{\beta/4}.
$$
\end{proof}

\medskip
\noindent
{\sc Claim 3:} $\tfrac{I_1}{I_3}={\rm exp}[\pm Q(y)^{\beta/4}]$
and $\tfrac{I_2}{I_4}={\rm exp}[\pm Q(y)^{\beta/4}]$.

\begin{proof}[Proof of Claim $3$]
We first bound $\tfrac{I_1}{I_3}$. Since $t_0,t_1\geq \tfrac{\inf(r_\Lambda)}{2}$, we have
$I_1,I_3\geq 4e^{4\rho}\tfrac{\inf(r_\Lambda)}{2}e^{-8\rho}=2e^{-4\rho}\inf (r_\Lambda)$ are uniformly
bounded away from zero. We have
$$
I_1-I_3=4e^{4\rho}\int_0^{t_0}e^{2\chi t}(\|\Phi^{-t}n^s_y\|^{-2}-\|\Phi^{-t}n^s_z\|^{-2})dt
-4e^{4\rho}\int_{t_0}^{t_1}e^{2\chi t}\|\Phi^{-t}n^s_z\|^{-2}dt
$$
We estimate each integral separately.
\begin{enumerate}[$\circ$]
\item By Claims 1 and 2:
\begin{align*}
&\ 4e^{4\rho}\int_0^{t_0}e^{2\chi t}\left(\|\Phi^{-t}n^s_y\|^{-2}-\|\Phi^{-t}n^s_z\|^{-2}\right)dt\leq 
4e^{4\rho}\int_0^{t_0}e^{2\rho}2e^{12\rho}\left|\|\Phi^{-t}n^s_y\|-\|\Phi^{-t}n^s_z\|\right|dt\\
&\leq 8\rho e^{18\rho}\mathfrak C[d(y,z)^\beta+\|n^s_y-P_{z,y}n^s_z\|]
\leq 16\rho e^{18\rho}\mathfrak CQ(y)^{\beta/3}<\ve^{1/8}Q(y)^{\beta/4}.
\end{align*}
\item By Lemma \ref{Lemma-regularity-q-t}(3) and the proof of Claim 2: 
\begin{align*}
&\ 4e^{4\rho}\int_{t_0}^{t_1}e^{2\chi t}\|\Phi^{-t}n^s_z\|^{-2}dt\leq 4e^{14\rho}|t_1-t_0|
\leq 4e^{14\rho}d(y,z)<16e^{14\rho}\ve Q(y)<Q(y).
\end{align*}
\end{enumerate}
Therefore $|I_1-I_3|<\ve^{1/8}Q(y)^{\beta/4}+Q(y)<2\ve^{1/8}Q(y)^{\beta/4}$, and so
$$
\left|\tfrac{I_1}{I_3}-1\right|<[2e^{-4\rho}\inf(r_\Lambda)]^{-1}2\ve^{1/8}Q(y)^{\beta/4}<\tfrac{1}{2}Q(y)^{\beta/4}.
$$
Since $e^{-2t}<1-t<1+t<e^{2t}$ for small $t>0$, the above inequality implies that
$\tfrac{I_1}{I_3}={\rm exp}[\pm Q(y)^{\beta/4}]$.

The estimate of $\tfrac{I_2}{I_4}$ is easier. We have
$\tfrac{I_4}{I_2}=e^{2\chi(t_1-t_0)}\tfrac{\|\Phi^{-t_0}n^s_y\|^2}{\|\Phi^{-t_0}n^s_z\|^2}$, and:
\begin{enumerate}[$\circ$]
\item $2\chi(t_1-t_0)=\pm 4\chi d(y,z)=\pm 4\chi Q(y)=\pm\tfrac{1}{2}Q(y)^{\beta/4}$,
hence $e^{2\chi(t_1-t_0)}={\rm exp}[\pm\tfrac{1}{2}Q(y)^{\beta/4}]$.
\item By Claim 1,
\begin{align*}
\left|\tfrac{\|\Phi^{-t_0}n^s_y\|}{\|\Phi^{-t_0}n^s_z\|}-1\right|\leq\mathfrak C[d(y,z)^\beta+\|n^s_y-P_{z,y}n^s_z\|]
<2\mathfrak C Q(y)^{\beta/3}<\tfrac{1}{8}Q(y)^{\beta/4},
\end{align*}
therefore $\tfrac{\|\Phi^{-t_0}n^s_y\|}{\|\Phi^{-t_0}n^s_z\|}={\rm exp}[\pm\tfrac{1}{4}Q(y)^{\beta/4}]$
and so $\tfrac{\|\Phi^{-t_0}n^s_y\|^2}{\|\Phi^{-t_0}n^s_z\|^2}={\rm exp}[\pm\tfrac{1}{2}Q(y)^{\beta/4}]$.
\end{enumerate}
These two items together imply that $\tfrac{I_2}{I_4}={\rm exp}[\pm Q(y)^{\beta/4}]$.
\end{proof}

Now we complete the proof of the lemma. By Claim 3, we can write
$\tfrac{s(g_y^-(z))^2}{s(g_y^-(y))^2}=\tfrac{I_3+I_4s(z)^2}{I_1+I_2s(y)^2}=
{\rm exp}[\pm Q(y)^{\beta/4}]\tfrac{I_1+I_2s(z)^2}{I_1+I_2s(y)^2}$. Since we want
to show that $\tfrac{s(g_y^-(z))^2}{s(g_y^-(y))^2}={\rm exp}[\pm(2\xi-4Q(y)^{\beta/4})]$,
it remains to prove that $\tfrac{I_1+I_2s(z)^2}{I_1+I_2s(y)^2}={\rm exp}[\pm(2\xi-5Q(y)^{\beta/4})]$.
We show one side of the inequality and leave the other to the reader. By assumption,
$s(z)\leq e^{\xi}s(y)$, hence
$$
\tfrac{I_1+I_2s(z)^2}{I_1+I_2s(y)^2}\leq\tfrac{I_1+e^{2\xi}I_2s(y)^2}{I_1+I_2s(y)^2}=
e^{2\xi}-\tfrac{I_1(e^{2\xi}-1)}{I_1+I_2s(y)^2}=e^{2\xi}\left[1-\tfrac{I_1(1-e^{-2\xi})}{I_1+I_2s(y)^2}\right].
$$
It is enough to show that $\tfrac{I_1(1-e^{-2\xi})}{I_1+I_2s(y)^2}>5Q(y)^{\beta/4}$,
since this implies
$$e^{2\xi}\left[1-\tfrac{I_1(1-e^{-2\xi})}{I_1+I_2s(y)^2}\right]<e^{2\xi}(1-5Q(y)^{\beta/4})<e^{2\xi-5Q(y)^{\beta/4}}.$$
Note that:
\begin{enumerate}[$\circ$]
\item $I_1\geq 2e^{-4\rho}\inf(r_\Lambda)$, as established in the proof of Claim 3.
\item $1-e^{-2\xi}\geq 1-e^{-2\ve^{1/2}}>\ve^{1/2}$ when $\ve>0$ is small enough.
\item Since $\sup(r_\Lambda)<1$, we have
$I_1< 4e^{14\rho}$
and $I_2\leq e^{10\rho}$. Since $s(y)\geq \sqrt{2}$,
it follows that $I_1+I_2s(y)^2<5e^{14\rho}s(y)^2$.
\end{enumerate}
Altogether, we get that
\begin{align*}
&\,\tfrac{I_1(1-e^{-2\xi})}{I_1+I_2s(y)^2}>\tfrac{2}{5}e^{-18\rho}\inf(r_\Lambda)\ve^{1/2}s(y)^{-2}
\geq\tfrac{2}{5}e^{-18\rho}\inf(r_\Lambda)\ve^{1/2}\|C(y)^{-1}\|^{-2}\\
&\geq\tfrac{2}{5}e^{-18\rho}\inf(r_\Lambda)Q(y)^{\beta/6}=
\tfrac{2}{5}e^{-18\rho}\inf(r_\Lambda)Q(y)^{-\beta/12}Q(y)^{\beta/4}\\
&\geq\tfrac{2}{5}e^{-18\rho}\inf(r_\Lambda)\ve^{-1/4}Q(y)^{\beta/4}>5Q(y)^{\beta/4},
\end{align*}
since $\tfrac{2}{5}e^{-18\rho}\inf(r_\Lambda)\ve^{-1/4}>5$ for $\ve>0$ small enough.
\end{proof}

Now we can prove that $x\in\nuh$.

\begin{proposition}\label{Prop-image-pi}
If $0<\ve\ll \rho\ll 1$, then $\pi[\Sigma^\#]\subset\nuh$.
\end{proposition}

\begin{proof}
Let $\un v=\{\Psi_{x_n}^{p^s_n,p^u_n}\}_{n\in\Z}\in\Sigma^\#$, and let $x=\pi(\un v)$.
We need to prove properties (NUH1) and (NUH2) for $x$. We prove the first property (the second is symmetric). 
The proof that $s(x)<+\infty$ for surface diffeomorphisms is contained in Claims 1 and 2 in \cite[Prop. 7.3]{Sarig-JAMS}, and uses four facts, which we also have here:
\begin{enumerate}[$\circ$]
\item The derivative of the diffeomorphism is continuous: in our context, the induced linear
Poincar\'e flow $\Phi$ is continuous.
\item Every vertex of the alphabet $\mathfs A$ is relevant: in our context, this is 
Theorem \ref{Thm-coarse-graining}(3).
\item Bounded distortion along invariant manifolds: in our context, this is Theorem \ref{Thm-stable-manifolds}(4).
\item Improvement lemma: in our context, this is Lemma \ref{Lemma-improvement}.
\end{enumerate}
Let us give the details. Let $n_k\to+\infty$ such that $(v_{n_k})_{k\geq 0}$ is constant. 
Since $\pi[\Sigma^\#]$ and $\nuh$ are invariant,
we can assume that $n_0=0$. Since $v_0$ is relevant, there is $\un w=\{w_n\}_{n\in\Z}$
with $w_0=v_0$ such that $y=\pi(\un w)\in \nuh^\#$. In particular, $s(y)<+\infty$.
Let $V:=V^s[\un w]$. We claim that $\sup_{y'\in V}s(y')<+\infty$. To prove this, fix
$y'\in V$. Using the same notation of Proposition \ref{Prop-disjointness},
let 
$$
t_n=\sum_{k=0}^{n-1}T_k[(g_{x_{k-1}}^+\circ\cdots\circ g_{x_0}^+)(y)]\ \text{ and }\
t_n'=\sum_{k=0}^{n-1}T_k[(g_{x_{k-1}}^+\circ\cdots\circ g_{x_0}^+)(y')].
$$
In particular, $(g_{x_{n-1}}^+\circ \cdots g_{x_0}^+)(y)=\vf^{t_n}(y)$
and $(g_{x_{n-1}}^+\circ \cdots g_{x_0}^+)(y')=\vf^{t_n'}(y')$.
By Lemma \ref{Lemma-regularity-q-t}(3)
and Theorem \ref{Thm-stable-manifolds}(3), we have 
\begin{align*}
&|t_n-t_n'|\leq 
\sum_{k=0}^{n-1}{\rm Lip}(T_k) d((g_{x_{n-1}}^+\circ \cdots g_{x_0}^+)(y),(g_{x_{n-1}}^+\circ \cdots g_{x_0}^+)(y'))\\
&\leq 
d(\Psi_{x_0}^{-1}(y),\Psi_{x_0}^{-1}(y'))\sum_{k=0}^{n-1}e^{-\frac{\chi\inf(r_\Lambda)}{2}k}<\frac{4p^s_0}{1-e^{-\frac{\chi\inf(r_\Lambda)}{2}}}
\ll \ve\ll \rho.
\end{align*}
Since $\|d(g_{x_{n-1}}^+\circ \cdots g_{x_0}^+)_y w\|=\|\Phi^{t_n}n^s_y\|$
and $\|d(g_{x_{n-1}}^+\circ \cdots g_{x_0}^+)_{y'}w'\|=\|\Phi^{t_n'}n^s_{y'}\|$, it follows 
from Theorem \ref{Thm-stable-manifolds}(4) and estimate (\ref{definition-d}) that
$\tfrac{\|\Phi^{t_n}n^s_y\|}{\|\Phi^{t_n}n^s_{y'}\|}=\tfrac{\|\Phi^{t_n}n^s_y\|}{\|\Phi^{t_n'}n^s_{y'}\|}\cdot
\tfrac{\|\Phi^{t_n'}n^s_{y'}\|}{\|\Phi^{t_n}n^s_{y'}\|}=e^{\pm(Q(x_0)^{\beta/4}+4\rho)}=e^{\pm 6\rho}$.
Now we interpolate this estimate. Given $t\geq 0$, let $n$ such that 
$t_n\leq t\leq t_{n+1}$. Since $|t_{n+1}-t_n|\leq \rho$, using estimate (\ref{definition-d}) again
gives that
$$
\tfrac{\|\Phi^{t}n^s_y\|}{\|\Phi^{t}n^s_{y'}\|}=\tfrac{\|\Phi^{t}n^s_y\|}{\|\Phi^{t_n}n^s_{y}\|}\cdot
\tfrac{\|\Phi^{t_n}n^s_y\|}{\|\Phi^{t_n}n^s_{y'}\|}\cdot \tfrac{\|\Phi^{t_n}n^s_{y'}\|}{\|\Phi^{t}n^s_{y'}\|}=e^{\pm 14\rho}.
$$ 
This implies that
$\tfrac{s(y)}{s(y')}=e^{\pm 14\rho}$. Since $y'\in V$ is arbitrary, $L_0:=\sup_{y'\in V}s(y')<+\infty$.

The next step is to prove that $s(x)<+\infty$. Recalling that $V\in \mathfs M^s(v_0)= \mathfs M^s(v_{n_k})$, define 
$V_k:=(\mathfs F^s_{v_0,v_1}\circ\mathfs F^s_{v_1,v_2}\circ\cdots\circ \mathfs F^s_{v_{n_k-1},v_{n_k}})[V]$.
By Section \ref{ss.graph.transform}, $(V_k)_{k\geq 0}$ converges in the $C^1$ topology to 
$V^s[\un v]$. In other words, if $G$ is the representing function of $V^s[\un v]$ and $G_k$
is the representing function of $V_k$, then $(G_k)_{k\geq 0}$ converges to $G$ in the $C^1$ topology.
Writing $x=\Psi_{x_0}(t,G(t))$, let $z_k=\Psi_{x_0}(t,G_k(t))\in V_k$ and 
$y_k=(g_{x_{n_k-1}}^+\circ\cdots\circ g_{x_0}^+)(z_k)$.
By Theorem \ref{Thm-stable-manifolds}(2), 
we have $y_k\in V$ and so $s(y_k)\leq L_0$. Consider the ratio $\tfrac{s(y_k)}{s(x_0)}$,
which is bounded by $\tfrac{L_0}{s(x_0)}$. Since $x_0=x_k$ by our choice of $n_k$, we can apply Lemma \ref{Lemma-improvement} along the sequence
of edges $v_0\to v_1\to\cdots\to v_{n_k}$. We obtain that
$\tfrac{s(z_k)}{s(x_0)}\leq \max\left\{e^{\sqrt{\ve}},\tfrac{L_0}{s(x_0)}\right\}=:L_1$.
Since $\Phi$ is continuous and $n^s_{x_k}\to n^s_x$ as $k\to+\infty$,
for every $T\geq 0$ we have
$$
4e^{4\rho}\int_0^T e^{2\chi t}\|\Phi^t n^s_x\|^2dt\leq 
\limsup_{k\to+\infty}4e^{4\rho}\int_0^T e^{2\chi t}\|\Phi^t n^s_{x_k}\|^2
\leq s(x_k)^2\leq L_1^2s(x_0)^2.
$$
Taking $T\to+\infty$, we conclude that $s(x)\leq L_1s(x_0)$.

Now we prove that 
$\liminf\limits_{t\to+\infty}\tfrac{1}{t}\log\|\Phi^{-t}n^s_x\|>0$.
Let $t_n=r_n(\un v)$ (see before the statement of Theorem \ref{Thm-inverse} for the definition
of $r_n(\un v)$). For $n\geq 0$, we have $0\leq -t_{-n}\leq n\sup(r_\Lambda)$, hence it is enough to prove that
$\liminf\limits_{n\to+\infty}\tfrac{1}{n}\log\|\Phi^{t_{-n}}n^s_x\|>0$.
 This can also be done as in the case of diffeomorphisms, as follows.
The second estimate of Theorem \ref{Thm-stable-manifolds}(3) gives that 
$\|\Phi^{r_n(\un v)}n^s_x\|\leq \|C(x_0)^{-1}\|e^{-\frac{\chi\inf (r_\Lambda)}{2}n}$
for every $n\geq 0$. Applying this to $\sigma^{-n}(\un v)$ and $G_{-n}(x)=\pi[\sigma^{-n}(\un v)]$,
we get that
$$
\|\Phi^{t_{-n}}n^s_x\|=\|\Phi^{-t_{-n}}n^s_{G_{-n}(x)}\|^{-1}\geq \|C(x_{-n})^{-1}\|^{-1}e^{\frac{\chi\inf (r_\Lambda)}{2}n}.
$$
Since
$\|C(x_{-n})^{-1}\|^{-1}\geq Q(x_{-n})^{\frac{\beta}{12}}\geq (p^s_n\wedge p^u_n)^{\frac{\beta}{12}}
\geq (e^{-2\ve n}p^s_0\wedge p^u_0)^{\frac{\beta}{12}}$, we have that 
$$
\liminf_{n\to+\infty}\tfrac{1}{n}\log\|\Phi^{t_{-n}}n^s_x\|\geq \tfrac{\chi\inf(r_\Lambda)}{2}-\tfrac{\beta\ve}{6},
$$
which is positive if $\ve>0$ is small enough.
\end{proof}

\subsection{Control of $\alpha(x_n),s(x_n),u(x_n),Q(x_n)$}

We now prove parts (2)--(4) of Theorem \ref{Thm-inverse}.
Part (2) follows directly from \ref{Lemma-admissible-manifolds}(2), as follows:
since $\vf^{r_n(\un v)}(x)=\pi[\sigma^n(\un v)]$ is the intersection point of a
$s$--admissible and a $u$--admissible manifold at $\Psi_{x_n}^{p^s_n,p^u_n}$, we have
$$
\tfrac{\sin\alpha(x_n)}{\sin\alpha(\vf^{r_n(\un v)}(x))}=e^{\pm (p^s_n\wedge p^u_n)^{\beta/4}}
\text{ and }|\cos\alpha(x_n)-\cos\alpha(\vf^{r_n(\un v)}(x))|<2(p^s_n\wedge p^u_n)^{\beta/4}.
$$

Now we proceed to control $s(x_n)$ and $u(x_n)$.

\begin{proposition}
The following holds for all $0<\ve\ll \rho\ll 1$.
If $\un v=\{\Psi_{x_n}^{p^s_n,p^u_n}\}_{n\in\Z}\in\Sigma^\#$ and $x=\pi(\un v)$ then for all $n\in\Z$:
$$
\tfrac{s(x_n)}{s(\vf^{r_n(\un v)}(x))}=e^{\pm\sqrt{\ve}}\text{ and }\tfrac{u(x_n)}{u(\vf^{r_n(\un v)}(x))}=e^{\pm\sqrt{\ve}}
$$
\end{proposition}

\begin{proof}
When $M$ is compact and $f$ is a $C^{1+\beta}$ diffeomorphism,
this is \cite[Prop. 7.3]{Sarig-JAMS}, and our proof follows the same methods.
To ease notation, write $z_n=\vf^{r_n(\un v)}(x)$, $n\in\Z$.
We sketch the proof for the first estimate:
\begin{enumerate}[$\circ$]
\item By Proposition \ref{Prop-image-pi}, $\pi[\Sigma^\#]\subset\nuh$ 
hence $s(x)<\infty$.
\item As in Claim 1 of \cite[Prop. 7.3]{Sarig-JAMS}, there is $\xi\geq\sqrt{\ve}$
and a sequence $n_k\to+\infty$ such that $\tfrac{s(x_{n_k})}{s(z_{n_k})}=e^{\pm\xi}$.
\item Since $g_{x_n}^-(z_n)=z_{n-1}$, we can apply Lemma \ref{Lemma-improvement} along
$\un v$ and the points $z_n$: if
$v_n=v$ for infinitely many $n>0$, then the ratio improves at each of these indices.
\end{enumerate}
The conclusion is that $\tfrac{s(x_n)}{s(z_n)}=e^{\pm\sqrt{\ve}}$ for all $n\in\Z$.
\end{proof}

\medskip
Part (4) is consequence of parts (2) and (3).
Remind that
$$
Q(x):=\ve^{3/\beta}\|C(x)^{-1}\|_{\rm Frob}^{-12/\beta}=
\ve^{3/\beta}\left(\tfrac{\sqrt{s(x)^2+u(x)^2}}{|\sin\alpha(x)|}\right)^{-12/\beta}.
$$
By part (2), $\tfrac{\sin\alpha(x_n)}{\sin\alpha(z_n)}=e^{\pm\sqrt{\ve}}$.
By part (3), $\tfrac{\sqrt{s(x_n)^2+u(x_n)^2}}{\sqrt{s(z_n)^2+u(z_n)^2}}=e^{\pm\sqrt{\ve}}$.
Therefore $\tfrac{\|C(x_n)^{-1}\|_{\rm Frob}}{\|C(z_n)^{-1}\|_{\rm Frob}}=e^{\pm 2\sqrt{\ve}}$, and so
$\tfrac{Q(x_n)}{Q(z_n)}=\tfrac{\|C(x_n)^{-1}\|_{\rm Frob}^{-12/\beta}}{\|C(z_n)^{-1}\|_{\rm Frob}^{-12/\beta}}
={\rm exp}[\pm\tfrac{24}{\beta}\sqrt{\ve}]={\rm exp}[\pm\sqrt[3]{\ve}]$ when $\ve>0$ is small enough.

\subsection{Control of $p^s_n,p^u_n$}\label{section-control-ps-pu}

Up to now, we have proved that $x\in\nuh$ and Parts (1)--(4) of Theorem 
\ref{Thm-inverse}. Now we prove Part (5). In particular, it follows that
$x\in\nuh^\#$. 
We continue to write  $z_n=\vf^{r_n(\un v)}(x)$, as in the previous section.
The control of $p^{s/u}_n$ consists on proving that it is comparable to $p^{s/u}(z_n)$.
To have the control from below, we will use that $\{\Psi_{x_n}^{p^s_n,p^u_n}\}_{n\in\Z}\in\Sigma^\#$
implies that the parameters $p^{s/u}_n$ are almost maximal infinitely often.
Proposition \ref{Prop-Z-par}(3) is the statement of maximality for $p^{s/u}(z_n)$.
The statement for $p^{s/u}_n$ is in the next lemma.
For simplicity of notation, write $T_k=T(v_k,v_{k+1})$.

\begin{lemma}\label{Lemma-max-in-chart}
If $\{\Psi_{x_n}^{p^s_n,p^u_n}\}_{n\in\Z}\in\Sigma^\#$ then
$\min\{e^{\ve T_n}p^s_{n+1},e^{-\ve}\ve Q(x_n)\}=e^{-\ve}\ve Q(x_n)$
for infinitely many $n>0$, and $\min\{e^{\ve T_n}p^u_n,e^{-\ve}\ve Q(x_{n+1})\}=e^{-\ve}\ve Q(x_{n+1})$
for infinitely many $n<0$.
\end{lemma}

\begin{proof}
The strategy is the same used in the proof of Proposition \ref{Prop-Z-par}(3). 
We prove the first statement (the second is identical). By contradiction, assume that
there exists $n\in\Z$ such that $\min\{e^{\ve T_N}p^s_{N+1},e^{-\ve}\ve Q(x_N)\}=e^{\ve T_N}p^s_{N+1}$
for all $N\geq n$. By (GPO2), it follows that $p^s_N\geq e^{\ve(T_N-p^s_N)}p^s_{N+1}$ for all $N\geq n$.
Let $\lambda=\exp{}[\ve^{1.5}]$,
then $\ve(T_N-p^s_N)\geq \ve(\inf(r_\Lambda)-\ve)>\lambda$ when $\ve>0$ is sufficiently small.
Hence $p^s_N> \lambda p^s_{N+1}$ for all $N\geq n$, and so $p^s_n\geq \lambda^{N-n}p^s_N$
for all $N\geq n$. This is contradiction, since $p^s_n<\ve$ and $\limsup\limits_{N\to+\infty}p^s_N>0$.
\end{proof}

Now we prove Theorem \ref{Thm-inverse}(5).
We will prove the statement for $p^s_n$ and $p^s(z_n)$
(the proof for $p^u_n$ and $p^u(z_n)$ is identical). 

\medskip
\noindent
{\bf Step 1.} $p^s_n\geq e^{-\sqrt[3]{\ve}}p^s(z_n)$ for all $n\in\Z$.

\medskip
We divide the proof into two cases, according to whether $n$ satisfies Lemma \ref{Lemma-max-in-chart}
or not. Assume first that it does, i.e. $\min\{e^{\ve T_n}p^s_{n+1},e^{-\ve}\ve Q(x_n)\}=e^{-\ve}\ve Q(x_n)$.
By (GPO2), we have $p^s_n\geq e^{-\ve p^s_n}e^{-\ve}\ve Q(x_n)\geq e^{-2\ve}\ve Q(x_n)$.
By Theorem \ref{Thm-inverse}(4), we get that
$$
p^s_n\geq  e^{-2\ve}\ve Q(x_n)\geq e^{-2\ve-O(\sqrt{\ve})}\ve Q(z_n)
\geq e^{-2\ve-O(\sqrt{\ve})}p^s(z_n)
\geq e^{-\sqrt[3]{\ve}}p^s(z_n).
$$

\medskip
Now assume that $n$ does not satisfy Lemma \ref{Lemma-max-in-chart}. Take the
smallest $m>n$ that satisfies Lemma \ref{Lemma-max-in-chart}. Hence
$\min\{e^{\ve T_k}p^s_{k+1},e^{-\ve}\ve Q(x_k)\}=e^{\ve T_k}p^s_{k+1}$
for $k=n,\ldots,m-1$. By (GPO2), we get that $p^s_k\geq e^{\ve(T_k-p^s_k)}p^s_{k+1}>\lambda p^s_{k+1}$
for $k=n,\ldots,m-1$. Therefore
$p^s_k\leq \lambda^{n-k}p^s_n$ for $k=n,\ldots,m-1$. Writing $\Delta_k=(t_{k+1}-t_k)-T_k\geq 0$,
this latter estimate gives two consequences:
\begin{enumerate}[$\circ$]
\item $\displaystyle\sum_{k=n}^{m-1}p^s_k<\ve$: indeed,
$$
\sum_{k=n}^{m-1}p^s_k\leq p^s_n\sum_{k=n}^{m-1}\lambda^{n-k}\leq \ve^{\frac{3}{\beta}}\frac{1}{1-\lambda^{-1}}
<2\ve^{\frac{3}{\beta}-1.5}<\ve,
$$
since $\lim\limits_{\ve\to 0}\tfrac{\ve^{1.5}}{1-\lambda^{-1}}=1$.
\item $\displaystyle\sum_{k=n}^{m-1} \Delta_k<\ve$: since the transition time from $x_k$ to $x_{k+1}$
is 2--Lipschitz (Lemma \ref{Lemma-regularity-q-t}(3)), we have
$$
\sum_{k=n}^{m-1}\Delta_k\leq 4\sum_{k=n}^{m-1}p^s_k<8\ve^{\frac{3}{\beta}-1.5}<\ve.
$$
\end{enumerate}
Using that $p^s_k\geq e^{\ve(T_k-p^s_k)}p^s_{k+1}=e^{-\ve(p^s_k+\Delta_k)}e^{\ve(t_{k+1}-t_k)}p^s_{k+1}$
for $k=n,\ldots,m-1$, we conclude that
\begin{align*}
&\ p^s_n\geq \exp{}\left[-\ve\sum_{k=n}^{m-1}p^s_k-\ve\sum_{k=n}^{m-1}\Delta_k\right]e^{\ve(t_m-t_n)}p^s_m\\
&\geq \exp{}\left[-2\ve^2-2\ve-O(\sqrt{\ve})\right]e^{\ve(t_m-t_n)}p^s(z_m)
\geq e^{-\sqrt[3]{\ve}}p^s(z_n),
\end{align*}
where in the last inequality we used Proposition \ref{Prop-Z-par}(2).

\medskip
\noindent
{\bf Step 2.} $p^s(z_n)\geq e^{-\sqrt[3]{\ve}}p^s_n$ for all $n\in\Z$.

\medskip
The motivation for this inequality is that $p^s(z_n)$ grows at least as much as $p^s_n$, since
$p^s(z_n)$ satisfies the recursive equality $p^s(z_n)=\min\{e^{\ve(t_{n+1}-t_n)}p^s(z_{n+1}),\ve Q(z_n)\}$ while by (GPO2) we have the recursive inequality $p^s_n\leq \min\{e^{\ve T_n}p^s_{n+1},\ve Q(x_n)\}$
and $t_{n+1}-t_n\geq T_n$. For ease of notation, let $n=0$ (the general case is identical).
By the above recurve equality and inequality, we have
$$p^s(z_0)=\ve\inf\{e^{\ve t_n}Q(z_n):n\geq 0\}\ \text{ and }\
p^s_0\leq \ve\inf\{e^{\ve(T_0+\cdots+T_{n-1})}Q(x_n):n\geq 0\}.
$$
Using Part (4) and that $t_n=\sum\limits_{k=0}^{n-1}(t_{k+1}-t_k)\geq \sum\limits_{k=0}^{n-1}T_k$, 
we conclude that
$$
\ p^s(z_0)=\ve\inf\{e^{\ve t_n}Q(z_n):n\geq 0\}\geq e^{-\sqrt[3]{\ve}}\ve \inf\{e^{\ve(T_0+\cdots+T_{n-1})}Q(x_n):n\geq 0\}=e^{-\sqrt[3]{\ve}} p^s_0.
$$
Steps 1 and 2 conclude the proof of Part (5). In particular, since $\{\Psi_{x_n}^{p^s_n,p^u_n}\}_{n\in\Z}\in\Sigma^\#$,
it follows that $x\in\nuh^\#$.

\subsection{Control of $\Psi_{x_0}^{-1}\circ\Psi_x$}

In the case of diffeomorphisms, this is \cite[Thm. 5.2]{Sarig-JAMS}, whose idea
of proof is the following: if $\un v=\{\Psi_{x_n}^{p^s_n,p^u_n}\}_{n\in\Z},
\un w=\{\Psi_{y_n}^{q^s_n,q^u_n}\}_{n\in\Z}\in\Sigma^\#$ with $\pi(\un v)=\pi(\un w)=x$,
then the parameters of $\Psi_{x_n}^{p^s_n,p^u_n}$ and $\Psi_{y_n}^{q^s_n,q^u_n}$ are comparable,
hence $\Psi_{y_n}^{-1}\circ\Psi_{x_n}$ is close to $\pm{\rm Id}$.
In our case, we know that $x\in\nuh^\#$, hence the Pesin charts along the orbit of $x$ are well-defined.
By parts (1)--(5), the parameters of $\Psi_{x_0}^{p^s_0,p^u_0}$ and
$\Psi_x^{q^s(x),q^u(x)}$ are comparable, therefore we can apply the same proof 
of \cite[Thm 5.2]{Sarig-JAMS} to conclude that both $\Psi_{x_0}^{-1}\circ\Psi_{x}$ and
$\Psi_{x}^{-1}\circ\Psi_{x_0}$ can be written in the form $(-1)^\sigma v+\delta+\Delta(v)$ for $v\in R[10Q(x)]$,
where $\sigma\in\{0,1\}$ and $\Delta$ is a vector field such that $\Delta(0)=0$ and $\|d\Delta\|_{C^0}<\sqrt[3]{\ve}$
on $R[10Q(x)]$. The proof will be complete once we estimate $\|\delta\|$.

Assume that $(\Psi_{x_0}^{-1}\circ\Psi_{x})(v)=(-1)^\sigma v+\delta+\Delta(v)$ as above, and write
$p=p^s_0\wedge p^u_0$. By Lemma \ref{Lemma-admissible-manifolds}(1),
$x=\Psi_{x_0}(\eta)$ for some $\eta\in R[10^{-2}p]$. In particular $\|\eta\|\leq 10^{-2}\sqrt{2}p<50^{-1}p$.
Since $\Psi_x(0)=x$, taking $v=0$ we conclude that
$\eta=\delta$, hence $\|\delta\|<50^{-1}p$.
Similarly, if $\Psi_x^{-1}\circ\Psi_{x_0}=(-1)^\sigma v+\delta+\Delta(v)$ then $v=\eta$ gives
$0=(-1)^\sigma\eta+\delta+\Delta(\eta)$ and so
$$
\|\delta\|\leq \|\eta\|+\|\Delta(\eta)\|\leq (1+\|d\Delta\|_{C^0})\|\eta\|\leq (1+\sqrt[3]{\ve})10^{-2}\sqrt{2}p< 50^{-1}p.
$$

\section{A countable locally finite section}\label{Section-locally-finite-section}

Up to now, we have:
\begin{enumerate}[$\circ$]
\item Constructed a countable family $\mathfs A$ of $\ve$--double charts,
see Theorem \ref{Thm-coarse-graining}.
\item Letting $\Sigma$ be the TMS defined by $\mathfs A$ with the edge condition
defined in Section \ref{ss.pseudo.orbits}, we constructed a H\"older continuous map
$\pi:\Sigma\to \widehat\Lambda$ that ``captures'' all orbits in $\nuh^\#$, see
Propositions \ref{Prop-pi} and \ref{Prop-pi_R}. The map $\pi$ is defined as
$\{\pi(\un v)\}:=V^s[\un v]\cap V^u[\un v]$.
\item Although $\pi$ is not finite-to-one, we solved the inverse problem by analyzing 
when $\pi$ loses injectivity, see Theorem \ref{Thm-inverse}.
\end{enumerate}
We now use these information to construct a countable family $\mathfs Z$ of subsets of $\widehat\Lambda$
with the following properties:
\begin{enumerate}[$\circ$]
\item The union of elements of $\mathfs Z$, from now on also denoted by $\mathfs Z$,
is a section that contains $\Lambda\cap\nuh^\#$.
\item $\mathfs Z$ is {\em locally finite}: each point $x\in\mathfs Z$ belongs to at most finitely many
rectangles $Z\in\mathfs Z$.
\item Every element $Z\in \mathfs Z$ is a {\em rectangle}: each point $x\in Z$ has
{\em invariant fibres} $W^s(x,Z)$, $W^u(x,Z)$ in $Z$,
and these fibres induce a local product structure on $Z$.
\item $\mathfs Z$ satisfies a {\em symbolic Markov property}.

\end{enumerate} 
In this section, all statements assume that $0<\ve\ll \rho\ll 1$, so we will omit this information.

\subsection{The Markov cover $\mathfs Z$}
Let $\mathfs Z:=\{Z(v):v\in\mathfs A\}$, where
$$
Z(v):=\{\pi(\un v):\un v\in\Sigma^\#\text{ and }v_0=v\}.
$$
In other words, $\mathfs Z$ is the family of sets induced by $\pi$ under
the natural partition of $\Sigma^\#$ into
cylinders at the zeroth position. 
Using admissible manifolds, we define {\em invariant fibres} inside each $Z\in\mathfs Z$. Let $Z=Z(v)$.

\medskip
\noindent
{\sc $s$/$u$--fibres in $\mathfs Z$:} Given $x\in Z$, let $W^s(x,Z):=V^s[\{v_n\}_{n\geq 0}]\cap Z$
be the {\em $s$--fibre} of $x$ in $Z$ for some (any) $\un v=\{v_n\}_{n\in\Z}\in\Sigma^\#$
such that $\pi(\un v)=x$ and $v_0=v$. Similarly, let $W^u(x,Z):=V^u[\{v_n\}_{n\leq 0}]\cap Z$ be
the {\em $u$--fibre} of $x$ in $Z$.

\medskip
By Proposition \ref{Prop-disjointness}, the above definitions do not depend on the choice of $\un v$, 
and any two $s$--fibres ($u$--fibres) either coincide or are disjoint. We also
define $V^s(x,Z):=V^s[\{v_n\}_{n\geq 0}]$ and $V^u(x,Z):=V^u[\{v_n\}_{n\leq 0}]$.
We can make two distinctions between $V^{s/u}(x,Z)$ and $W^{s/u}(x,Z)$:
\begin{enumerate}[$\circ$]
\item $V^{s/u}(x,Z)$ are smooth curves, while $W^{s/u}(x,Z)$ are usually fractal sets.
\item $V^{s/u}(x,Z)$ are {\em not} subsets of $Z$, while $W^{s/u}(x,Z)$ are.
\end{enumerate}

\subsection{Fundamental properties of $\mathfs Z$}\label{subsec-fundpropZ}

\medskip
Although $\mathfs Z$ is usually a fractal set (and hence not a proper section),
we can still define its Poincar\'e return map. Indeed, if $x=\pi(\un v)\in\mathfs Z$ with
$\un v\in\Sigma^\#$ then $\vf^{r_n(\un v)}(x)=\pi[\sigma^n(\un v)]\in\mathfs Z$ for all $n\in\N$.
Define $r_{\mathfs Z}:\mathfs Z\to (0,\rho)$ by $r_{\mathfs Z}(x):=\min\{t>0:\vf^t(x)\in\mathfs Z\}$. 

\medskip
\noindent
{\sc The return map $H$:} It is the map $H:\mathfs Z\to\mathfs Z$ defined by $H(x):=\vf^{r_{\mathfs Z}}(x)$.

\medskip
Below we collect the main properties of $\mathfs Z$.

\begin{proposition}\label{Prop-Z}
The following are true.
\begin{enumerate}[{\rm (1)}]
\item {\sc Covering property:} $\mathfs Z$ is a cover of $\Lambda\cap\nuh^\#$.
\item {\sc Local finiteness:} For every $Z\in\mathfs Z$,
$$
\#\left\{Z'\in\mathfs Z:\left(\bigcup_{|n|\leq 1}H^n[Z]\right)\cap Z'\neq\emptyset\right\}<\infty.
$$
\item {\sc Local product structure:} For every $Z\in\mathfs Z$ and every $x,y\in Z$, the intersection
$W^s(x,Z)\cap W^u(y,Z)$ consists of a single point, and this point belongs to $Z$.
\item {\sc Symbolic Markov property:} If $x=\pi(\un v)\in\mathfs Z$ with
$\un v=\{v_n\}_{n\in\Z}=\{\Psi_{x_n}^{p^s_n,p^u_n}\}_{n\in\Z}\in\Sigma^\#$, then
\begin{align*}
&g_{x_0}^+(W^s(x,Z(v_0)))\subset W^s(g_{x_0}^+(x),Z(v_1))\text{ and }\\
&g_{x_1}^{-}(W^u(g_{x_0}^+(x),Z(v_1)))\subset W^u(x,Z(v_0)).
\end{align*}
\end{enumerate}
\end{proposition}

Before proceeding to the proof, we use part (3) to give the following definition: 
for $x,y\in Z$, let $[x,y]_Z:=$ intersection point of $W^s(x,Z)$ and $W^u(y,Z)$, and
call it the {\em Smale bracket} of $x,y$ in $Z$.

\begin{proof}
We have $\mathfs Z=\pi[\Sigma^\#]$. Since $\pi[\Sigma^\#]\supset\Lambda\cap\nuh^\#$
by Proposition \ref{Prop-pi}(3), it follows that $\mathfs Z$ contains $\Lambda\cap\nuh^\#$.
This proves (1).

\medskip
\noindent
(2) Write $Z=Z[\Psi_x^{p^s,p^u}]$, and take $Z'=Z[\Psi_y^{q^s,q^u}]$ such that
$$\left(\bigcap_{|n|\leq 1}H^n[Z]\right)\cap Z'\neq\emptyset.$$ We will estimate the ratio
$\tfrac{p^s\wedge p^u}{q^s\wedge q^u}$. By assumption, there is $x\in Z$ such that $x'=H^n(x)\in Z'$
for some $|n|\leq 1$. Let $\un v\in\Sigma^\#$ with $v_0=\Psi_x^{p^s,p^u}$ such that $x=\pi(\un v)$.
Recalling that $p^{s/u}(x)=p^{s/u}(x,\mathcal T,0)$ for $\mathcal T=\{R_n(\un v)\}_{n\in\Z}$,
the following holds:
\begin{enumerate}[$\circ$]
\item $x\in Z$, hence by Theorem \ref{Thm-inverse}(5) we have
$\tfrac{p^s}{p^s(x)}=e^{\pm\sqrt[3]{\ve}}\text{ and }\tfrac{p^u}{p^u(x)}=e^{\pm\sqrt[3]{\ve}}$,
and so $\tfrac{p^s\wedge p^u}{p^s(x)\wedge p^u(x)}=e^{\pm\sqrt[3]{\ve}}$.
By Proposition \ref{Prop-Z-par}(1), we have $\tfrac{p^s(x)\wedge p^u(x)}{q(x)}=e^{\pm\mathfrak H}$.
The conclusion is that $\tfrac{p^s\wedge p^u}{q(x)}=e^{\pm(\sqrt[3]{\ve}+\mathfrak H)}$.
\item $x'\in Z'$, hence by the same reason $\tfrac{q^s\wedge q^u}{q(x')}=e^{\pm(\sqrt[3]{\ve}+\mathfrak H)}$.
\item $x'=\vf^t(x)$ with $|t|\leq 2\rho$, hence by Lemma \ref{Lemma-q} we have
$\tfrac{q(x)}{q(x')}=e^{\pm 2\ve}$.
\end{enumerate}
Altogether, we conclude that $\tfrac{p^s\wedge p^u}{q^s\wedge q^u}=e^{\pm2(\sqrt[3]{\ve}+\ve+\mathfrak H)}$,
and so
$$
\left\{Z'\in\mathfs Z:\left(\bigcap_{|n|\leq 1}H^n[Z]\right)\cap Z'\neq\emptyset\right\}\subset\{\Psi_y^{q^s,q^u}\in\mathfs A:(q^s\wedge q^u)\geq e^{-2(\sqrt[3]{\ve}+\ve+\mathfrak H)}(p^s\wedge p^u)\}.
$$
By Theorem \ref{Thm-coarse-graining}(1), this latter set is finite.

\medskip
\noindent
(3) We proceed as in \cite[Prop. 10.5]{Sarig-JAMS}.
Let $Z=Z(v)$, and take
$x,y\in Z$, say $x=\pi(\un v),y=\pi(\un w)$ with $\un v,\un w\in\Sigma^\#$, where
$\un v=\{v_n\}_{n\in\Z}=\{\Psi_{x_n}^{p^s_n,p^u_n}\}_{n\in\Z}$ and
$\un w=\{w_n\}_{n\in\Z}=\{\Psi_{y_n}^{q^s_n,q^u_n}\}_{n\in\Z}$ with $v_0=w_0=v$.
We let $z=\pi(\un u)$ where $\un u=\{u_n\}_{n\in\Z}$ is defined by
$$
u_n=\left\{\begin{array}{ll}v_n&,n\geq 0\\ w_n&,n\leq 0.\end{array}\right.
$$
We claim that $\{z\}=W^s(x,Z)\cap W^u(y,Z)$. To prove this, first remember that
$V^s[\{u_n\}_{n\geq 0}]\cap V^u[\{u_n\}_{n\geq 0}]$ intersects at a single point
(Lemma \ref{Lemma-admissible-manifolds}(1)), and that $z$ belongs to such intersection.
Therefore, it is enough to show that $z\in\pi[\Sigma^\#]$, which is clear since
$\un u\in\Sigma^\#$.

\medskip
\noindent
(4) Proceed exactly as in \cite[Prop. 10.9]{Sarig-JAMS}.
\end{proof}

Let $Z=Z(v), Z'=Z(w)$ where $v=\Psi_x^{p^s,p^u},w=\Psi_y^{q^s,q^u}\in\mathfs A$, and assume
that $Z\cap \vf^{[-2\rho,2\rho]}Z'\neq\emptyset$. Let $D,D'$ be the 
connected components of $\widehat\Lambda$ such that $Z\subset D$ and $Z'\subset D'$. We wish to 
compare $s$--fibres of $Z$ with $u$--fibres of $Z'$ and vice-versa. 
To do that, we apply the holonomy maps $\mathfrak q_D$ and $\mathfrak q_{D'}$. Given $z\in Z,z'\in Z'$, define
\begin{align*}
\{[z,z']_Z\}&:=V^s(z,Z)\cap \mathfrak q_D[V^u(z',Z')]\\
\{[z,z']_{Z'}\}&:=\mathfrak q_{D'}[V^s(z,Z)]\cap V^u(z',Z').
\end{align*}
The next proposition proves that $[z,z']_Z$ and $[z,z']_{Z'}$ consist of single points, and some
compatibility properties that will be used in the next section. 

\begin{proposition}\label{Prop-overlapping-charts}
Let $Z=Z(v), Z'=Z(w)$ where $v=\Psi_x^{p^s,p^u},w=\Psi_y^{q^s,q^u}\in\mathfs A$, and assume
that $Z\cap \vf^{[-2\rho,2\rho]}Z'\neq\emptyset$. Let $D,D'$ be the 
connected components of $\widehat\Lambda$ such that $Z\subset D$ and $Z'\subset D'$.
The following are true.
\begin{enumerate}[{\rm (1)}]
\item $\mathfrak q_{D'}\circ \Psi_x(R[\tfrac{1}{2}(p^s\wedge p^u)])\subset \Psi_y(R[q^s\wedge q^u])$.
\item If $z\in Z$ with $z'=\mathfrak q_{D'}(z)\in Z'$, then $\mathfrak q_{D'}[W^{s/u}(z,Z)]\subset V^{s/u}(z',Z')$.
\item If $z\in Z,z'\in Z'$ then $[z,z']_Z$, $[z,z']_{Z'}$ are points with
$[z,z']_Z=\mathfrak q_{D}([z,z']_{Z'})$.
\end{enumerate}
\end{proposition}

When $M$ is compact and $f$ is a diffeomorphism, this is \cite[Lemmas 10.8 and 10.10]{Sarig-JAMS}.
A very similar method of proof works in our case: Theorem \ref{Thm-non-linear-Pesin} also works when
we change $g_x^+$ to $\mathfrak q_{D'}$, so we can control the composition
$\Psi_y^{-1}\circ\mathfrak q_{D'}\circ\Psi_x$. The details are in Appendix \ref{Appendix-proofs}.
We will also need more information regarding the Smale product of nearby charts.

\begin{proposition}\label{Prop-overlapping-charts-2}
Let $Z, Z',Z''$ such that $Z\cap \vf^{[-2\rho,2\rho]}Z'\neq\emptyset$, $Z\cap \vf^{[-2\rho,2\rho]}Z''\neq\emptyset$,
and let $D$ be the connected components of $\widehat\Lambda$ such that $Z\subset D$.
Assume that $z'\in Z'$ such that $\vf^t(z')\in Z''$ for some $|t|\leq 2\rho$. For every $z\in Z$, it holds
$$
[z,z']_Z=[z,\vf^t(z')]_Z.
$$
\end{proposition}

Note that $[z,z']_Z$ is defined by $Z,Z'$ while $[z,\vf^t(z')]_Z$ is defined by 
$Z,Z''$. The equality shows a compatibility of the Smale product along small flow displacements. 
It holds because such displacements barely change the sizes of invariant
fibres, hence the unique intersection is preserved.

\section{A refinement procedure}\label{Section-refinement}

Up to now, we have constructed a countable family $\mathfs Z$ of subsets of $\widehat\Lambda$
with the following properties:
\begin{enumerate}[$\circ$]
\item The union of elements of $\mathfs Z$, from now on also denoted by $\mathfs Z$,
is a section that contains $\Lambda\cap\nuh^\#$.
\item $\mathfs Z$ is locally finite: each point $x\in\mathfs Z$ belongs to at most finitely many
rectangles $Z\in\mathfs Z$.
\item Every element $Z\in \mathfs Z$ is a rectangle: each point $x\in Z$ has
invariant fibres $W^s(x,Z)$, $W^u(x,Z)$ in $Z$,
and these fibres induce a local product structure on $Z$.
\item $\mathfs Z$ satisfies a {\em symbolic} Markov property.
\end{enumerate}
In this section, we will refine $\mathfs Z$ to generate a countable family of disjoint sets
$\mathfs R$ that satisfy a {\em geometrical} Markov property.
We stress the difference from a symbolic to a geometrical Markov property:
by Proposition \ref{Prop-Z}(4), $g_{x_0}^\pm$ satisfy a symbolic Markov property;
our goal is to obtain a Markov property for the first return map $H$.
In general the orbit of $x$ can intersect $\mathfs Z$
between $x$ and $g_{x_0}^+(x)$, in which case we will have that $g_{x_0}^+(x)\neq H(x)$.
Therefore the symbolic Markov property
of Proposition \ref{Prop-Z}(4) does not directly translate into a geometrical Markov
property for $H$. To accomplish this latter property, we will use a refinement procedure
developed by Bowen \cite{Bowen-Symbolic-Flows}, motivated by the work of
Sina{\u\i} \cite{Sinai-Construction-of-MP,Sinai-MP-U-diffeomorphisms}. The difference
from our setup to Bowen's is that, while in Bowen's case all families are finite, in ours
it is usually countable. Fortunately, as implemented in \cite{Sarig-JAMS}, the refinement procedure
works well for countable covers with the local finiteness property, which we have by
Proposition \ref{Prop-Z}(2).

\subsection{The partition $\mathfs R$}
We first see that the map $g_{x_0}^+$ can be deduced from $H$ by a bounded time change.
\begin{lemma}\label{l.time}
There exists $N\geq 1$ such that for any $x=\pi(\un v)\in\mathfs Z$ there exists $0< n < N$
such that $g_{x_0}^+(x)=H^n(x)$.
\end{lemma}
\begin{proof}
We have $g_{x_0}^+(x)\in\mathfs Z$, so $g_{x_0}^+(x)=H^n(x)$
for some $n>0$.
Remember that $\widehat\Lambda=\bigcup_{i=1}^\mathfrak{n} D_i$ is a proper section of size $\rho/2$
(see Section \ref{Section-sections} for the definitions). In particular $\inf(r_{\widehat\Lambda})>0$.
Since $\mathfs Z\subset\widehat\Lambda$, every hit of $x$
to $\mathfs Z$ is also a hit to $\widehat\Lambda$. Writing $g_{x_0}^+(x)=\vf^t(x)$ for some $t\leq\rho$,
we conclude that $n\inf(r_{\widehat\Lambda})\leq t\leq\rho$, therefore
$n\leq \left\lceil \tfrac{\rho}{\inf(r_{\widehat\Lambda})}\right\rceil$.
We thus define $N:=\left\lceil \tfrac{\rho}{\inf(r_{\widehat\Lambda})}\right\rceil+1$.
\end{proof}

Therefore Proposition \ref{Prop-Z}(4) implies that for every $x\in\mathfs Z$ there are 
$0<k,\ell<N$ such that $H^k(x)$ satisfies a Markov property in the stable direction and
$H^{-\ell}(x)$ satisfies a Markov property in the unstable direction.

At this point, it is worth mentioning the method that Bowen used
to construct Markov partitions for Axiom A flows \cite{Bowen-Symbolic-Flows}:
\begin{enumerate}[(1)]
\item[(1)] Fix a global section for the flow; inside this section, construct a finite family of rectangles
(sets that are closed under the Smale bracket operation). Let $H$ be the Poincar\'e return map
of this family.
\item[(2)] Apply the method of Sina{\u\i} of successive approximations to get a new family
of rectangles $\mathfs Z$ with the following property: if $H$ is the Poincar\'e return map of $\mathfs Z$,
then for every $x\in\mathfs Z$ there are $k,\ell>0$ such that $H^k(x)$ satisfies a Markov property in the stable
direction and $H^{-\ell}(x)$ satisfies a Markov property in the unstable direction. In addition, there is a global
constant $N>0$ such that $k,\ell<N$.
\item[(3)] Apply a refinement procedure to $\mathfs Z$ such that the resulting partition $\mathfs R$ is a
disjoint family of rectangles satisfying the Markov property for $H$.
\end{enumerate}
The attentive reader might have note that, so far, we did implement steps (1) and (2) above,
with the difference that while Bowen used the method of successive approximations, we used
the method of $\ve$--gpo's. It remains to establish step (3), and we will do this closely
following Bowen \cite{Bowen-Symbolic-Flows}.

For each $Z\in\mathfs Z$, let
$$
\mathfs I_Z:=\left\{Z'\in\mathfs Z:\vf^{[-\rho,\rho]}Z\cap Z'\neq\emptyset\right\}.
$$
By Theorem \ref{Thm-inverse}, $\mathfs I_Z$ is finite.
Let $D$ is the connected component of $\widehat\Lambda$ such that $Z\subset D$.
By continuity, having chosen the discs $D_i$ small enough
the following property holds:
\begin{equation}\label{e.2rho}
\text{If $Z'\in\mathfs I_Z$ then $Z'\subset \vf^{[-2\rho,2\rho]}D$.}
\end{equation}
Therefore $\mathfrak q_D(Z')$ is a well-defined subset of $D$.
For each $Z'\in\mathfs I_Z$ we consider the partition of $Z$ into four subsets as follows:
\begin{align*}
E_{Z,Z'}^{su}&=\{x\in Z: W^s(x,Z)\cap\mathfrak{q}_{D}(Z')\neq\emptyset,
W^u(x,Z)\cap\mathfrak{q}_{D}(Z')\neq\emptyset\}\\
E_{Z,Z'}^{s\emptyset}&=\{x\in Z: W^s(x,Z)\cap\mathfrak{q}_{D}(Z')\neq\emptyset,
W^u(x,Z)\cap\mathfrak{q}_{D}(Z')=\emptyset\}\\
E_{Z,Z'}^{\emptyset u}&=\{x\in Z: W^s(x,Z)\cap\mathfrak{q}_{D}(Z')=\emptyset,
W^u(x,Z)\cap\mathfrak{q}_{D}(Z')\neq\emptyset\}\\
E_{Z,Z'}^{\emptyset\emptyset}&=\{x\in Z: W^s(x,Z)\cap\mathfrak{q}_{D}(Z')=\emptyset,
W^u(x,Z)\cap\mathfrak{q}_{D}(Z')=\emptyset\}.
\end{align*}
Call this partition
$\mathfs P_{Z,Z'}:=\{E_{Z,Z'}^{su},E_{Z,Z'}^{s\emptyset},E_{Z,Z'}^{\emptyset u},E_{Z,Z'}^{\emptyset\emptyset}\}$.
Clearly, $E^{su}_{Z,Z'}=Z\cap \mathfrak{q}_D(Z')$.

\medskip
\noindent
{\sc The partition $\mathfs E_Z$:} It is the coarser partition of $Z$ that refines
all of $\mathfs P_{Z,Z'}$, $Z'\in\mathfs I_Z$.

\medskip
To define a partition of $\mathfs Z$, we define an equivalence relation on $\mathfs Z$.

\medskip
\noindent
{\sc Equivalence relation $\simN$ on $\mathfs Z$:} For $x,y\in\mathfs Z$, we write
$x\simN y$ if for any $|k|\leq N$:
\begin{enumerate}[aa)]
\item[(i)] For all $Z\in\mathfs Z$: $H^k(x)\in Z\Leftrightarrow H^k(y)\in Z$.
\item[(ii)] For all $Z\in\mathfs Z$ such that $H^k(x),H^k(y)\in Z$, the points $H^k(x),H^k(y)$
belong to the same element of $\mathfs E_Z$. 
\end{enumerate}

\medskip
Clearly $\simN$ is an equivalence relation in $\mathfs Z$, hence it defines a partition of $\mathfs Z$.
Before proceeding, let us state a fact that will be used in the sequel: if $x\simN y$ with $x\in Z=Z(\Psi_{x_0}^{p^s_0,p^u_0})\in\mathfs Z$, then there exists $|k|\leq N$ such that
$g_{x_0}^+(x)=H^k(x)$ and $g_{x_0}^+(y)=H^k(y)$. To see this, write
$x=\pi(\un v)$ with $v_0=\Psi_{x_0}^{p^s_0,p^u_0}$,
and let $D'$ be the connected component of $\widehat\Lambda$ with $Z(v_1)\subset D'$.
On one hand, $g_{x_0}^+(y)=\mathfrak q_{D'}(y)$. On the other hand,
since $H^k(x)\in Z(v_1)\subset D'$ for some $|k|\leq N$, the definition of $\simN$ implies that
$H^k(y)\in Z(v_1)\subset D'$, hence $H^k(y)=\mathfrak q_{D'}(y)$. A similar result holds for $g_{x_0}^-$.

\medskip
\noindent
{\sc The Markov partition $\mathfs R$:} It is the partition of $\mathfs Z$ whose elements are the
equivalence classes of $\simN$.

\medskip
By definition, $\mathfs R$ is a refinement of $\mathfs Z$.

\begin{lemma}\label{Lemma-local-finite}
The partition $\mathfs R$ satisfies the following properties.
\begin{enumerate}[{\rm (1)}]
\item For every $Z\in\mathfs Z$, $\#\{R\in\mathfs R:R\subset \vf^{[-\rho,\rho]}Z\}<\infty$.
\item For every $R\in\mathfs R$, $\#\{Z\in\mathfs Z:R\subset \vf^{[-\rho,\rho]}Z\}<\infty$.
\end{enumerate}
\end{lemma}

\begin{proof}\mbox{}

\noindent
(1) Start noting that, for every $Z\in\mathfs Z$,
$\#\{R\in\mathfs R:R\subset Z\}\leq 4^{\#\mathfs I_Z}$. Hence
$$
\#\{R\in\mathfs R:R\subset \vf^{[-\rho,\rho]}Z\}\leq
\sum_{Z'\in\mathfs I_Z}\#\{R\in\mathfs R:R\subset Z'\}
\leq \sum_{Z'\in\mathfs I_Z}4^{\#\mathfs I_{Z'}}<+\infty
$$
since the last summand is the finite sum of finite numbers.

\medskip
\noindent
(2) For any $Z'\in\mathfs Z$ such that $Z'\supset R$, we have
$\{Z\in\mathfs Z:R\subset \vf^{[-\rho,\rho]}Z\}\subset \mathfs I_{Z'}$.
Since each $\mathfs I_{Z'}$ is finite, the result follows.
\end{proof}

\subsection{The Markov property}
The final step in the refinement procedure is to show that $\mathfs R$ is a Markov
partition for the map $H$, in the sense of Sina{\u\i} \cite{Sinai-MP-U-diffeomorphisms}. 

\medskip
\noindent
{\sc $s$/$u$--fibres in $\mathfs R$:} Given $x$ in $R\in\mathfs R$, we define the {\em $s$--fibre}
and {\em $u$--fibre} of $x$ by:
\begin{align*}
W^s(x,R):=
\bigcap_{Z\in\mathfs Z:Z\supset R} V^s(x,Z)\cap R
\, \text{, }\quad W^u(x,R):=
\bigcap_{Z\in\mathfs Z:Z\supset R} V^u(x,Z)\cap R.
\end{align*}

By Proposition \ref{Prop-disjointness}, any two $s$--fibres ($u$--fibres) either coincide or are disjoint.

\begin{proposition}\label{Prop-R}
The following are true.
\begin{enumerate}[{\rm (1)}]
\item {\sc Product structure:} For every $R\in\mathfs R$ and every $x,y\in R$, the intersection
$W^s(x,R)\cap W^u(y,R)$ is a single point, and this point is in $R$. Denote it by $[x,y]$.
\item {\sc Hyperbolicity:} If $z,w\in W^s(x,R)$ then $d(H^n(z),H^n(w))\xrightarrow[n\to\infty]{}0$, and
if $z,w\in W^u(x,R)$ then $d(H^n(z),H^n(w))\xrightarrow[n\to-\infty]{}0$. The rates are exponential.
\item {\sc Geometrical Markov property:} Let $R_0,R_1\in\mathfs R$. If $x\in R_0\cap H^{-1}(R_1)$ then 
$$
H(W^s(x,R_0))\subset W^s(H(x),R_1)\, \text{ and }\, H^{-1}(W^u(H(x),R_1))\subset W^u(x,R_0).
$$
\end{enumerate}
\end{proposition}
\begin{proof}
The sets $R\in\mathfs R$ are defined from the sets $Z\in\mathfs Z$
and the partitions $\mathfs E_Z$. By Proposition~\ref{Prop-Z} and by the definition
of the partitions $\mathfs P_{Z,Z'}$, each $Z$ and each element of $\mathfs E_Z$
is a rectangle. Note that rectangles are preserved under the holonomy maps
$\mathfrak{q}_{D_i}$ and that rectangles contained in a same disc $D_i$
are preserved under intersections.
Consequently the sets $R\in\mathfs R$ are also rectangles
and so part (1) follows. Part (2) is a direct consequence of the properties of the
stable and unstable manifolds obtained in Theorem~\ref{Thm-stable-manifolds}(3).
It remains to prove part (3).

Fix $R_0,R_1\in\mathfs R$ and $x\in R_0\cap H^{-1}(R_1)$.
We check that $H(W^s(x,R_0))\subset W^s(H(x),R_1)$ (the other inclusion is proved
similarly). Let $y\in W^s(x,R_0)$. By Proposition~\ref{Prop-overlapping-charts}(2)
and the definition of $W^s(H(x),R_1)$, it is enough to check that $H(x)\simN H(y)$.
Since $x\simN y$, we already know that $H^k(x),H^k(y)$ satisfy the properties (i) and (ii)
defining the relation $\simN$ when $-N\leq k\leq N$, hence it is enough to prove that this is also
true for $k=N+1$. The property (ii) for $k=N$ says that
$H^{N}(x), H^{N}(y)$ belong to the same elements of the partitions
$\mathfs E_Z$. We claim that this implies that $H^{N+1}(x),H^{N+1}(y)$ belong to the same
sets $Z\in \mathfs Z$, which gives (i) for $k=N+1$. To see this, let $Z'\in\mathfs Z$
such that $H^{N+1}(x)\in Z'$, and let $D'$ be the connected component of $\widehat\Lambda$
that contains $Z'$.
Let $Z\in\mathfs Z$ containing $H^N(x),H^N(y)$. Noting that 
$H^N(x)\in E^{su}_{Z,Z'}$, it follows from property (ii) for $k=N$ that 
$H^N(y)\in E^{su}_{Z,Z'}$, hence $\mathfrak q_{D'}(H^N(y))\in Z'$. If
$\mathfrak q_{D'}(H^N(y))=H^{N+1}(y)$, the claim is proved. If not, 
there is $Z''\in\mathfs Z$ such that $H^{N+1}(y)\in Z''$,
and so repeating the same argument with the roles of $x,y$ interchanged
gives that $\mathfrak q_{D''}(H^N(x))\in Z''$, a contradiction since 
the time transition from $Z$ to $Z''$ is smaller than time transitions from $Z$ to $Z'$.
Hence property (i) for $k=N+1$ is proved, and
it remains to prove property (ii) for $k=N+1$.

Let $Z\in \mathfs Z$ be a rectangle which contains $H^{N+1}(x),H^{N+1}(y)$
and let $D$ be the connected component of $\widehat\Lambda$ that contains $Z$.
We need to show that $H^{N+1}(x),H^{N+1}(y)$ belong to the same element of $\mathfs E_Z$.
We first note that $W^s(H^{N+1}(x),Z)=W^s(H^{N+1}(y),Z)$:
since $x,y$ belong to the same $s$--fibre of a rectangle in $\mathfs Z$,
this can be checked by applying Proposition~\ref{Prop-overlapping-charts}(2) inductively.
In particular, we have the following property:
\begin{equation}\label{e.stable}
\forall Z'\in \mathfs I_Z,\quad
W^s(H^{N+1}(x),Z)\cap \mathfrak{q}_{D}(Z')\neq \emptyset
\iff W^s(H^{N+1}(y),Z)\cap \mathfrak{q}_{D}(Z')\neq \emptyset.
\end{equation}
We then prove the analogous property for the sets $W^u(H^{N+1}(x),Z)$, $W^u(H^{N+1}(y),Z)$.
In Figure \ref{figure-markov} we draw the points we will define below.

Let us consider $Z'\in  \mathfs I_Z$ and assume for instance that
$W^u(H^{N+1}(x),Z)\cap \mathfrak{q}_{D}(Z')$ contains a point $z$
(the case when $W^u(H^{N+1}(y),Z)\cap \mathfrak{q}_{D}(Z')\neq \emptyset$ is treated analogously).
Write $H^{N+1}(x)=\pi(\un v)$ with
$\un v=\{v_n\}_{n\in\Z}=\{\Psi_{x_n}^{p^s_n,p^u_n}\}_{n\in\Z}\in\Sigma^\#$ and $Z=Z(v_0)$.
By Lemma~\ref{l.time}, there exists $0\leq k\leq N$ such that 
the point $\widetilde x:=H^k(x)$ coincides with $\pi[\sigma^{-1}(\un v)]$.
The rectangle $\widetilde Z:=Z(v_{-1})$ contains $\widetilde x$.
The symbolic Markov property in Proposition~\ref{Prop-Z}(4)
implies that the image of $W^u(\widetilde x,\widetilde Z)$ under $g^+_{x_{-1}}$
contains $W^u(H^{N+1}(x),Z)$, hence the point $z$.
In particular, the backward orbit of $z$ under the flow intersects $W^u(\widetilde x,\widetilde Z)$
at some point $\widetilde z$.

By the definition of $z$ and Property~\ref{e.2rho}, we have $\varphi^s(z)\in Z'$
for some $|s|\leq 2\rho$, thus we can write $\varphi^s(z)=\pi(\un w)$ with $\un w=\{w_n\}_{n\in\Z}\in\Sigma^\#$ and
$Z'=Z(w_0)$. Since all transition times of holonomy maps 
are bounded by $\rho$, necessarily the piece of orbit $\vf^{[0,\rho]}(\widetilde z)$
contains some $\pi[\sigma^{-b}(\un w)]$ with $b\geq 1$. Let $b\geq 1$ and $0\leq \widetilde s\leq \rho$
with $\pi[\sigma^{-b}(\un w)]=\varphi^{\widetilde s}(\widetilde z)$.
Consequently the rectangle
$\widetilde Z':=Z(w_{-b})$ belongs to $ \mathfs I_{\widetilde Z}$. 
Moreover, $\widetilde z$ belongs to the intersection between $W^u(\widetilde x,\widetilde Z)$ and 
$\mathfrak q_{D'}(\widetilde Z')$, where $\widetilde D$ is the connected component of $\widehat\Lambda$ 
containing $\widetilde Z$.

\begin{figure}[hbt!]
\centering
\def\svgwidth{14cm}
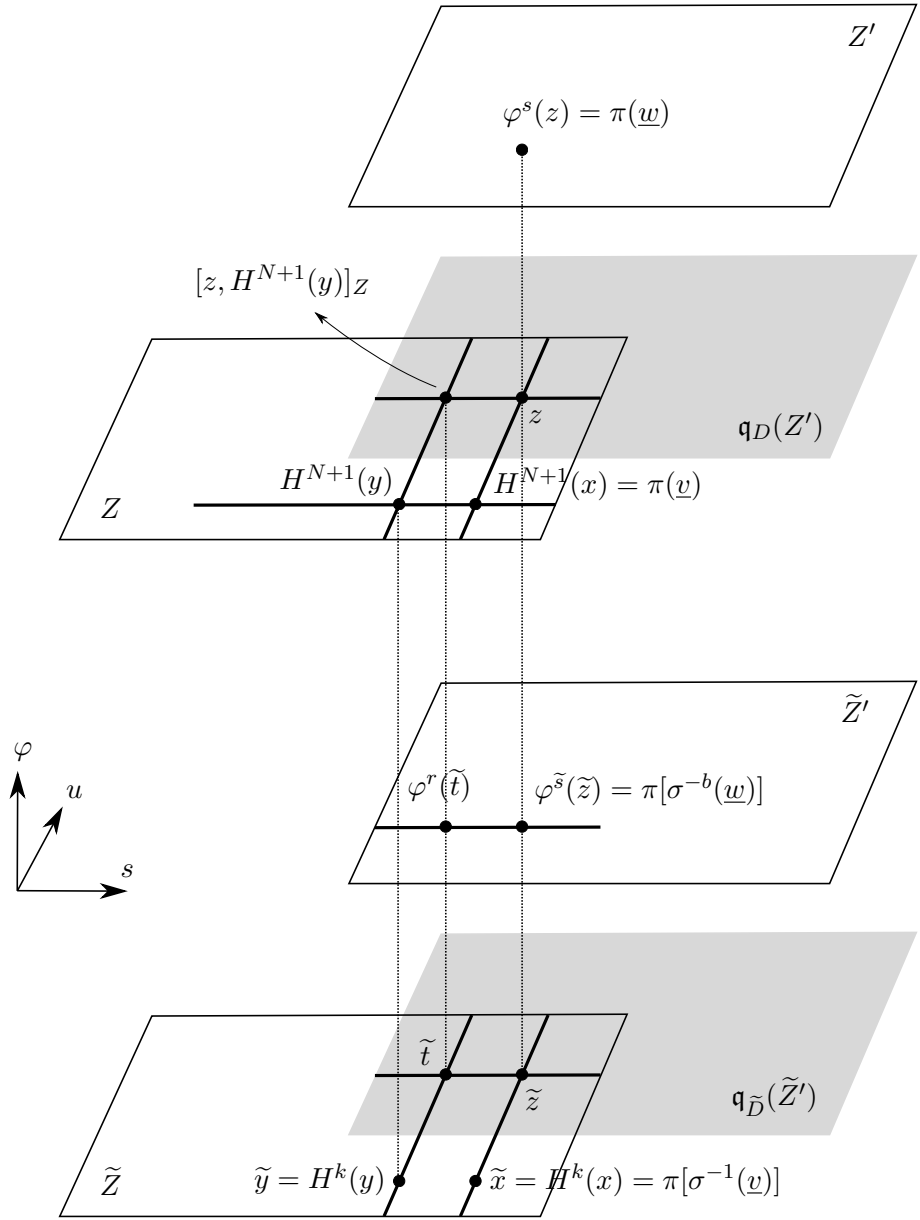\caption{Proof of the Markov partition.}
\label{figure-markov}
\end{figure}

By the induction assumption, the point $\widetilde y:=H^k(y)$ also belongs to $\widetilde Z$ and
to the same element of the partition $\mathfs P_{\widetilde Z, \widetilde Z'}$ as $\widetilde x$.
Since $W^u(\widetilde x,\widetilde Z)$ intersects $\mathfrak q_{\widetilde D}(\widetilde Z')$,
the $u$--fibre $W^u(\widetilde y,\widetilde Z)$ intersects it as well at some point $\widetilde t$.
Note that $[\widetilde z, \widetilde t]_{\widetilde Z}=[\widetilde z, \widetilde y]_{\widetilde Z}$
also belongs to $W^u(\widetilde y,\widetilde Z)$ and to
$\mathfrak q_{\widetilde D}(\widetilde Z')$ (this latter property follows from
Proposition \ref{Prop-overlapping-charts}(3), noting that 
$\widetilde z,\widetilde t\in \widetilde Z\cap \mathfrak q_{\widetilde D}(\widetilde Z')$),
hence we can replace $\widetilde t$
by any point in $W^u(\widetilde y,\widetilde Z)\cap \mathfrak q_{\widetilde D}(\widetilde Z')$.
Take $\widetilde t :=[\widetilde z,\widetilde y]_{\widetilde Z}$.

Let $0<r\leq 2\rho$ such that $\varphi^r(\widetilde t)\in W^s(\vf^{\widetilde s}(\widetilde z), \widetilde Z')$.
The symbolic Markov property in Proposition~\ref{Prop-Z}(4) then implies that
its forward orbit under the flow will meet the rectangles
$Z(w_{-b})$,\dots, $Z(w_0)$.

Note that $\widetilde z\in \widetilde Z=Z(v_{-1})$ and $z=g^+_{x_{-1}}(\widetilde z)\in Z=Z(v_0)$.
The same property holds for $\widetilde y$ and $H^{N+1}(y)=g^+_{x_{-1}}(\widetilde y)$ since
the points $H^i(x)$ and $H^i(y)$ belong to the same rectangles in $\mathfs Z$ for each $i=k,\dots,N+1$.
Using Proposition~\ref{Prop-overlapping-charts}(3), it follows that
the image of $\widetilde t=[\widetilde z,\widetilde y]_{\widetilde Z}$
by $g^+_{x_{-1}}$ belongs to $Z$ and coincides with the Smale product
$[z,H^{N+1}(y)]_Z$.

The properties found in the two previous paragraphs imply that
$W^u(H^{N+1}(y),Z)$ intersects $\mathfrak{q}_{D}(Z')$
at a point of the orbit of $\widetilde t$, contained in $W^s(z,Z)$.
In particular, the intersection $W^u(H^{N+1}(y),Z)\cap \mathfrak{q}_{D}(Z')$ is non-empty.
We have thus shown:
\begin{equation}\label{e.unstable}
\forall Z'\in \mathfs I_Z,\quad
W^u(H^{N+1}(x),Z)\cap \mathfrak{q}_{D}(Z')\neq \emptyset
\iff  W^u(H^{N+1}(y),Z)\cap \mathfrak{q}_{D}(Z')\neq \emptyset.
\end{equation}
Properties~\eqref{e.stable} and~\eqref{e.unstable} mean that
$H^{N+1}(x)$ and $H^{N+1}(y)$ belong to the same element of $\mathfs E_Z$
for any rectangle $Z\in \mathfs Z$ containing $H^{N+1}(x),H^{N+1}(y)$.
This concludes the proof that $H(x)\simN H(y)$, and of part (3) of the proposition.
\end{proof}

\section{A finite-to-one extension}

In this section, we construct a finite-to-one extension and deduce the Main Theorem.
We rely on the
family of disjoint sets $\mathfs R$ satisfying a geometrical Markov property.
This family was obtained in the previous section
as a refinement of the family $\mathfs Z$ constructed in Section \ref{Section-locally-finite-section}, which
was itself induced by the coding $\pi$ introduced in Section \ref{ss.first.coding}.
One important property of $\mathfs Z$ is that, due to the inverse theorem
(Theorem \ref{Thm-inverse}), it satisfies a local finiteness property, see Proposition \ref{Prop-Z}(2).
Having these facts in mind,
we construct a symbolic coding of the return map $H$.

\subsection{A detailed statement}\label{s.detailed}

The theorem below implies the Main Theorem and includes additional properties that will
be useful for some applications, including the one we will obtain in Section \ref{sec.homoclinic}.
We begin defining a Bowen relation for flows. This notion was formalized for diffeomorphisms in \cite{Boyle-Buzzi},
and the following is an adaptation for flows. We refer to \cite{Buzzi-JMD} for a discussion on the notion,
and in particular on the non-uniqueness of such a relation.

\newcommand\hS{\widehat S}

Let  $T_r: S_r\to S_r$ be a suspension flow over a symbolic system $S$ that is an extension
of some flow $U:X\to X$ by a semiconjugacy map $\pi:S_r\to X$,  i.e.
$U^t\circ\pi=\pi\circ T^t_r$ for all $t\in\R$.\\

\noindent
{\sc Bowen relation:} A \emph{Bowen relation} $\sim$ for $(T_r,\pi,U)$ is a symmetric binary
relation on the alphabet of $S$ satisfying the following two properties:
\begin{enumerate}[i,]
\item[{\rm (i)}] $\forall\omega,\omega'\in S_r,\;\; \pi(\omega)=\pi(\omega')\implies \operatorname{v}(\omega)\sim\operatorname{v}(\omega')$, where $\operatorname{v}(x,t):=x_0$ for $x\in S$;
\item[{\rm (ii)}] $\exists \gamma>0$ with the following property:
$$\forall\omega,\omega'\in S_r,\;\; \left[ \forall t\in\R,\; \operatorname{v}(T_r^t\omega)\sim\operatorname{v}(T_r^t\omega') \right] \implies \left[ \exists |t|<\gamma, \pi(\omega)= U^t(\pi(\omega')) \right].
$$
  \end{enumerate}

\begin{theorem}\label{t.main}
Let $X$ be a non-singular $C^{1+\beta}$ vector field ($\beta>0$) on a closed $3$-manifold $M$.
Given $\chi>0$, there exist a locally compact topological Markov flow
$(\widehat \Sigma_{\widehat r},\widehat\sigma_{\widehat r})$ with graph
$\widehat{\mathfs G}=(\widehat V,\widehat E)$ and roof function $\widehat{r}$
and a map $\widehat \pi_{\widehat r}:\widehat \Sigma_{\widehat r}\to M$ such that
$\widehat \pi_{\widehat r}\circ {\widehat\sigma}_{\widehat r}^t=\vf^t\circ\widehat \pi_{\widehat r}$, for all $t\in\R$, and satisfying:
\smallskip
\begin{enumerate}[{\rm (1)}]
\item $\widehat r$ and $\widehat \pi_{\widehat r}$ are H\"older continuous.
\smallskip

\item $\widehat \pi_{\widehat r}[\widehat \Sigma_{\widehat r}^\#]=\nuh^\#$ has full measure for every $\chi$--hyperbolic measure; for every ergodic $\chi$--hyperbolic measure $\mu$,
there is an ergodic $\widehat\sigma_{\widehat r}$--invariant measure $\overline \mu$ 
on $\widehat \Sigma_{\widehat r}$
such that $\overline \mu\circ\widehat \pi_{\widehat r}^{-1}=\mu$.
\smallskip

\item If
$(\un R,t)\in \widehat \Sigma_{\widehat r}^{\#}$ satisfies
$R_n=R$ and $R_m=S$ for infinitely many $n<0$ and $m>0$, then $\operatorname{Card}\{z\in \widehat \Sigma_{\widehat r}^\#:\widehat \pi_{\widehat r}(z)=\widehat \pi_{\widehat r}(\un R,t)\}$
is bounded by a number $C(R,S)$, depending only on $R,S$.
\smallskip

\item\label{i.splitting} There is $\lambda>0$ and for $x\in \widehat \pi_{\widehat r}(\widehat \Sigma_{\widehat r})$ there is a unique splitting
$N_{x}=N^s_{x}\oplus N^u_{x}$ such that:
\begin{align*}
\limsup_{t\to +\infty} \tfrac{1}{t}\log \|\Phi^t|_{N^s_{x}}\|\leq -\lambda
\quad \text{ and }\quad\liminf_{t\to +\infty} \tfrac{1}{t}\log \|\Phi^{-t}|_{N^s_{x}}\|\geq \lambda\\
\quad \limsup_{t\to +\infty} \tfrac{1}{t}\log \|\Phi^{-t}|_{N^u_{x}}\|\leq -\lambda
\quad \text{ and }\quad\liminf_{t\to +\infty} \tfrac{1}{t}\log \|\Phi^{t}|_{N^u_{x}}\|\geq \lambda.
\end{align*}
The splitting is $\Phi$--equivariant, and the maps $z\mapsto N^{s/u}_{\widehat \pi_{\widehat r}(z)}$ are H\"older continuous on $\widehat \Sigma_{\widehat r}$.
\smallskip

\item\label{i.manifold} For every $z\in \widehat \Sigma_{\widehat r}$,
there are $C^1$ submanifolds $V^{cs}(z),V^{cu}(z)$ passing through $x:=\widehat \pi_{\widehat r}(z)$
such that:
\begin{enumerate}[{\rm (a)}]
\item
$T_{x}V^{cs}(z)=N^s_x+\mathbb{R}\cdot X(x)$ and $T_{x}V^{cu}(z)=N^{u}_x+\mathbb{R}\cdot X(x)$.
\item For all $y\in V^{cs}(z)$, there is $\tau\in \mathbb{R}$ such that
$d(\varphi^t(x),\varphi^{t+\tau}(y))\leq e^{-\lambda t}$, $\forall t\geq 0$.
\item For all $y\in V^{cu}(z)$, there is $\tau\in \mathbb{R}$ such that
$d(\varphi^{-t}(x),\varphi^{-t+\tau}(y))\leq e^{-\lambda t}$, $\forall t\geq 0$.
\end{enumerate}
\smallskip

\item\label{i.Bowen} There is a symmetric binary relation $\sim$ on the alphabet $\widehat V$
satisfying:
\begin{enumerate}[{\rm (a)}]
\item For any $R\in \widehat V$, the set $\{S\in \widehat V:R\sim S\}$ is finite.
\item The relation $\sim$ is a Bowen relation for $(\widehat\sigma_{\widehat r},\widehat\pi_{\widehat r}|_{\widehat\Sigma^\#_{\widehat r}},\vf^t)$.
\end{enumerate}
\smallskip

\item\label{i.canonical} There exists a measurable set $\mathfs R$ with a measurable partition
indexed by $\widehat V$, which we denote by $\{R:R\in\widehat V\}$, such that:
\begin{enumerate}[{\rm (a)}]
\item The orbit of any point $x\in \nuh^\#$ intersects $\mathfs R$.
\item The first return map $H\colon \mathfs R\to \mathfs R$ induced by $\varphi$ is a well-defined bijection.
\item For any $x\in \mathfs R$, if $\un R=\{R_n\}_{n\in\Z}$ satisfies $H^n(x)\in R_n$ for all $n\in\Z$,
then $(\un R,0)\in \widehat \Sigma^\#_{\widehat r}$ and $\widehat \pi_{\widehat r}(\un R,0)=x$.
\end{enumerate}

\item\label{i.lift}
For any compact transitive invariant hyperbolic set $K\subset M$ whose ergodic $\vf$--invariant
measures are all $\chi$--hyperbolic, there is a {\em transitive} invariant compact set 
$X\subset \widehat \Sigma_{\widehat r}$ such that $\widehat\pi_{\widehat r}(X)=K$.
\end{enumerate}
\end{theorem}

Part~\eqref{i.Bowen}
provides a combinatorial characterization of the noninjectivity of  the coding.
It is an adaptation for flows of the \emph{Bowen property}, which was
introduced in \cite{Boyle-Buzzi} for diffeomorphisms and motivated by the work of
Bowen \cite{Bowen-Regional-Conference}. Note that, in contrast to \cite{Bowen-Regional-Conference},
we {\em do not} claim that the flow restricted to $\widehat\pi_{\widehat r}[\widehat\Sigma_{\widehat r}^\#]$
is topologically equivalent to the corresponding quotient dynamics.

The relation $\sim$ will be the \emph{affiliation}, which will be introduced in Section~\ref{s.finite-one},
following a similar notion introduced in \cite{Sarig-JAMS}. 
Note that the assumption 
$\bigl[\operatorname{v}(\widehat \sigma_{\widehat r}^t(z))\sim \operatorname{v}(\widehat\sigma_{\widehat r}^t(z'))$
for all $t\in \mathbb{R}\bigr]$ consists of countably many
affiliation conditions: if $z=(\un R,s)$ and $z'=(\un S,s')$, then varying $t$ in the interval $[\widehat r_n(\un R),\widehat r_{n+1}(\un R))$
provides $i\leq \tfrac{\sup(\widehat r)}{\inf(\widehat r)}$ affiliations of the form
$R_n\sim S_{m+1},\ldots,R_n\sim S_{m+i}$.

Part~\eqref{i.canonical} provides for any
$x\in  \nuh^\#$ a particular pair $(\un R,t)\in \widehat \Sigma_{\widehat r}^\#$
such that $\widehat \pi_{\widehat r}(\un R,t)=x$
(here $t$ is the smallest non-negative number such that $\varphi^{-t}(x)\in \mathfs R$). We call
the pair $(\un R,t)$ the \emph{canonical lift} of $x$.
This is a measurable embedding of $\nuh^\#$ into $\widehat\Sigma_{\widehat r}$.
\color{black}

Part~\eqref{i.lift} is a version of \cite[Proposition 3.9]{BCS-MME} in our context, and the proof
is very similar, see Section~\ref{ss.conclusion}.

\subsection{Second coding}\label{s.second}
Let $\widehat{\mathfs G}=(\widehat V,\widehat E)$ be the oriented graph with vertex set
$\widehat V=\mathfs R$ and edge set $\widehat E=\{R\to S:R,S\in\mathfs R\text{ s.t. }H(R)\cap S\neq\emptyset\}$,
and let $(\widehat\Sigma,\widehat\sigma)$ be the TMS induced by $\widehat{\mathfs G}$.
We note that the ingoing and outgoing degree of every vertex in $\widehat\Sigma$ is finite. We show this
for the outgoing edges, since the proof for the ingoing edges is symmetric. Fix $R\in\mathfs R$, and fix
$Z\in\mathfs Z$ such that $Z\supset R$. 
If $(R,S)\in\widehat E$ then $\vf^{[0,\rho]}(R)\cap S\ne\emptyset$, hence for any $Z'\in\mathfs Z$ with $S\subset Z'$, we have $Z'\in\mathfs I_Z$. In particular,
$$
\#\{(R,S)\in\widehat E\}\leq \sum_{Z'\in\mathfs I_Z}\#\{S\in\mathfs R:S\subset Z'\}<+\infty,
$$
since both $\mathfs I_Z$ and each $\{S\in\mathfs R:S\subset Z'\}$ are finite sets
(see Lemma \ref{Lemma-local-finite}(1)).

For $\ell\in\Z$ and a path $R_m\to\cdots\to R_n$ on $\widehat{\mathfs G}$ define
$$
_\ell[R_m,\ldots,R_n]:=H^{-\ell}(R_m)\cap\cdots\cap H^{-\ell-(n-m)}(R_n),
$$
the set of points whose itinerary
under $H$ from $\ell$ to $\ell+(n-m)$ visits the rectangles $R_m,\ldots,R_n$ respectively.
The crucial property that
gives the new coding is that $_\ell[R_m,\ldots,R_n]\neq\emptyset$. This follows by induction, using the
Markov property of $\mathfs R$ (Proposition \ref{Prop-R}(3)).

The map $\pi$ defines similar sets: for $\ell\in\Z$ and a path
$v_m\overset{\ve}{\to}\cdots\overset{\ve}{\to}v_n$ on $\Sigma$, let
$$
Z_\ell[v_m,\ldots,v_n]:=\{\pi(\un w):\un w\in\Sigma^\#\text{ and }w_\ell=v_m,\ldots,w_{\ell+(n-m)}=v_n\}.
$$
There is a relation between these sets we just defined. Before stating such a relation, we will define the coding of $H$,
and then collect some of its properties.

\medskip
\noindent
{\sc The map $\widehat\pi:\widehat\Sigma\to M$:} Given $\un R=\{R_n\}_{n\in\Z}\in\widehat\Sigma$,
$\widehat\pi(\un R)$ is defined by the identity
$$
\{\widehat\pi(\un R)\}:=\bigcap_{n\geq 0}\overline{_{-n}[R_{-n},\ldots,R_n]}.
$$

\medskip
Note that $\widehat\pi$ is well-defined, because the right hand side is an intersection of
nested compact sets with diameters going to zero.
The proposition below states relations between $\Sigma$ and $\widehat\Sigma$, and between $\pi$ and $\widehat\pi$.
For $\un v=\{\Psi_{x_n}^{p^s_n,p^u_n}\}_{n\in\Z}\in\Sigma$, define
$$
G_{\un v}^n=\left\{
\begin{array}{ll}
g_{x_{n-1}}^+\circ\cdots\circ g_{x_0}^+ &,n\geq 0\\
g_{x_{n+1}}^-\circ\cdots\circ g_{x_0}^- &,n<0.
\end{array}
\right.
$$
Recall the integer $N$ introduced in Lemma \ref{l.time}.

\begin{proposition}\label{Prop-relation-codings}
For each $\un R=(R_n)_{n\in\Z}\in\widehat\Sigma$ and $Z\in\mathfs Z$ with
$Z\supset R_0$, there are an $\ve$--gpo
$\un v=\{v_k\}_{k\in\Z}\in\Sigma$ with $Z(v_0)=Z$ and a sequence
$(n_k)_{k\in\Z}$ of integers with $n_0=0$ and $1\le n_k-n_{k-1}\le N$ for all $k\in\Z$ such that:
\begin{enumerate}[{\rm (1)}]
\item For each $k\ge1$, $$ _{n_{-k}}[R_{n_{-k}},\ldots,R_{n_k}]\subset Z_{-k}[v_{-k},\ldots,v_k].$$
In particular, $\widehat\pi(\un R)=\pi(\un v)$.
Moreover, $R_{n_k}\subset Z(v_k)$ for all $k\in\Z$.

\item The map $\widehat\pi$ is H\"older continuous over $\widehat\Sigma$. In fact,
$\{v_i\}_{|i|\le k}$ depends only on $\{R_j\}_{|j|\le kN}$ for each $k\ge1$.

\item If $\un R\in {\widehat \Sigma}^\#$, then $\un v\in {\Sigma}^\#$.

\item The two codings have the same regular image: $\pi[\Sigma^\#]=\widehat\pi[\widehat\Sigma^\#]$.
\end{enumerate}
\end{proposition}

For diffeomorphisms, the above lemma is \cite[Lemma 12.2]{Sarig-JAMS}. The difference from the case
of diffeomorphisms relies on our definitions of $\mathfs G$ and $\widehat{\mathfs G}$. While the edges of
$\widehat{\mathfs G}$ correspond to possible time evolutions of $H$, the edges of $\mathfs G$ correspond
to $\ve$--overlaps. In particular, not every edge of $\widehat{\mathfs G}$ corresponds to an edge of $\mathfs G$,
and this is the reason we have to introduce the sequence $(n_k)_{k\in\Z}$. 
In fact, each edge $v_k\to v_{k+1}$ of $\mathfs G$ corresponds to a sequence of edges $R_{n_k}\to\dots\to R_{n_{k+1}}$ of $\widehat{\mathfs G}$.

\begin{proof}
We begin proving part (1).
Fix $\{R_n\}_{n\in\Z}\in\widehat\Sigma$. The proof consists of successive uses of the following fact.

\medskip
\noindent
{\sc Claim:}
\emph{For all $i\in\Z$ and $v\in\mathfs A$ such that $R_i\subset Z(v)$, there are $1\leq k\leq N$
and $w\in\mathfs A$ such that $_{0}[R_i,\ldots,R_{i+k}]\subset Z_0[v,w]$ and $R_{i+k}\subset Z(w)$.
Similarly, there are $1\leq \ell\leq N$ and $u\in\mathfs A$ such that 
$_0[R_{i-\ell},\ldots,R_i]\subset Z_0[u,v]$ and $R_{i-\ell}\subset Z(u)$.
}

\begin{proof}[Proof of the claim.]
We prove the first statement (the second is proved similarly).  Let $v=\Psi_x^{p^s,p^u}\in\mathfs A$
such that $R_i\subset Z(v)$.  Since $\un R\in\widehat\Sigma$,  there is $y^*\in\, _{0}[R_i,\ldots,R_{i+N}]$.
Moreover,  there is $\un{v}^*\in\Sigma^\#$ such that $\pi(\un v^*)=y^*$ and $v_0^*=v$.  We set $w:=v^*_1$ so that $v\to w$.
By construction, $g_{x}^+(\pi(\un v^*))=\pi(\sigma(\un v^*))$ so $Z(w)$ contains $g_{x}^+(y^*)$.
Also, there is $1\leq k\leq N$ such that $g_{x}^+(y^*)=H^k(y^*)$.  

We claim that 
$_{0}[R_i,\ldots,R_{i+k}]\subset Z_0[v,w]$. To see that, let $y\in\, _{0}[R_i,\ldots,R_{i+k}]$.
We have $y\simN y^*$, thus the following occur:
\begin{enumerate}[$\circ$]
\item $y\in\, _{0}[R_i,\ldots,R_{i+k}]\subset R_i\subset Z(v)$, hence $y=\pi(\un v)$ for some $\un v\in\Sigma^\#$ with $v_0=v$.
\item $g_x^+(y^*)=H^k(y^*)\in Z(w)\Rightarrow g_x^+(y)=H^k(y)\in Z(w)$,
hence $\pi(\sigma(\un v))=g_x^+(y)=\pi(\un w)$
for some $\un w\in\Sigma^\#$ with $w_0=w$.
\end{enumerate}
Define $\un u=\{u_n\}_{n\in\Z}$ by
$$
u_n=\left\{\begin{array}{ll}v_n&,n\leq 0\\ w_{n-1}&,n\geq 1.\end{array}\right.
$$
Note that $\un u$ belongs to $\widehat\Sigma^\#$ since $\un v,\un w\in\widehat\Sigma^\#$ and $v\to w$ on $\mathfs G$.
To prove that $z=\pi(\un u)$,
note that:
\begin{enumerate}[$\circ$]
\item If $n\leq 0$, then $G^n_{\un u}(z)=G^n_{\un v}(z)\in Z(v_n)$.
\item If $n\geq 1$, then $G^n_{\un u}(z)=G^{n-1}_{\un w}[g_x^+(z)]\in Z(w_{n-1})$.
\end{enumerate}
By Proposition \ref{Prop-shadowing}, it follows that $y=\pi(\un u)\in Z_0[v,w]$, proving the inclusion.

The rest of the claim follows by symmetry, replacing $g_x^+,H,\sigma$ by $g_x^-,H^{-1},\sigma^{-1}$
and noting that $\simN$ considers $H^k$ for all $|k|\le N$.
\end{proof}

Now we prove part (1). Fix $n_0=0$ and $v_0\in\mathfs A$
such that $R_0\subset Z(v_0)$. Applying the claim for $i=0$ and $v_0$, we get $0<n_1\leq N$ and
$v_1\in\mathfs A$ such that $_{0}[R_0,\ldots,R_{n_1}]\subset Z_0[v_0,v_1]$
and $R_{n_1}\subset Z(v_1)$.
By induction, we obtain
an increasing sequence $n_0=0<n_1<n_2<\cdots$ such that $n_k<n_{k+1}\leq n_k+N$,
$_{0}[R_{n_k},\ldots,R_{n_{k+1}}]\subset Z_0[v_k,v_{k+1}]$, and $R_{n_k}\subset Z(v_k)$ for all $k\geq 0$. 
Doing the same for negative
iterates, we get a decreasing sequence $n_0=0>n_{-1}>n_{-2}>\cdots$ such that
$n_k-N\leq n_{k-1}<n_k$, $_{0}[R_{n_k},\ldots,R_{n_{k+1}}]\subset Z_0[v_k,v_{k+1}]$,
and $R_{n_{k}}\subset Z(v_{k})$ for all $k<0$.
We claim that the sequence $\un v=\{v_k\}_{k\in\Z}$ satisfies the proposition.

Fix $k\geq 0$. We wish to show that $_{n_{-k}}[R_{n_{-k}},\ldots,R_{n_k}]\subset Z_{-k}[v_{-k},\ldots,v_k]$,
i.e. given $y\in\, _{n_{-k}}[R_{n_{-k}},\ldots,R_{n_k}]$ we want to find $\un u\in\Sigma^\#$ such that
$(u_{-k},\ldots,u_k)=(v_{-k},\ldots,v_k)$ and $\pi(\un u)=y$.
Since $H^{n_{-k}}(y)\in R_{n_{-k}}\subset Z(v_{-k})$,
there is $\un w^-\in\Sigma^\#$ with $w^-_0=v_{-k}$ and $H^{n_{-k}}(y)=\pi(\un w^-)$.
Similarly, since $H^{n_k}(y)\in R_{n_k}\subset Z(v_k)$,
there is $\un w^+\in\Sigma^\#$ with $w^+_0=v_k$ and
$H^{n_k}(y)=\pi(\un w^+)$. Define $\un u=\{u_i\}_{i\in\Z}$ by:
$$
u_i=\left\{\begin{array}{ll}w^-_{i+k}&,i\leq -k\\ v_i&,i=-k,\ldots,k\\ w^+_{i-k}&,i\geq k.\end{array}\right.
$$
Clearly $\un u\in\Sigma^\#$. We claim that $\pi(\un u)=y$. Indeed:
\begin{enumerate}[$\circ$]
\item $-k\leq i\leq k$: we have $G^i_{\un u}(y)=H^{n_i}(y)\in R_{n_i}\subset Z(v_i).$
\item $i\leq -k$: since $G^{-k}_{\un u}(y)=H^{n_{-k}}(y)$ and $G^{i+k}_{\sigma^{-k}(\un u)}=G^{i+k}_{\un w^-}$
(the sequences $\sigma^{-k}(\un u)$ and $\un w^-$ coincide in the past), we have
$G^i_{\un u}(y)=G^{i+k}_{\sigma^{-k}(\un u)}[G^{-k}_{\un u}(y)]=
G^{i+k}_{\un w^-}[H^{n_{-k}}(y)]\in Z(w^-_{i+k})=Z(u_i)$.
\item $i\geq k$: as in the previous case, $G^i_{\un u}(y)\in Z(u_i)$.
\end{enumerate}
Therefore $G^i_{\un u}(y)\in Z(u_i)$ for all $i\in\Z$, hence by Proposition \ref{Prop-shadowing}
it follows that $\pi(\un u)=y$.

Now we show that $\widehat\pi(\un R)=\pi(\un v)$. Indeed, since $n_k\pm\infty$ as $k\to\pm\infty$,
we have
$$
\{\widehat\pi(\un R)\}=\displaystyle\bigcap_{k\geq 0}\overline{_{n_{-k}}[R_{n_{-k}},\ldots,R_{n_k}]}
\subset\displaystyle\bigcap_{k\geq 0}\overline{Z_{-k}[v_{-k},\ldots,v_k]}.
$$
On one hand, this latter set is, by Theorem~\ref{Thm-stable-manifolds}(3), the intersection of a descending chain of closed sets with diameter going to zero, hence it is a singleton. On the other hand, it contains
$\displaystyle\bigcap_{k\geq 0}Z_{-k}[v_{-k},\ldots,v_k]=\{\pi(\un v)\}$. Thus
$\widehat\pi(\un R)=\pi(\un v)$, which concludes the proof of part (1).

To check part (2), note that its second statement is immediate from the above argument.
It implies the rest, since $\pi$ is H\"older-continuous.

We turn to part (3). Assume that $\un R\in\widehat\Sigma^\#$. Let $R\in\mathfs R$ and $m_j\to+\infty$ such that
$R_{m_j}=R$ for all $j$. Since $\widehat\Sigma$ is locally compact (the degrees of $\widehat{\mathfs G}$ are
all finite), the set
$$
\mathfs P=\{S\in\mathfs R:\exists \text{ path }S_0=R\to S_1\to\cdots\to S_i=S\text{ with }i\leq N\}
$$
is finite. Given $j$, let $k=k(j)$ be the unique integer such that $n_{k-1}<m_j\leq n_k$. Since $n_k-n_{k-1}\leq N$,
it follows that $R_{n_k}\in \mathfs P$. By Lemma \ref{Lemma-local-finite}(2), it follows that
$v_k$ belongs to the finite set
$\{Z\in\mathfs Z:\exists S\in \mathfs P\text{ such that }S\subset Z\}$, and so there is a sequence
$k_i\to+\infty$ such that $\{v_{k_i}\}_{i\geq 0}$ is a constant sequence. Proceeding similarly for the negative indices,
we conclude that $\un v\in\Sigma^\#$. This proves part~(3).

Now we prove part (4). By part (3), we have $\widehat\pi[\widehat\Sigma^\#]\subset \pi[\Sigma^\#]$. 
To prove the converse inclusion, 
let $\un v=\{v_n\}_{n\in\Z}\in\Sigma^\#$ and write $x=\pi(\un v)$.
Let $R_n\in\mathfs R$ such that $H^n(x)\in R_n$. Clearly, $\un R=\{R_n\}\in\widehat\Sigma$ 
and $x=\widehat\pi(\un R)$. It remains to prove that $\un R\in\widehat\Sigma^\#$. Let
$v\in \mathfs A$ and $k_i\to+\infty$ such that $v_{k_i}=v$ for all $i\geq 0$.
Letting $m_i:=n_{k_i}\to+\infty$ so that $H^{m_i}(x)=\pi[\sigma^{k_i}(\un v)]$,
we have $H^{m_i}(x)\in R_{m_i}\cap Z(v)$ and so 
$R_{m_i}\subset Z(v)$. By Lemma \ref{Lemma-local-finite}(1), there is a subsequence
${m_{\ell_j}}$ such that $(R_{m_{\ell_j}})$ is constant. Proceeding similarly for negative
indices, it follows that $\un R\in\widehat\Sigma^\#$ and so $\pi[\Sigma^\#]\subset\widehat\pi[\widehat\Sigma^\#]$. This concludes the proof of part (4), and of the proposition.
\end{proof}

We now define the topological Markov flow (TMF) and coding that satisfy the Main Theorem.
For that, recall the definition of TMF in Section \ref{Section-Preliminaries}.

\medskip
\noindent
{\sc The triple $({\widehat\Sigma}_{\widehat r},{\widehat\sigma}_{\widehat r},{\widehat\pi}_{\widehat r})$:} 
The topological Markov flow $({\widehat\Sigma}_{\widehat r},{\widehat\sigma}_{\widehat r})$ is the suspension of $(\widehat\Sigma,\widehat\sigma)$ by the roof function $\widehat r:\widehat\Sigma\to(0,\rho)$ defined by
$$
 \widehat r(\un R) := \min\{t>0:\vf^t(\widehat\pi(\un R)) = \widehat\pi(\widehat\sigma(\un R))\},
$$
and the factor map ${\widehat\pi}_{\widehat r}:\widehat\Sigma_{\widehat r}\to M$ is 
given by $\widehat\pi_{\widehat r}(\un R,s):=\vf^s(\widehat\pi(\un R))$.

\medskip
As claimed above, we have $\sup \widehat r<\rho$. Indeed, by 
Proposition \ref{Prop-relation-codings} there is $\un v=\{v_n\}_{n\in\Z}\in\Sigma$
such that $\widehat\pi(\un R)=\pi(\un v)$, and 
there are integers $n_{-1}<0<n_1$ such that  
$ _{n_{-1}}[R_{n_{-1}},\ldots,R_{n_1}]\subset Z_{-1}[v_{-1},v_0,v_1]$, hence
$\widehat r(\un R)\leq \widehat r_{n_1}(\un R)=r(\un v)<\rho$.
The rest of this section is devoted to proving that
$({\widehat\Sigma}_{\widehat r},{\widehat\sigma}_{\widehat r},{\widehat\pi}_{\widehat r})$
satisfies Theorem~\ref{t.main}. We start with some fundamental properties.

\begin{proposition}\label{Prop-pi_r}
The following holds for all $\ve>0$ small enough.
\begin{enumerate}[{\rm (1)}]
\item $\widehat r:\widehat\Sigma\to(0,\infty)$ is well-defined and H\"older continuous.
\item $\widehat\pi_{\widehat r}\circ\widehat\sigma_{\widehat r}^t=\vf^t\circ\widehat \pi_{\widehat r}$,
for all $t\in\R$.
\item $\widehat\pi_{\widehat r}$ is H\"older continuous with respect to the Bowen-Walters distance.
\item $\widehat\pi_{\widehat r}[\widehat\Sigma_{\widehat r}^\#]=\nuh^\#$.
\end{enumerate}
\end{proposition}

\begin{proof}
To prove part (1), note that, by construction of $\widehat\sigma$ and of the sections $\Lambda\subset\widehat\Lambda$, $\widehat r$ is well-defined over $\widehat\Sigma$. Now, let $\un R\in\widehat\Sigma$ and notice that $U:=\{\un S\in\widehat\Sigma:(S_0,S_1)=(R_0,R_1)\}$ is a neighborhood of $\un R$. Moreover, 
there are $v,w\in\mathfs A$ and discs $D_i,D_j$ from $\widehat\Lambda$ such that 
$\widehat\pi(U)\subset Z(v)\subset D_i$ and $\widehat\pi(\widehat\sigma(U))\subset Z(w)\subset D_j$. 
Setting $\tau(x)=\inf\{t>0:\vf^t(x)\in D_j\}$ for $x$ on a neighborhood of $Z(v)$,
we have $\widehat r=\tau\circ\widehat\pi$ on $U$. 
Since $\tau$ is a continuous passage time between the two smooth disks, transverse to the flow, 
it is well-defined and smooth, see Lemma \ref{Lemma-regularity-q-t}(3).
To finish the proof of part (1), recall that $\widehat\pi$ is 
H\"older continuous by Proposition~\ref{Prop-relation-codings}(2).

Part (2) follows from the definition of $\widehat r$ by a routine argument, which we quickly recall. 
For $n\in\Z$, let $\widehat r_n$ be the $n$--th Birkhoff sum of $\widehat r$ (see Section~\ref{Section-Preliminaries}). 
Let $(\un R,s)\in\widehat\Sigma_{\widehat r}$. Given $t\in\R$, let $n\in\Z$ be defined by 
$\widehat r_n(\un R)\le t+s<\widehat r_{n+1}(\un R)$ so that 
$\widehat\sigma_{\widehat r}^{t}(\un R,s)=(\widehat\sigma^n(\un R),t+s-\widehat r_n(\un R))$. We have
\begin{align*}
&\, (\widehat\pi_{\widehat r}\circ \widehat\sigma_{\widehat r}^t)(\un R,s)
=\widehat\pi_{\widehat r}\left(\widehat\sigma^n(\un R),t+s-\widehat r_n(\un R)\right) 
=\vf^{t+s-\widehat r_n(\un R)}(\widehat\pi(\widehat\sigma^n(\un R)))\\
&= \vf^{t+s-\widehat r_n(\un R)}(\vf^{\widehat r_n(\un R)}(\widehat\pi(\un R)))=
\vf^{t+s}(\widehat\pi(\un R))
= (\vf^t\circ\widehat\pi_{\widehat r})(\un R,s),
\end{align*}
and so part (2) is established.

Now we prove part (3). By Proposition~\ref{Prop-relation-codings}(2), $\widehat\pi$ is H\"older continuous.
Applying the same arguments of \cite[Lemma 5.9]{Lima-Sarig}, we conclude that
$\widehat\pi_{\widehat r}$ is H\"older continuous with respect to the Bowen-Walters distance. 

\medbreak
We finally arrive at part (4). Recall from Proposition \ref{Prop-relation-codings}(4) that
$\widehat\pi[\widehat\Sigma^\#]=\pi[\Sigma^\#]$, hence Proposition~\ref{Prop-pi}(3) rewrites as
$\widehat\pi[\widehat\Sigma^\#]\supset\Lambda\cap\nuh^\#$. The flow saturation of 
$\widehat\pi[\widehat\Sigma^\#]$ is $\widehat\pi_{\widehat r}[\widehat\Sigma_{\widehat r}^\#]$
by definition, and the flow saturation of $\Lambda\cap\nuh^\#$ is $\nuh^\#$ since $\Lambda$ is a global section
and $\nuh^\#$ is $\vf$--invariant. Therefore 
$\widehat\pi_{\widehat r}[\widehat\Sigma_{\widehat r}^\#]\supset\nuh^\#$. Reversely,
$\widehat\pi[\widehat\Sigma^\#]=\pi[\Sigma^\#]$ is contained in $\nuh^\#$ by Theorem 
\ref{Thm-inverse}. Saturating this inclusion under the flow, we obtain that $\widehat\pi_{\widehat r}[\widehat\Sigma_{\widehat r}^\#]\subset\nuh^\#$. This concludes the proof of part (4).
\end{proof}

By Proposition \ref{Prop-adaptedness}, the above proposition establishes Parts
(1) and (2) of the Main Theorem. In the next sections, we focus on proving part (3)
and the other properties stated in Theorem~\ref{t.main}.

\subsection{The map $\widehat\pi_r$ is finite-to-one}\label{s.finite-one}

Given $Z\in\mathfs Z$, remember that $\mathfs I_Z=\{Z'\in\mathfs Z:\vf^{[-\rho,\rho]}Z\cap Z'\neq\emptyset\}$.
The loss of injectivity of $\widehat\pi_{\widehat r}$ is related to the following notion.

\medskip
\noindent
{\sc Affiliation:} We say that two rectangles $R,S\in\mathfs R$ are {\em affiliated}, and write
$R\sim S$, if there are
$Z,Z'\in\mathfs Z$ such that $R\subset Z$, $S\subset Z'$ and $Z'\in\mathfs I_Z$. 
This is a symmetric relation.

\begin{lemma}\label{Lemma-affiliation}
If $\widehat\pi(\un R)=\vf^t[\widehat\pi(\un S)]$ with $\un R,\un S\in\widehat\Sigma^\#$ and $|t|\leq\rho$, then $R_0\sim S_0$.
More precisely, if $\un v,\un w\in\Sigma^\#$ are such that $\pi(\un v)=\widehat\pi(\un R)$ and $\pi(\un w)=\widehat\pi(\un S)$, then $R_0\subset Z(v_0)$ and $S_0\subset Z(w_0)$ with $Z(w_0)\in\mathfs I_{Z(v_0)}$.
\end{lemma}

\begin{proof}
Let $y=\widehat\pi(\un R)$ and $z=\widehat\pi(\un S)$, so that $y=\vf^{t}(z)$.
Applying Proposition \ref{Prop-relation-codings} to $\un R$ and $\un S$, we find
two $\ve$--gpo's $\un v,\un w\in \Sigma^\#$ such that:
\begin{enumerate}[$\circ$]
\item $\pi(\un v)=y$ and $R_0\subset Z(v_0)$,
\item $\pi(\un w)=z$ and $S_0\subset Z(w_0)$.
\end{enumerate}
The lemma thus follows with $Z=Z(v_0)$ and $Z'=Z(w_0)$, since $\vf^t(z)\in Z(v_0)$.
\end{proof}

\begin{remark} We observe that the condition $\widehat\pi(\un R)=\vf^t[\widehat\pi(\un S)]$ in the above lemma 
actually implies more than just $R_0\sim S_0$. It implies a {\em strong} affiliation: for {\em any}
$Z,Z'\in\mathfs Z$ such that $Z\supset R_0$ and $Z' \supset S_0$, we have
$Z'\in\mathfs I_Z$. Indeed, if $\un R,\un S\in\widehat\Sigma^\#$ and $|t|\leq\rho$
satisfy $\widehat\pi(\un R)=\vf^t[\widehat\pi(\un S)]$ and $Z,Z'\in\mathfs Z$ satisfy $Z\supset R_0$
and $Z' \supset S_0$, Proposition \ref{Prop-relation-codings} gives the existence of $\un v,\un w\in\Sigma^\#$
such that $\pi(\un v)=\widehat\pi(\un R)$ and $\pi(\un w)=\widehat\pi(\un S)$ with $Z(v_0)=Z$ and $Z(w_0)=Z'$,
and so $Z'\in\mathfs I_Z$.
\end{remark}

For each $R\in\mathfs R$, define
$$
A(R):=\{(S,Z')\in\mathfs R\times\mathfs Z:R\sim S\text{ and }S\subset Z'\}\text{ and }N(R):=\#A(R).
$$
We can use Lemma \ref{Lemma-local-finite} and proceed as in the proof of \cite[Lemma 12.7]{Sarig-JAMS}
to show that $N(R)<\infty$, $\forall R\in\mathfs R$. Having this in mind, we now prove the finiteness-to-one
property of $\widehat\pi_{\widehat r}$, i.e. part (3) of the Main Theorem and of Theorem~\ref{t.main}.

\begin{theorem}\label{Thm-finite-extension}
Every $x\in \widehat\pi_{\widehat r}[\widehat\Sigma_{\widehat r}^\#]$ has finitely many
$\widehat\pi_{\widehat r}$--preimages inside $\widehat\Sigma_{\widehat r}^\#$.
More precisely, if $x=\widehat\pi_{\widehat r}(\un R,t)$ with $R_n=R$ for infinitely many $n>0$
and $R_n=S$ for infinitely many $n<0$, then
$\#\{(\un S,t')\in\widehat\Sigma^\#_{\widehat r}:\widehat\pi_{\widehat r}(\un S,t')=x\}\leq N(R)N(S)$.
\end{theorem}

\begin{proof}
The proof is by contradiction. Assuming that
$\#\{(\un S,t')\in\widehat\Sigma^\#_{\widehat r}:\widehat\pi_{\widehat r}(\un S,t')=x\}$
contains $N(R)N(S)+1$ distinct elements $(\un R^{(i)},t_i)$,  we are going to show that,
up to permutation of these preimages,  there are arbitrarily large integers $k<0<\ell$ such that 
\begin{equation}\label{eq-coincides}
(R^{(1)}_{k},\ldots,R^{(1)}_{\ell})=(R^{(2)}_{k},\ldots,R^{(2)}_{\ell}),
\end{equation} 
i.e. $\un R^{(1)}$ and $\un R^{(2)}$ agree between positions $k\to-\infty$ and $\ell\to+\infty$. 
This implies that $\un R^{(1)}=\un R^{(2)}$ and so $t_1\ne t_2$. But then $x$ is periodic with period 
$|t_2-t_1|<\widehat r(\un R^{(1)})<\rho$,  a contradiction to the choice of  $\rho$ (see Section~\ref{section-discs}).

The proof of equality~\eqref{eq-coincides} uses, as in \cite[Theorem 12.8]{Sarig-JAMS},
an idea of Bowen \cite[pp. 13--14]{Bowen-Regional-Conference}: it exploits the (non-uniform) expansiveness of $\vf$,
expressed in terms of the uniqueness of shadowing (Proposition \ref{Prop-shadowing}).
For simplicity of notation, we assume without loss of generality that $t=0$.
Recall that $r_n$ and $\widehat r_n$ denote Birkhoff sums for $n\in\Z$, see Section~\ref{Section-Preliminaries}.

\medskip
Let $x_n:=\vf^{\widehat r_n(\un R)}(x)=\widehat\pi[\widehat\sigma^n(\un R)]$, a point in the trajectory of $x$.
Fix two integers $k<0<\ell$ such that $R_k=S$ and $R_\ell=R$.

\begin{figure}
\includegraphics[width=0.9\linewidth]{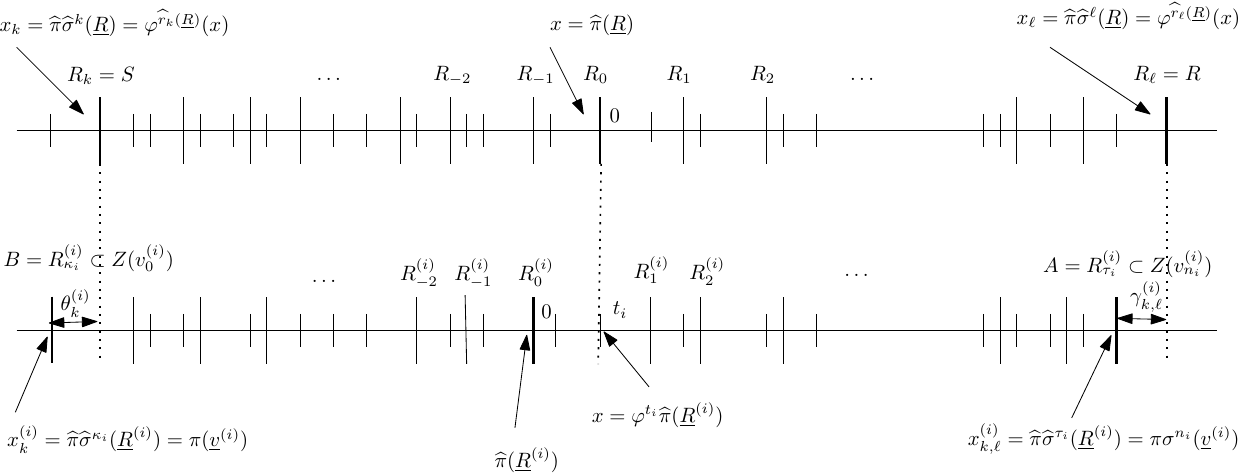}
\caption{The objects in the proof of Theorem~\ref{Thm-finite-extension}.  
The line above depicts the points associated to $x=\widehat\pi(\underline R)$ and the line below 
to $x=\varphi^{t_i}(\underline R^{(i)})$.  
Vertical segments represent visits to the section (long segments correspond to the symbols from $\un R$ or $\un R^{(i)}$).  
The origins $\widehat\pi(\un R)$ and $\widehat\pi(\un R^{(i)})$ are marked by a zero.}
\end{figure}

For each $i=1,\ldots,N(R)N(S)+1$, consider the following objects:
\begin{enumerate}[$\circ$]
\item Let $\kappa_i\in\Z$ be the unique integer such that
$\widehat r_{\kappa_i}(\un R^{(i)})\leq \widehat r_k(\un R)+t_i<\widehat r_{\kappa_i+1}(\un R^{(i)})$,
so that $(\vf^{t_i}\circ\widehat\pi\circ\widehat\sigma^{k})(\un R)$ belongs to the orbit segment between 
$(\widehat\pi\circ\widehat\sigma^{\kappa_i})(\un R^{(i)})$ and $(\widehat\pi\circ\widehat\sigma^{\kappa_i+1})(\un R^{(i)})$.
\item Let $\theta^{(i)}_k:=\widehat r_k(\un R)+t_i-\widehat r_{\kappa_i}(\un R^{(i)})$,
then $0\le\theta^{(i)}_k<\rho$.
\item Let $x^{(i)}_k:=\vf^{\widehat r_{\kappa_i}(\un R^{(i)})-t_i}(x)=\widehat\pi[\widehat\sigma^{\kappa_i}(\un R^{(i)})]$,
a point in the trajectory of $x$. Note that $\vf^{\theta^{(i)}_k}(x^{(i)}_k)=\vf^{\widehat r_k(\un R)}(x)=x_k$
and $\vf^{\theta^{(i)}_k+\widehat r_\ell(\un R)-\widehat r_k(\un R)}(x^{(i)}_k)=x_\ell$.
\item By Proposition \ref{Prop-relation-codings}, there is an $\ve$--gpo $\un v^{(i)}\in\Sigma^\#$
such that $\pi[\un v^{(i)}]=x^{(i)}_k$ and $R_{\kappa_i}^{(i)}\subset Z(v_0^{(i)})$.
\item Let $n_i$ be the unique integer such that
$r_{n_i}(\un v^{(i)})\leq \theta^{(i)}_k+\widehat r_\ell(\un R)-\widehat r_k(\un R)<r_{n_i+1}(\un v^{(i)})$.
Hence $x_\ell$ belongs to the orbit segment between $(\pi\circ\sigma^{n_i})(\un v^{(i)})$
and $(\pi\circ\sigma^{n_i+1})(\un v^{(i)})$.
\item Let $\tau_i>\kappa_i$ be the unique integer such that
$\widehat r_{\tau_i-\kappa_i}[\widehat\sigma^{\kappa_i}(\un R^{(i)})]=r_{n_i}(\un v^{(i)})$.
The existence of such an integer is ensured by Proposition~\ref{Prop-relation-codings} which also gives
$R^{(i)}_{\tau_i}\subset Z(v^{(i)}_{n_i})$.
\item Let $x^{(i)}_{k,\ell}:=\vf^{r_{n_i}(\un v^{(i)})}(x^{(i)}_k)=\widehat\pi[\widehat\sigma^{\tau_i}(\un R^{(i)})]$,
a point in the trajectory of $x$.
\item Let $\gamma^{(i)}_{k,\ell}:=\theta^{(i)}_k+\widehat r_\ell(\un R)-\widehat r_k(\un R)-r_{n_i}(\un v^{(i)})$,
then $|\gamma^{(i)}_{k,\ell}|<\rho$.
This is the time displacement between $x^{(i)}_{k,\ell}$ and $x_\ell$,
i.e. $x_\ell=\vf^{\gamma^{(i)}_{k,\ell}}(x^{(i)}_{k,\ell})$.
\end{enumerate}
Therefore, for each $i$, we have:
\begin{enumerate}[$\circ$]
\item $(R_{\kappa_i}^{(i)},Z(v_0^{(i)}))\in A(S)$: this follows from Lemma \ref{Lemma-affiliation},
since $x^{(i)}_k=\widehat\pi[\widehat\sigma^{\kappa_i}(\un R^{(i)})]$, $x_k=\widehat\pi[\widehat\sigma^k(\un R)]$
and $x_k=\vf^{\theta^{(i)}_k}(x^{(i)}_k)$.
\item $(R_{\tau_i}^{(i)},Z(v_{n_i}^{(i)}))\in A(R)$: this also follows from Lemma \ref{Lemma-affiliation},
since $x^{(i)}_{k,\ell}=\widehat\pi[\widehat\sigma^{\tau_i}(\un R^{(i)})]$, $x_\ell=\widehat\pi[\widehat\sigma^\ell(\un R)]$
and $x_\ell=\vf^{\gamma^{(i)}_{k,\ell}}(x^{(i)}_{k,\ell})$.
\end{enumerate}

The previous paragraph implies that every quadruple
$(R_{\kappa_i}^{(i)},Z(v_0^{(i)});R_{\tau_i}^{(i)},Z(v_{n_i}^{(i)}))$ we constructed
belongs to the cartesian product $A(R)\times A(S)$. This latter set has cardinality $N(R)N(S)$,
hence by the pigeonhole principle there are distinct $i,j$ such that
$$
(R_{\kappa_i}^{(i)},Z(v_0^{(i)});R_{\tau_i}^{(i)},Z(v_{n_i}^{(i)}))=
(R_{\kappa_j}^{(j)},Z(v_0^{(j)});R_{\tau_j}^{(j)},Z(v_{n_j}^{(j)})).
$$
For simplicity of notation, we assume $i=1$ and $j=2$ and write $R^{(1)}_{\kappa_1}=R^{(2)}_{\kappa_2}=:B$ and
$R^{(1)}_{\tau_1}=R^{(2)}_{\tau_2}=:A$.

Set $\alpha_i:=\widehat r_{\kappa_i}(\un R)-t_i$ and $\beta_i:=\widehat r_{\tau_i}(\un R)-t_i$ for $i=1,2$. 
By definition, we have $\alpha_i\in [\widehat r_k(\un R)-\rho,\widehat r_k(\un R)]$, and so $|\alpha_1-\alpha_2|\le\rho$. Since $\vf^{\alpha_1}(x)=\widehat\pi[\widehat\sigma^{\kappa_1}(\un R^{(1)})]$ and 
$\vf^{\alpha_2}(x)=\widehat\pi[\widehat\sigma^{\kappa_2}(\un R^{(2)})]$ both belong to $\overline{B}$, 
we must have $\alpha_1=\alpha_2$.
An analogous argument shows that $\beta_1=\beta_2$. We denote these common values by $\alpha,\beta$.

Since $R^{(1)}_{\kappa_1}\to\cdots\to R^{(1)}_{\tau_1}$
and $R^{(2)}_{\kappa_2}\to\cdots\to R^{(2)}_{\tau_2}$ are admissible paths on $\widehat\Sigma$,
we can find non-periodic points
$$
y\in {_{0}[}R^{(1)}_{\kappa_1},\ldots,R^{(1)}_{\tau_1}]\textrm{ and }
z\in{_{0}[}R^{(2)}_{\kappa_2},\ldots,R^{(2)}_{\tau_2}].
$$
Let $y'=H^{\tau_1-\kappa_1}(y)$ and $z'=H^{\tau_2-\kappa_2}(z)$.
We have $y,z\in B$ and $y',z'\in A$. By Proposition~\ref{Prop-R}(1), we can define two
points $w,w'$ by the equalities
\begin{align*}\{w\}&:=\{[y,z]\}= W^s(y,B)\cap W^u(z,B)\\
\{w'\}&:=\{[y',z']\}= W^s(y',A)\cap W^u(z',A).
\end{align*}
Note that neither $w$ nor $w'$ can be periodic.

\medskip
\noindent
{\sc Claim:} $w,w'$ belong to the same trajectory of $\vf$. More precisely, $w'=H^{\tau_1-\kappa_1}(w)$.

\begin{proof}[Proof of the claim.]
This is a consequence of Proposition \ref{Prop-overlapping-charts-2}: we can obtain $w'$
from $w$ by applying small flow displacements of Smale products of points at nearby rectangles.

To implement this idea, we first divide the interval $[\alpha,\beta]$ by visits to the rectangles 
$\left\{R^{(1)}_k\right\}_{\kappa_1\le k\le\tau_1}$ and $\left\{R^{(2)}_k\right\}_{\kappa_2\le k\le\tau_2}$.
Since these visits are $\rho$--dense in this interval, we can select times:
 \begin{equation}\label{eq-interleave}
   \delta_0=\alpha<\ve_0<\delta_1<\ve_1<\dots<\delta_T\le\ve_T=\beta
   \text{ such that } 0<\ve_s-\delta_s,\delta_{s+1}-\ve_s\le\rho
 \end{equation}
where each $\delta_t=\widehat r_m(\un R^{(1)})$ for some $m=m(t)\in[\kappa_1,\tau_1]$ and each
$\ve_t=\widehat r_n(\un R^{(2)})$ for some $n=n(t)\in[\kappa_2,\tau_2]$.

By Lemma \ref{Lemma-affiliation}, this implies that the successive rectangles implied by 
eq.~\eqref{eq-interleave} are affiliated:
$R^{(1)}_{m(t)}\sim R^{(2)}_{n(t)}$ and $R^{(2)}_{n(t)}\sim R^{(1)}_{m(t+1)}$.
Applying Proposition \ref{Prop-relation-codings},
find rectangles $Z,Z',Z''$ of $\mathfs Z$ that contain $R^{(1)}_{m(t)},R^{(2)}_{n(t)},R^{(1)}_{m(t+1)}$
and satisfy the conditions of Proposition~\ref{Prop-overlapping-charts-2}. The same applies to
the three rectangles $R^{(2)}_{n(t)},R^{(1)}_{m(t+1)},R^{(2)}_{n(t+1)}$.

Now let $y_k=H^{k-\kappa_1}(y)$ for $\kappa_1\leq k\leq \tau_1$ and $z_\ell=H^{\ell-\kappa_2}(z)$ for 
$\kappa_2\leq \ell\leq \tau_2$. 
For each $t=0,1,\ldots,T$,  note that $y_{m(t)}\in R^{(1)}_{m(t)}$ and $z_{n(t)}\in R^{(2)}_{n(t)}$. 
We let $D_t^{(1)}$ and $D^{(2)}_t$ be the connected components
of $\widehat\Lambda$ containing
$R^{(1)}_{m(t)}$ and $R^{(2)}_{n(t)}$ respectively.

On the one hand, since $y_{m(t+1)}=\vf^u(y_{m(t)})$ with $0\le u\le\rho$, Proposition \ref{Prop-overlapping-charts-2} implies that
 $$
   [y_{m(t)},z_{n(t)}]_{D^{(2)}_t}=[y_{m(t+1)},z_{n(t)}]_{D^{2}_t}.
 $$
On the other hand, Proposition \ref{Prop-overlapping-charts}(3) yields:
\begin{align*}
&\ [y_{m(t+1)},z_{n(t)}]_{D_{t+1}^{(1)}}= \mathfrak q_{D^{(1)}_{t+1}}([y_{m(t+1)},z_{n(t)}]_{D^{(2)}_t})\\
&=\mathfrak q_{D^{(1)}_{t+1}}([y_{m(t)},z_{n(t)}]_{D^{(2)}_t})=
(\mathfrak q_{D^{(1)}_{t+1}}\circ \mathfrak q_{D^{(2)}_t})([y_{m(t)},z_{n(t)}]_{D^{(1)}_t}).
\end{align*}
Finally, applying Proposition \ref{Prop-overlapping-charts-2} again, we conclude that
$$
[y_{m(t+1)},z_{n(t+1)}]_{D^{(1)}_{t+1}}= [y_{m(t+1)},z_{n(t)}]_{D^{(1)}_{t+1}}=
(\mathfrak q_{D^{(1)}_{t+1}}\circ \mathfrak q_{D^{(2)}_t})([y_{m(t)},z_{n(t)}]_{D^{(1)}_t}).
$$

Proceeding inductively,
\begin{align*}
w'&=[y',z']=\mathfrak q_{D^{(2)}_T}([y_{m(T)},z_{n(T)}]_{D^{(1)}_T})\\
&=(\mathfrak q_{D^{(2)}_T}\circ \mathfrak q_{D^{(1)}_T}\circ\mathfrak q_{D^{(2)}_{T-1}})([y_{m(T-1)},z_{n(T-1)}]_{D^{(1)}_{T-1}})\\
&=
\cdots \\
&= (\mathfrak q_{D^{(2)}_T}\circ \mathfrak q_{D^{(1)}_T}\circ\cdots \circ \mathfrak q_{D^{(1)}_1}\circ  \mathfrak q_{D^{(2)}_0})([y_{m(0)},z_{n(0)}]_{D^{(1)}_0})\\
&= (\mathfrak q_{D^{(2)}_T}\circ \mathfrak q_{D^{(1)}_T}\circ\cdots \circ \mathfrak q_{D^{(1)}_1}\circ  \mathfrak q_{D^{(2)}_0})([y,z]_{D^{(1)}_0})\\
&= (\mathfrak q_{D^{(2)}_T}\circ \mathfrak q_{D^{(1)}_T}\circ\cdots \circ \mathfrak q_{D^{(1)}_1}\circ  \mathfrak q_{D^{(2)}_0})(w),
\end{align*}
which proves that $w$ and $w'$ belong to the same trajectory. Repeating the argument using the holonomy maps corresponding to the sequence $(R^{(1)}_{\kappa_1},\dots,R^{(1)}_{\tau_1})$, we get that their composition sends $w$ to $w'$. By the Markov property in the stable direction, these holonomy maps correspond to first returns. This proves that $w'=H^{\tau_1-\kappa_1}(w)$.
\end{proof}

Now it is easy to conclude the proof of the theorem. 
A symmetric version of the claim implies that $w=H^{-(\tau_2-\kappa_2)}(w')$.
Since $w$ is not periodic, we obtain $\tau_1-\kappa_1=\tau_2-\kappa_2$. 
It follows that $(R^{(1)}_{\kappa_1},\dots,R^{(1)}_{\tau_1})=(R^{(2)}_{\kappa_2},\dots,R^{(2)}_{\tau_2})$,
since both correspond to the rectangles in $\mathfs R$ that contain $H^{\kappa_1}(w),\ldots,H^{\tau_1}(w)$.
This concludes the proof.
\end{proof}

\subsection{Conclusion of the proof of Theorem~\ref{t.main}}\label{ss.conclusion}

We already proved parts (1) and the first half of part (2). Also, Theorem \ref{Thm-finite-extension}
establishes part (3). For the second half of part (2), we note that
every point of $\nuh^\#$ has a finite and nonzero number of lifts to $\widehat \Sigma^\#_{\widehat r}$,
hence every ergodic $\chi$--hyperbolic measure on $M$, which is supported in $\nuh^\#$,
can be lifted to an ergodic $\widehat\sigma_{\widehat r}$--invariant measure $\overline\mu$,
exactly as in the argument performed in~\cite[Section 13]{Sarig-JAMS}. This concludes the proof of 
part (2) of Theorem~\ref{t.main}.

We now prove the remaining parts (4)--(8) stated in Theorem~\ref{t.main}.

\paragraph{\bf Part~(\ref{i.splitting}).} Using Theorem~\ref{Thm-stable-manifolds}, 
we define $N^{s/u}_{z}$ as follows:
\begin{enumerate}[$\circ$]
\item For $z=(\un R,0)\in\widehat\Sigma_{\widehat r}$, define first $V^{s/u}(z)=W^{s/u}(\widehat\pi(\un R),R_0)$
and $N^{s/u}_{z}=T_{\widehat\pi(\un R)}V^{s/u}(z)$. By definition,
$V^s(z)$ and $V^u(z)$ are transverse.
\item For $z=(\un R,t)\in\widehat\Sigma_{\widehat r}$, define
$N^{s/u}_{z}=\Phi^t\left(N^{s/u}_{(\un R,0)}\right)$.
Since $\Phi$ is an isomorphism, $N_{\widehat\pi_{\widehat r}(\un R,t)}=N^{s}_{z}\oplus N^{u}_{z}$.
\end{enumerate}
The geometrical Markov property of Proposition \ref{Prop-R}(3) 
implies that the families $\{N^{s/u}_z\}$ are invariant under $\Phi$.
The convergence rates along $N^{s/u}_z$ follow from Theorem \ref{Thm-stable-manifolds}(3),
taking $\lambda:=\tfrac{1}{\sup(r_\Lambda)}\left(\tfrac{\chi \inf(r_\Lambda)}{2}-\tfrac{\beta\ve}{6}\right)$.
These estimates show, in particular, that these spaces only depend on $x:=\widehat\pi_{\widehat r}(z)$,
hence one can set $N^{s/u}_x:=N^{s/u}_z$.
Finally, the H\"older continuity follows from 
Theorem \ref{Thm-stable-manifolds}(5). This concludes the proof of part~(\ref{i.splitting}).
\medskip

\paragraph{\bf Part~\eqref{i.manifold}.}
For any $z=(\un R,0)\in \widehat \Sigma$, Theorem~\ref{Thm-stable-manifolds}
associates curves $V^{s/u}(z)$ tangent to $\widehat \Lambda$, hence transverse to the flow direction.
For general $z=(\un R,t)\in\widehat\Sigma_{\widehat r}$, 
one then defines the manifolds $V^{cs/cu}(z):=\varphi^{[t-1,t+1]}(V^{s/u}(\un R,0))$.
By construction, $V^{cs/cu}(z)$ is tangent to $N^{s/u}_z+\mathbb{R}\cdot X(\widehat \pi_{\widehat r}(z))$.
Moreover, by Proposition~\ref{Prop-center-stable}, for any $y\in V^{cs}(z)$ there exists $\tau\in \mathbb{R}$
such that $d(\varphi^t(\widehat\pi_{\widehat r}(z)),\varphi^{t+\tau}(y))\leq \exp{}(-\lambda t)$
for all $t\geq 0$. The same holds for $V^{cu}(z)$, thus concluding the proof of Part~\eqref{i.manifold}.

\medskip

\paragraph{\bf Part~\eqref{i.canonical}.} The proof of this part is almost automatic.
The measurable set $\mathfs Z=\mathfs R$ contains $\Lambda\cap \nuh^\#$, hence
the orbit of any point $x\in \nuh^\#$ intersects $\mathfs R$, which proves item (a).
Item (b) was proved in the beginning of Section~\ref{subsec-fundpropZ}.
Finally, any $x\in \mathfs R$ defines $\{R_n\}_{n\in\Z}$
such that $H^n(x)\in R_n$ for all $n\in\Z$.
In particular, $H(R_n)\cap R_{n+1}\neq \emptyset$ for all $n\in\Z$ and so
$\un R=\{R_n\}\in\widehat\Sigma$. Since $\mathfs R=\pi[\Sigma^\#]$, we also
have $x=\pi(\un v)$ for some $\un v=\{v_n\}_{n\in\Z}\in\Sigma^\#$. For each $k\in\Z$,
the point $\pi[\sigma^k(\un v)]$ is a return of $x$ to $\mathfs R$, hence
there is an increasing sequence such that $\pi[\sigma^k(\un v)]=H^{n_k}(x)$.
Therefore $R_{n_k}\subset Z(v_k)$.
Using that $\un v\in\Sigma^\#$ and Lemma \ref{Lemma-local-finite}(1), it follows that
$\un R\in\widehat\Sigma^\#$.
\medskip

\paragraph{\bf Part~\eqref{i.lift}.}
Assume $K\subset M$ is a compact, transitive, invariant, hyperbolic set
such that all $\vf$--invariant measures supported by it are $\chi$--hyperbolic. Let
$TK=E^s\oplus X\oplus E^u$ be the continuous hyperbolic splitting.
Proceeding as in \cite[Proposition 2.8]{BCS-MME}, there are constants $C>0$ and $\kappa>\chi$
such that
$$
\|d\vf^t v^s\|\leq Ce^{-\kappa t}\|v^s\| \text{ and } \|d\vf^{-t} v^u\|\leq Ce^{-\kappa t}\|v^u\|\text{,\ \ for all }v^s\in E^s,
v^u\in E^u\text{ and }t\geq 0.
$$ 
Now we proceed as in the proof of Proposition \ref{Prop-NUH}. Using the notation 
of equation (\ref{definition-normal-vectors}), the functions $x\in K\mapsto \gamma^{s/u}(x)$ are continuous.
Therefore there is a constant $C_1=C_1(K)$ such that $s(x),u(x)<C_1$
and $\alpha(x)=\angle(n^s_x,n^u_x)>C_1^{-1}$ for all $x\in K$. This implies 
that $\inf_{x\in K}Q(x)>0$, which in turn implies that $\inf_{x\in K}q(x)>0$. In particular,
$K\subset\nuh^\#$. This is enough to reproduce the method of proof of \cite[Prop. 3.9]{BCS-MME}, as follows.
We recall that $X\subset\widehat\Sigma_{\widehat r}$ is $\widehat\sigma_{\widehat r}$--invariant if 
$\widehat\sigma_{\widehat r}^t(X)=X$ for all $t\in\R$.

\medskip
\noindent
{\sc Step 1:} There is a $\widehat\sigma_{\widehat r}$--invariant compact set $X_0\subset \widehat\Sigma_{\widehat r}$ such that $\widehat\pi_{\widehat r}(X_0)\supset K$.

\begin{proof}[Proof of Step $1$]
For each $x\in K\cap\mathfs R$, consider its canonical coding $\un R(x)=\{R_n(x)\}_{n\in\Z}$.
Since $\inf_{x\in K}q(x)>0$, $K$ intersects finitely many rectangles of $\mathfs R$. Hence there is a finite
set $V_0\subset\mathfs R$ such that $R_0(x)\in V_0$ for all $x\in K\cap\mathfs R$.
By invariance, the same happens for all $n\in\Z$, i.e.
$R_n(x)\in V_0$ for all $x\in K\cap\mathfs R$. Therefore 
the subshift $\Sigma_0$ induced by $V_0$, which is compact since $V_0$ is finite,
satisfies $\widehat\pi(\Sigma_0)\supset K\cap\mathfs R$.
Let $X_0$ be the TMF defined by $(\Sigma_0,\sigma)$ with roof function $\widehat r\restriction_{\Sigma_0}$.
Saturating the latter inclusion under $\vf$ and using part~(7)(a), we conclude that
$\widehat\pi_{\widehat r}(X_0)\supset K$.
\end{proof}

\medskip
\noindent
{\sc Step 2:} There is a transitive $\widehat\sigma_{\widehat r}$--invariant compact subset
$X\subset X_0$ such that $\widehat\pi_{\widehat r}(X)=K$.

\begin{proof}[Proof of Step $2$] Among all compact $\widehat\sigma_{\widehat r}$--invariant sets
$X\subset X_0$ with $\widehat\pi_{\widehat r}(X)\supset K$,
consider one which is minimal for the inclusion (it exists by Zorn's lemma). We claim that such an $X$ satisfies
Step 2. 
To see that, let $z\in K$ whose forward orbit is dense in $K$, let
$x\in X$ be a lift of $z$,
and let $Y$ be the $\omega$--limit set of the forward orbit of $x$,
$$
Y=\{y\in \widehat\Sigma_{\widehat r}:\exists t_n\to+\infty \text{ s.t. }\widehat\sigma_{\widehat r}^{t_n}(x)\to y\}.
$$
For any $n\geq 1$, the set
$Y_n:=\{\sigma^t_{\widehat r}(x), t\geq n\}\cup Y\subset X$ is compact and forward invariant.
Hence the projection $\widehat\pi_{\widehat r}(Y_n)$ is compact and contains
$\{\varphi^t(z), t\geq n\}$. Since the forward orbit of $z$ is dense in $K$,
we have $\widehat\pi_{\widehat r}(Y_n)\supset K$.
Taking the intersection over $n$, one deduces that the projection of the
$\sigma^t_{\widehat r}$--invariant compact set $Y$ contains $K$.
By the minimality of $X$, it follows that $X=Y$.
\end{proof}

This concludes the proof of Part~\eqref{i.lift}.
\bigskip

\paragraph{\bf Part~\eqref{i.Bowen}, items (a) and (b)-(i).} Item (a) of Part~\eqref{i.Bowen}, the local finiteness of the affiliation, was proved at the beginning of Section~\ref{s.finite-one}.
Item (b) claims that the affiliation $\sim$ is a Bowen relation. This splits into two properties (i)~and~(ii).

To prove item (i) of the Bowen relation,  let $(\un R,t),(\un S,s)\in \widehat \Sigma^\#_{\widehat r}$ with 
$\widehat \pi_{\widehat r}(\un R,t)=\widehat \pi_{\widehat r}(\un S,s)$, i.e.
$\widehat \pi(\un R)=\vf^{s-t}\widehat \pi(\un S)$. 
Since $|s-t|\leq\sup(\widehat r)\le \rho$, Lemma~\ref{Lemma-affiliation} implies that $R_0\sim S_0$.
\bigskip

\paragraph{\bf Part~\eqref{i.Bowen}, item (b)-(ii).}
We turn to property (ii) of a Bowen relation. We take $\gamma=3\rho$.
Let $z,z'\in\widehat\Sigma_{\widehat r}^\#$ 
such that $\operatorname{v}(\widehat\sigma_{\widehat r}^t z)\sim\operatorname{v}(\widehat\sigma_{\widehat r}^tz')$
for all $t\in\R$. By flowing the two orbits, we can assume that $z=(\un R,0)$ and $z'=(\un S,s)$.
Let $x=\widehat\pi(\un R)$ and $y=\widehat\pi(\un S)$. We wish to show that $x=\vf^{t+s}(y)$ for some 
$|t|< \gamma$. We will deduce from the affiliation condition that the orbit of $y$ must be shadowed by an
$\varepsilon$--gpo that shadows $x$.  By Proposition~\ref{Prop-shadowing}, the two orbits are equal
and the time shift between $x$ and $\vf^s(y)$ will be easily bounded.

To do this, we first apply Proposition \ref{Prop-relation-codings}(1) and get $\varepsilon$--gpo's
$\un v,\un w\in \Sigma^\#$ such that $x=\widehat\pi(\un R)=\pi(\un v)$ and $y=\widehat\pi(\un S)=\pi(\un w)$ 
with $R_0\subset Z(v_0)$ and $S_0\subset Z(w_0)$. Moreover, there are increasing integer sequences $(n_i)_{i\in\Z}$, $(\widetilde m_i)_{i\in\Z}$ such that $R_{n_i}\subset Z(v_i)$ and $S_{\widetilde m_i}\subset Z(w_i)$. For each $i\in\Z$, we locate affiliated symbols in the codings of $x$ and $y$ as follows.

We start with $\vf^t(x)\in Z(v_i)$ for $t=r_i(\un v)=\widehat r_{n_i}(\un R)$.
We have $\widehat\sigma^t_{\widehat r}(\un R,0)=(\widehat\sigma^{n_i}(\un R),0)$,
hence $\operatorname{v}(\widehat\sigma_{\widehat r}^t(z))=R_{n_i}$.
We also have $\widehat\sigma^t_{\widehat r}(\un S,s)=(\widehat\sigma^{\ell_i}(\un S),t+s-\widehat r_{\ell_i}(\un S))$,
where $\ell_i$ is the unique integer such that $\widehat r_{\ell_i}(\un S)\le t+s<\widehat r_{\ell_i+1}(\un S)$.
Thus $\operatorname{v}(\widehat\sigma_{\widehat r}^t(z'))=S_{\ell_i}$ and, by assumption, $R_{n_i}\sim S_{\ell_i}$.

Let $a_i\in\Z$ be the largest integer such that $m_i:=\widetilde m_{a_i}\le\ell_i$. Hence, $S_{m_i}\subset Z(w_{a_i})$.
We have $R_{n_i}\subset Z(v_i)\subset D_i$ and likewise $S_{m_i}\subset Z(w_{a_i})\subset E_i$ for some unique connected components $D_i,E_i$  of the section $\widehat \Lambda$.

We write $\Psi_{X_i}^{P^s_i,P^u_i}$ for $v_i$ and $\Psi_{Y_i}^{Q^s_i,Q^u_i}$ for $w_{a_i}$ for all $i\in\Z$.
Finally, we set $\widetilde y_i:=\pi(\sigma^{a_i}\un w)\in Z(w_{a_i})$  and $y_i:=\mathfrak q_{D_i}(\widetilde y_i)$.
We are going to show that, for all $i\in\Z$:
 \begin{enumerate}[(1)]\label{id.for.bowen}
  \item[(1)] $y_i$ is well-defined, and for $i=0$ we have $y_0=\vf^{u}(\widetilde y_0)$
  with $|u|\leq 2\rho$;
  \item[(2)] $y_{i+1}=g_{X_i}^+(y_i)$. 
 \end{enumerate}
Proposition~\ref{Prop-shadowing} will then imply that 
$x=y_0=\vf^u(\widetilde y_0)=\vf^u(y)=\vf^{u-s}(\widehat\pi_{\widehat r}(\un S,s))$, where
$|u-s|\le 2\rho+\sup\widehat r<3\rho$. 
Property (ii) and therefore the Bowen relation claimed by Part (6)(b) will be established.

\begin{figure}\label{fig-Bowen}
\includegraphics[width=0.9\linewidth]{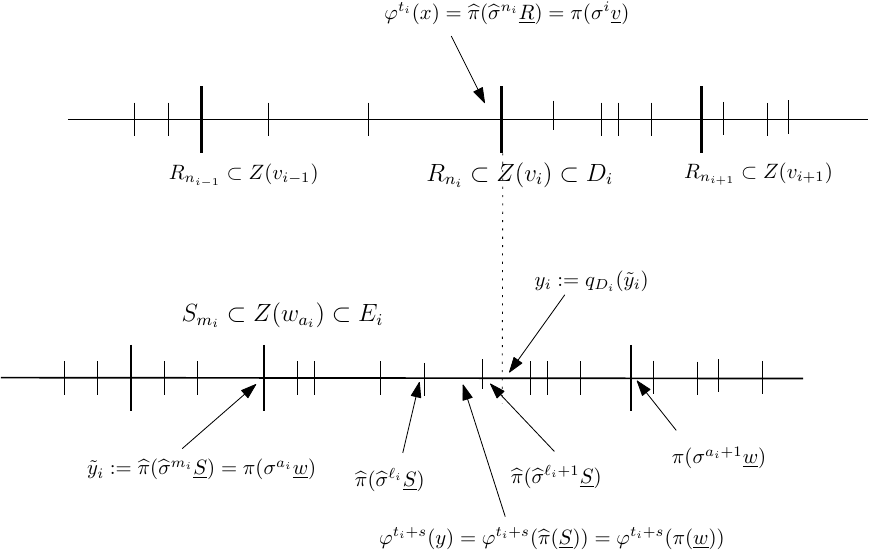}
\caption{The objects in the proof of Theorem~\ref{t.main}, part (6)(b).  
The time $t_i$ is $r_i(\underline v)=\widehat r_{n_i}(\underline R)$ for an arbitrary $i\in\Z$.
The line above depicts the coding of $x=\widehat\pi(\underline R)=\pi(\underline v)$:
large vertical lines correspond to $R_{n_j}\subset Z(v_j)$,  shorter ones to other $R_n$'s.
The line below is related to the coding of $y=\widehat\pi(\underline S)=\pi(\underline w)$ with $S_{m_i}\subset Z(w_{a_i})$, the symbol that our proof relates to $R_{n_i}\subset Z(v_i)$.
By construction $R_{n_i}\sim S_{\ell_i}$ and $S_{\ell_i}\sim S_{m_i}$.
The point $y_i$ is the trace of the orbit of $y$ on $D_i$,  the connected component of the section containing $Z(v_i)$, figured by a dotted line.}
\end{figure}

It remains to prove the above identities. They require checking that some holonomies along the flow are compatible. 
We will prove this using that affiliation implies
that charts have comparable parameters and their images fall inside $\widehat\Lambda$
far from its boundary. The claims below are not sharp but enough for our purposes.
We begin by proving some variants of Proposition \ref{Prop-overlapping-charts}(1).

\medskip
\noindent
{\sc Claim 1:} Let $Z_1,Z_2\in\mathfs Z$ such that $Z_1\cap \vf^{[-\rho,\rho]}Z_2\neq\emptyset$.
Write $Z_i=Z(\Psi_{x_i}^{p^s_i,p^u_i})$ and let $D_i$ be the connected component
of $\widehat\Lambda$ containing $Z_i$. 
Then $\tfrac{p^s_1\wedge p^u_1}{p^s_2\wedge p^u_2}=e^{\pm(O(\sqrt[3]{\ve})+O(\rho))}$ and  
$$
\mathfrak q_{D_1}(\Psi_{x_2}(R[c(p^s_2\wedge p^u_2)]))\subset \Psi_{x_1}(R[2c(p^s_1\wedge p^u_1)])
$$
for all $1\leq c\leq 64$.\\

\begin{proof}[Proof of Claim $1$.] Same of Proposition \ref{Prop-overlapping-charts}(1).
\end{proof}

\medskip
\noindent
{\sc Claim 2:} Let $R_1,R_2\in\mathfs R$ such that $R_1\sim R_2$. For $i=1,2$,
let $D_i$ be the connected component of
$\widehat\Lambda$ containing $R_i$, and let $Z_i=Z(\Psi_{x_i}^{p^s_i,p^u_i})\in\mathfs Z$ 
such that $Z_i\supset R_i$. 
Then $\tfrac{p^s_1\wedge p^u_1}{p^s_2\wedge p^u_2}=e^{\pm(O(\sqrt[3]{\ve})+O(\rho))}$ and
$$
\mathfrak q_{D_1}(\Psi_{x_2}(R[c(p^s_2\wedge p^u_2)]))\subset \Psi_{x_1}(R[8c(p^s_1\wedge p^u_1)]).
$$
for all $1\leq c\leq 16$.

\begin{proof}[Proof of Claim $2$.]
Since $R_1\sim R_2$, there are $W_1,W_2\in\mathfs Z$ such that $W_i\supset R_i$
and $W_1\cap\vf^{[-\rho,\rho]}W_2\neq\emptyset$. Write $W_i=Z(\Psi_{y_i}^{q^s_i,q^u_i})$.
We apply Claim 1 three times:
\begin{enumerate}[$\circ$]
\item Since $W_2,Z_2\supset R_2$, we have $W_2\cap Z_2\neq\emptyset$, hence 
$\tfrac{p^s_2\wedge p^u_2}{q^s_2\wedge q^u_2}=e^{\pm(O(\sqrt[3]{\ve})+O(\rho))}$ and
$$
\Psi_{x_2}(R[c(p^s_2\wedge p^u_2)])\subset \Psi_{y_2}(R[2c(q^s_2\wedge q^u_2)]).
$$
\item Since $W_1\cap\vf^{[-\rho,\rho]}W_2\neq\emptyset$, we have
$\tfrac{q^s_2\wedge q^u_2}{q^s_1\wedge q^u_1}=e^{\pm(O(\sqrt[3]{\ve})+O(\rho))}$ and
$$
\mathfrak q_{D_1}(\Psi_{y_2}(R[2c(q^s_2\wedge q^u_2)]))\subset \Psi_{y_1}(R[4c(q^s_1\wedge q^u_1)]).
$$
\item Since $W_1,Z_1\supset R_1$, we have $W_1\cap Z_1\neq\emptyset$, hence 
$\tfrac{q^s_1\wedge q^u_1}{p^s_1\wedge p^u_1}=e^{\pm(O(\sqrt[3]{\ve})+O(\rho))}$ and
$$
\Psi_{y_1}(R[4c(q^s_1\wedge q^u_1)])\subset \Psi_{x_1}(R[8c(p^s_1\wedge p^u_1)].
$$
\end{enumerate}
Plugging these inclusions together, Claim 2 is proved.
\end{proof}

\medskip
\noindent
{\sc Claim 3:} Let $R_1,R_2,R_3\in\mathfs R$ such that $R_1\sim R_2$ and $R_2\sim R_3$.
For $i=1,2,3$, let $D_i$ be the connected component of
$\widehat\Lambda$ containing $R_i$, and let $Z_i=Z(\Psi_{x_i}^{p^s_i,p^u_i})\in\mathfs Z$ 
such that $Z_i\supset R_i$. 
Then $\tfrac{p^s_3\wedge p^u_3}{p^s_1\wedge p^u_1}=e^{\pm(O(\sqrt[3]{\ve})+O(\rho))}$ and
$$
(\mathfrak q_{D_1}\circ \mathfrak q_{D_2})(\Psi_{x_3}(R[c(p^s_3\wedge p^u_3)]))=
\mathfrak q_{D_1}(\Psi_{x_3}(R[c(p^s_3\wedge p^u_3)]))\subset 
\Psi_{x_1}(R[64c(p^s_1\wedge p^u_1)])
$$
for all $1\leq c\leq 2$.

\begin{proof}[Proof of Claim $3$.]
The estimate $\tfrac{p^s_3\wedge p^u_3}{p^s_1\wedge p^u_1}=e^{\pm(O(\sqrt[3]{\ve})+O(\rho))}$
follows directly from Claim 2. Also by Claim 2, we have the inclusions
$\mathfrak q_{D_2}(\Psi_{x_3}(R[c(p^s_3\wedge p^u_3)]))\subset \Psi_{x_2}(R[8c(p^s_2\wedge p^u_2)])$
and $\mathfrak q_{D_1}(\Psi_{x_2}(R[8c(p^s_2\wedge p^u_2)]))\subset \Psi_{x_1}(R[64c(p^s_1\wedge p^u_1)])$.
This implies that $(\mathfrak q_{D_1}\circ \mathfrak q_{D_2})(\Psi_{x_3}(R[c(p^s_3\wedge p^u_3)]))\subset 
\Psi_{x_1}(R[64c(p^s_1\wedge p^u_1)])$. In particular, it proves that 
we can project $\Psi_{x_3}(R[c(p^s_3\wedge p^u_3)])$ to $D_1$, and so the equality follows. 
\end{proof}

Now we apply the above claims to our particular situation. Write $v_i=\Psi_{x_i}^{p^s_i,p^u_i}$
and $w_i=\Psi_{z_i}^{q^s_i,q^u_i}$, so that $Q^{s/u}_i=q^{s/u}_{a_i}$.

\medskip
\noindent
{\sc Claim 4:} Let $i\in\Z$. We have
$$
\mathfrak q_{E_{i+1}}(\Psi_{Y_{i}}(R[Q^s_{i}\wedge Q^u_{i}]))\subset
\Psi_{Y_{{i+1}}}(R[2(Q^s_{{i+1}}\wedge Q^u_{{i+1}})]).
$$

\begin{proof}[Proof of Claim $4$.] 
By Lemma \ref{Lemma-minimum} and Claim 3, we have 
$$
\tfrac{q^s_{a_{i+1}}\wedge q^u_{a_{i+1}}}{q^s_{a_i}\wedge q^u_{a_i}}=
\tfrac{q^s_{a_{i+1}}\wedge q^u_{a_{i+1}}}{p^s_{i+1}\wedge p^u_{i+1}}\cdot 
\tfrac{p^s_{i+1}\wedge p^u_{i+1}}{p^s_{i}\wedge p^u_{i}}\cdot
\tfrac{p^s_{i}\wedge p^u_{i}}{q^s_{a_i}\wedge q^u_{a_i}}=e^{\pm(O(\sqrt[3]{\ve})+O(\rho))}.
$$
This estimate allows to apply the same proof of Proposition \ref{Prop-overlapping-charts}(1), 
and so we can obtain the claimed inclusion in the same manner.
\end{proof}

\medskip
\noindent
{\sc Claim 5:} Let $i\in\Z$. Restricted to the set $\Psi_{Y_{i}}(R[Q^s_{i}\wedge Q^u_{i}])$, we have the equality
$\mathfrak q_{D_{i+1}}\circ \mathfrak q_{E_{i+1}}=\mathfrak q_{D_{i+1}}=g_{X_i}^+\circ \mathfrak q_{D_i}$.
A similar statement holds for $i\leq 0$.

\begin{proof}[Proof of Claim $5$.] It is enough to prove the equality for $i=0$, i.e. that
$\mathfrak q_{D_1}\circ\mathfrak q_{E_1} =\mathfrak q_{D_1}=g_{X_0}^+\circ \mathfrak q_{D_0}$ 
when restricted to $\Psi_{Y_0}(R[Q^s_0\wedge Q^u_0])$. 
By Claim 4, $\mathfrak q_{E_1}[\Psi_{Y_0}(R[Q^s_0\wedge Q^u_0])]\subset \Psi_{Y_1}(R[2(Q^s_1\wedge Q^u_1)])$.
Applying Claim 3 with $c=2$ to the triple $(R_{n_1},S_{\ell_1},S_{m_1})$, we get that
$\mathfrak q_{D_1}[\Psi_{Y_1}(R[2(Q^s_1\wedge Q^u_1)])]$
is well-defined, hence $\mathfrak q_{D_1}\circ \mathfrak q_{E_1}=\mathfrak q_{D_1}$ when
restricted to $\Psi_{Y_0}(R[Q^s_0\wedge Q^u_0])$. On the other hand, applying Claim 3
with $c=1$ to the triple $(R_{n_0},S_{\ell_0},S_{m_0})$, we have that
$\mathfrak q_{D_0}[\Psi_{Y_0}(R[Q^s_0\wedge Q^u_0])]\subset \Psi_{X_0}(R[64(P^s_0\wedge P^u_0)])$.
By definition, $g_{X_0}^+=\mathfrak q_{D_1}$ when restricted to $R[64(P^s_0\wedge P^u_0)]$.
Therefore, $g_{X_0}^+\circ \mathfrak q_{D_0}=\mathfrak q_{D_1}$ when
restricted to $\Psi_{Y_0}(R[Q^s_0\wedge Q^u_0])$. This proves Claim 5.
\end{proof}

We now complete the proof of identities (1) and (2) of page \pageref{id.for.bowen}, 
which in turn will complete the proof of part (6) of Theorem \ref{t.main}.
For that, we use the claims we just proved.

Firstly we check that $y_i:=q_{D_i}(\widetilde y_i)$ is well-defined. By assumption $R_{n_i}\sim S_{\ell_i}$,
and by construction the orbit of $y$ between $S_{m_i}$ and $S_{\ell_i}$ flows for a time at most $\sup(r)<\rho$,
hence $S_{\ell_i}\sim S_{m_i}$. This allows us to apply Claim 3 for $c=1$ and get 
that $y_i:=\mathfrak q_{D_i}(\widetilde y_i)$ is well-defined. To calculate the time displacement for
$i=0$, recall that $m_0=\ell_0=0$. Since $R_0\sim S_0$, inclusion (\ref{e.2rho}) 
implies that $y_0=\vf^u(\widehat y_0)$ with $|u|\leq 2\rho$.

Finally, Claim 5 implies that 
 $$
   g_{X_i}^+(y_i)= g_{X_i}^+\circ \mathfrak q_{D_i}(\widetilde y_i)= q_{D_{i+1}}\circ q_{E_{i+1}}(\widetilde y_i)
    =  q_{D_{i+1}}(\widetilde y_{i+1})=y_{i+1},
  $$
finishing  the proof of Theorem~\ref{t.main}.

\newcommand\ignore[1]{}

\ignore{
Now we prove the implication $(\Longrightarrow)$ of Part~\eqref{i.Bowen}(c). 
We are assuming that 
$\operatorname{v}(\widehat\sigma_{\widehat r}^t(z))\sim \operatorname{v}(\widehat\sigma_{\widehat r}^t(z'))$
for all $t\in\R$.
We can assume that $z=(\un R,0)$ and $z'=(\un S,s)$ with $\un R=\{R_n\}_{n\in\Z}$ and $\un S=\{S_n\}_{n\in\Z}$. Let $x=\widehat\pi(\un R)$ and $y=\widehat\pi(\un S)$.
We are going to show $x,y$ belong to the same orbit by showing that they are shadowed by some $\varepsilon$-gpo defining $x$.

\medbreak

We need to show the following for all $i\in\Z$:
 \begin{itemize}
  \item $y_i:=q_{D_i}(\widehat\pi\widehat\sigma^{m_i}\underline S)$ is well-defined and belongs to $\Psi_{x_{i}}(R[10Q(x_{i})]\subset D_i$;
  \item $y_{i+1}=g_{x_i}^+(y_i)$.
 \end{itemize}
It will then follow by uniqueness of the orbit shadowed by the $\varepsilon$-gpo $\underline v$ that $y_0=x$.

*****

Firstly, note that if $D_0$ is the connected component of $\widehat\Lambda$
containing $R_0$, then $t-s$ is uniquely defined by the intersection of $D_0$ and the 
orbit segment $\vf^{[-\rho,\rho]}(y)$. This intersection is equal to $\mathfrak q_{D_0}(y)$, call it
$x'$, so we want to prove that $x=x'$. By Proposition \ref{Prop-relation-codings}(1), 
there are $\un v,\un w\in \Sigma^\#$ such that $x=\widehat\pi(\un R)=\pi(\un v)$ and $y=\widehat\pi(\un S)=\pi(\un w)$ 
with $R_0\subset Z(v_0)$ and $S_0\subset Z(w_0)$. Let $(n_i)$ such that $R_{n_i}\subset Z(v_i)$
as in Proposition \ref{Prop-relation-codings}(1). 
Applying the assumption $\operatorname{v}(\widehat\sigma_{\widehat r}^t(z))\sim \operatorname{v}(\widehat\sigma_{\widehat r}^t(z'))$ for $t=\widehat r_{n_i}(\un R)$, there is $\ell_i$ such that
$R_{n_i}\sim S_{\ell_i}$. The integer $\ell_i$ is uniquely defined by the inequalities
$\widehat r_{\ell_i}(\un S)\leq \widehat r_{n_i}(\un R)+s<\widehat r_{\ell_i+1}(\un S)$.
We locate the rectangle of $\un w$ of largest index appearing in $\un S$ before index $\ell_i$. In other words,
let $(k_j)$ be a sequence such that $S_{k_j}\subset Z(w_j)$ as in Proposition \ref{Prop-relation-codings}(1),
and make the following definitions: $m_i$ is the largest $k_j$ that is $\leq \ell_i$, and 
$a_i=j$. Notice that $S_{\ell_i}\sim S_{m_i}$. We want to prove a compatibility between the holonomy maps
defined by $\un v$ and $\un w$. Write $\un v=\{v_n\}_{n\in\Z}=\{\Psi_{x_n}^{p^s_n,p^u_n}\}_{n\in\Z}$
and $\un w=\{w_n\}_{n\in\Z}=\{\Psi_{y_n}^{q^s_n,q^u_n}\}_{n\in\Z}$. Let $D_i,E_i$
be the connected components of $\widehat\Lambda$ containing $R_{n_i},S_{m_i}$ respectively.

\medskip
Let us show how to conclude the proof assuming Claim 5. Since the equality can be applied to
$y$ and its time displacements, we have that 
\begin{align*}
[(g_{x_{i-1}^+}\circ\cdots\circ g_{x_0}^+)\circ \mathfrak q_{D_0}](y)&=
[q_{D_i}\circ(\mathfrak q_{E_{i}}\circ\cdots\circ  \mathfrak q_{E_1})](y)\\
(g_{x_{i-1}}^+\circ\cdots\circ g_{x_0}^+)(x')&=q_{D_i}[(\mathfrak q_{E_{i}}\circ\cdots\circ  \mathfrak q_{E_1})(y)].
\end{align*}
Since $(\mathfrak q_{E_{i}}\circ\cdots\circ \mathfrak q_{E_1})(y)=(g_{y_{a_i-1}}^+\circ\cdots g_{y_0}^+)(y)$,
we obtain that 
$(g_{x_{i-1}}^+\circ\cdots\circ g_{x_0}^+)(x')\in q_{D_i}(\Psi_{y_{a_i}}(R[q^s_{a_i}\wedge q^u_{a_i}]))\subset 
\Psi_{x_n}(R[64(p^s_i\wedge p^u_i)])$, where the last inclusion follows from Claim 3 with $c=1$
applied to the triple $(R_{n_i},S_{\ell_i},S_{m_i})$.
Applying a similar reasoning for $i\leq 0$,
it follows from Theorem \ref{Thm-stable-manifolds}(1) and Proposition \ref{Prop-shadowing} that
$x=x'$.

}

\section{Homoclinic classes of measures}\label{sec.homoclinic}
In this final section, we prove Theorem \ref{thm.homoclinic} stated in the introduction,
as well as Corollary~\ref{cor.local-uniq}.

\subsection{The homoclinic relation}
For any hyperbolic measure $\mu$ and $\mu$--a.e. $x$,
the stable set $W^s(x)$ of the orbit of $x$ is
the set of points $y$ such that there exists an increasing homeomorphism
$h\colon \mathbb{R}\to \mathbb{R}$ satisfying $d(\varphi^t(x),\varphi^{h(t)}(y))\to 0$
as $t\to +\infty$. This is an injectively immersed submanifold which is tangent to $E^s_x\oplus X(x)$ 
and invariant under the flow. 
We define similarly the unstable manifold $W^u(x)$ by considering past orbits.

\medskip
\noindent
{\sc Homoclinic relation of measures:} 
We say that two ergodic hyperbolic measures $\mu,\nu$ are
\emph{homoclinically related} if for $\mu$--a.e. $x$ and $\nu$--a.e. $y$
there exist transverse intersections $W^s(x)\pitchfork W^u(y)\ne\emptyset$
and  $W^u(x)\pitchfork W^s(y)\ne\emptyset$, i.e., points $z_1\in W^s(x)\cap W^u(y)$
and $z_2\in W^u(x)\cap W^s(y)$ satisfying
$T_{z_1}M=T_{z_1}W^s(x)+T_{z_1}W^u(y)$
and $T_{z_2}M=T_{z_2}W^u(x)+T_{z_2}W^s(y)$.

\medskip
Note that the invariance of the stable and unstable manifolds makes this notion slightly
\emph{simpler} than it is for diffeomorphisms.
Since any hyperbolic periodic orbit supports a (unique) ergodic measure,
the above homoclinic relation is also defined between
hyperbolic periodic orbits, in which case it coincides with the classical notion,
see, e.g., \cite{Newhouse-Lectures-dynamical-systems}.

\begin{proposition}
The homoclinic relation is an equivalence relation among ergodic hyperbolic measures.
\end{proposition}

\begin{proof}
The only property that is not obvious is the transitivity of the relation.
Its proof uses the following standard lemma. 
\medskip

\noindent
{\bf Inclination lemma.}
\emph{For any hyperbolic measure $\mu$,
there is a set $Y\subset M$ of full $\mu$--measure satisfying the following:
if $x\in Y$, $D\subset W^u(x)$ is a two-dimensional disc 
and $\Delta$ is a two-dimensional disc tangent to $X$ having a transverse intersection point with $W^s(x)$,
then there are discs $\Delta_k\subset \varphi_{(k,+\infty)}(\Delta)$ which converge to $D$ in the $C^1$ topology.}

\begin{figure}
\includegraphics[width=0.9\linewidth]{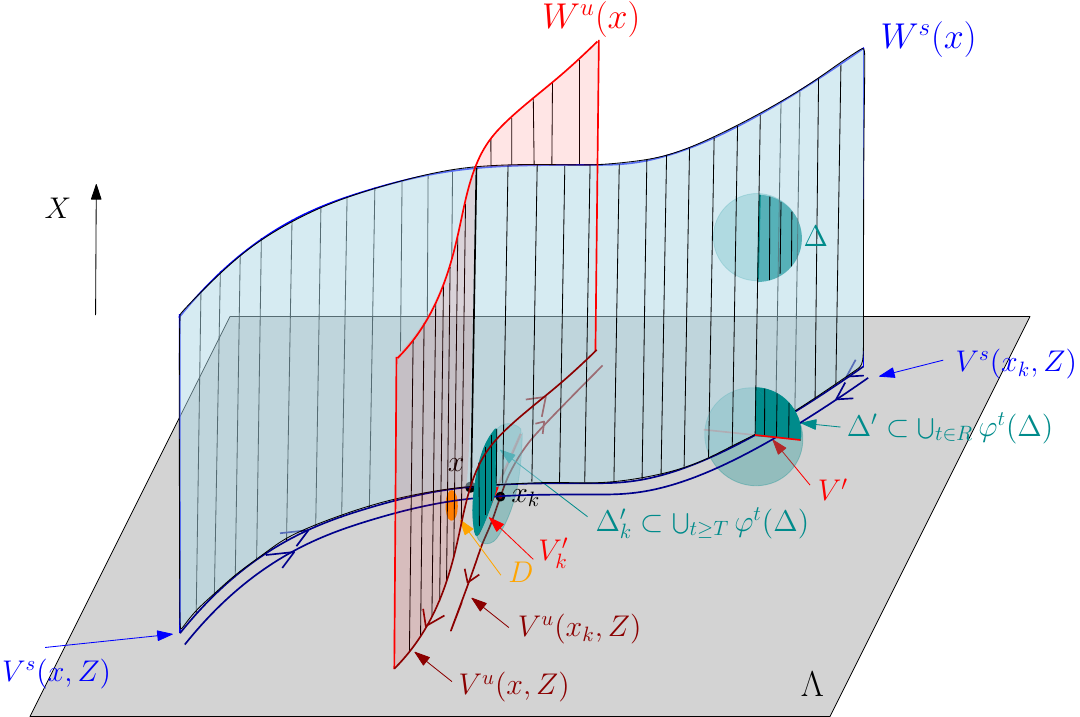}
\caption{The objects in the proof of the inclination lemma. }\label{fig-inclination}
\end{figure}
\begin{proof}[Sketch of the proof]
Taking $\chi>0$ small, the measure $\mu$ is $\chi$--hyperbolic and the constructions
done in the other sections apply.
Consequently one may replace $x$ by an iterate in the section~$\Lambda$
and assume that it is the projection under $\pi$ of a regular sequence $\un v\in \Sigma^\#$.
We denote $Z=Z(v_0)$ and let $n_k\to +\infty$ such that
$\sigma^{n_k}(\un v)\to \un v$. 

We let $x_k:=\pi(\sigma^{n_k}(\un v))$. We consider the curves $V^s(x_k,Z)\to V^s(x,Z)$, $V^u(x_k,Z)\to V^u(x,Z)$ as in Section~\ref{ss.graph.transform} and especially Theorem~\ref{Thm-stable-manifolds}(5).
The intersections $W^s(x)\cap\Lambda$ and $W^u(x)\cap\Lambda$ contain the stable and unstable 
curves $V^s(x,Z)$ and $V^u(x,Z)$.
See Figure~\ref{fig-inclination} for the various objects.

Now, the orbit of $\Delta$ contains a disk $\Delta'$ transversally  intersecting $V^s(x,Z)\subset\Lambda$. Thus $\Delta'$ transversally intersects $\Lambda$ along some curve $V'$. This curve $V'$ intersects the stable curve $V^s(x,Z)$ transversally inside the section $\Lambda$.

Hence, if $k$ is large enough, then the curve $V^s(x_k,Z)$ also intersects $\Delta'$ transversally 
inside $\Lambda$ along the curve $V'$. 
Now, the graph transform argument in Section~\ref{ss.graph.transform} shows that the images of $V'$ 
(by suitable holonomies of the flow mapping $x_k$ to $x_{k'}$ for $k'\gg k$) 
contain curves $V'_k$ that $C^1$--approximate $V^u(x_{k'},Z)$.

It follows that the orbit of $\Delta$ contains a curve which is arbitrarily $C^1$--close to $V^u(x,Z)$.
By invariance, this orbit contains discs which are arbitrarily $C^1$--close to the arbitrary subset $D\subset W^u(x)$.
\end{proof}

In order to prove the proposition, let us consider three measures $\mu_1,\mu_2,\mu_3$
such that $\mu_1,\mu_2$ are homoclinically related and $\mu_2,\mu_3$ are homoclinically related.
For each measure $\mu_i$, let $x_i$ be a point in the full measure set implied by the homoclinic relation.
In particular, there exist a disc $\Delta\subset W^u(x_1)$ which intersects transversally $W^s(x_2)$
and a disc $D\subset W^u(y_2)$ which intersects transversally $W^s(x_3)$.
By the inclination lemma, the orbit of $\Delta$ contains discs that converge to $D$ for the $C^1$--topology.
This proves that $W^u(x_1)$ has a transverse intersection point with $W^s(x_3)$.
The same argument shows that $W^u(x_3)$ has a transverse intersection with $W^s(x_1)$.
Hence $\mu_1$ and $\mu_3$ are homoclinically related.
\end{proof}

\medskip
\noindent
{\sc Homoclinic classes of measures:}
The equivalence classes for the homoclinic relation on the set of hyperbolic measures
are called \emph{homoclinic classes of measures.}

\subsection{Proof of Theorem \ref{thm.homoclinic}}
The proof follows closely the argument in~\cite[Section 3]{BCS-MME}.
We consider the setting of the Main Theorem and especially a topological Markov flow
$(\widehat \Sigma_{\widehat r},\widehat\sigma_{\widehat r})$ satisfying the properties
stated in Theorem~\ref{t.main}.

We begin by some preliminary lemmas.
The first two correspond to properties (C6), (C7) in~\cite{BCS-MME}.

\begin{lemma}\label{C6}
For any two ergodic measures supported on a common irreducible component of $\widehat \Sigma_{\widehat r}$,
their projections under $\widehat\pi_{\widehat r}$ are hyperbolic ergodic measures that are homoclinically related.
\end{lemma}

\begin{proof}
Let us consider two ergodic measures $\overline\mu$ and $\overline \nu$ on a same irreducible component of $\widehat \Sigma_{\widehat r}$ and their projections $\mu=\overline\mu\circ{\widehat\pi_{\widehat r}}^{-1}$ and
$\nu=\overline\nu\circ{\widehat\pi_{\widehat r}}^{-1}$.
These two measures are obviously ergodic. They are hyperbolic by Theorem~\ref{t.main}\eqref{i.splitting}.

Let $x,y$ be points in full measure sets for $\mu$ and $\nu$ respectively:
they are the projections of points $\overline x,\overline y$ which are in the irreducible component
supporting the measures $\overline \mu,\overline \nu$.
Note that one can replace $x,y,\overline x,\overline y$ by iterates and assume that
$\overline x=(\un R,0)$, $\overline y=(\un S,0)$. Since $\overline x,\overline y$ belong to the same irreducible component,
there exists a finite word $w=w_0w_1\cdots w_\ell$ such that $w_0=R_0$ and $w_\ell=S_0$.
One can thus consider the point $\overline z=(\un T,0)$ such that
$T_{-n}=R_{-n}$ and $T_{\ell+n}=S_{n}$ for any $n\geq 0$ and
$T_n=w_n$ for $1\leq n\leq \ell$.
One deduces from the H\"older-continuity of $\widehat\pi_{\widehat r}$ that the
projection $z=\widehat \pi_{\widehat r}(\overline z)=\widehat \pi(\un T)$ belongs to the intersection between $W^s(x)$ and $W^u(y)$.
In particular $W^s(z)=W^s(x)$, hence using Theorem~\ref{t.main}\eqref{i.splitting}\eqref{i.manifold} we have
$$\limsup_{t\to +\infty} \tfrac{1}{t}\log \|\Phi^t|_{T_z(W^s(x)\cap \Lambda)} \|\leq -\lambda <0.$$
Therefore $N^s_z=T_z(W^s(x)\cap \Lambda)$, and 
similarly $N^u_z=T_z(W^u(x)\cap \Lambda)$. Since $N^s_z\oplus N^u_z=N_z$,
one deduces that the intersection between $W^s(x)$ and $W^u(y)$ at $z$ is transverse.
By the same argument, one finds a transverse intersection between
$W^u(x)$ and $W^s(y)$.
Since the points $x$ and $y$ can be taken in full measure sets for $\mu$ and $\nu$ respectively,
this proves that $\mu$ and $\nu$ are homoclinically related.
\end{proof}

\begin{lemma}\label{C7}
For any $\chi’>0$, the set of ergodic measures on $\widehat \Sigma_{\widehat r}$
whose projection is $\chi’$--hyperbolic is open for the weak--* topology.
\end{lemma}

\begin{proof}
The two Lyapunov exponents of the projection of any measure $\overline\mu$ on $\widehat \Sigma_{\widehat r}$  are obtained by integration of the bounded continuous functions
$\overline x\mapsto \log \|\Phi^t|_{N^s_x}\|$ and $\overline x\mapsto \log \|\Phi^t|_{N^u_x}\|$
($x=\widehat \pi_{\widehat r}(\overline x)$).
Hence they vary continuously with the measure $\overline\mu$ in the weak--* topology.
\end{proof}

The next lemma finds an irreducible component that lifts periodic orbits.

\begin{lemma}
There exists an irreducible component $\widehat\Sigma'_{\widehat r}\subset \widehat \Sigma_{\widehat r}$
to which one can lift all $\chi$--hyperbolic periodic orbits that are homoclinically related to $\mu$.
\end{lemma}

\begin{proof}
Periodic orbits that are homoclinically related to $\mu$ are homoclinically related together.
Hence, given any finite set of such periodic orbits, there exists a transitive $\chi$--hyperbolic set $K$
which contains all of them. By Theorem~\ref{t.main}(8), there is a transitive invariant compact
set $X\subset  \widehat \Sigma_{\widehat r}$ such that $\widehat\pi_{\widehat r}(X)=K$.
In particular, $X$ is contained in an irreducible component of $\widehat \Sigma_{\widehat r}$.

Note that $X\subset\widehat \Sigma_{\widehat r}^\#$, the regular set of
$\widehat\Sigma_{\widehat r}$, since $X$ sees only finitely many vertices.
In particular,  $\widehat \pi_{\widehat r}:X\to K$ is not only onto but finite-to-one and all 
periodic orbits of $K$ lift to periodic orbits of $X$ (though with perhaps larger periods).

Let us enumerate all the $\chi$--hyperbolic periodic orbits $\mathcal{O}_i\sim \mu$, $i=1,2,\dots$.
For each $n$, the set of irreducible components which contains periodic lifts of all
the periodic orbits $\mathcal{O}_i$ with $1\leq i\leq n$ is non-empty (by the previous paragraph),
finite (by the finiteness-to-one property of the coding) and is non-increasing with $n$.
Hence their intersection is nonempty, and any irreducible component which belongs to it satisfies the claim.
\end{proof}

Let $\nu$ be a $\chi$--hyperbolic ergodic measure that is homoclinically related to $\mu$.
By Theorem~\ref{t.main}(2), there exists an ergodic lift $\overline\nu$ of $\nu$ to
$\widehat \Sigma_{\widehat r}$.
Consider a point $ q\in \widehat \Sigma_{\widehat r}$ that is recurrent (such that there exists a sequence of forward iterates
$\widehat \sigma^{k_i}(q)$ which converges to $q$) and generic for $\overline\nu$,
and let $x=\widehat\pi_{\widehat r}(q)$.

The recurrence of $q$ gives rise to a sequence of periodic points $ q^i$ in $\widehat \Sigma_{\widehat r}$
which converge to $q$ (hence are in a same irreducible component)
and whose orbits weak--* converge to $\overline\nu$.
By Lemma~\ref{C7} the projections of these periodic orbits are $\chi$--hyperbolic and
by Lemma~\ref{C6} they are homoclinically related to $\mu$. Therefore there are periodic
orbits $p^i$ in the irreducible component $\widehat\Sigma'_{\widehat r}$
which have the same projections as the periodic orbits $ q^i$.

Let us write $q^i=(\underline R^i,t^i)$ and $p^i=(\underline S^i,s^i)$.
Since $(q^i)$ is converging
and $\widehat \Sigma$ is locally compact,
the sequence $(\underline R^i)$ is relatively compact.
The Bowen property of Theorem~\ref{t.main}(6) implies that
$\operatorname{v}(\widehat \sigma_{\widehat r}^t(q^i))\sim \operatorname{v}(\widehat\sigma_{\widehat r}^t(p^i))
\text{ for all }t\in \mathbb{R}$ so, by the local finiteness of the affiliation,
the sequence $(\underline S^i)$ is relatively compact.
This implies that $(p^i)$ is relatively compact and (up to taking a subsequence)
converges to some $p\in \widehat\Sigma'_{\widehat r}$.
By continuity of the projection, $\widehat\pi_{\widehat r}(p)=\widehat\pi_{\widehat r}(q)=x$.

We claim that $p\in  \widehat{\Sigma'}_{\widehat r}^\#$.
This follows from the fact that $q$ is recurrent and that the Bowen relation is locally finite.
More precisely, there are some vertex $A\in\widehat V$ and integers $m_k,n_k\to\infty$ such that $q_{m_k}=q_{-n_k}=A$. 
In particular, for each $k\ge1$ we have $q^i_{m_k}=q^i_{-n_k}=A$ for all large $i$.
Hence $p^i_{m_k},p^i_{-n_k}$ are related to $A$, and so they belong to the set $\{B\in\widehat V:B\sim A\}$.
Since this latter set is finite, some symbol must repeat
as required and this passes to the limit $p$, proving the claim.

We have proved that $\nu$--almost every point has a lift in $\widehat{\Sigma'}_{\widehat r}^\#$.
The finiteness-to-one property of Theorem~\ref{t.main}(3) and
the same averaging argument used in the proof of Theorem~\ref{t.main}(2) imply
that $\nu$ has a lift in $\widehat\Sigma'_{\widehat r}$. Considering the ergodic decomposition,
we can choose an ergodic lift, as claimed.
Theorem \ref{thm.homoclinic} is now proved.\qed

\subsection{Proof of Corollary~\ref{cor.local-uniq}}
Let $\mathcal H$ be some homoclinic class of hyperbolic ergodic measures.
Let us deduce from Theorem~\ref{thm.homoclinic} that there is at most one
$\nu\in\mathcal H$ such that  $h(\vf,\nu)=\sup\{h(\vf,\mu):\mu\in\mathcal H\}$. 
Let  $\nu,\nu'\in\mathcal H$ be two measures with this property.  They are both hyperbolic, 
hence $\chi$--hyperbolic for some $\chi>0$. For one such fixed parameter $\chi$,
let $\pi_r:\Sigma_r\to M$ be the coding given by the Main Theorem.

By Theorem~\ref{thm.homoclinic}, there is an irreducible component $\Sigma'_r$ of $\Sigma_r$
to which both $\nu$ and $\nu'$ lift. Since the factor map $\pi_r$ preserves the entropy
and since the projection of any ergodic measure on $\Sigma'_r$ is homoclinically related 
to $\nu$ and $\nu'$ by Lemma~\ref{C6},
the two lifts are measures of maximal entropy for  $\Sigma'_r$.  But the measure of maximal entropy of an irreducible component of a topological Markov flow with a H\"older continuous roof function $r$ is unique (see
e.g. \cite[Proof of Theorem 6.2]{Lima-Sarig}). Hence $\nu=\nu'$, which proves Corollary~\ref{cor.local-uniq}.

\appendix

\section{Standard proofs}\label{Appendix-proofs}

\renewcommand\thetheorem{\Alph{section}.\arabic{theorem}}

Remind we are assuming that $\|\nabla X\|\leq 1$, and that this implies two facts:
\begin{enumerate}[$\circ$]
\item Every Lyapunov exponent of $\vf$ has absolute values $\leq 1$, hence we consider $\chi\in (0,1)$.
\item $\|\Phi^t\|\leq e^{2\rho+|t|}$, $\forall t\in\R$, see Section \ref{section-induced}.
\end{enumerate}

\begin{proof}[Proof of Lemma \ref{Lemma-linear-reduction}.]
We begin with some preliminary calculations. Fix $t\in\R$. 
We prove that $e_1$ is an eigenvector of $C(\vf^t(x))^{-1}\circ \Phi^t\circ C(x)$,
and calculate its eigenvalue. By the proof of Proposition \ref{Prop-NUH},
$\Phi^t n^s_x=\pm\|\Phi^t n^s_x\|n^s_{\vf^t(x)}$, therefore
$[\Phi^t\circ C(x)](e_1)=\pm\tfrac{\|\Phi^t n^s_x\|}{s(x)}n^s_{\vf^t(x)}$.
This implies that $[C(\vf^t(x))^{-1}\circ\Phi^t\circ C(x)](e_1)=\pm\|\Phi^t n^s_x\|\tfrac{s(\vf^t(x))}{s(x)}e_1$,
hence $e_1$ is an eigenvector with eigenvalue $A_t(x)=\pm\|\Phi^t n^s_x\|\tfrac{s(\vf^t(x))}{s(x)}$.
Similarly, $e_2$ is an engenvector with eigenvalue $B_t(x)=\pm\|\Phi^t n^u_x\|\tfrac{u(\vf^t(x))}{u(x)}$.
Note that
\begin{align*}
&\,s(x)^2=4e^{4\rho}\int_0^t e^{2\chi t'}\|\Phi^{t'}n^s_x\|^2dt'+4e^{4\rho}\int_t^\infty e^{2\chi t'}\|\Phi^{t'}n^s_x\|^2dt'\\
&=4e^{4\rho}\int_0^t e^{2\chi t'}\|\Phi^{t'}n^s_x\|^2dt'+
e^{2\chi t}\|\Phi^tn^s_x\|^2 s(\vf^t(x))^2\\
\end{align*}
and so
$$
e^{2\chi t}\|\Phi^tn^s_x\|^2\tfrac{s(\vf^t(x))^2}{s(x)^2}=1-\tfrac{4e^{4\rho}}{s(x)^2}\int_0^t e^{2\chi t'}\|\Phi^{t'}n^s_x\|^2dt'.
$$
When $0<t\leq 2\rho$, we have 
$\tfrac{4e^{4\rho}}{s(x)^2}\int_0^t e^{2\chi t'}\|\Phi^{t'}n^s_x\|^2dt'\leq 4\rho e^{16\rho}<5\rho$ 
for $\rho>0$ small enough, therefore
\begin{equation}\label{equation-s}
e^{-4\rho}<e^{\chi t}\|\Phi^tn^s_x\|\tfrac{s(\vf^t(x))}{s(x)}<1.
\end{equation}
Similarly,
\begin{align*}
&\,u(\vf^t(x))^2=4e^{4\rho}\int_0^{t}e^{2\chi t'}\|\Phi^{-t'}n^u_{\vf^t(x)}\|^2dt'+
4e^{4\rho}\int_{t}^\infty e^{2\chi t'}\|\Phi^{-t'}n^u_{\vf^t(x)}\|^2dt'\\
&=4e^{4\rho}\int_0^{t}e^{2\chi t'}\|\Phi^{-t'}n^u_{\vf^t(x)}\|^2dt'+
e^{2\chi t}\|\Phi^tn^u_x\|^{-2}u(x)^2
\end{align*}
since $1=\|\Phi^{-t}\Phi^t n^u_x\|=\|\Phi^t n^u_x\|\cdot\|\Phi^{-t}n^u_{\vf^t(x)}\|$, and so 
\begin{equation}\label{equation-u}
e^{-4\rho}<e^{\chi t}\|\Phi^tn^u_x\|^{-1}\tfrac{u(x)}{u(\vf^t(x))}<1.
\end{equation}
We will use (\ref{equation-s}) and (\ref{equation-u}) to prove (2)--(3).

\medskip
\noindent
(1) In the basis $\{e_1,e_2\}$ of $\R^2$ and the basis $\{n^s_x,(n^s_x)^\perp\}$ of $N_x$,
$C(x)$ takes the form
$\left[\begin{array}{cc}\tfrac{1}{s(x)}& \tfrac{\cos\alpha(x)}{u(x)}\\ 0&\tfrac{\sin\alpha(x)}{u(x)}\end{array}\right]$,
hence $\|C(x)\|_{\rm Frob}^2=\tfrac{1}{s(x)^2}+\tfrac{1}{u(x)^2}\leq 1$. Now observe that the inverse of
$C(x)$ is
$\left[\begin{array}{cc}s(x)& -\tfrac{s(x)\cos\alpha(x)}{\sin\alpha(x)}\\ 0&\tfrac{u(x)}{\sin\alpha(x)}\end{array}\right]$,
hence $\|C(x)^{-1}\|_{\rm Frob}=\tfrac{\sqrt{s(x)^2+u(x)^2}}{|\sin\alpha(x)|}$.

\medskip
\noindent
(2) The first part was already proved, so we concentrate on the second part. Fix $0<t\leq 2\rho$.
By (\ref{equation-s}), $e^{-4\rho}<e^{\chi t}|A_t(x)|<1$ and so
$e^{-8\rho}<|A_t(x)|<e^{-\chi t}$.
Similarly, (\ref{equation-u}) implies that
$e^{-4\rho}<e^{\chi t}|B_t(x)|^{-1}<1$, and so $e^{\chi t}<|B_t(x)|<e^{8\rho}$.

\medskip
\noindent
(3) For $|t|\leq 2\rho$, we have $e^{\chi t}\|\Phi^t n_s^x\|=e^{\pm 6\rho}$,
therefore by (\ref{equation-s}) it follows that $e^{-10\rho}<\tfrac{s(\vf^t(x))}{s(x)}<e^{6\rho}$,
so that $\tfrac{s(\vf^t(x))}{s(x)}=e^{\pm 10\rho}$. Similarly, 
$\tfrac{u(\vf^t(x))}{u(x)}=e^{\pm 10\rho}$.
To estimate $\tfrac{|\sin\alpha(\vf^t(x))|}{|\sin\alpha(x)|}$,
we use the general inequality for an invertible linear transformation $L$:
\begin{equation}\label{gen-ineq-angles}
\frac{1}{\|L\|\|L^{-1}\|}\leq \frac{|\sin\angle(Lv,Lw)|}{|\sin\angle(v,w)|}\leq \|L\|\|L^{-1}\|.
\end{equation}
Apply this to $L=\Phi^t$, $v=n^s_x$, $w=n^u_x$ to get that
$\tfrac{|\sin\alpha(\vf^t(x))|}{|\sin\alpha(x)|}=e^{\pm 8\rho}$. Finally, the above estimates
and part (1) imply that $\tfrac{\|C(\vf^t(x))\|_{\rm Frob}}{\|C(x)\|_{\rm Frob}}=e^{\pm 18\rho}$.
\end{proof}

For the proof of the next theorem we will need some estimates on $Q(x)$.
By Lemma \ref{Lemma-linear-reduction}(3) proved above,
$\tfrac{Q(\vf^t(x))}{Q(x)}=e^{\pm\frac{200\rho}{\beta}}$ for all $x\in\nuh$ and
$|t|\leq 2\rho$. Therefore, if $x\in\Lambda\cap\nuh$ then
$\tfrac{Q(f(x))}{Q(x)}=e^{\pm\frac{200\rho}{\beta}}$.
Hence the following bounds hold for $Q(x)$:
\begin{align*}
&Q(x)\leq \ve^{3/\beta}\text{ and } \|C(x)^{-1}\|Q(x)^{\beta/12}\leq \ve^{1/4}\text{ for all }x\in\nuh,\\
&Q(x)^{\beta/2}\leq e^{100\rho}Q(f(x))^{\beta/2}\text{ for all }x\in\Lambda\cap\nuh.
\end{align*}

\begin{proof}[Proof of Theorem \ref{Thm-non-linear-Pesin}.]
Recall that $B_x=B(x,2\mathfrak r)$.
If $\ve>0$ is small enough then Lemma \ref{Lemma-Pesin-chart}(1)
implies 
$$\Psi_x(R[10Q(x)])\subset B(x,40Q(x))\subset B_x,
$$ and in this ball (Exp1)--(Exp4) are valid.
We first show that $f_x^+:R[10Q(x)]\to\R^2$ is well-defined.
Since $C(x)$ is a contraction, we have $C(x)R[10Q(x)]\subset B_x[20Q(x)]$.
Since $C(f(x))^{-1}$ is globally defined, it is enough to show that
$$
(g_x^+\circ\exp{x})(B_x[20Q(x)])\subset \exp{f(x)}(B_{f(x)}[2\mathfrak r]).
$$
For small $\ve>0$ we have:
\begin{enumerate}[$\circ$]
\item $20Q(x)<2\mathfrak r$, hence $\exp{x}$ is well-defined on $B_x[20Q(x)]$. By (Exp2),
$\exp{x}$ maps $B_x[20Q(x)]$ diffeomorphically into $B(x,40Q(x))$.
\item $40Q(x)<2\mathfrak r\Rightarrow B(x,40Q(x))\subset B_x$,
hence Lemma \ref{Lemma-map-g} implies that $g_x^+$ maps the ball $B(x,40Q(x))$
diffeomorphically into $B(f(x),80Q(x))$.
\item $80Q(x)<2\mathfrak r\Rightarrow B(f(x),80Q(x))\subset
B_{f(x)}$. By condition (Exp2), $\exp{f(x)}^{-1}$ maps $B(f(x),80Q(x))$
diffeomorphically onto its image.
\end{enumerate}
The conclusion is that $f_x^+:R[10Q(x)]\to\R^2$ is a diffeomorphism onto its image.

\medskip
Now we check (1)--(2). Using the equalities $d(\Psi_x)_0=C(x)$, $d(\Psi_{f(x)})_0=C(f(x))$
and Lemma \ref{Lemma-map-g}, we get that
$d(f_x^+)_0=C(f(x))^{-1}\circ \Phi^{r_\Lambda(x)}\circ C(x)$.
By Lemma \ref{Lemma-linear-reduction}(2),
$d(f_x^+)_0=\left[\begin{array}{cc}A & 0 \\ 0 & B\end{array}\right]$ with $e^{-4\rho}<|A|<e^{-\chi r_\Lambda(x)}$ and
$e^{\chi r_\Lambda(x)}<|B|<e^{4\rho}$. This proves part (1). Items (a)--(b) of part (2) are automatic,
hence we focus on (c).

\medskip
\noindent
{\sc Claim:} $\|d(f_x^+)_{v_1}-d(f_x^+)_{v_2}\|\leq \tfrac{\ve}{3}\|v_1-v_2\|^{\beta/2}$
for all $v_1,v_2\in R[10Q(x)]$.

\medskip
Before proving the claim, we show how to conclude (c).
If $\ve>0$ is small enough then $R[10Q(x)]\subset B_x[1]$. Applying the claim with $v_2=0$,
we get $\|dH_v\|\leq \frac{\ve}{3}\|v\|^{\beta/2}<\tfrac{\ve}{3}$. By the mean value inequality,
$\|H(v)\|\leq \tfrac{\ve}{3}\|v\|<\tfrac{\ve}{3}$, hence $\|H\|_{C^{1+\frac{\beta}{2}}}<\ve$.

\begin{proof}[Proof of the claim.]
Let us choose $L>\Hol{\beta}(dg_x^+)$.
For $i=1,2$, write $w_i=C(x)v_i$ and let
$$
A_i= \widetilde{d(\exp{f(x)}^{-1})_{(g_x^+\circ \exp{x})(w_i)}}\,,\
B_i=\widetilde{d(g_x^+)_{\exp{x}(w_i)}}\,,\ C_i=\widetilde{d(\exp{x})_{w_i}}.
$$
We first estimate $\|A_1 B_1 C_1-A_2 B_2 C_2\|$.
\begin{enumerate}[$\circ$]
\item By (Exp2), $\|A_i\|\leq 2$. By (Exp2), (Exp3) and Lemma \ref{Lemma-map-g}:
\begin{align*}
\|A_1-A_2\|\leq \mathfrak Kd((g_x^+\circ \exp{x})(w_1),(g_x^+\circ \exp{x})(w_2))
\leq 4\mathfrak K\|w_1-w_2\|.
\end{align*}
\item By Lemma \ref{Lemma-map-g}, $\|B_i\|\leq 2$.
By (Exp2) and Lemma \ref{Lemma-map-g}:
$$
\|B_1-B_2\|\leq Ld(\exp{x}(w_1),\exp{x}(w_2))^{\beta}\leq 2L\|w_1-w_2\|^\beta.
$$
\item By (Exp2), $\|C_i\|\leq 2$. By (Exp3), $\|C_1-C_2\|\leq \mathfrak K\|w_1-w_2\|$.
\end{enumerate}
Applying some triangle inequalities, we get that
$$
\|A_1 B_1 C_1-A_2 B_2 C_2\|\leq 24\mathfrak K L\|w_1-w_2\|^\beta
\leq 24\mathfrak K L\|v_1-v_2\|^\beta.
$$
Now we estimate $\|d(f_x^+)_{v_1}-d(f_x^+)_{v_2}\|$:
\begin{align*}
&\,\|d(f_x^+)_{v_1}-d(f_x^+)_{v_2}\|\leq \|C(f(x))^{-1}\| \|A_1 B_1 C_1-A_2 B_2 C_2\| \|C(x)\|\\
&\leq 24\mathfrak K L\|C(f(x))^{-1}\|\|v_1-v_2\|^\beta.
\end{align*}
Using estimate (\ref{ratio-Q}) and that $\|v_1-v_2\|<40Q(x)$, 
we conclude that for $\ve>0$ small:
\begin{align*}
&\,24\mathfrak K L\|C(f(x))^{-1}\|\|v_1-v_2\|^{\beta/2}\leq 
200\mathfrak K L\|C(f(x))^{-1}\| Q(x)^{\beta/2}\\
&\leq 200\mathfrak K L e^{125\rho}\|C(f(x))^{-1}\| Q(f(x))^{\beta/2}
\leq 200\mathfrak K L e^{125\rho}\ve^{3/2} \|C(f(x))^{-1}\|^{-5}\\
&\leq 200\mathfrak K L e^{125\rho}\ve^{3/2}<\ve.
\end{align*}
Hence the claim is proved.
\end{proof}
This completes the proof of the theorem.
\end{proof}

\begin{remark}\label{rmk-holonomy}
The sole property of $g_x^+$ used in the above proof is Lemma \ref{Lemma-map-g}.
Since any holonomy map $\mathfrak q_{D_j}$ also satisfies this lemma, we conclude that $\mathfrak q_{D_j}$
satisfies a statement analogous to Theorem \ref{Thm-non-linear-Pesin}. We will use this
observation in the proof of Proposition \ref{Prop-overlapping-charts}.
\end{remark}

\begin{proof}[Proof of Proposition \ref{Lemma-overlap}.]
Write $C_i=\widetilde{C(x_i)}:\R^2\to T_{x_1}\Lambda$.
By assumption, $d(x_1,x_2)+\|C_1-C_2\|<(\eta_1\eta_2)^4$.
Note that $\Psi_{x_i}=\exp{x_i}\circ P_{x_1,x_i}\circ C_i$.

\medskip
\noindent
(1) We prove the estimate for $s$ (the calculation for $u$ is similar).
Since $\ve>0$ is small, it is enough to prove that $\left|\tfrac{s(x_1)}{s(x_2)}-1\right|<\ve^{3/\beta}(\eta_1\eta_2)^3$.
We have $s(x_i)^{-1}=\|C(x_i)e_1\|=\|C_ie_1\|$, hence
$|s(x_1)^{-1}-s(x_2)^{-1}|=|\|C_1e_1\|-\|C_2e_1\||\leq \|C_1-C_2\|<(\eta_1\eta_2)^4$.
Also
$s(x_1)=\|C(x_1)e_1\|^{-1}\leq \|C(x_1)^{-1}\|\leq\tfrac{\ve^{3/\beta}}{Q(x_1)}
<\tfrac{\ve^{3/\beta}}{\eta_1\eta_2}$,
therefore
$$
\left|\tfrac{s(x_1)}{s(x_2)}-1\right|=s(x_1)|s(x_1)^{-1}-s(x_2)^{-1}|<\ve^{3/\beta}(\eta_1\eta_2)^3.
$$

\medskip
\noindent
(2) Apply (\ref{gen-ineq-angles}) to $L=C_1C_2^{-1}$, $v=C_2e_1$, $w=C_2e_2$ to get that
$$
\frac{1}{\|C_1C_2^{-1}\|\|C_2C_1^{-1}\|}\leq \frac{\sin\alpha(x_1)}{\sin\alpha(x_2)}\leq \|C_1C_2^{-1}\|\|C_2C_1^{-1}\|.
$$
We have $\|C_1C_2^{-1}-{\rm Id}\|\leq\|C_1-C_2\|\|C_2^{-1}\|<\ve^{3/\beta}(\eta_1\eta_2)^3$,
and by symmetry $\|C_2C_1^{-1}-{\rm Id}\|<\ve^{3/\beta}(\eta_1\eta_2)^3$, therefore
$\|C_1C_2^{-1}\|\|C_2C_1^{-1}\|<[1+\ve^{3/\beta}(\eta_1\eta_2)^3]^2<e^{2\ve^{3/\beta}(\eta_1\eta_2)^3}<e^{(\eta_1\eta_2)^3}$.
The left hand side estimate is proved similarly.

\medskip
\noindent
(3) We prove that $\Psi_{x_1}(R[e^{-2\ve}\eta_1])\subset \Psi_{x_2}(R[\eta_2])$.
If $v\in R[e^{-2\ve}\eta_1]$ then
$\|C(x_1)v\|\leq \sqrt{2}e^{-2\ve}\eta_1<2\mathfrak r$,
hence by (Exp1):
$$\Sas(C(x_1)v,C(x_2)v)\leq 2(d(x_1,x_2)+\|C_1v-C_2v\|)\leq 2(\eta_1\eta_2)^4.$$
By (Exp2), 
$d(\Psi_{x_1}(v),\Psi_{x_2}(v))\leq 4(\eta_1\eta_2)^4\Rightarrow\Psi_{x_1}(v)\in B(\Psi_{x_2}(v),4(\eta_1\eta_2)^4)$.
By Lemma \ref{Lemma-Pesin-chart}(1),
$B(\Psi_{x_2}(v),4(\eta_1\eta_2)^4)\subset \Psi_{x_2}(B)$ where
$B\subset \R^2$ is the ball with center $v$ and radius $8\|C_2^{-1}\|(\eta_1\eta_2)^4$,
hence it is enough to show that $B\subset R[\eta_2]$. If $w\in B$ then
$\|w\|_\infty\leq \|v\|_\infty+8\|C_2^{-1}\|(\eta_1\eta_2)^4\leq (e^{-\ve}+8\ve^{3/\beta})\eta_2<\eta_2$
for $\ve>0$ small enough.

\medskip
\noindent
(4) The proof that $\Psi_{x_2}^{-1}\circ \Psi_{x_1}$ is well-defined in
$R[\mathfrak r]$ is similar to the proof of (3). The only difference is in the last calculation:
if $\ve>0$ is small enough then for $w\in B$ it holds
\begin{align*}
\|w\|\leq \|v\|+8\|C_2^{-1}\|(\eta_1\eta_2)^4\leq \sqrt{2}\mathfrak r+8(\eta_1\eta_2)^3
\leq [\sqrt{2}+8\ve^{3/\beta}]\mathfrak r<2\mathfrak r,
\end{align*}
therefore $B$ is contained in the ball of $\R^2$ with center 0 and radius $2\mathfrak r$,
and in this latter ball $\Psi_{x_2}$
is a diffeomorphism onto its image.
Now:
\begin{align*}
&\,\Psi_{x_2}^{-1}\circ \Psi_{x_1}-{\rm Id}=C(x_2)^{-1}\circ\exp{x_2}^{-1}\circ\exp{x_1}\circ C(x_1)-{\rm Id}\\
&=[C_2^{-1}\circ P_{x_2,x_1}]\circ[\exp{x_2}^{-1}\circ\exp{x_1}-
P_{x_1,x_2}]\circ [P_{x_1,x_1}\circ C_1]+C_2^{-1}(C_1-C_2)\\
&=[C_2^{-1}\circ P_{x_2,x_1}]\circ[\exp{x_2}^{-1}-P_{x_1,x_2}\circ\exp{x_1}^{-1}]\circ\Psi_{x_1}+C_2^{-1}(C_1-C_2).
\end{align*}
We calculate the $C^2$ norm of $[\exp{x_2}^{-1}-P_{x_1,x_2}\circ\exp{x_1}^{-1}]\circ\Psi_{x_1}$
in the domain $R[\mathfrak r]$.
By Lemma \ref{Lemma-Pesin-chart}, $\|d\Psi_{x_1}\|_{C^0}\leq 2$ and $\Lip(d\Psi_{x_1})\leq\mathfrak K$.
Call $\Theta:=\exp{x_2}^{-1}-P_{x_1,x_2}\circ\exp{x_1}^{-1}$. For $\ve>0$ small enough, inside $B_{x_1}$ we have:
\begin{enumerate}[$\circ$]
\item By (Exp2),
$\|\Theta(y)\|\leq \Sas(\exp{x_2}^{-1}(y),\exp{x_1}^{-1}(y))\leq 2d(x_1,x_2)\leq 2\ve^{6/\beta}(\eta_1\eta_2)^3$
thus $\|\Theta\circ \Psi_{x_1}\|_{C^0}<\ve^{2/\beta}(\eta_1\eta_2)^3$.
\item By (Exp3), $\|d\Theta_y\|=\|\tau(x_2,y)-\tau(x_1,y)\|\leq \mathfrak Kd(x_1,x_2)
<\ve^{3/\beta}(\eta_1\eta_2)^3$.
Hence $\|d\Theta\|_{C^0}<\ve^{3/\beta}(\eta_1\eta_2)^3$ and
$\|d(\Theta\circ\Psi_{x_1})\|_{C^0}\leq 2\ve^{3/\beta}(\eta_1\eta_2)^3<\ve^{2/\beta}(\eta_1\eta_2)^3$.
\item By (Exp4),
\begin{align*}
&\,\|\widetilde{d\Theta_y}-\widetilde{d\Theta_z}\|=\|[\tau(x_2,y)-\tau(x_1,y)]-[\tau(x_2,z)-\tau(x_1,z)]\|\\
&\leq \mathfrak Kd(x_1,x_2)d(y,z)
\end{align*}
hence ${\rm Lip}(d\Theta)\leq \mathfrak Kd(x_1,x_2)$.
\item Using that $\Lip(d(\Theta_1\circ\Theta_2))\leq \|d\Theta_1\|_{C^0}\Lip(d\Theta_2)+\Lip(d\Theta_1)\|d\Theta_2\|_{C^0}^2$,
we get that
\begin{align*}
&\ \Lip[d(\Theta\circ\Psi_{x_1})]\leq \|d\Theta\|_{C^0}\Lip(d\Psi_{x_1})+
{\rm Lip}(d\Theta)\|d\Psi_{x_1}\|_{C^0}^2\\
&<\mathfrak K\ve^{3/\beta}(\eta_1\eta_2)^3+4\mathfrak K(\eta_1\eta_2)^4<
5\mathfrak K\ve^{3/\beta}(\eta_1\eta_2)^3<
\ve^{2/\beta}(\eta_1\eta_2)^3.
\end{align*}
\end{enumerate}
This implies that $\|\Theta\circ\Psi_{x_1}\|_{C^2}<3\ve^{2/\beta}(\eta_1\eta_2)^3$, hence
$$
\|C_2^{-1}\circ P_{x_2,x}\circ\Theta\circ\Psi_{x_1}\|_{C^2}\leq \|C_2^{-1}\|3\ve^{2/\beta}(\eta_1\eta_2)^3
\leq 3\ve^{2/\beta}(\eta_1\eta_2)^2.
$$
Thus
$\|\Psi_{x_2}^{-1}\circ \Psi_{x_1}-{\rm Id}\|_2\leq
3\ve^{2/\beta}(\eta_1\eta_2)^2+\|C_2^{-1}\|(\eta_1\eta_2)^4<
3\ve^{2/\beta}(\eta_1\eta_2)^2+\ve^{3/\beta}(\eta_1\eta_2)^3$ $<4\ve^{2/\beta}(\eta_1\eta_2)^2<\ve(\eta_1\eta_2)^2$.
\end{proof}

\begin{proof}[Proof of Proposition \ref{Prop-overlapping-charts}]
Let $z\in Z$, $z'=\vf^t(z)\in Z'$ with $|t|\leq2\rho$, and assume that $Z'\subset D'$.
Define $\Upsilon:=\Psi_y^{-1}\circ \mathfrak q_{D'}\circ\Psi_x$. We will write $\Upsilon$
as a small perturbation of $\pm{\rm Id}$. For ease of notation, write $p:=p^s\wedge p^u$
and $q:=q^s\wedge q^u$.
Start noting that, by Lemma \ref{Lemma-q}, Proposition \ref{Prop-Z-par}(1), and Theorem \ref{Thm-inverse}(5),
$$
\tfrac{p}{q}=\tfrac{p}{p^s(z)\wedge p^u(z)}\cdot\tfrac{p^s(z)\wedge p^u(z)}{q(z)}\cdot\tfrac{q(z)}{q(z')}\cdot
\tfrac{q(z')}{p^s(z')\wedge p^u(z')}\cdot \tfrac{p^s(z')\wedge p^u(z')}{q}=e^{\pm[O(\sqrt[3]{\ve})+O(\rho)]}.$$
We have
$\Upsilon=
(\Psi_y^{-1}\circ \Psi_{z'})\circ(\Psi_{z'}^{-1}\circ\mathfrak q_{D'}\circ\Psi_z)\circ(\Psi_z^{-1}\circ \Psi_x)$.
By Theorem \ref{Thm-inverse}(6), we have:
\begin{enumerate}[$\circ$]
\item $(\Psi_z^{-1}\circ \Psi_x)=(-1)^{\sigma_1}{\rm Id}+\Delta_1(v)$ where $\sigma_1\in\{0,1\}$,
$\|\Delta_1(0)\|<50^{-1}p$, and $\|d\Delta_1\|_{C^0}<\sqrt[3]{\ve}$ on $R[10Q(x)]$.
\item $(\Psi_y^{-1}\circ\Psi_{z'})=(-1)^{\sigma_2}{\rm Id}+\Delta_2(v)$ where $\sigma_2\in\{0,1\}$,
$\|\Delta_2(0)\|<50^{-1}q$, and $\|d\Delta_2\|_{C^0}<\sqrt[3]{\ve}$ on $R[10Q(z')]$.
\end{enumerate}
Assume, for simplicity, that $\sigma_1=\sigma_2=0$.
Applying the same method of proof of Theorem \ref{Thm-non-linear-Pesin} to
$\mathfrak q_{D'}$ (see Remark \ref{rmk-holonomy}), we conclude that 
$\Psi_{z'}^{-1}\circ\mathfrak q_{D'}\circ\Psi_z$ can be written in the form
$(v_1,v_2)\mapsto \begin{bmatrix} A & 0 \\ 0 & B \end{bmatrix} + H$,
where $A,B,H$ satisfy Theorem \ref{Thm-non-linear-Pesin}(2) with $\rho$ changed to $2\rho$.
Assuming for simplicity that $\vf$ preserves orientation\footnote{If not, we can apply an argument similar
to \cite{Ben-Ovadia-high-dimension}.}, we have $AB>0$, hence we can rewrite
$\Psi_{z'}^{-1}\circ\mathfrak q_{D'}\circ\Psi_z=\pm[{\rm Id}+\Delta_3(v)]$ on $R[10Q(z)]$,
where $\pm$ is the sign of $A,B$. Clearly $\Delta_3(0)=0$. If $A,B>0$ then
$d(\Delta_3)_0=\left[
\begin{array}{cc}
A-1 & 0\\
0 & B-1
\end{array}\right]$, thus
$\|d(\Delta_3)_0\|=|B-1|<e^{8\rho}-1$. The same estimate holds if $A,B<0$.
Using Theorem \ref{Thm-non-linear-Pesin}(2)(c), we get that
$\|d\Delta_3\|_{C^0}<e^{8\rho}-1+O(\ve)$.
Therefore $\Upsilon=\pm({\rm Id}+\Delta_1)({\rm Id}+\Delta_3)({\rm Id}+\Delta_2)$ where:
\begin{enumerate}[$\circ$]
\item $\|\Delta_1(0)\|<50^{-1}p$ and $\|d\Delta_1\|_{C^0}=O(\ve^{1/3})$.
\item $\Delta_3(0)=0$ and $\|d\Delta_3\|_{C^0}<e^{8\rho}-1+O(\ve)$.
\item $\|\Delta_2(0)\|<50^{-1}q$ and $\|d\Delta_2\|_{C^0}=O(\ve^{1/3})$.
\end{enumerate}
So $\Upsilon=\pm({\rm Id}+\Delta)$,
$\Delta=\Delta_1+\Delta_2+\Delta_3+\Delta_1\Delta_2+\Delta_1\Delta_3+\Delta_3\Delta_2+
\Delta_1\Delta_3\Delta_2$. We have:
\begin{enumerate}[$\circ$]
\item $\|d\Delta\|_{C^0}\leq [1+O(\ve^{1/3})]^2[1+e^{8\rho}-1+O(\ve)]-1=
[1+O(\ve^{1/3})][e^{8\rho}+O(\ve)]-1=e^{8\rho}-1+O(\ve^{1/3})$.
Hence $\|d\Upsilon\|_{C^0}\leq e^{8\rho}+O(\ve^{1/3})$.
\item Let $a_i:=\Delta_i(0)$, $b_i:=\Lip(\Delta_i)$.
Using that $\|\Delta_i\Delta_j(0)\|\leq \|\Delta_i(0)\|+\Lip(\Delta_i)\|\Delta_j(0)\|$, direct
calculations give
$\|\Delta(0)\|\leq 4a_1+(1+b_1)(2a_2+a_3+a_3b_2)=4a_1+2a_2(1+b_1)$.
Since $p\leq e^{\pm[O(\sqrt[3]{\ve})+O(\rho)]}q=[1+O(\ve^{1/3})+O(\rho)]q$, it follows that
\begin{align*}
&\,\|\Delta(0)\|\leq \tfrac{2}{25}p+\tfrac{1}{25}q[1+O(\ve^{1/3})]\leq\tfrac{3}{25}[1+O(\ve^{1/3})+O(\rho)]q.
\end{align*}
Hence $\|\Upsilon(0)\|\leq\tfrac{3}{25}[1+O(\ve^{1/3})+O(\rho)]q$.
\end{enumerate}
We now proceed to prove the proposition.

\medskip
\noindent
(1) We have $\Upsilon(R[\tfrac{1}{2}p])\subset \Upsilon(B_0[\tfrac{1}{\sqrt{2}}p])
\subset B_{\Upsilon(0)}[\tfrac{1}{\sqrt{2}}\Lip(\Upsilon)p]\subset B$,
where $B\subset\R^2$ is the ball with center $0$ and radius
$\|\Upsilon(0)\|+\tfrac{1}{\sqrt{2}}\Lip(\Upsilon)p$. By the estimates obtained above,
\begin{align*}
&\,\|\Upsilon(0)\|+\tfrac{1}{\sqrt{2}}\Lip(\Upsilon)p\leq\tfrac{3}{25}[1+O(\ve^{1/3})+O(\rho)]q+
\tfrac{1}{\sqrt{2}}\left[e^{8\rho}+O(\ve^{1/3})\right]p\\
&\leq\tfrac{3}{25}[1+O(\ve^{1/3})+O(\rho)]q+\tfrac{1}{\sqrt{2}}\left[1+O(\ve^{1/3})+
O(\rho)\right]\left[e^{8\rho}+O(\ve^{1/3})\right]q\\
&<\left[\tfrac{3}{25}+\tfrac{1}{\sqrt{2}}\right]\left[e^{8\rho}+O(\ve^{1/3})\right]q.
\end{align*}
Since $\tfrac{3}{25}+\tfrac{1}{\sqrt{2}}<1$, for $0<\ve\ll\rho\ll1 $ we get that $B\subset B_0[q]\subset R[q]$.

\medskip
\noindent
(2) Fix $z\in Z$ such that $z'=\mathfrak q_{D'}(z)\in Z'$. We will show that
$\mathfrak q_{D'}[W^{s}(z,Z)]\subset V^{s}(z',Z')$ (the other case is identical).
Write $W=\mathfrak q_{D'}[W^s(z,Z)]$ and $V=V^s(z',Z')$. Our goal is to show that $W\subset V$.
Let $\un v=\{\Psi_{x_n}^{p^s_n,p^u_n}\}_{n\in\Z},\un w=\{\Psi_{y_n}^{q^s_n,q^u_n}\}_{n\in\Z}$
such that $z=\pi(\un v)$ and $z'=\pi(\un w)$. For $n\geq 0$, let
$G^n_{\un v}=g_{x_{n-1}}^+\circ\cdots\circ g_{x_0}^+ \text{ and } G^n_{\un w}=g_{y_{n-1}}^+\circ\cdots\circ g_{y_0}^+$.
By Theorem \ref{Thm-stable-manifolds}(1), we need to show that
$G^n_{\un w}[W]\subset\Psi_{y_n}(R[10Q(y_n)])$ for all $n\geq 0$.

Fix $n\geq 0$. If $z'=\vf^t(z)$, $|t|\leq2\rho$, then there is a unique $m\geq 0$ such that
$r_m(\un v)<r_n(\un w)+t\leq r_{m+1}(\un v)$. Let $D_k$ be the disc containing
$\vf^{r_n(\un w)}(z')$. We claim that $G^n_{\un w}\circ{\mathfrak q}_{D'}=\mathfrak q_{D_k}\circ G^m_{\un v}$
wherever these maps are well-defined. To see this, firstly note that these maps are both
of the form $\vf^{\tau}$ for some continuous function $\tau$. Secondly, we claim that they coincide
at $z$. Indeed, $(G^n_{\un w}\circ{\mathfrak q}_{D'})(z)=G^n_{\un w}(z')=\vf^{r_n(\un w)}(z')$ and
$(\mathfrak q_{D_k}\circ G^m_{\un v})(z)=\mathfrak q_{D_k}[\vf^{r_m(\un v)}(z)]$.
Writing $\vf^{r_n(\un w)}(z')=z_n'$ and $\vf^{r_m(\un v)}(z)=z_m$, we have $z_n'=\vf^{t'}(z_m)$
for $t'=r_n(\un w)+t-r_m(\un v) \in (0,\rho]$, therefore $\mathfrak q_{D_k}(z_m)=z_n'$. Hence
$G^n_{\un w}[W]=(G^n_{\un w}\circ{\mathfrak q}_{D'})[W^s(z,Z)]=(\mathfrak q_{D_k}\circ G^m_{\un v})[W^s(z,Z)]
\subset \mathfrak q_{D_k}[W^s(\vf^{r_m(\un v)}(z),Z(v_m))]$, where we used Proposition \ref{Prop-Z}(4)
in the last inclusion. Since
$W^s(\vf^{r_m(\un v)}(z),Z(v_m))\subset \Psi_{x_m}(R[10^{-2}(p^s_m\wedge p^u_m)])$,
part (1) gives that $\mathfrak q_{D_k}[W^s(\vf^{r_m(\un v)}(z),Z(v_m))]\subset \Psi_{y_n}(R[q^s_n\wedge q^u_n])$,
and this last set is contained in $\Psi_{y_n}(R[10Q(y_n)])$.

\medskip
\noindent
(3) When $M$ is compact and $f$ is a $C^{1+\beta}$ diffeomorphism, the proof 
that $[z,z']_{Z'}$ is well-defined is \cite[Lemma 10.8]{Sarig-JAMS},
and the proof uses that the change of coordinates from one Pesin chart to the other is so close to the identity
that the representing function of an $s$--admissible manifold satisfies properties similar to
(AM1)--(AM3), with the constants $10^{-3},\tfrac{1}{2}$ slightly increased. We can apply the same method,
since we showed above that our change of coordinates $\Upsilon$ is a small perturbation of the identity.
The details can be easily carried out with the estimates we already obtained above. Similarly,
$[z,z']_Z$ is well-defined. It remains to prove that $[z,z']_Z=\mathfrak q_D([z,z']_{Z'})$.
To see this, observe that the composition $\mathfrak q_D\circ\mathfrak q_{D'}$ is the identity
where it is defined, hence
$$
\mathfrak q_D([z,z']_{Z'})=\mathfrak q_D(\mathfrak q_{D'}[V^s(z,Z)]\cap V^u(z',Z'))
= V^s(z,Z)\cap \mathfrak q_D[V^u(z',Z')]=[z,z']_Z.
$$
This completes the proof of the proposition.
\end{proof}

\begin{proof}[Proof of Proposition \ref{Prop-overlapping-charts-2}]
Let $Z, Z',Z''$ such that $Z\cap \vf^{[-2\rho,2\rho]}Z'\neq\emptyset$, $Z\cap \vf^{[-2\rho,2\rho]}Z''\neq\emptyset$,
and assume that $z'\in Z'$ such that $\vf^t(z')\in Z''$ for some $|t|\leq 2\rho$. 
We are asked to show that for every $z\in Z$ it holds
$$
[z,z']_Z=[z,\vf^t(z')]_Z.
$$
The idea is the following:
\begin{enumerate}[$\circ$] 
\item $V^u(z',Z')$ and $V^u(\vf^t(z'),Z'')$ coincide in a small window.
\item If $Z=Z(\Psi_x^{p^s,p^u})$ and $G$ is the representing function of $V^s(z,Z)$,
then $[z,z']_Z=\Psi_x(s,G(s))$ for some $|s|\leq \tfrac{1}{3}(p^s\wedge p^u)$.
\end{enumerate}
The precise statements are in the next claims. Write $Z'=Z(\Psi_y^{q^s,q^u})$,
$p=p^s\wedge p^u$ and $q=q^s\wedge q^u$, and
let $D$ be the connected components of $\widehat\Lambda$ with $Z\subset D$.

\medskip
\noindent
{\sc Claim 1:} $\mathfrak q_D[V^u(z',Z')\cap \Psi_y(R[\tfrac{1}{2}q])]$ contains
$\Psi_x\{(H(t),t):|t|\leq \tfrac{1}{3}p\}$ for some function $H:[-\tfrac{1}{3}p,\tfrac{1}{3}p]\to\R$
such that $H(0)<\tfrac{4}{25}p$ and $\|H'\|_{C^0}<\tfrac{1}{2}$.
Furthermore, $[z,z']_Z=\Psi_x(s,G(s))$ for some $|s|\leq \tfrac{1}{3}p$.

\medskip
\noindent
{\sc Claim 2:} If $D''$ is the connected components of $\widehat\Lambda$ such that $Z''\subset D''$,
then
$$
\mathfrak q_{D''}[V^{s/u}(z',Z')\cap \Psi_y(R[\tfrac{1}{2}q])]\subset V^{s/u}(z'',Z'').
$$

Once we prove these claims, the proposition follows:
Claim 2 implies that $\mathfrak q_D[V^u(z',Z')\cap \Psi_y(R[\tfrac{1}{2}q])]
\subset \mathfrak q_D[V^u(z'',Z'')]$ and so by Claim 1
\begin{align*}
&\ \{[z,z']_Z\}=V^s(z,Z)\cap \mathfrak q_D[V^u(z',Z')\cap \Psi_y(R[\tfrac{1}{2}q])]\\
&\subset V^s(z,Z)\cap \mathfrak q_D[V^u(z'',Z'')]=\{[z,z'']_Z\}.
\end{align*}

\begin{proof}[Proof of Claim $1$]
With the estimates obtained in the beginning of the proof of Proposition \ref{Prop-overlapping-charts},
we just need to proceed as in the proof of \cite[Lemma 10.8]{Sarig-JAMS}.
We will include the calculations for completeness. By the proof of Proposition \ref{Prop-overlapping-charts},
$\Upsilon:=\Psi_x^{-1}\circ \mathfrak q_D\circ \Psi_y={\rm Id}+\Delta$ where:
\begin{enumerate}[$\circ$]
\item $\|d\Delta\|_{C^0}\leq e^{8\rho}-1+O(\ve^{1/3})=O(\ve^{1/3})+O(\rho)$.
\item $\|\Delta(0)\|\leq \tfrac{3}{25}\left[1+O(\ve^{1/3})+O(\rho)\right]p$.
\end{enumerate} 
In particular, $\|\Delta\|_{C^0}\leq \tfrac{3}{25}\left[1+O(\ve^{1/3})+O(\rho)\right]p$.
Write $\Delta=(\Delta_1,\Delta_2)$, and let $F$ be the representing function of $V^u(z',Z')$, i.e.
$V^u(z',Z')=\Psi_y\{(F(t),t):|t|\leq q^u\}$. Hence
$V^u(z',Z')\cap \Psi_y(R[\tfrac{1}{2}q])=\Psi_y\{(F(t),t):|t|\leq \tfrac{1}{2}q\}$, and since
$\mathfrak q_D\circ \Psi_y=\Psi_x\circ\Upsilon$ we have
\begin{align*}
&\ \mathfrak q_D[V^u(z',Z')\cap \Psi_y(R[\tfrac{1}{2}q])]=
(\Psi_x\circ\Upsilon)\{(F(t),t):|t|\leq \tfrac{1}{2}q\}\\
&=\Psi_x\{(F(t)+\Delta_1(F(t),t),t+\Delta_2(F(t),t)):|t|\leq \tfrac{1}{2}q\}\}.
\end{align*}
We represent the pair inside $\Psi_x$ above as a graph on the second coordinate.
Call $\tau(t):=t+\Delta_2(F(t),t))$. We have:
\begin{enumerate}[$\circ$]
\item $|\tau(0)|=|\Delta_2(F(0),0)|\leq \|\Delta(F(0),0)\|\leq \|\Delta(0)\|+\|d\Delta\|_{C^0}|F(0)|\leq
\tfrac{3}{25}[1+O(\ve^{1/3})+O(\rho)]p+[O(\ve^{1/3})+O(\rho)]10^{-3}q\leq
\tfrac{3}{25}[1+O(\ve^{1/3})+O(\rho)]p$.
\item $|\tau'(t)|=1\pm\|d\Delta\|_{C^0}(1+\|F'\|_{C^0})=1+[O(\ve^{1/3})+O(\rho)](1+\ve)=1+O(\ve^{1/3})+O(\rho)$
for every $|t|\leq\tfrac{1}{2}q$.
\end{enumerate}
In particular,
\begin{align*}
&\ \tau(\tfrac{1}{2}q)\geq \tfrac{1}{2}q-|\Delta_2(F(0),0)|\geq \tfrac{1}{2}q-\tfrac{3}{25}[1+O(\ve^{1/3})+O(\rho)]p\\
&\geq \left( \tfrac{1}{2}e^{-[O(\ve^{1/3})+O(\rho)]}-\tfrac{3}{25}[1+O(\ve^{1/3})+O(\rho)]\right)p>\tfrac{1}{3}p,
\end{align*}
for $\rho,\ve>0$ small, since $\tfrac{1}{2}-\tfrac{3}{25}>\tfrac{1}{3}$.
Therefore, the image of $\tau:[-\tfrac{1}{2}q,\tfrac{1}{2}q]\to \R$ contains $[-\tfrac{1}{3}p,\tfrac{1}{3}p]$.

Now, we write the first coordinate $F(t)+\Delta_1(F(t),t)$ as a function of $\tau$.
Start noting that, since the derivative of $\tau$ is positive, it has an inverse $\theta:\tau[-\tfrac{1}{2}q,\tfrac{1}{2}q]\to 
[-\tfrac{1}{2}q,\tfrac{1}{2}q]$ such that $|\theta'(\tau(t))|=|\tau'(t)|^{-1}=1+O(\ve^{1/3})+O(\rho)$
for every $\tau(t)\in \tau[-\tfrac{1}{2}q,\tfrac{1}{2}q]$. In particular,
$$
|\theta(0)|=|\theta(0)-\theta(\tau(0))|\leq \|\theta'\|_{C^0}|\tau(0)|\leq \tfrac{3}{25}[1+O(\ve^{1/3})+O(\rho)]p<\tfrac{1}{5}p.
$$
Defining $H:[-\tfrac{1}{3}p,\tfrac{1}{3}p]\to\R$ by
$$
H(\tau)=F(t)+\Delta_1(F(t),t)=F(\theta(\tau))+\Delta_1(F(\theta(\tau)),\theta(\tau)),
$$
we have:
\begin{enumerate}[$\circ$]
\item $|H(0)|\leq |F(\theta(0))|+|\Delta_1(F(\theta(0),\theta(0))|\leq |F(0)|+\|F'\|_{C^0}|\theta(0)|+\|\Delta\|_{C^0}\leq
10^{-3}q+\ve \tfrac{1}{5}p+\tfrac{3}{25}\left[1+O(\ve^{1/3})+O(\rho)\right]p<\tfrac{4}{25}p$.
\item $\|H'\|_{C^0}\leq \|F'\|_{C^0}\|\theta'\|_{C^0}+\|d\Delta\|_{C^0}(1+\|F'\|_{C^0})\|\theta'\|_{C^0}\leq 
2\ve+2[O(\ve^{1/3})+O(\rho)][1+\ve]=O(\ve^{1/3})+O(\rho)$
which is smaller than $\tfrac{1}{2}$ for $\rho,\ve>0$ small.
\end{enumerate}
This proves the first part of Claim 1. For the second part, note that
$|H(\tau)|\leq |H(0)|+\|H'\|_{C^0}|\tau|\leq \tfrac{4}{25}p+\tfrac{1}{2}\cdot\tfrac{1}{3}p<\tfrac{1}{3}p$, thus
$H:[-\tfrac{1}{3}p,\tfrac{1}{3}p]\to[-\tfrac{1}{3}p,\tfrac{1}{3}p]$ is a contraction.
We have $[z,z']_Z=\Psi_x(t,G(t))$,
where $t$ is the unique $t\in[-p^s,p^s]$ such that $(t,G(t))=(H(\tau),\tau)$. Necessarily 
$H(G(t))=t$, i.e. $t$ is a fixed point of $H\circ G$. Using the admissibility of $G$ and the above estimates,
the restriction of $H\circ G$ to $[-\tfrac{1}{3}p,\tfrac{1}{3}p]$ is a contraction into 
$[-\tfrac{1}{3}p,\tfrac{1}{3}p]$, and so it has a unique fixed point in this interval, proving that
$|t|\leq \tfrac{1}{3}p$.
\end{proof}

\begin{proof}[Proof of Claim $2$]
The proof is very similar to the proof of Proposition \ref{Prop-disjointness}.
Let us prove the inclusion for $V^s$.
Let $V^s=V^s(z'',Z'')=V^s[\un v^+]$ with $\un v^+=\{\Psi_{y_n}^{q^s_n,q^u_n}\}$,
and let $G_n=g_{y_{n-1}}^+\circ\cdots\circ g_{y_0}^+$.
Let  $U^s=\mathfrak q_{D''}[V^s(z',Z')\cap \Psi_y(R[\tfrac{1}{2}q])]$.
By Proposition \ref{Prop-overlapping-charts}(1), $U^s\subset \Psi_{y_0}(R[q^s_0\wedge q^u_0])$.
Now we proceed as in the proof of Proposition \ref{Prop-disjointness} to get that:
\begin{enumerate}[$\circ$]
\item If $n$ is large enough then $G_n(U^s)\subset \Psi_{y_n}(R[Q(y_n)])$: this is exactly 
Claim 2 in the proof of Proposition \ref{Prop-disjointness}.
\item $U^s\subset V^s$: this is exactly Claim 3 in the proof of Proposition \ref{Prop-disjointness}.
\end{enumerate}
Hence Claim 2 is proved.
\end{proof}
The proof of the proposition is complete.
\end{proof}

\bigskip
\small
\bibliographystyle{plain-like-initial}
\bibliography{bibliography}{}

\bigskip
\bigskip

\hspace{-2.8cm}
\begin{tabular}{l l l l l}
\emph{J\'er\^ome Buzzi}
& &\emph{Sylvain Crovisier}
& &\emph{Yuri Lima}\\
Laboratoire de Math\'ematiques
&& Laboratoire de Math\'ematiques
&& Departamento de Matem\'atica\\
 d'Orsay, CNRS - UMR 8628
&&  d'Orsay, CNRS - UMR 8628
&&  Centro de Ci\^encias, Campus do Pici\\
Universit\'e Paris-Saclay
&&  Universit\'e Paris-Saclay
&& Universidade
Federal do Cear\'a (UFC)\\
Orsay 91405, France
&& Orsay 91405, France
&& Fortaleza -- CE, CEP 60455-760, Brasil\\
\tt{jerome.buzzi}
&& \tt{sylvain.crovisier}
&&\tt{yurilima@gmail.com}\\
\; \tt{@universite-paris-saclay.fr}
&& \; \tt{@universite-paris-saclay.fr}
&&
\end{tabular}

\end{document}

%% file: continuous-return.pdf_tex
\begingroup%
  \makeatletter%
  \providecommand\color[2][]{%
    \errmessage{(Inkscape) Color is used for the text in Inkscape, but the package 'color.sty' is not loaded}%
    \renewcommand\color[2][]{}%
  }%
  \providecommand\transparent[1]{%
    \errmessage{(Inkscape) Transparency is used (non-zero) for the text in Inkscape, but the package 'transparent.sty' is not loaded}%
    \renewcommand\transparent[1]{}%
  }%
  \providecommand\rotatebox[2]{#2}%
  \newcommand*\fsize{\dimexpr\f@size pt\relax}%
  \newcommand*\lineheight[1]{\fontsize{\fsize}{#1\fsize}\selectfont}%
  \ifx\svgwidth\undefined%
    \setlength{\unitlength}{551.74918662bp}%
    \ifx\svgscale\undefined%
      \relax%
    \else%
      \setlength{\unitlength}{\unitlength * \real{\svgscale}}%
    \fi%
  \else%
    \setlength{\unitlength}{\svgwidth}%
  \fi%
  \global\let\svgwidth\undefined%
  \global\let\svgscale\undefined%
  \makeatother%
  \begin{picture}(1,0.65822703)%
    \lineheight{1}%
    \setlength\tabcolsep{0pt}%
    \put(0,0){\includegraphics[width=\unitlength,page=1]{continuous-return.pdf}}%
    \put(0.87449119,0.635199){\color[rgb]{0,0,0}\makebox(0,0)[lt]{\lineheight{0}\smash{\begin{tabular}[t]{l}$D_j$\end{tabular}}}}%
    \put(0.85683207,0.17079751){\color[rgb]{0,0,0}\makebox(0,0)[lt]{\lineheight{0}\smash{\begin{tabular}[t]{l}$D_i$\end{tabular}}}}%
    \put(0.86281665,0.40820372){\color[rgb]{0,0,0}\makebox(0,0)[lt]{\lineheight{0}\smash{\begin{tabular}[t]{l}$D_\ell$\end{tabular}}}}%
    \put(0.53652878,0.09314956){\color[rgb]{0,0,0}\makebox(0,0)[lt]{\lineheight{0}\smash{\begin{tabular}[t]{l}$x$\end{tabular}}}}%
    \put(0,0){\includegraphics[width=\unitlength,page=2]{continuous-return.pdf}}%
    \put(0.60842586,0.09467024){\color[rgb]{0,0,0}\makebox(0,0)[lt]{\lineheight{0}\smash{\begin{tabular}[t]{l}$y$\end{tabular}}}}%
    \put(0.66378897,0.30688552){\color[rgb]{0,0,0}\makebox(0,0)[lt]{\lineheight{0}\smash{\begin{tabular}[t]{l}$f(y)$\end{tabular}}}}%
    \put(0.4780642,0.04184792){\color[rgb]{0,0,0}\makebox(0,0)[lt]{\lineheight{0}\smash{\begin{tabular}[t]{l}$B_x$\end{tabular}}}}%
    \put(0.3376791,0.59100809){\color[rgb]{0,0,0}\makebox(0,0)[lt]{\lineheight{0}\smash{\begin{tabular}[t]{l}$g_x^+(x)=f(x)$\end{tabular}}}}%
    \put(0.43323424,0.49069933){\color[rgb]{0,0,0}\makebox(0,0)[lt]{\lineheight{0}\smash{\begin{tabular}[t]{l}$g_x^+(B_x)$\end{tabular}}}}%
    \put(0.75247242,0.55967996){\color[rgb]{0,0,0}\makebox(0,0)[lt]{\lineheight{0}\smash{\begin{tabular}[t]{l}$g_x^+(y)$\end{tabular}}}}%
    \put(0,0){\includegraphics[width=\unitlength,page=3]{continuous-return.pdf}}%
  \end{picture}%
\endgroup%

%% file: markov-property.pdf_tex
\begingroup%
  \makeatletter%
  \providecommand\color[2][]{%
    \errmessage{(Inkscape) Color is used for the text in Inkscape, but the package 'color.sty' is not loaded}%
    \renewcommand\color[2][]{}%
  }%
  \providecommand\transparent[1]{%
    \errmessage{(Inkscape) Transparency is used (non-zero) for the text in Inkscape, but the package 'transparent.sty' is not loaded}%
    \renewcommand\transparent[1]{}%
  }%
  \providecommand\rotatebox[2]{#2}%
  \newcommand*\fsize{\dimexpr\f@size pt\relax}%
  \newcommand*\lineheight[1]{\fontsize{\fsize}{#1\fsize}\selectfont}%
  \ifx\svgwidth\undefined%
    \setlength{\unitlength}{496.70708899bp}%
    \ifx\svgscale\undefined%
      \relax%
    \else%
      \setlength{\unitlength}{\unitlength * \real{\svgscale}}%
    \fi%
  \else%
    \setlength{\unitlength}{\svgwidth}%
  \fi%
  \global\let\svgwidth\undefined%
  \global\let\svgscale\undefined%
  \makeatother%
  \begin{picture}(1,1.15325894)%
    \lineheight{1}%
    \setlength\tabcolsep{0pt}%
    \put(0,0){\includegraphics[width=\unitlength,page=1]{markov-property.pdf}}%
    \put(0.10526533,0.32338521){\color[rgb]{0,0,0}\makebox(0,0)[lt]{\lineheight{0}\smash{\begin{tabular}[t]{l}$s$\end{tabular}}}}%
    \put(0.05380456,0.39873644){\color[rgb]{0,0,0}\makebox(0,0)[lt]{\lineheight{0}\smash{\begin{tabular}[t]{l}$u$\end{tabular}}}}%
    \put(0.00388728,0.43906846){\color[rgb]{0,0,0}\makebox(0,0)[lt]{\lineheight{0}\smash{\begin{tabular}[t]{l}$\vf$\end{tabular}}}}%
    \put(0,0){\includegraphics[width=\unitlength,page=2]{markov-property.pdf}}%
    \put(0.23233178,0.02755254){\color[rgb]{0,0,0}\makebox(0,0)[lt]{\lineheight{0}\smash{\begin{tabular}[t]{l}$\widetilde y=H^k(y)$\end{tabular}}}}%
    \put(0.45608443,0.02692051){\color[rgb]{0,0,0}\makebox(0,0)[lt]{\lineheight{0}\smash{\begin{tabular}[t]{l}$\widetilde x=H^k(x)=\pi[\sigma^{-1}(\underline v)]$\end{tabular}}}}%
    \put(0.48958415,0.1023474){\color[rgb]{0,0,0}\makebox(0,0)[lt]{\lineheight{0}\smash{\begin{tabular}[t]{l}$\widetilde z$\end{tabular}}}}%
    \put(0.08588735,0.02096367){\color[rgb]{0,0,0}\makebox(0,0)[lt]{\lineheight{0}\smash{\begin{tabular}[t]{l}$\widetilde Z$\end{tabular}}}}%
    \put(0,0){\includegraphics[width=\unitlength,page=3]{markov-property.pdf}}%
    \put(0.78938167,0.47017217){\color[rgb]{0,0,0}\makebox(0,0)[lt]{\lineheight{0}\smash{\begin{tabular}[t]{l}$\widetilde Z'$\end{tabular}}}}%
    \put(0,0){\includegraphics[width=\unitlength,page=4]{markov-property.pdf}}%
    \put(0.38992783,0.14462584){\color[rgb]{0,0,0}\makebox(0,0)[lt]{\lineheight{0}\smash{\begin{tabular}[t]{l}$\widetilde t$\end{tabular}}}}%
    \put(0,0){\includegraphics[width=\unitlength,page=5]{markov-property.pdf}}%
    \put(0.49864382,0.39829647){\color[rgb]{0,0,0}\makebox(0,0)[lt]{\lineheight{0}\smash{\begin{tabular}[t]{l}$\vf^{\widetilde s}(\widetilde z)=\pi[\sigma^{-b}(\underline w)]$\end{tabular}}}}%
    \put(0,0){\includegraphics[width=\unitlength,page=6]{markov-property.pdf}}%
    \put(0.37784828,0.40131636){\color[rgb]{0,0,0}\makebox(0,0)[lt]{\lineheight{0}\smash{\begin{tabular}[t]{l}$\vf^r(\widetilde t)$\end{tabular}}}}%
    \put(0,0){\includegraphics[width=\unitlength,page=7]{markov-property.pdf}}%
    \put(0.68549918,0.1062628){\makebox(0,0)[lt]{\lineheight{1.25}\smash{\begin{tabular}[t]{l}$\mathfrak q_{\widetilde D}(\widetilde Z')$\end{tabular}}}}%
    \put(0,0){\includegraphics[width=\unitlength,page=8]{markov-property.pdf}}%
    \put(0.25649089,0.69192827){\color[rgb]{0,0,0}\makebox(0,0)[lt]{\lineheight{0}\smash{\begin{tabular}[t]{l}$H^{N+1}(y)$\end{tabular}}}}%
    \put(0.45910432,0.68827635){\color[rgb]{0,0,0}\makebox(0,0)[lt]{\lineheight{0}\smash{\begin{tabular}[t]{l}$H^{N+1}(x)=\pi(\underline v)$\end{tabular}}}}%
    \put(0.49260404,0.75464358){\color[rgb]{0,0,0}\makebox(0,0)[lt]{\lineheight{0}\smash{\begin{tabular}[t]{l}$z$\end{tabular}}}}%
    \put(0.08588735,0.66420019){\color[rgb]{0,0,0}\makebox(0,0)[lt]{\lineheight{0}\smash{\begin{tabular}[t]{l}$Z$\end{tabular}}}}%
    \put(0,0){\includegraphics[width=\unitlength,page=9]{markov-property.pdf}}%
    \put(0.79542145,1.11340869){\color[rgb]{0,0,0}\makebox(0,0)[lt]{\lineheight{0}\smash{\begin{tabular}[t]{l}$Z'$\end{tabular}}}}%
    \put(0,0){\includegraphics[width=\unitlength,page=10]{markov-property.pdf}}%
    \put(0.17551575,0.8814789){\color[rgb]{0,0,0}\makebox(0,0)[lt]{\lineheight{0}\smash{\begin{tabular}[t]{l}$[z,H^{N+1}(y)]_Z$\end{tabular}}}}%
    \put(0,0){\includegraphics[width=\unitlength,page=11]{markov-property.pdf}}%
    \put(0.46844493,1.04153298){\color[rgb]{0,0,0}\makebox(0,0)[lt]{\lineheight{0}\smash{\begin{tabular}[t]{l}$\vf^s(z)=\pi(\underline w)$\end{tabular}}}}%
    \put(0,0){\includegraphics[width=\unitlength,page=12]{markov-property.pdf}}%
    \put(0.69153896,0.74345954){\makebox(0,0)[lt]{\lineheight{1.25}\smash{\begin{tabular}[t]{l}$\mathfrak q_D(Z')$\end{tabular}}}}%
    \put(0,0){\includegraphics[width=\unitlength,page=13]{markov-property.pdf}}%
  \end{picture}%
\endgroup%

%% file: universal-markov-partitions-v5.bbl
\def\cprime{$'$} \def\cprime{$'$} \def\cprime{$'$} \def\cprime{$'$}
  \def\cprime{$'$}
\begin{thebibliography}{10}

\bibitem{Adler-Weiss-PNAS}
R.~L. Adler and B.~Weiss.
\newblock Entropy, a complete metric invariant for automorphisms of the torus.
\newblock {\em Proc. Nat. Acad. Sci. U.S.A.} \textbf{57} (1967), 1573--1576.

\bibitem{ALP}
E.~Araujo, Y.~Lima, and M.~Poletti.
\newblock Symbolic dynamics for nonuniformly hyperbolic maps with singularities
  in high dimension.
\newblock preprint arXiv:2010.11808, 2020.

\bibitem{Ben-Ovadia-high-dimension}
S.~Ben~Ovadia.
\newblock Symbolic dynamics for non-uniformly hyperbolic diffeomorphisms of
  compact smooth manifolds.
\newblock {\em J. Mod. Dyn.} \textbf{13} (2018), 43--113.

\bibitem{Ben-Ovadia-ETDS}
S.~Ben~Ovadia.
\newblock The set of points with {M}arkovian symbolic dynamics for
  non-uniformly hyperbolic diffeomorphisms.
\newblock {\em Ergodic Theory Dynam. Systems} \textbf{41} (2021), 3244--3269.

\bibitem{Bowen-LNM}
R.~Bowen.
\newblock {\em Equilibrium states and the ergodic theory of {A}nosov
  diffeomorphisms}, volume 470 of {\em Lecture Notes in Mathematics}.
\newblock Springer-Verlag, Berlin, 2008.

\bibitem{Bowen-Walters-Metric}
R.~Bowen and P.~Walters.
\newblock Expansive one-parameter flows.
\newblock {\em J. Diff. Equations} \textbf{12} (1972), 180--193.

\bibitem{Bowen-MP-Axiom-A}
R.~Bowen.
\newblock Markov partitions for {A}xiom {${\rm A}$} diffeomorphisms.
\newblock {\em Amer. J. Math.} \textbf{92} (1970), 725--747.

\bibitem{Bowen-Symbolic-Flows}
R.~Bowen.
\newblock Symbolic dynamics for hyperbolic flows.
\newblock {\em Amer. J. Math.} \textbf{95} (1973), 429--460.

\bibitem{Bowen-mme-flow}
R.~Bowen.
\newblock Maximizing entropy for a hyperbolic flow.
\newblock {\em Math. Systems Theory} \textbf{7} (1974), 300--303.

\bibitem{Bowen-Regional-Conference}
R.~Bowen.
\newblock {\em On {A}xiom {A} diffeomorphisms}.
\newblock American Mathematical Society, Providence, R.I., 1978.
\newblock Regional Conference Series in Mathematics, No. 35.

\bibitem{Boyle-Buzzi}
M.~Boyle and J.~Buzzi.
\newblock The almost {B}orel structure of surface diffeomorphisms, {M}arkov
  shifts and their factors.
\newblock {\em J. Eur. Math. Soc.} \textbf{19} (2017), 2739--2782.

\bibitem{BCFT}
K.~Burns, V.~Climenhaga, T.~Fisher, and D.~J. Thompson.
\newblock Unique equilibrium states for geodesic flows in nonpositive
  curvature.
\newblock {\em Geom. Funct. Anal.} \textbf{28} (2018), 1209--1259.

\bibitem{Burns-Masur-Wilkinson}
K.~Burns, H.~Masur, and A.~Wilkinson.
\newblock The {W}eil-{P}etersson geodesic flow is ergodic.
\newblock {\em Ann. of Math.} \textbf{175} (2012), 835--908.

\bibitem{Buzzi-IJM}
J.~Buzzi.
\newblock Intrinsic ergodicity of smooth interval maps.
\newblock {\em Israel J. Math.} \textbf{100} (1997), 125--161.

\bibitem{Buzzi-Invent}
J.~Buzzi.
\newblock Subshifts of quasi-finite type.
\newblock {\em Invent. Math.} \textbf{159} (2005), 369--406.

\bibitem{Buzzi-JMD}
J.~Buzzi.
\newblock The degree of {B}owen factors and injective codings of
  diffeomorphisms.
\newblock {\em J. Mod. Dyn.} \textbf{16} (2020), 1--36.

\bibitem{BCS-continuity}
J.~Buzzi, S.~Crovisier, and O.~Sarig.
\newblock Continuity properties of {L}yapunov exponents for surface
  diffeomorphisms.
\newblock {\em Invent. Math.} \textbf{230} (2022), 767--849.

\bibitem{BCS-MME}
J.~Buzzi, S.~Crovisier, and O.~Sarig.
\newblock Finiteness of measures maximizing the entropy for surface
  diffeomorphisms.
\newblock {\em Ann. of Math.} \textbf{195} (2022), 421--508.

\bibitem{CKP-20}
D.~Chen, L.-Y. Kao, and K.~Park.
\newblock Unique equilibrium states for geodesic flows over surfaces without
  focal points.
\newblock {\em Nonlinearity} \textbf{33} (2020), 1118--1155.

\bibitem{CKW}
V.~Climenhaga, G.~Knieper, and K.~War.
\newblock Uniqueness of the measure of maximal entropy for geodesic flows on
  certain manifolds without conjugate points.
\newblock {\em Adv. Math.} \textbf{376} (2021), Paper No. 107452, 44.

\bibitem{climenhaga-thompson}
V.~Climenhaga and D.~Thompson.
\newblock Unique equilibrium states for flows and homeomorphisms with
  non-uniform structure.
\newblock {\em Adv. Math.} \textbf{303} (2016), 745--799.

\bibitem{Dolgopyat-Annals}
D.~Dolgopyat.
\newblock On decay of correlations in {A}nosov flows.
\newblock {\em Ann. of Math. (2)} \textbf{147} (1998), 357--390.

\bibitem{Gelfert-Ruggiero}
K.~Gelfert and R.~O. Ruggiero.
\newblock Geodesic flows modelled by expansive flows.
\newblock {\em Proc. Edinb. Math. Soc. (2)} \textbf{62} (2019), 61--95.

\bibitem{Hofbauer-PMM}
F.~Hofbauer.
\newblock On intrinsic ergodicity of piecewise monotonic transformations with
  positive entropy.
\newblock {\em Israel J. Math.} \textbf{34} (1979), 213--237.

\bibitem{KatokIHES}
A.~Katok.
\newblock Lyapunov exponents, entropy and periodic orbits for diffeomorphisms.
\newblock {\em Inst. Hautes \'Etudes Sci. Publ. Math.} \textbf{51} (1980),
  137--173.

\bibitem{Knieper-Rank-One-Entropy}
G.~Knieper.
\newblock The uniqueness of the measure of maximal entropy for geodesic flows
  on rank {$1$} manifolds.
\newblock {\em Ann. of Math. (2)} \textbf{148} (1998), 291--314.

\bibitem{Kunzinger-flow}
M.~Kunzinger, H.~Schichl, R.~Steinbauer, and J.~A. Vickers.
\newblock Global gronwall estimates for integral curves on riemannian
  manifolds.
\newblock {\em Rev. Mat. Complut.} \textbf{19} (2006), 133--137.

\bibitem{Ledrappier-Lima-Sarig}
F.~Ledrappier, Y.~Lima, and O.~Sarig.
\newblock Ergodic properties of equilibrium measures for smooth three
  dimensional flows.
\newblock {\em Comment. Math. Helv.} \textbf{91} (2016), 65--106.

\bibitem{Lima-Matheus}
Y.~Lima and C.~Matheus.
\newblock Symbolic dynamics for non-uniformly hyperbolic surface maps with
  discontinuities.
\newblock {\em Ann. Sci. \'Ec. Norm. Sup\'er.} \textbf{51} (2018), 1--38.

\bibitem{Lima-Sarig}
Y.~Lima and O.~Sarig.
\newblock Symbolic dynamics for three dimensional flows with positive entropy.
\newblock {\em J. Eur. Math. Soc.} \textbf{21} (2019), 199--256.

\bibitem{Lima-AIHP}
Y.~Lima.
\newblock Symbolic dynamics for one dimensional maps with nonuniform expansion.
\newblock {\em Ann. Inst. H. Poincar\'{e} C Anal. Non Lin\'{e}aire} \textbf{37}
  (2020), 727--755.

\bibitem{Newhouse-Lectures-dynamical-systems}
S.~E. Newhouse.
\newblock Lectures on dynamical systems.
\newblock In {\em Dynamical systems ({C}.{I}.{M}.{E}. {S}ummer {S}chool,
  {B}ressanone, 1978)}, volume~8 of {\em Progr. Math.}, pages 1--114.
  Birkh\"{a}user, Boston, Mass., 1980.

\bibitem{pacifico-yang-yang}
M.~J. Pacifico, F.~Yang, and J.~Yang.
\newblock Equilibrium states for the classical lorenz attractor and
  sectional-hyperbolic attractors in higher dimensions.
\newblock arXiv:2209.10784.

\bibitem{Pesin-Izvestia-1976}
J.~B. Pesin.
\newblock Families of invariant manifolds that correspond to nonzero
  characteristic exponents.
\newblock {\em Izv. Akad. Nauk SSSR Ser. Mat.} \textbf{40} (1976), 1332--1379,
  1440.

\bibitem{Ratner-MP-n-dimensions}
M.~Ratner.
\newblock Markov partitions for {A}nosov flows on {$n$}-dimensional manifolds.
\newblock {\em Israel J. Math.} \textbf{15} (1973), 92--114.

\bibitem{Ratner-MP-three-dimensions}
M.~E. Ratner.
\newblock Markov decomposition for an {U}-flow on a three-dimensional manifold.
\newblock {\em Mat. Zametki} \textbf{6} (1969), 693--704.

\bibitem{Sarig-JMD}
O.~Sarig.
\newblock Bernoulli equilibrium states for surface diffeomorphisms.
\newblock {\em J. Mod. Dyn.} \textbf{5} (2011), 593--608.

\bibitem{Sarig-JAMS}
O.~Sarig.
\newblock Symbolic dynamics for surface diffeomorphisms with positive entropy.
\newblock {\em J. Amer. Math. Soc.} \textbf{26} (2013), 341--426.

\bibitem{Sinai-Construction-of-MP}
J.~G. Sina{\u\i}.
\newblock Construction of {M}arkov partitionings.
\newblock {\em Funkcional. Anal. i Prilo{\v z}en.} \textbf{3} (1968), 70--80.

\bibitem{Sinai-MP-U-diffeomorphisms}
J.~G. Sina{\u\i}.
\newblock Markov partitions and {U}-diffeomorphisms.
\newblock {\em Funkcional. Anal. i Prilo{\v z}en} \textbf{2} (1968), 64--89.

\end{thebibliography}
